\documentclass[12pt,openany]{amsbook} 
\usepackage[left=2.8cm,top=2cm,bottom=2.5cm,right=2.8cm]{geometry}
\usepackage{amsmath, amsthm}
\usepackage{mathtools,thmtools}
\usepackage{graphicx,xspace}
\usepackage{enumitem}
\setlist{leftmargin=5.5mm}
\usepackage[usenames,dvipsnames]{xcolor}
\usepackage[T1]{fontenc}
\usepackage{bm}
\usepackage[export]{adjustbox}
\usepackage{setspace}
\usepackage{hyperref}
\usepackage[noabbrev]{cleveref}
\usepackage{subfig}
\usepackage{cases}

\usepackage[many]{tcolorbox}
\tcbuselibrary{skins,breakable,theorems}
\tcolorboxenvironment{proof}{blanker, before skip=\topsep, after skip=\topsep, borderline west={0.4pt}{0.4pt}{black}, breakable, left=12pt, right=12pt}

\definecolor{mythm}{RGB}{136,216,216}% teal
\definecolor{myprop}{RGB}{216,136,216}% violet
\definecolor{mycor}{RGB}{216,216,136}% beige

\newtcbtheorem[number within=chapter]{theorem}{Theorem}%
{label type=theorem, colframe=white, before skip=\topsep, after skip=\topsep, detach title, before upper={\tcbtitle\quad}, coltitle=black, breakable, colback=mythm,fonttitle=\bfseries, fontupper=\sffamily}{thm}

\newtcbtheorem[number within=chapter]{proposition}{Proposition}%
{label type=proposition, colframe=white, before skip=\topsep, after skip=\topsep, detach title, before upper={\tcbtitle\quad}, coltitle=black, breakable, colback=myprop,fonttitle=\bfseries, fontupper=\sffamily}{thm}

\newtcbtheorem[number within=chapter]{corollary}{Corollary}%
{label type=corollary, colframe=white, before skip=\topsep, after skip=\topsep, detach title, before upper={\tcbtitle\quad}, coltitle=black, breakable, colback=mycor,fonttitle=\bfseries, fontupper=\sffamily}{thm}

\usepackage[charter]{mathdesign}

\newtheorem{lemma}{Lemma}[section]

\newtheoremstyle{mythm}% 〈name〉
{}% 〈Space above〉1
{}% 〈Space below 〉1
{\sffamily}
{}% 〈Indent amount〉2
{\sffamily\bfseries}% 〈Theorem head font〉
{.}% 〈Punctuation after theorem head 〉
{.5em}% 〈Space after theorem head 〉3
{}% 〈Theorem head spec (can be left empty, meaning ‘normal’ )
\theoremstyle{mythm}
\newtheorem{definition}{Definition}[section]
\newtheorem{definition-lemma}{Definition-Lemma}[section]

\newcommand{\splus}{\!+\!}
\newcommand{\sminus}{\!-\!}

\newcommand{\m}{\mathfrak{m}}
\newcommand{\trisp}{\mathfrak{t}}
\newcommand{\graph}{G}
\newcommand{\simp}{s}
\newcommand{\cM}{\mathcal{M}}
\newcommand{\cG}{\mathcal{G}}
\newcommand{\cT}{\mathcal{T}}
\newcommand{\bb}{B}
\newcommand{\quart}{Q}

\newcommand{\CBB}{\mathfrak{b}}
\newcommand{\colset}{\mathfrak{c}}
\newcommand{\id}{\mathrm{id}}
\newcommand{\tree}{t}

\newcommand{\ptimes}{\vec{p}}
\newcommand{\qtimes}{\vec{q}}
\DeclareMathOperator{\hook}{hook}
\DeclareMathOperator{\Pf}{Pf}
\DeclareMathOperator{\sgn}{sgn}
\DeclareMathOperator{\Pl}{Pl}
\DeclareMathOperator{\SIF}{SIF}

\DeclareMathOperator*{\Res}{Res}

\usepackage{slashed}
\declareslashed{}{\Big/}{.1}{0}{\sum}
\newcommand{\hsum}{\mathop{\slashed{\sum}}}

\theoremstyle{remark}
\newtheorem{remark}{Remark}

\definecolor{SectionColor}{RGB}%{41,134,204}% Blue
{211,68,133}% Pink
\makeatletter
\renewcommand\subsection{\@startsection{subsection}{2}%
\normalparindent{.3\linespacing\@plus.6\linespacing}{-.5em}{\normalfont\bfseries\color{SectionColor}}}
\makeatother

\title[]{Some structural and enumerative aspects of discrete surfaces and PL-manifolds}
\author[V.~Bonzom]{Valentin Bonzom}
\address{Universit\'e Sorbonne Paris Nord, LIPN, CNRS, UMR 7030, F-93430 Villetaneuse, France}
\email{bonzom@lipn.univ-paris13.fr}

\begin{document}

\begin{abstract}
	This manuscript recounts some of the author's contributions to algebraic and enumerative combinatorics. We have focused on two types of generalizations of bipartite maps, which are bipartite graphs embedded on surfaces. Maps are known to appear in many areas of theoretical physics and discrete mathematics, but one key interest for fundamental computer science is how multi-facetted they are in the sense that multiple encodings exist which are not fully interchangeable, like Tutte's equations, the topological recursion, the KP hierarchy and numerous bijections which made the field so rich. One generalization we considered is weighted Hurwitz numbers, including constellations, monotone Hurwitz numbers and the unoriented versions of Chapuy-Do\l\k{e}ga. We have investigated whether some universal structures of maps lift to weighted Hurwitz numbers, such that the topological recursion (it does, for orientable, double, weighted Hurwitz numbers), and the passage from the KP to the BKP hierarchy for some unoriented weighted Hurwitz numbers (like monotone ones). The other generalization of bipartite maps we considered is colored triangulations in dimensions three and higher. They provide a nice meeting ground for topology and combinatorics, where universality classes above dimension 2 can be investigated. In particular, we found that the gluings of 3-balls which maximize the number of edges at fixed number of tetrahedra are in bijection with trees. In even dimensions however, we found that more universality classes can be reached depending on the choice of building blocks. Beyond those universality classes we also wondered whether some of the universal structures featured by maps lift to higher dimensions. In particular, we proved the blobbed topological recursion with only mild assumptions.
\end{abstract}
\maketitle
\tableofcontents

\chapter*{Remerciements}
Je remercie vivement toutes les personnes avec qui j'ai eu plaisir à travailler, échanger et collaborer. En premier lieu, Razvan Gurau et Vincent Rivasseau qui m'ont emmené à l'insu de mon plein gré vers des thématiques de plus en plus combinatoires et m'ont scientifiquement beaucoup appris. Ç'aura été un immense plaisir de travailler avec eux, et j'espère que nous aurons la chance de retravailler ensemble au cours de nos déambulations scientifiques. 

Merci à Frédérique Bassino et Adrian Tanasa pour l'intégration au sein de l'équipe CALIN du LIPN, et à Olivier Bodini pour l'accueil dans le bureau A111. J'ai largement bénéficié de leurs conseils pendant toutes ces années. Olivier et Adrian m'ont de plus toujours fait confiance et impliqué scientifiquement dans leurs projets. Je ne peux pas évoquer le bureau A111 sans mentionner Nabil Mustafa et son entrain quotidien qui ont fait de ce bureau un havre de confort dionysien dédié à la science, aux plantes et aux gourmandises (merci pour les amandes !). Au-delà du LIPN, j'ai vivement apprécié l'accueil de la communauté de combinatoire énumérative et algébrique française et du GT ALEA, qui en outre ne ménage pas ses efforts pour faire vivre de si belles thématiques.

Un grand merci à mes collaborateurs de ces dernières années, avec qui j'ai partagé l'intimité de la recherche, Guillaume Chapuy, Stéphane Dartois, Maciej Do\l\k{e}ga, Luca Lionni, Victor Nador, Adrian Tanasa. Un remerciement spécial à Guillaume pour son accueil à l'IRIF en 2019 et auprès de qui j'ai énormément appris (et c'est toujours le cas !) dans des domaines non-bornés. Un autre remerciement spécial pour Marie Albenque avec qui j'ai eu plaisir à discuter de tout, et qui en partageant de nombreux détails de son article à paraître \cite{AlbenqueBouttier2022} m'a immensément aidé dans la complétion un brin précipitée de \cite{BCCGF22}. Les tranches de cartes sont magnifiques, et c'était très excitant de partager cet intérêt (sans parler du vélo, évidemment).

Je remercie également mes collègues de GEA de l'IUT de Saint-Denis, dont l'engagement pédagogique et social à offrir une formation à la fois accessible et pointue m'impressionne chaque jour, et l'équipe dirigeante actuelle d'Hélène Breda et Maria Grat qui m'épaule quotidiennement.

Merci à Marie Albenque, Guillaume Chapuy, Sylvie Cortéel et Karine Lamiaux-Charet  pour leur amitié, soutien, précieux conseils et encouragements à rédiger. 

Enfin, merci taille mammouth à ma famille, en particulier Sara qui m'a toujours soutenu quoi qu'il arrive et avec qui j'ai toujours plaisir et bonheur à partager ma vie, Léon et Clarisse qui chaque jour me font m'évader.

%%%%%%%%%% KP chapter
% KP hierarchy
% \bar{GL} and Thm about Schur exp with det minors
% Evolution eq. and => constraints
% Say that there is a parallel story giving a geometric picture, and alternative proofs of the KP hierarchy

%%%%%%% A enlever
% Équation d'évolution ?
% GL_infty ?
% If necessary, half-infinite wedge

%%%%%%%%%%%%%%%%%%%%%%%%%%%%%%%%
\chapter*{Introduction}
\section*{Motivation}

This manuscript describes some of the author's work of the past 10 years and it is not meant to represent a state of the art in any of the topics discussed.

\subsection*{A (shamefully short) history of maps} Combinatorial maps, or simply \emph{maps}, are ubiquitous in modern science. Maps are graphs with an additional structure, an embedding on a surface without crossings, and up to deformations \cite{MoharThomassen2001}. For instance, they encode the adjacency relations of polygon meshes in geometric modeling, and are key objects in the theory of forbidden minors.

W.~T.~Tutte, in the 60s, while interested in the four-color problem, found some remarkable enumerative formulas \cite{Tutte1962a, Tutte1962b, Tutte1962c}, e.g. the number of rooted planar (i.e. embedded on the sphere) maps is \cite{Tutte1963}, 
\begin{equation*}
	\frac{2\times 3^n\times (2n)!}{n!(n+2)!}
\end{equation*}
Tutte derived them by writing functional equations on the generating series of maps, which he found by performing some map decomposition, well, {\it à la} Tutte as it is now said. His work found echoes in French computer science. His decompositions of maps was of interest to computer scientists who worked on decompositions based on grammars. His enumerative formulas generated a strong interest and were eventually understood through bijections with decorated trees, like the Cori-Vauquelin-Schaeffer bijection. Those bijections are obviously crucial for efficient coding and random sampling of maps and can be used to classify maps according to their genus \cite{ChapuyMarcusSchaeffer2009}. The functional equations {\it à la} Tutte have also not only survived but also flourished, see for instance \cite{BousquetJehanne2006} for a theorem ensuring algebraicity of the generating series in some of those equations, and they also appear in the enumeration of quarter-plane walks.

\subsection*{Dynamical triangulations and quantum gravity} In the late 70s, physicists realized that the Feynman expansions of matrix integrals give rise to generating series of maps \cite{BrezinItzyksonParisiZuber1978}. At first seen as 't Hooft's large $N$ approximation ($N$ being the size of the matrix) in quantum field theory, it was also realized that another interpretation holds, as a discrete approach to 2-dimensional quantum gravity (see \cite{DiFrancescoGinspargZinnJustin} and references therein). Indeed, general relativity is a theory where the spacetime metric is dynamical, so at the quantum level, it seems necessary to be able to define and perform integrals over the space of equivalence classes of metrics up to diffeomorphisms. A way to achieve this is to discretize, say {\it à la} Regge \cite{Regge1961}, so that everything is (hopefully) well-defined, and then take a continuum limit.

The discretization proposed in \emph{Dynamical Triangulations} is to consider equilateral triangulations (in any dimension), assuming that one triangulation is like one equivalence class of metrics up to diffeomorphisms, and thus to sum over triangulations. The probability density of a metric in Einstein's theory is given by the exponential of the Einstein-Hilbert action. On equilateral triangulations, the latter gives rise to a \emph{volume} term, i.e. counting the number of top-dimensional simplices, say $d$, and a \emph{curvature} term measuring the deficit angle around every $(d-2)$-dimensional simplices. The deficit angle at $\sigma_{d-2}$ is 
\begin{equation*}
	2\pi- (\text{number of $d$-simplices meeting at $\sigma_{d-2}$})\times \text{value of dihedral angle in regular $d$-simplex}
\end{equation*}
So on a $d$-manifold $M$, the ``partition function'' would be
\begin{equation*}
	Z_M(x,y) = \sum_{\text{Triangulations $\trisp$ of $M$}} x^{\Delta_d(\trisp)} y^{\Delta_{d-2}(\trisp)}
\end{equation*}
where $\Delta_k(\trisp)$ is the number of $k$-dimensional simplices of $\trisp$.

This approach has been extremely successful in two dimensions. In particular, the Cori-Vauquelin-Schaeffer-like bijections give access to distances and geodesics on the maps and are therefore interesting for quantum gravity purposes. The relationship to Liouville quantum gravity has also been a major topic \cite{GwynneMillerSheffield}, as well as the limiting objects emerging out of maps at large scale in the probabilistic setting \cite{Miermont2013, LeGall2013}.

In higher dimensions however, while several interesting families of triangulations have been introduced \cite{Ambjorn3D,Durhuus1995,BenedettiZiegler,BenedettiPavelka2021}, it has been difficult to extract any enumerative result analytically and for a long time people resorted to numerical simulations instead (see Loll's review \cite{LollReview2019}). Some basic questions still remain unanswered, e.g. it is not known if the number of triangulations of the 3-sphere is exponentially bounded with respect to the number of tetrahedra, meaning that it is not even known if the sum above, over triangulations of the 3-sphere, converges at all \cite{Rivasseau2013,ChapuyPerarnau2020}.

%Matrix integrals and map combinatorics have since then percolated in many areas like string theory, foldings and gluings of polymers and DNA and many other applications. Matrix integrals generalize to tensor integrals, which as it turns out, generate higher-dimensional PL-manifolds \cite{Uncoloring} and this has been the original context of my work. Tensor integrals have also been related to holographic duality for black holes \cite{KT2017} and I have taken an active participation in providing all the relevant combinatorial help possible, in particular in terms of the scheme classification.

\subsection*{The objects studied in this manuscript} We study two generalizations of bipartite maps: $m$-constellations and $d$-dimensional colored triangulations. A bipartite map is a map whose underlying graph is bipartite, with, say, black and white vertices. A bipartite map with $n$ labeled edges can be encoded as a pair of permutations $\sigma_\circ, \sigma_\bullet\in\mathfrak{S}_n$, by recording the labels of the edges meeting counter-clockwise at every white (respectively black) vertex into a cycle of $\sigma_\circ$ (respectively $\sigma_\bullet$).

A map not only has edges and vertices but also \emph{faces} which are the connected components of the graph complement and are by definition homeomorphic to discs\footnote{Here there is a slight confusion of terminology since we use ``faces'' as is traditional in the literature on combinatorial maps, but ``faces'' has a different meaning in the literature on higher-dimensional triangulations. We will avoid the latter meaning.}. Since the permutations record which edges are next to one another (i.e. share a corner), they also encode the faces, as the cycles of the permutation $\phi$ such that
\begin{equation*}
	\phi = \sigma_\circ\sigma_\bullet.
\end{equation*}
This is thus a simple example of \emph{ordered factorization in the symmetric group} $\mathfrak{S}_n$. The more general problem of ordered factorizations is to enumerate the tuples $(\sigma_0, \dotsc, \sigma_{m-1}, \sigma_m)$ such that 
\begin{equation*}
	\sigma_0 \dotsm \sigma_m=\mathbb{1}
\end{equation*}
and $\sigma_i$ has a fixed conjugacy class $\lambda^{(i)}$. We recall that conjugacy classes of the symmetric group are the cycle types of permutations, i.e. how many cycles of length 1, 2, and so on, and are thus labeled by partitions, by listing the length of the cycles in decreasing order. Counting ordered factorization is a purely combinatorial problem, but one of its origins is as counting ramified coverings of the sphere by surfaces of genus $g$ (given by the Riemann-Hurwitz formula). Moreover, the factorization problem can be given a map-like interpretation known as \emph{$m$-constellations}, introduced by Lando and Zvonkin \cite{LandoZvonkin2004}. The key question regarding constellations is whether properties of maps extend \emph{universally} to constellations. Or to an even larger family of objects generalizing constellations and counted by the so-called weighted Hurwitz numbers of \cite{Guay-PaquetHarnad2017}. Their generating series can be given the remarkable form
\begin{equation*}
	\tau_G(\ptimes, \qtimes) = \sum_\lambda s_\lambda(\ptimes) s_\lambda(\qtimes) \prod_{\Box\in\lambda} G(c(\Box)).
\end{equation*}
Here $\ptimes = (p_1, p_2, \dotsc)$ and $\qtimes = (q_1, q_2, \dotsc)$ are two sets of indeterminates and $s_\lambda(\ptimes)$ is the Schur function at partition $\lambda$, where the $p_i$s are seen as power-sums. Moreover, $G$ is the \emph{weight} function and $c(\Box)$ the cell content (all definitions in due time). For constellations, $G$ is polynomial, but it is also natural to allow for rational $G$s to include the so-called monotone Hurwitz numbers, and an exponential factor to include simple Hurwitz numbers. This series $\tau_G$ is one of the main objects studied here.

Another generalization of bipartite maps are \emph{colored triangulations}. They are built out of $d$-simplices whose boundary $(d-1)$-simplices are colored in $0..d$ and those colors can be used to define an unambiguous attaching map between $d$-simplices. It is known that those colored triangulations produce PL-pseudomanifolds and that every PL-manifold admits a colored triangulation. 

Due to the specificity of the attaching map for colored simplices, the triangulation is determined by the data of which $d$-simplex is glued to which along which colors. In the dual picture, it means that they can be represented as bipartite graphs with properly-colored edges and this is the reason why those objects were introduced \cite{Ferri1982, Ferri1986, LinsMandel1985}. %Indeed, it is enough to equip the edges of the 1-skeleton of the dual with the colors of the corresponding facets to completely determine the triangulation. 

Those objects have thus been introduced in topology some time ago. Almost the same objects have been introduced by Stanley as balanced complexes \cite{Stanley1983} (in the simplicial setting), with different questions in mind, this time of combinatorial nature, but not enumerative. However, none of the celebrated enumerative result about maps had been extended to this framework.

Here we stress the core differences between the analysis required for maps and constellations one the one hand and for those higher-dimensional triangulations on the other hand. For maps, there is a beautiful matching between
\begin{itemize}
	\item the topological classification of surfaces. The topology is characterized by a single non-negative integer, the genus $g$,
	\item and a purely combinatorial classification based on the difference between the number of vertices $V$ and the number of faces $F$.
\end{itemize}
This matching is the famous Euler's formula, which for instance for $p$-angulations (faces are $p$-gons) yields 
\begin{equation} \label{EulerIntro}
	V-\frac{p-2}{2}F=2-2g.
\end{equation}
In higher dimensions however this coincidence is lost and it is necessary to choose whether one works at fixed topology, or one relies on a combinatorial criterion instead (or something else). The program started by R.~Gurau and V.~Rivasseau was the second option, based on Gurau's discovery of a relation similar to \eqref{EulerIntro} for higher-dimensional colored triangulations. Gurau's relation is a bound on the number of $(d-2)$-dimensional simplices with respect to the number of $d$-dimensional simplices. In two dimensions, this bound is attained by planar maps, and the natural question is thus what the equivalent objects are in higher dimensions.

%On labeled $d$-simplices, there is a coding by permutations $\sigma_0, \dotsc, \sigma_d$ where the number of $(d-2)$-dimensional simplices are given by the sum of the number of cycles of $\sigma_i \sigma_j^{-1}$ for $0\leq i<j\leq d$.
\subsection*{Universality I: What is ``planar'' in higher dimensions?} We will discuss two types of universality in this manuscript. One is the usual notion of universality in combinatorics: although different models are, well, different at the ``microscopic'' level, they may exhibit similar large scale behaviors. A typical manifestation is the ``entropy'' exponent. Let $\mathcal{A}$ be a combinatorial class and $A_n$ the number of objects of size $n$ and assume an asymptotics of the form
\begin{equation*}
	A_n \underset{n\to\infty}{\sim} a \rho^{-n} n^\gamma.
\end{equation*}
The constant $a$ and the growth rate $\rho^{-1}$ are model-dependent but the exponent $\gamma$ is typical of a universality class \cite{FlajoletSedgewick2009book}. For instance, $\gamma=-3/2$ is the universality class of trees, in which we find rooted plane trees, binary trees and trees of any fixed arity while $\gamma=-5/2$ is the so-called universality class of pure two-dimensional gravity, in which we find all types of planar maps, like triangulations, quadrangulations, bipartite maps and the formula given for general planar maps at the beginning of the manuscript\footnote{This notion of universality can be pushed further by investigating the universality of the scaling limit. In the two cases mentioned, the first is the continuous random tree, the second is the Brownian sphere. We will not discuss scaling limits in this manuscript.}. Universality classes of gravity coupled to conformal matter can be found for example by coupling the $O(n)$-loop model to planar maps \cite{EynardZinnJustin1992}.

My work on colored triangulations in dimensions $d\geq 3$ has focused on extracting this type of universality classes. This first involves being able to test them, by defining \emph{colored building blocks} (CBBs) which generalize polygons and include for instance octahedra and bipyramids at $d=3$. Then I focused on identifying and enumerating the gluings of CBBs which maximize the number of $(d-2)$-simplices at fixed number of $d$-simplices, and it is their universality classes that I was interested in. Some results I chose to highlight here are
\begin{itemize}
	\item The colored triangulations built out of arbitrary CBBs which are homeomorphic to 3-balls, and which maximize the number of edges at fixed number of tetrahedra, are structurally trees, and topologically 3-spheres \cite{Bonzom3D}. This proves the branched polymer phase observed in numerical simulations for all building blocks which can create 3-spheres.
	\item In even dimensions, subverting some expectations, different CBBs can lead to different universality classes \cite{Enhancing}. However, only known classes have been found. The most interesting one and lesser known lies inbetween that of trees and maps and is sometimes known as the ``proliferation of baby universes'', with $\gamma=-5/3$\footnote{And ongoing investigation by Z.~Salvy and W.~Fleurat shows that its scaling limit is the stable tree of index $3/2$. We are thankful to Z.~Salvy and their directors M.~Albenque and É.~Fusy for sharing this with us.}.
\end{itemize}

%We notice that 2d and 3d are radically different in many aspects. Remark on their discrepancies

\subsection*{Universality II: all genus structures} Another universality phenomenon is not at fixed genus or higher-dimensional equivalent, but in structures which relate genera. In this manuscript we present two methods through which universality has shown up in various models.

In maps, the Tutte equations relate the generating series of maps with $n$ boundaries and of genus $g$. B.~Eynard \cite{Eynard04} and collaborators found an intrinsic way to solve those equations, known as the (Chekhov-Eynard-Orantin) \emph{topological recursion} (TR) \cite{EO,ChekhovEynard2006}. It has since become a universal solution to many problems of enumerative geometry, such as Kontsevich-Witten intersection numbers and Weil-Petersson volumes \cite{EynardOrantinWeilPetersson}.

It was then only natural to ask whether any of the two generalizations of bipartite maps we consider here, i.e. weighted Hurwitz numbers and colored triangulations, satisfy the topological recursion.
\begin{itemize}
	\item The remarkable \cite{AlexandrovChapuyEynardHarnad2020} proves the topological recursion for constellations without internal faces (only boundary components). This was reproved in \cite{BychkovDuninBarkowskiKazarianShadrin2020}. We eventually proved the topological recursion for weighted Hurwitz numbers in \cite{BCCGF22}. It was out on the same day as \cite{BychkovDuninBarkowskiKazarianShadrin2022} which proves essentially the same thing, with a completely different approach.
	\item As for colored triangulations, it is easy to convince oneself that no topogical recursion can hold. The models (depending on the CBBs) are too complicated, and do not follow a classification by the genus, like the TR requires. However, we could show \cite{BonzomDartois2018, BonzomDub} that ``with the correct point of view'', some version of the TR kwown as the blobbed TR of \cite{Borot13, BorotShadrin2016} does hold. It really highlights the notion of universality, because the TR holds with respect to \emph{specific boundary CBBs} which are simple enough to control, while all the gluings of the other, arbitrary CBBs, which are incredibly too difficult to enumerate, are packaged into some blobs on which only limited data is required to prove the TR\footnote{Furthermore, this blobbed version of the TR \emph{has a non-universal part}, which results precisely from the too-complicated structures and which we do not know how to evaluate.}.
\end{itemize}

\subsection*{Universality III: Integrable hierarchies and non-oriented maps} The origin of the topological recursion for weighted Hurwitz numbers, whether from \cite{AlexandrovChapuyEynardHarnad2020} or \cite{BychkovDuninBarkowskiKazarianShadrin2020}, lies not in the Tutte equations (like it is usually done for maps) but in a different set of non-linear, partial derivative equations on the generating series, known as the \emph{KP hierarchy} \cite{GouldenJackson2008}. The origin of those equations is well-known from the theory of integrable systems in mathematical physics \cite{AlexandrovZabrodin2013}, and in the context of algebraic combinatorics \cite{Guay-PaquetHarnad2017}, yet their combinatorial interpretation remains a mystery.

They are powerful equations since it is possible to extract from them some not only very compact, but acutally the most efficient known recurrence formulas to evaluate the number of triangulations \cite{GouldenJackson2008}, maps and bipartite maps \cite{CarrellChapuy2015, KazarianZograf2015}, constellations \cite{Louf2019}, by their size and genus. Those recurrence formulas as well as the mysterious origin of the hierarchy are enough reasons to investigate its universality further. 

Integrability is considered to be fragile and to break down under continuous deformations, like the $\beta$-deformation of matrix models and the $b$-deformation of maps of Goulden and Jackson \cite{GouldenJackson1996}. Those deformations interpolate between orientable and non-oriented maps. Yet, it is known from van de Leur's \cite{VandeLeur2001} that non-oriented maps satisfy indeed an integrable hierarchy, the so-called BKP (because it is like KP but for a type $B$ group) hierarchy. The latter contains KP, i.e. solutions of the KP hierarchy are also solutions of the BKP one.

G.~Chapuy and M.~Do\l\k{e}ga in a beautiful article \cite{ChapuyDolega2020} have introduced a notion of non-oriented weighted Hurwitz numbers. However, among all of them we have only found two more instances of the BKP hierarchy: non-oriented bipartite maps (easily proved from \cite{VandeLeur2001}), and non-noriented \emph{monotone Hurwitz numbers} \cite{BonzomChapuyDolega2021}. Proving the latter involved equating beautiful expansions on symmetric functions, and as a byproduct we proved some technical conjectures of V.~Féray and of Oliveira and Novaes. Another key application of our result is a Pfaffian expression for the $O(n)$-BGW integral, which was a long-standing open problem\footnote{Since it is a bit outside the scope of this manuscript, we will not present that application.}.

To conclude this investigation of the universality of the BKP hierarchy, we derived from it some recurrence formulas for non-oriented triangulations, maps and bipartite maps by size and genus \cite{BCD-RecurrenceFormulas} and provided a Maple worksheet with all of them \cite{us:Maple}. 

\subsection*{Universality IV: Scheme decomposition} Here we briefly mention another method through which universality arises. {\it We will not discuss it further in this manuscript.} The idea of the scheme decomposition, say for maps of genus $g$, is to encode the latter, of which there is an infinite number, as only a finite number of objects known as \emph{schemes} then decorated by decorations which although they can be arbitrarily large, are simple like trees. This has been done for maps in \cite{ChapuyMarcusSchaeffer2009}, and re-adapted to colored graphs in \cite{GurauSchaeffer}. We found that a sort of universal scheme decomposition holds in several models, including some of interest for black hole holography and models of maps with 4-valent vertices decorated with crossing loops.  In particular, the ``most singular schemes'' have the same structure \cite{O(N)3, DoubleScalingMultiMatrix}.% However, the proof, while following the original strategy of \cite{GurauSchaeffer}, has to be adapted to the details of every model. 

\section*{Plan \& list of works used in this manuscript}

\subsection*{Chapter \ref{sec:Definitions}} Introduces the main protagonists: maps, constellations, stuffed maps, higher-dimensional colored triangulations and their dual representation as properly-edge-colored graphs.

\subsection*{Chapter \ref{sec:TR}} Introduces double weighted Hurwitz numbers, their Schur expansion, their enumeration in the planar case, and presents the topological recursion. It is based on
\begin{itemize}
	\item ``Topological recursion for Orlov-Scherbin tau functions, and constellations with internal faces'', V. Bonzom, G. Chapuy, S. Charbonnier and E. Garcia-Failde, \url{https://arxiv.org/abs/2206.14768}
\end{itemize}

\subsection*{Chapter \ref{sec:EnumerationIntegrability}} Introduces the KP hierarchy applied to map enumeration at all genus.

\subsection*{Chapter \ref{sec:NonOrientedMaps}} Is about non-oriented maps and the BKP hierarchy. It is based on
\begin{itemize}
	\item ``Enumeration of non-oriented maps via integrability'', V. Bonzom, G. Chapuy and M. Do\l\k{e}ga, Accepted in ALCO, \url{https://arxiv.org/abs/2110.12834}
\end{itemize}

\subsection*{Chapter \ref{sec:Monotone}} Focuses on the specific case of non-oriented monotone Hurwitz numbers and proves the BKP hierarchy, as done in
\begin{itemize}
	\item ``$b$-monotone Hurwitz numbers: Virasoro constraints, BKP hierarchy, and $O(N)$-BGW integral'', V. Bonzom, G. Chapuy and M. Do\l\k{e}ga, International Mathematics Research Notices (2022) rnac177
\end{itemize}

\subsection*{Chapter \ref{sec:PrescribedBubbles}} Introduces the main questions we tackled in higher-dimensional triangulations, focusing on the problem of maximizing the number of $(d-2)$-subsimplices at fixed number of $d$-simplices, and investigating universality using the notion of colored building blocks. The presentation is inspired from the review article
\begin{itemize}
	\item ``Large $N$ Limits in Tensor Models: Towards More Universality Classes of Colored Triangulations in Dimension $d\geq 2$'', V. Bonzom, SIGMA 12 (2016) 073
\end{itemize}
This program was triggered in the context of tensor models (generalizing matrix models) with
\begin{itemize}
	\item ``Critical behavior of colored tensor models in the large N limit'', V. Bonzom, R. Gurau, A. Riello and V. Rivasseau, Nucl. Phys. B 853 (2011) 174--195
	\item ``Random tensor models in the large N limit: Uncoloring the colored tensor models'',V. Bonzom, R. Gurau and V. Rivasseau, Phys. Rev. D 85 (2012) 084037
	\item ``New 1/N expansions in random tensor models'', V. Bonzom, JHEP 06 (2013) 062
\end{itemize}
where in the last one I realized that more behaviors than the melonic (i.e. trees) one may be achieved.

\subsection*{Chapter \ref{sec:1CBB}} Presents some enumerative results about triangulations made of a single colored building block (higher-dimensional equivalents of 1-face maps). We focus on a family of CBBs where the enumeration problem is related to meander systems, from
\begin{itemize}
	\item ``The calculation of expectation values in Gaussian random tensor theory via meanders'', V. Bonzom and F. Combes, Annales de l'Institut Henri Poincar{\'{e}} D 1 (2014) 443--485
\end{itemize}

\subsection*{Chapter \ref{sec:Universality3D}} Contains our most important results on the universality class of triangulations in dimensions 3. It is based on
\begin{itemize}
	\item ``Maximizing the number of edges in three-dimensional colored triangulations whose building blocks are balls'', V. Bonzom, \url{https://arxiv.org/abs/1802.06419}
	\item ``Counting Gluings of Octahedra'', V. Bonzom and L. Lionni, The Electronic Journal of Combinatorics 24 (2017)
\end{itemize}

\subsection*{Chapter \ref{sec:Universality2}} Presents a bijection with some colored stuffed maps, which we used to extract universality classes, in particular the ``proliferation of baby universes'' phase transition between maps and trees, and which we also used to prove the blobbed TR for some boundary variables of colored triangulations. The relevant references are
\begin{itemize}
	\item ``Colored Triangulations of Arbitrary Dimensions are Stuffed Walsh Maps'', V. Bonzom, L. Lionni and V. Rivasseau, The Electronic Journal of Combinatorics 24 (2017)
	\item ``Enhancing non-melonic triangulations: A tensor model mixing melonic and planar maps'', V. Bonzom, T. Delepouve and V. Rivasseau, Nucl. Phys. B 895 (2015) 161--191
	\item ``Tensor models with generalized melonic interactions'', V. Bonzom, \url{https://arxiv.org/abs/1905.01903}
	\item ``Blobbed topological recursion for the quartic melonic tensor model'', V. Bonzom and S. Dartois, J. Phys. A 51 (2018) 325201
	\item ``Blobbed topological recursion for correlation functions in tensor models'', V. Bonzom and N. Dub, \url{https://arxiv.org/abs/2011.09399}
\end{itemize}

%\section{Structural and enumerative aspects}
%%%%%%%%%%%%%%%%%%%%%%%%%%%%%
\chapter{Maps and higher-dimensional PL-manifolds} \label{sec:Definitions}

This chapter introduces some of the main objects of this memoir, some well-known like maps, others less known like colored triangulations. Because the latter are not wide-spread, I will provide short proofs for some propositions which in my opinion help grasp a better intuition.
%Below we define maps, both orientable or not. We restrict attention to orientable objects for constellations and higher dimensional PL-manifolds.
\section{Maps, stuffed maps and constellations}

\subsection{Maps as gluings of polygons} \label{sec:PolygonGluings} An $p$-gon is a polygon made of $p$ sides and vertices. A \emph{map} with $F$ faces is a gluing of $F$ polygons along their sides, such that every side is identified with another \cite{MoharThomassen2001, LandoZvonkin2004}, see Figure \ref{fig:PolygonsGluing}. The identification of a side with another\footnote{The reader who already notices that there are in fact two ways of identifying two sides together probably already knows this has to do with orientability. Here we focus on the orientable case. The reader is invited to go straight to Chapter \ref{sec:EnumerationIntegrability}, Section \ref{sec:NonOrientedMaps} for details about the roles of those two ways of gluing sides of polygons together.} is an example of an \emph{attaching map} between pieces of boundaries of PL-balls. This will be discussed in higher dimensions later. The identified sides form \emph{edges}. The polygons are called the \emph{faces} and a face has \emph{degree} $p$ it is a $p$-gon. The degree of a vertex is the number of incident edges. A corner is the portion of polygon between two consecutive edges meeting at the same vertex. A $p$-angulation is a map whose faces all have degree $p$, as in Figure \ref{fig:Quadrangulation}.
%A walk along a face is the sequence of edges/vertices/corners visited clockwise.
\begin{figure}
	\includegraphics[scale=.6,valign=c]{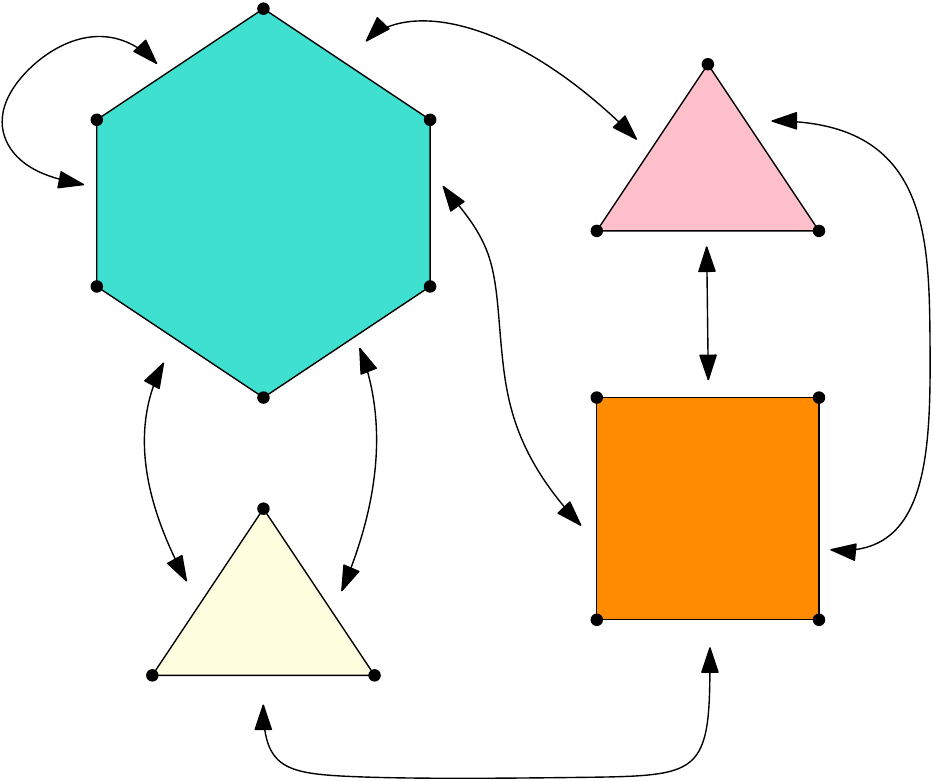} \hspace{2cm} \includegraphics[scale=.6,valign=c]{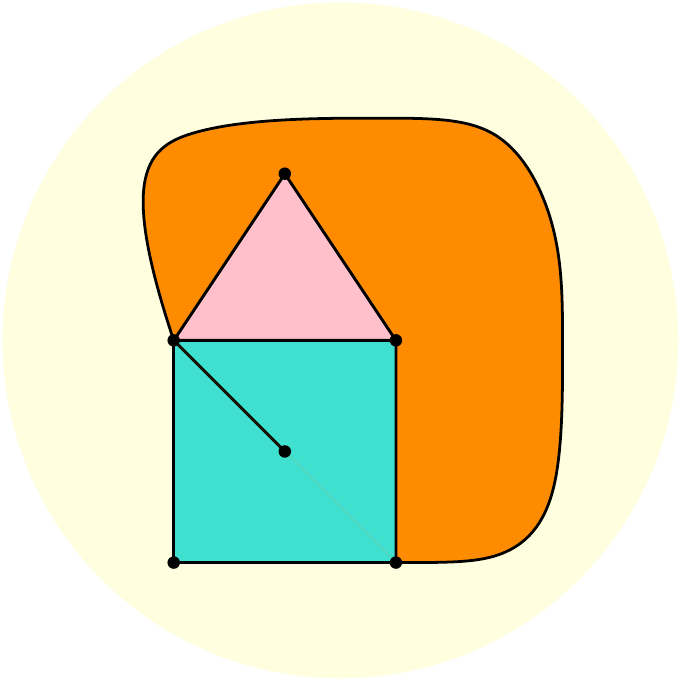}
	\caption{\label{fig:PolygonsGluing}A gluing of four polygons and the surface (the sphere) it represents on the right. The yellow triangle is represented as an outer face.}
\end{figure}
\begin{figure}
	\includegraphics[scale=.6]{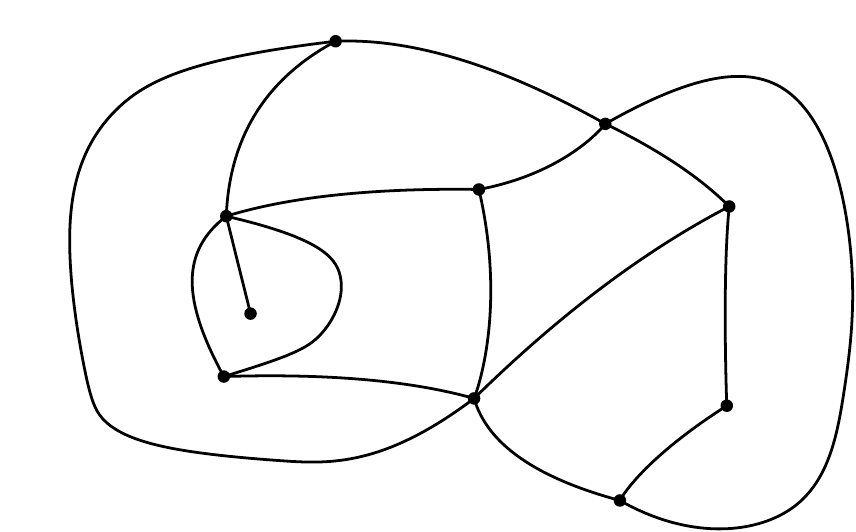}
	\caption{\label{fig:Quadrangulation} A planar quadrangulation.}
\end{figure}

\subsection{Euler's relation} A map may not be connected. As a topological space, it is a disjoint union of closed surfaces. The topology of each connected component is characterized by its genus (which is a half-integer if non-orientable surfaces are considered). For a connected map $\m$ with $F(\m)$ faces, $E(\m)$ edges and $V(\m)$ vertices, the genus $g(\m)\geq 0$ is found via Euler's formula,
\begin{equation} \label{Euler}
	2-2g(\m) = F(\m)-E(\m)+V(\m).
\end{equation}
Setting aside the topological content, there is still a purely combinatorial substance here: the fact that $g(\m)$ defined by the above formula is non-negative. More explicitly, and focusing on $p$-angulations for simplicity where we have $E(\m) = pF(\m)/2$, Euler's relation gives
\begin{equation} \label{EulerBound}
	V(\m) \leq 2 + \frac{p-2}{2} F(\m),
\end{equation}
in other words,
\begin{itemize}
	\item there is a bound on the number of vertices which grows linearly with the number of faces,
	\item maps which maximize the bound are precisely the planar maps.
\end{itemize}
This is the content we will be aiming at generalizing in higher dimensions.

\emph{Unless stated otherwise, all maps in the rest of this section are assumed to be orientable.}

\subsection{Maps as embedded graphs} \label{sec:EmbeddedGraphs} Another useful point of view to pick up the 1-skeleton of a map as defined above, which is a graph, and realize that the map is recovered by a ``proper'' embedding of the graph into the surface. This embedding is such that edges only intersect at vertices and the graph complement is a union of discs (the polygons of the above definition), see Figure \ref{fig:TorusCylinder}. This selects a unique genus (given by Euler's formula) and the embedding is considered up to isotopy. As a consequence, the embedding (hence the map itself) is equivalent to a \emph{rotation system}, i.e. a choice of cyclic ordering of the edges meeting at every vertex.

In practice, maps are represented as drawings on the plane by projection (which typically induces crossings outside of vertices). Faces in an oriented map are recovered by following the edges and the corners counter-clockwise. 
\begin{figure}
	\includegraphics[scale=.6]{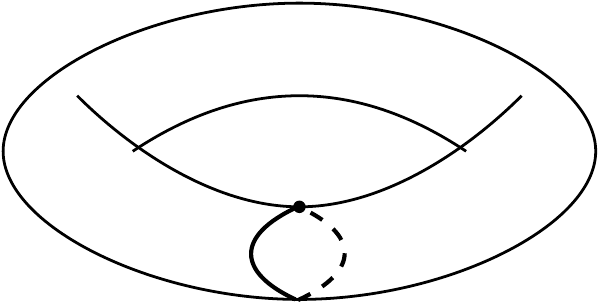}
	\caption{\label{fig:TorusCylinder} This is not a map, because the 1-skeleton is a vertex with an edge and its complement is a cylinder, not a disc. It cannot be constructed by gluings of polygons.}
\end{figure}

\subsection{Stuffed maps}\label{sec:StuffedMaps} Instead of gluing polygons, which have the topology of the disc, it is possible to glue higher topologies. Stuffed maps are defined as gluings of surfaces along their boundaries, where the surfaces are characterized by their genus $g$ and their $n$ boundaries of size $l_1, \dotsc, l_n$. They are called \emph{elementary cells of topology $(g, n)$ and perimeters $l_1, \dotsc, l_n$}. Usual maps correspond to only gluing discs, which are elementary cells of topology $(0,1)$ characterized by the perimeter of their boundary. Stuffed maps were formally introduced by Borot in \cite{Borot13} as a generalization of the gasket decomposition of maps decorated with loops. In the latter case, a map decorated with (self-avoiding, non-intersecting) loops can be transformed into a stuffed map by thickening the loop into a cylinder, respecting the inner and outer degrees of the loop. That gives stuffed maps where only discs and cylinders are allowed. 

\subsection{Constellations} A \emph{bipartite} map is such that the set of vertices can be written $A\sqcup B$ where vertices from $A$ (``white'' vertices) are only connected to vertices of $B$ (``black'' vertices) and the other way around. A bicoloring is a bipartition of the set of faces (a white face is adjacent to only black faces and the other way around) and a bicolored map is a map equipped with a bicoloring. A bicoloring induces a \emph{canonical orientation}, which orients edges so that black faces are to their left (and white faces to their right). 

Of special interest for us will be \emph{constellations} \cite{LandoZvonkin2004}. An (orientable\footnote{See \cite{ChapuyDolega2020} for a non-orientable generalization.}) $m$-constellation ($m\geq 1$) is a bicolored map equipped with its canonical orientation, such that (Figure \ref{fig:Faces})
\begin{itemize}
	\item each vertex has a color $c\in[1..m]$;
	\item if $v$ is a vertex of color~$c$, then any edge outgoing from $v$ points to a vertex of color $c-1\mod m$;
	\item black faces have degree $m$.
\end{itemize}
As a consequence, white faces have degrees multiples of $m$.
\begin{figure}
	\includegraphics[scale=.5,valign=c]{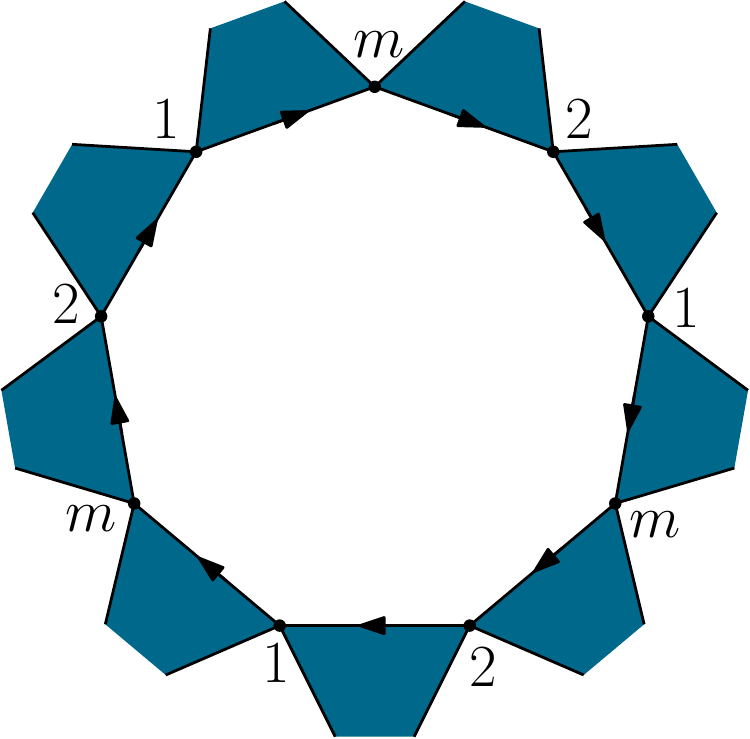} \hspace{2cm} %\includegraphics[scale=.4]{BlackFace.pdf} \hspace{2cm}
	\includegraphics[scale=.5,valign=c]{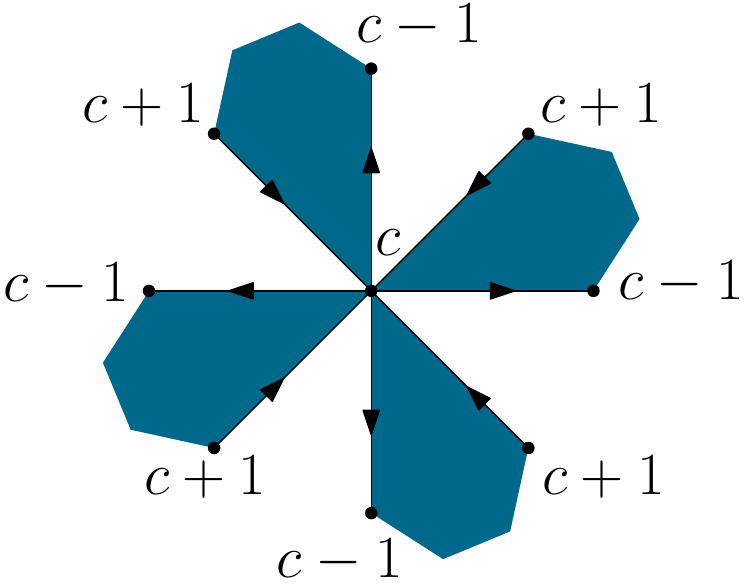}
	\caption{A white face of degree $3m$, and a vertex of color~$c$. Edges receive their canonical orientation so that the black faces are to their left.}
	\label{fig:Faces}
\end{figure}
The size of a constellation is the number of black faces. There are three connected 3-constellations of size 2, given in Figure \ref{fig:3ConstellationsSize2}.
\begin{figure}
	\includegraphics[scale=.55,valign=c]{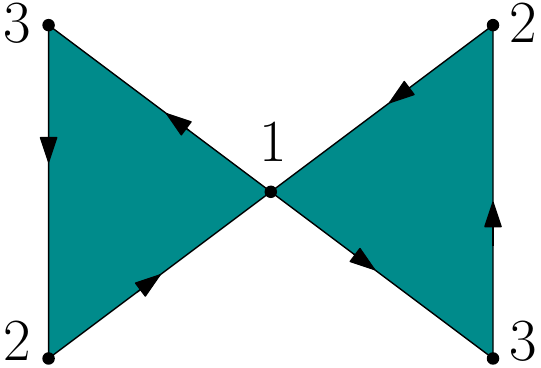} \hspace{2cm} \includegraphics[scale=.55,valign=c]{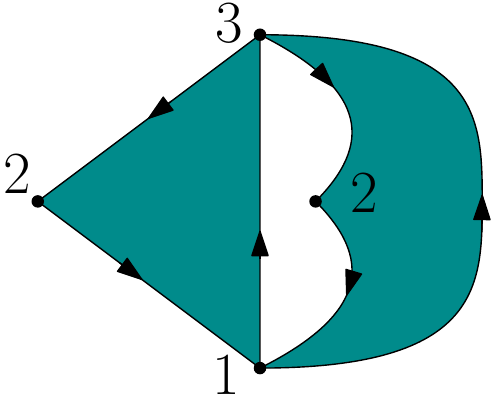} \hspace{2cm}
	\includegraphics[scale=.55,valign=c]{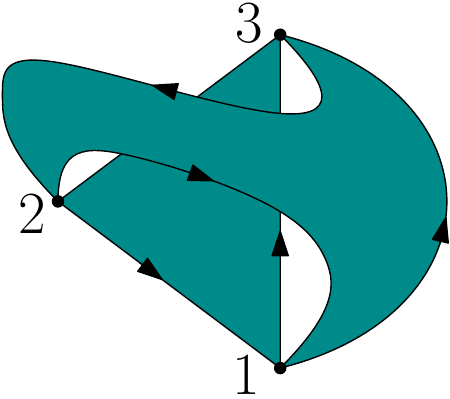}
	\caption{The three connected 3-constellations of size 2. Two are planar, the other has genus 1.}
	\label{fig:3ConstellationsSize2}
\end{figure}
2-constellations are simply bipartite maps, since black faces are then digons which can be replaced with edges (between vertices of colors 1 and 2). General maps can be seen as bipartite maps (thus 2-constellations) whose white vertices only have degree 2.

A \emph{labeled} constellation is a constellation with labeled black faces. A \emph{rooted} constellation has only a reference black face. There is thus a one-to-$(n-1)!$ correspondence due to the $(n-1)!$ possible labelings of the non-root black faces.

\subsection{Biconstellations}\label{sec:Biconstellations} An equivalent representation of $m$-constellations is as $(m-1)$-biconstellations\footnote{This terminology is not standard. We hope thereby to avoid confusion on the various representations of constellations and factorizations of permutations.}. An $(m-1)$-biconstellation is a bicolored map equipped with its canonical orientation, such that
\begin{itemize}
	\item each vertex has a color $c\in[1..m-1]$;
	\item if $v$ is a vertex of color~$c$, then any edge outgoing from $v$ points to a vertex of color $c-1\mod (m-1)$;
\end{itemize}
It comes that white and black faces have degrees multiples of $m-1$. There is a trivial bijection with $m$-constellations which exchanges vertices of color $m$ and degree $2k$ in the latter with black faces of degree $(m-1)k$ in the biconstellation, and white faces of degree $mk$ with white faces of degree $(m-1)k$,
\begin{equation} \label{VertexFaceBiconstellation}
	\includegraphics[scale=.5,valign=c]{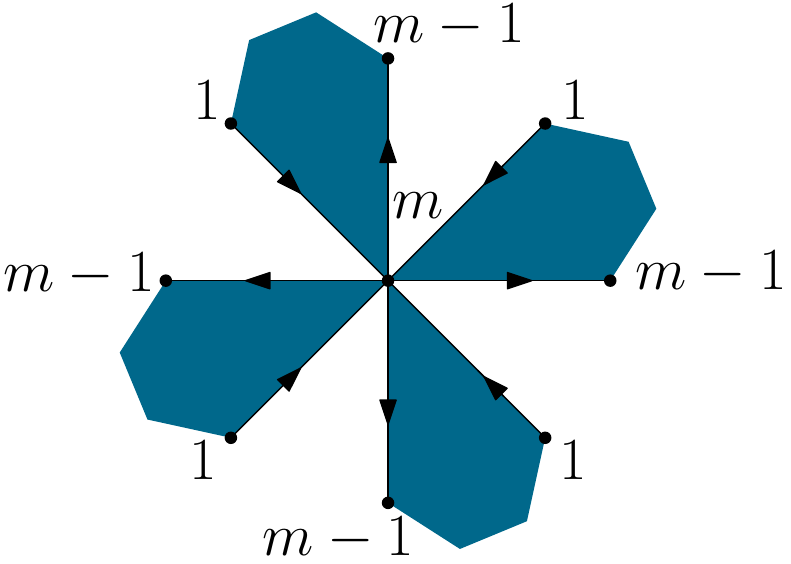} \quad \to\quad \includegraphics[scale=.5,valign=c]{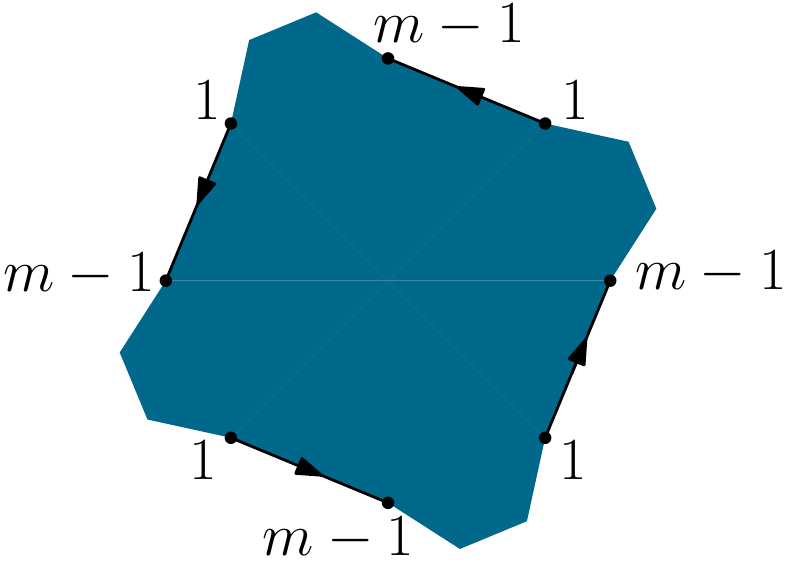}
\end{equation}

\emph{Rooted biconstellations} are defined by following the root of a constellation through that mapping. A black face of an $m$-constellation becomes a \emph{left path} in an $(m-1)$-biconstellation, i.e. an oriented sequence of $m-1$ edges from a vertex of color $m-1$ to the next vertex of color $m-1$ along a black face.

Biconstellations are a natural representation when attempting (and succeeding!) to perform some enumeration while controlling two sets of degrees. If one were to attempt that in terms of constellations, one would try to control the degrees of white faces and those of vertices of a given color like $m$, which is less symmetric. In particular, biconstellations generalize the famous bicolored maps/hypermaps generated by the 2-matrix model \cite{Eynard:book}, which simply correspond to 1-biconstellations. As shown originally in \cite{Eynard2002}, the enumeration with respect to the number of white and black faces with controlled degrees can take a very nice symmetric formulation (in planar case). It was moreover made completely clear in combinatorial terms in \cite{AlbenqueBouttier2022} using the notion of slices, for which the symmetry between white and black faces is obvious. We will come back to that in the next chapter.

%If one were to attempt that in terms of constellations, one would have to control the degrees of white faces and those of vertices of color $m$, which is less convenient, in particular because the former are multiple of 2 and the latter multiple of $m$. In contrast, they are both multiples of $m-1$ in biconstellations.

%\subsection{Vertex-Face duality}
%%%%%%%%%%%%%%%%%%%%%%%%%%%%%%%%%%
\section{A graphical encoding of PL-manifolds} \label{sec:PLmanifolds}

\subsection{Bipartite maps and tricolored graphs} A tricolored graph is a bipartite, cubic graph\footnote{In this manuscript, graphs are multigraphs.}, whose edges carry a color $c\in\{0,1,2\}$ so that each vertex is incident to exactly one edge of every color. Of particular importance are the bicolored cycles, meaning cycles whose edges alternate two colors. There are 3 types of bicolored cycles: $\{0,1\}$, $\{0,2\}$, $\{1,2\}$. There is exactly one of each incident at each vertex.
\begin{proposition}{}{}
	There is a bijection between bipartite maps and tricolored graphs. It maps white vertices to bicolored cycles with colors $\{0,1\}$, black vertices to bicolored cycles with colors $\{0,2\}$ and faces to bicolored cycles with colors $\{1,2\}$.
\end{proposition}

We thus define the genus of a tricolored graph as that of the associated map $g(\graph) := g(\m(\graph))$. If $C_{ab}(\graph)$ is the number of bicolored cycles with colors $\{a,b\}$ in a tricolored graph $\graph$ with $V(\graph)$ vertices, then Euler's relation becomes
\begin{equation} \label{EulerColoredGraph}
	2-2g(\graph) = C_{01}(\graph) + C_{02}(\graph) + C_{12}(\graph) - \frac{V(\graph)}{2}.
\end{equation}

A key intermediate object between a bipartite map and its associated tricolored graph is a tricolored triangulation. It is a triangulation $\trisp$
\begin{itemize}
	\item whose edges carry a color $c\in\{0,1,2\}$, such that every triangle has an edge of each color,
	\item triangles can only be glued along edges which carry the same color (but this is not enough to specify the attaching map),
	\item In every triangle, assign the pair of colors $\{c_1,c_2\}$ to the vertex shared by the edges with colors $c_1$ and $c_2$. Then, edges of the same color are glued in the only way which respects the pair-coloring of the vertices. For instance, we glue two sides of color 0 in the only way which identifies their vertices of colors $\{0,1\}$ and of colors $\{0,2\}$,
	\begin{equation} \label{GluingTriangles}
		\includegraphics[scale=.35,valign=c]{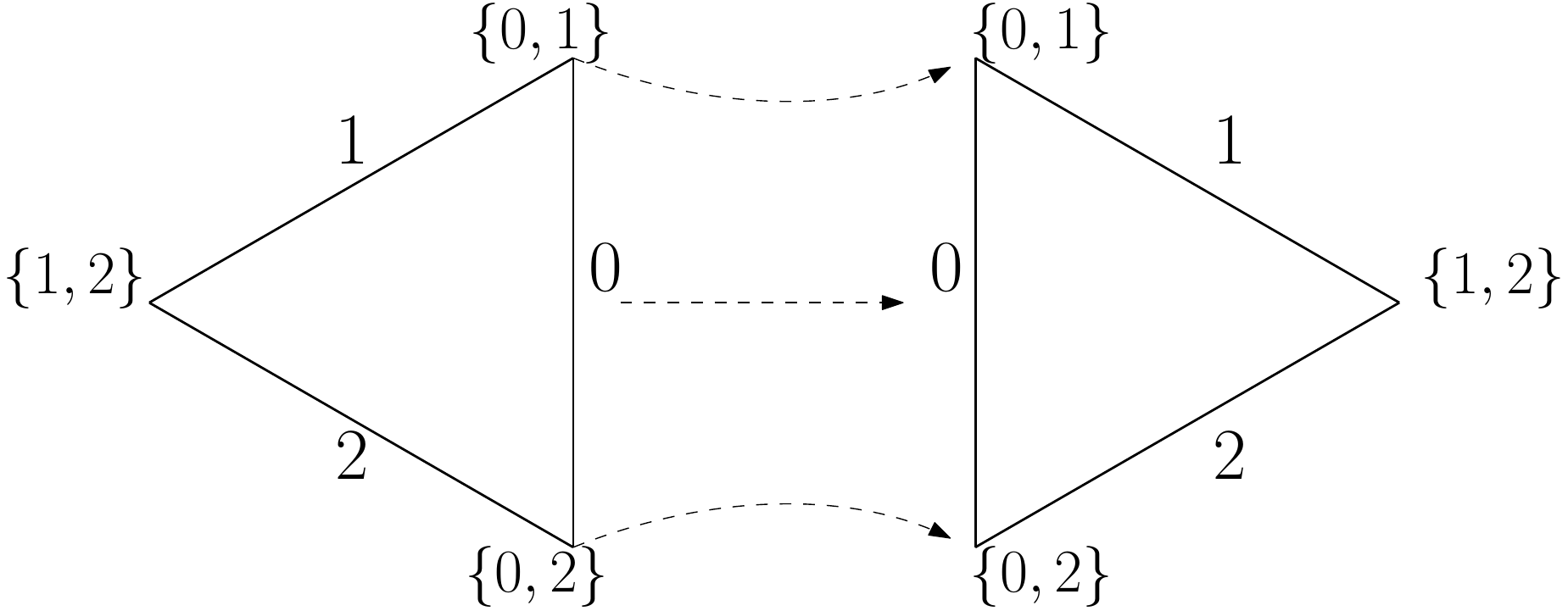}
	\end{equation}
\end{itemize}

\begin{figure}
	\includegraphics[scale=.43]{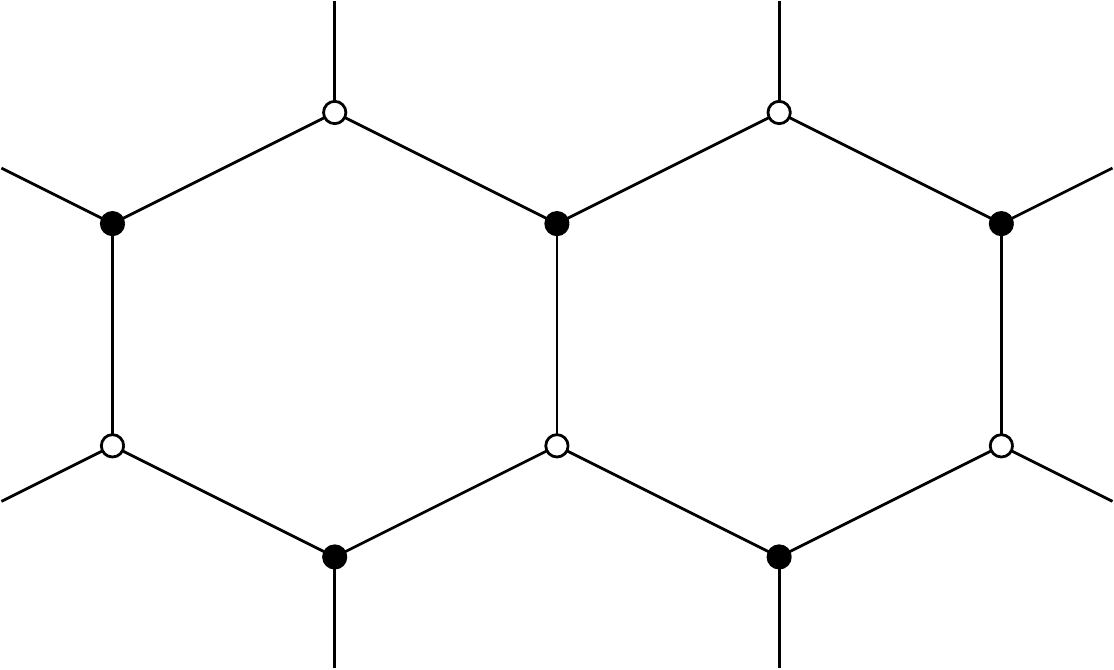} \hspace{.2cm}
	\includegraphics[scale=.43]{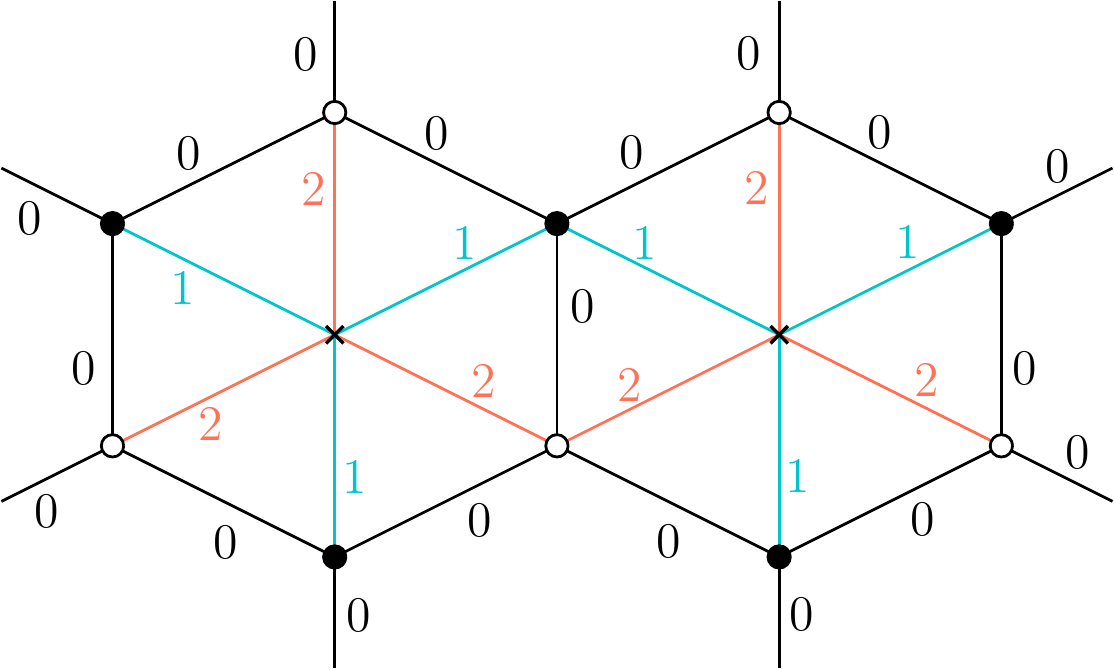} \hspace{.2cm}
	\includegraphics[scale=.43]{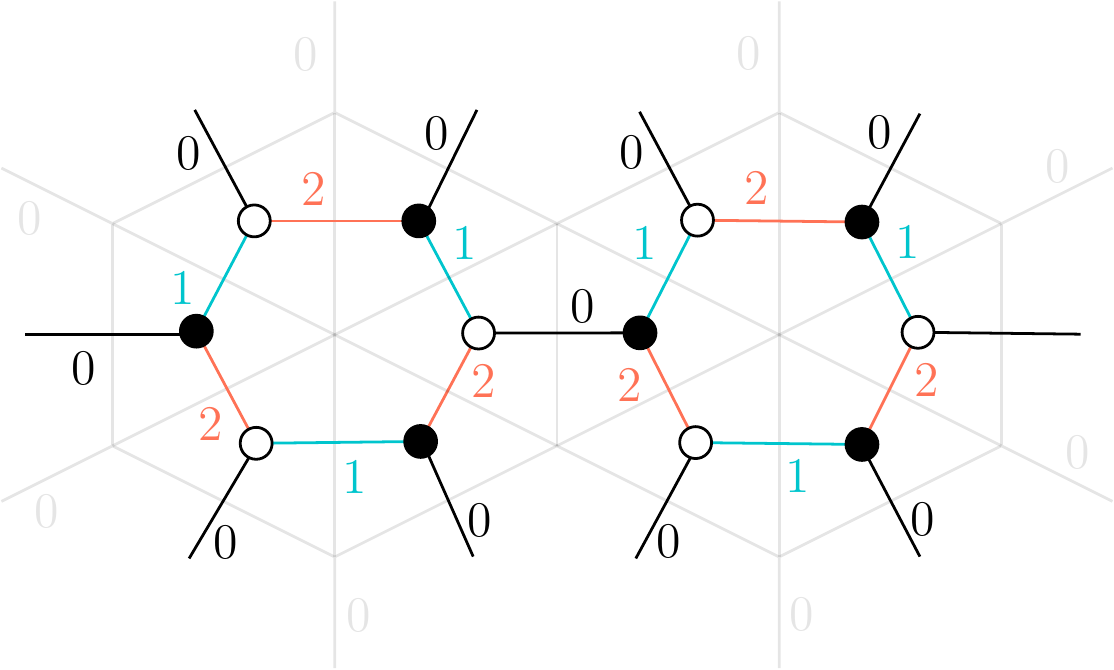}
	\caption{\label{fig:BipMapToColoredGraph} On the left is a piece of a bipartite map. In the middle, the corresponding tricolored triangulation. On the right, the corresponding colored graph.}
\end{figure}
\begin{proof}
	We first go between a bipartite map $\m$ and its associated tricolored triangulation $\trisp$. This is illustrated in Figure \ref{fig:BipMapToColoredGraph}. First assign the color 0 to all edges of $\m$. Then add a vertex $v_f$ inside every face $f$ of $\m$ and connect it with edges to the boundary vertices of $f$. To assign colors to those edges, we choose that edges which connect $v_f$ to a white vertex are given the color 1 and edges which connect $v_f$ to a black vertex are given the color 2. This is $\trisp$. The attaching rule of tricolored triangulations are satisfied because the white vertices of $\m$ are the vertices with colors $\{0,1\}$ in $\trisp$, the black vertices of $\m$ are the vertices with colors $\{0,2\}$ in $\trisp$, and the vertices with colors $\{1,2\}$ are the $v_f$s. %Orientability is established more generally in \cite{Ferri1986}, and corresponds in $\trisp$ to an independent bicoloring of the triangles.
	
	The tricolored graph $\graph$ associated to $\trisp$ is obtained by first taking the dual map to $\trisp$. It is a cubic map $\trisp^*$. Each edge being dual to an edge of $\trisp$, it inherits its color. Then the key observation which allows for reducing $\trisp^*$ to a graph $\graph$ is that the gluing between two adjacent triangle in $\trisp$ is in fact fully determined by the color of the edge they share. Therefore one can ignore the ordering between the edges around every vertex of $\trisp^*$, as long as one keeps the colors of the edges. Therefore $\graph$ is the colored 1-skeleton of $\trisp^*$.
\end{proof}

An obvious consequence of the construction in the above proof is the following.
\begin{proposition}{}{CanonicalEmbedding}
	There is a canonical embedding of the tricolored graph $\graph(\trisp)$ which turns it into the cubic map $\trisp^*$, by letting $(012)$ be the cyclic order around the white vertices of $\graph(\trisp)$ and $(021)$ be the cyclic order around its black vertices. It turns the bicolored cycles with colors $\{a,b\}$ for $a,b\in\{0,1,2\}$ of $\graph(\trisp)$ into the faces of $\trisp^*$.
\end{proposition}

\subsection{Colored graphs} Let $d\geq 2$. A properly $(d+1)$-edge-colored graph, or simply $(d+1)$-colored graph, or even colored graph when there is no ambiguity, is a (multi)graph such that each edge carries a color $c\in[0..d]$ and every vertex is incident to exactly one edge of every color. This is a generalization of the above tricolored graph, which will play the same dual role with respect to a set of higher-dimensional PL-manifolds, which we now introduce.

In this manuscript, all colored graphs are \emph{bipartite}. An example is given in Figure \ref{fig:ColoredGraph}.
\begin{figure}
	\includegraphics[scale=.48]{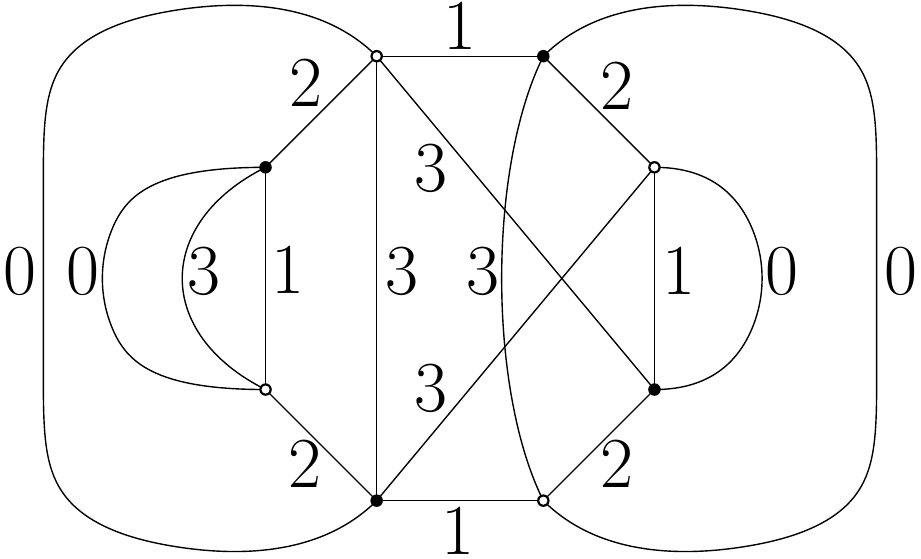}
	\caption{\label{fig:ColoredGraph} A (bipartite) colored graph at $d=3$.}
\end{figure}

\subsection{Colored triangulations} Let $d\geq 2$. An abstract $d$-simplex is the set of subsets of $[0..d]$. However, instead of thinking of the elements of $[0..d]$ as vertices, we will think of them as faces i.e. simplices of dimension $d-1$. Although they are often called faces or facets in the literature on simplicial complexes, we will avoid this terminology as it overlaps with the traditional use of faces for maps. Instead we will use ``subsimplices'' and specifiy the dimension as needed, like $k$-subsimplices.

A \emph{colored simplex} is a $d$-simplex whose $d+1$ subsimplices of dimension $d-1$ are each equipped with a color from $[0..d]$. It is key that every subsimplex of dimension $d-k$, $k=1,\dotsc, d$ is thus uniquely labeled with a $k$-tuple of colors. The tuple $\{i_0, \dotsc, i_k\}$ identifies the subsimplex which lies at the intersection of the $(d-1)$-subsimplices of colors $i_0, \dotsc, i_k$. For instance, a subsimplex $\simp^{(d-2)}$ of dimension $d-2$ is shared by exactly two $(d-1)$-subsimplices. If they have colors $c,c'$ then $\simp^{(d-2)}$ is uniquely identified by the pair of colors $\{c,c'\}$. This is depicted in Figure \ref{fig:ColoredSimplex}.

\begin{figure}
	\includegraphics[scale=.5,valign=c]{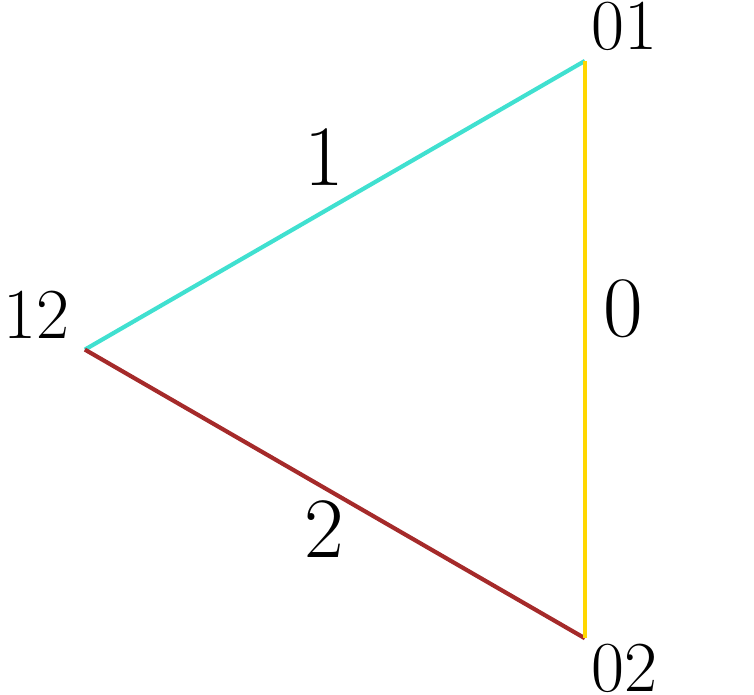} \hspace{2cm} \includegraphics[scale=.5,valign=c]{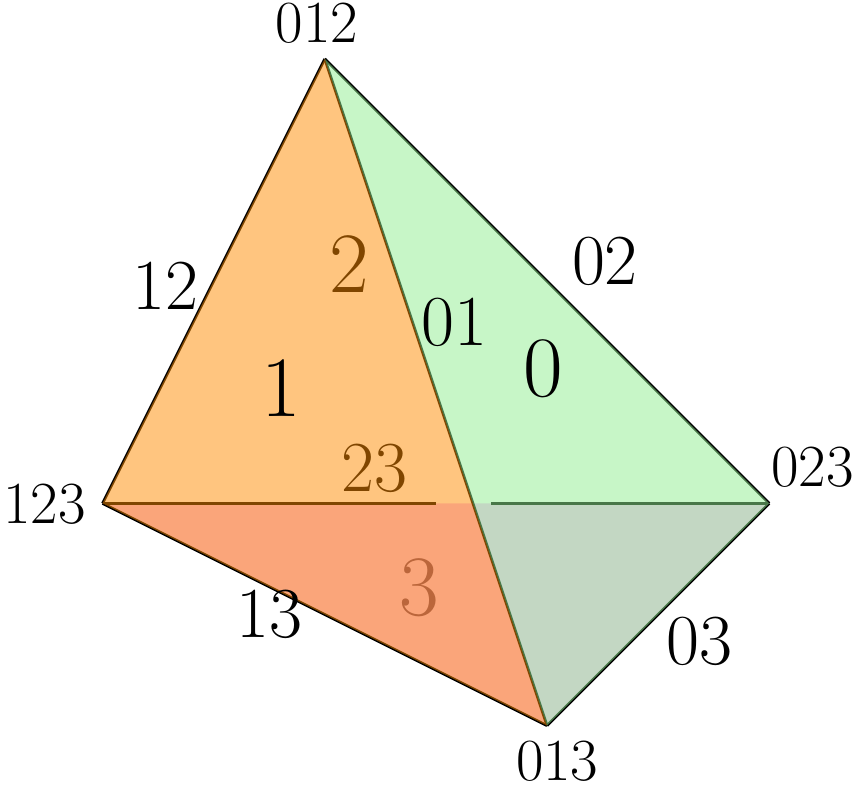}
	\caption{\label{fig:ColoredSimplex}On the left is a triangle equipped with a coloring of its edges which induces a coloring of its vertices by pairs. On the right, a tetrahedron with colored faces, which induces pairs of colors on edges and triples of colors on vertices.}
\end{figure}

Given two colored $d$-simplices $\simp_1, \simp_2$, we define a canonical attaching map which is fully specified by the color of the faces which are identified between them. We glue $\simp_1$ and $\simp_2$ along their $(d-1)$-subsimplex of color $c\in[0..d]$ by identifying all their subsimplices which contain $c$ in their labels, i.e. for all $k=0, \dotsc, d-1$, we identify the subsimplex with label $(i_0, \dotsc, i_k)$ of $\simp_1$ with the one of $\simp_2$ if $c=i_j$ for some $j\in\{0,\dotsc, k\}$. This was shown in \eqref{GluingTriangles} at $d=2$, and in Figure \ref{fig:GluingColoredSimplices} for $d=3$. For instance, if one glues two tetrahedra along their face of color 0, then this is done by identifying their edges of colors $\{0,1\}$, $\{0,2\}$, $\{0,3\}$ and their vertices of colors $\{0,1,2\}$, $\{0,1,3\}$ and $\{0,2,3\}$.

\begin{figure}
	\includegraphics[scale=.45,valign=c]{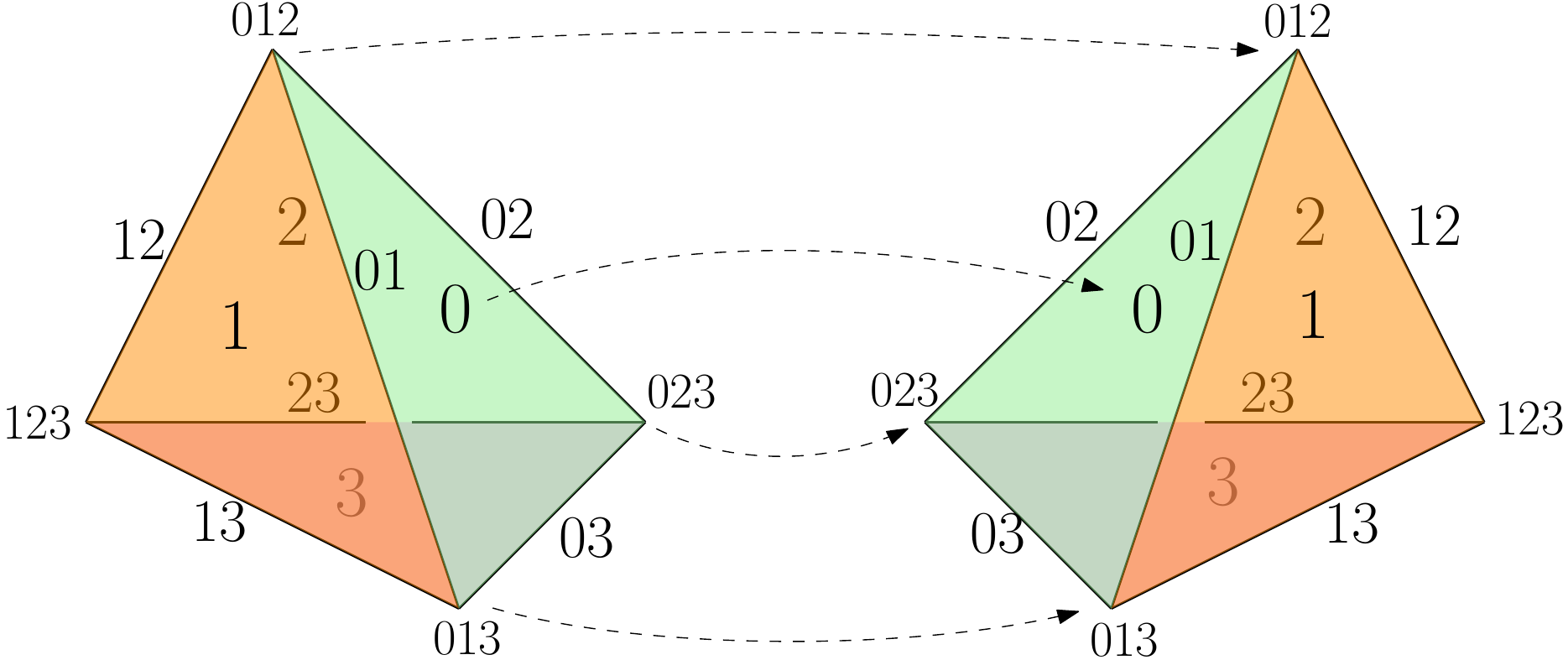}
	\caption{\label{fig:GluingColoredSimplices}The gluing of colored simplices at $d=2$ and $d=3$. One specifies the color, here 0, of the face being glued, then the attaching map is determined by requiring that all subcoloring containing 0 are respected.}
\end{figure}

A \emph{colored triangulation} $\trisp$ is a (connected or not) gluing of colored $d$-simplices where all $(d-1)$-subsimplices are shared by two $d$-simplices. A colored triangulation is a pseudomanifold (by definition). In two dimensions, a pseudomanifold is a manifold but in general, it can have singularities of codimension 2 (this will be clear later).

There is a single colored triangulation with two $d$-simplices, whose boundary $(d-1)$-simplices are identified two by two. It is homeomorphic to the $d$-sphere. The following theorem shows that colored triangulations are ``good'' topological objects.
\begin{theorem}{}{}\cite{Ferri1986,LinsMandel1985,KauffmanLins1994}
	Every PL-manifold admits a colored triangulation. For every $d$, there is a finite set of moves of the triangulations such that two homeomorphic triangulations can be transformed into each other through a finite sequence of those moves.
\end{theorem}
The main interest in colored triangulations yet comes from the existence of a graphical encoding. The canonical gluing rule for colored simplices ensures that there is a unique gluing between two colored simplices as soon as the color of their common $(d-1)$-subsimplex is specified. Therefore, a colored triangulation is completely determined by the data of which simplex is connected to which other simplex using which color. This data is fully encoded into a graph $\graph(\trisp)$ with properly colored edges:
\begin{itemize}
	\item each vertex of $\graph(\trisp)$ corresponds to a $d$-simplex of $\trisp$
	\item there is an edge of color $c\in\{0, 1, \dotsc, d\}$ between two vertices of $\graph(\trisp)$ if the two corresponding $d$-simplices of $\trisp$ are glued along a $(d-1)$-subsimplex of color $c$.
\end{itemize} 
For instance, the colored graph corresponding to the unique colored triangulation with two $d$-simplices is the graph with two vertices, connected by edges of all colors from $0$ to $d$.
\begin{theorem}{}{ColoredGraphs}
	There is a one-to-one correspondence between $d$-dimensional, closed, connected, colored triangulations and connected colored graphs. Moreover,
	\begin{itemize}
		\item If $\trisp$ is a colored triangulation and $\graph(\trisp)$ the corresponding colored graph, then the latter is obtained as the \emph{colored 1-skeleton} of the complex dual to $\trisp$, i.e. the 1-skeleton of the dual where the edges of $\graph(\trisp)$ retain the colors of their dual $(d-1)$-simplices in $\trisp$. We can denote this as
		\begin{equation}
			\graph(\trisp) = \operatorname{col-1-sk}(\trisp^*)
		\end{equation}
		where the star represents the dual and $\operatorname{col-1-sk}$ the colored 1-skeleton.
		\item $\trisp$ can be reconstructed from $\graph(\trisp)$ by applying the colored gluing rule to each pair of simplices represented in $\graph(\trisp)$ by two vertices connected by an edge.
		\item For any $k\in\{0, \dotsc, d-1\}$, let $\{c_1, \dotsc, c_{d-k}\} \subset \{0, \dotsc, d\}$ be a subset of colors and $\graph(c_1, \dotsc, c_{d-k})\subset \graph(\trisp)$ be the subgraph obtained by keeping all the vertices of $\graph(\trisp)$ and the edges of colors $c_1, \dotsc, c_{d-k}$ while removing the other edges. There is a one-to-one correspondence between the $k$-simplices of $\trisp$ labelled with $\{c_1, \dotsc, c_{d-k}\}$ and the connected components of $G(c_1, \dotsc, c_{d-k})$.
	\end{itemize}
\end{theorem}
Moreover, orientability of the colored triangulation translates into bipartiteness of its colored graph.
\begin{remark}
	In general, the 1-skeleton of the dual is not enough to reconstruct the whole triangulation. Here it works because one uses the \emph{colored} 1-skeleton, where the colors of the edges carry the necessary information to re-build $\trisp$ from $\graph(\trisp)$. The third item above shows that instead of considering directly the dual complex $\trisp^*$, it can be re-constructed from $\graph(\trisp)$ by looking at the connected components of $\graph(c_1, \dotsc, c_{d-k})$. So the duality is as follows.
	
	\begin{center}
		\begin{tabular}{|c|c|}
			\hline
			In $\graph(\trisp)$ & In $\trisp$\\
			\hline
			Vertex & $d$-simplex\\
			Edge & $(d-1)$-simplex\\
			Bicolored cycle & $(d-2)$-simplex\\
			Tricolored connected component & $(d-3)$-simplex\\
			\vdots & \vdots\\
			$(d-1)$-colored connected component & edge\\
			$d$-colored connected component & vertex\\
			\hline		\end{tabular}
	\end{center}
	
	We have already detailed the 2-dimensional case. At $d=3$, the tetrahedra of $\trisp$ are represented as the vertices of $\graph(\trisp)$, the triangles of $\trisp$ as the edges of $\graph(\trisp)$, the edges of $\trisp$ with colors $\{a, b\}$ as the bicolored cycles with colors $\{a, b\}$ of $\graph(\trisp)$ (this is illustrated in Figure \ref{fig:DualFace}) and the vertices of $\trisp$ with colors $\{a, b, c\}$ as the connected components of the subgraph of $\graph(\trisp)$ with colors $a, b, c$ only (this is illustrated in Figure \ref{fig:DualBubble}). The example of Figure \ref{fig:ColoredGraph} has a single bicolored cycle with colors $\{1,2\}$ and the graph $\graph(0,1,2)$ has a single connected component. This means that the corresponding colored triangulation has a single edge with colors $\{1,2\}$ and a single vertex with colors $\{0,1,2\}$.
	
	The connected components of $G(c_1, \dotsc, c_{d-k})$ are usually called $(d-k)$-bubbles. We will soon focus on a special case of $d$-bubbles obtained by removing the color 0. %We will nevertheless use the full notion of bubbles in Section \ref{sec:Topology} to investigate topology.
\end{remark}

\begin{figure}
	\includegraphics[scale=.4]{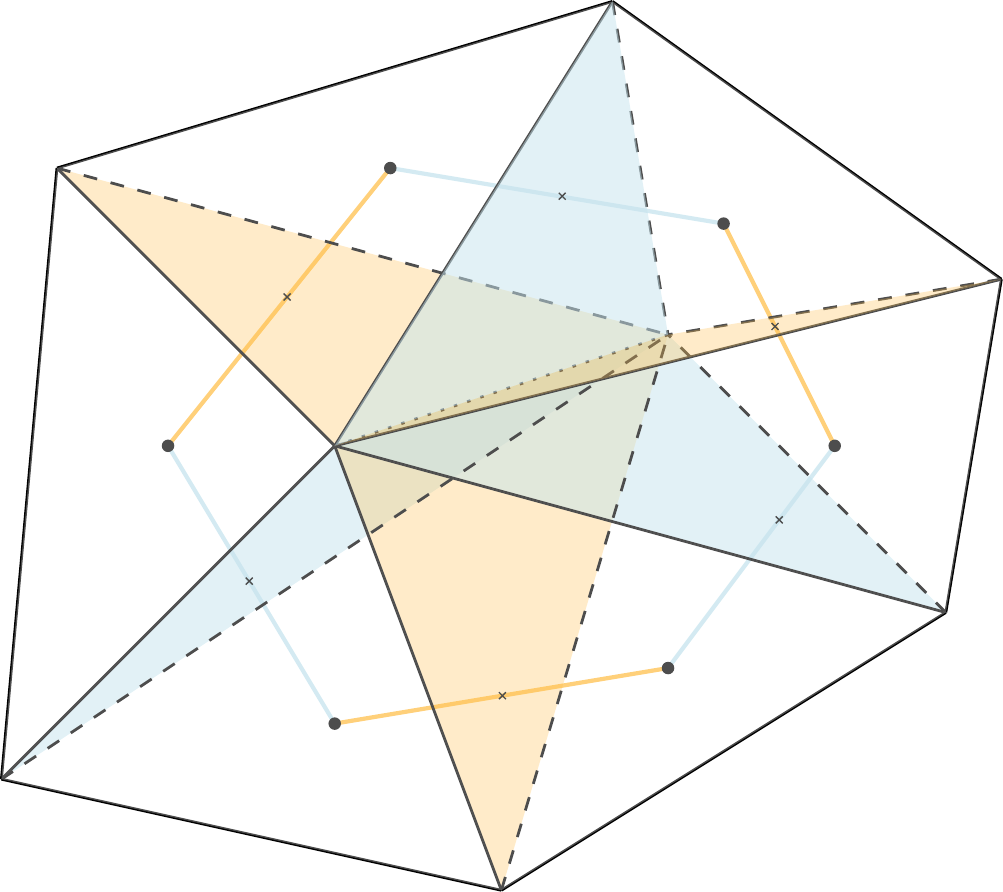}
	\caption{\label{fig:DualFace} The dual bicolored cycle of $\graph(\trisp)$ representing an edge of $\trisp$ in three dimensions.}
\end{figure}
\begin{figure}
	\includegraphics[scale=.45]{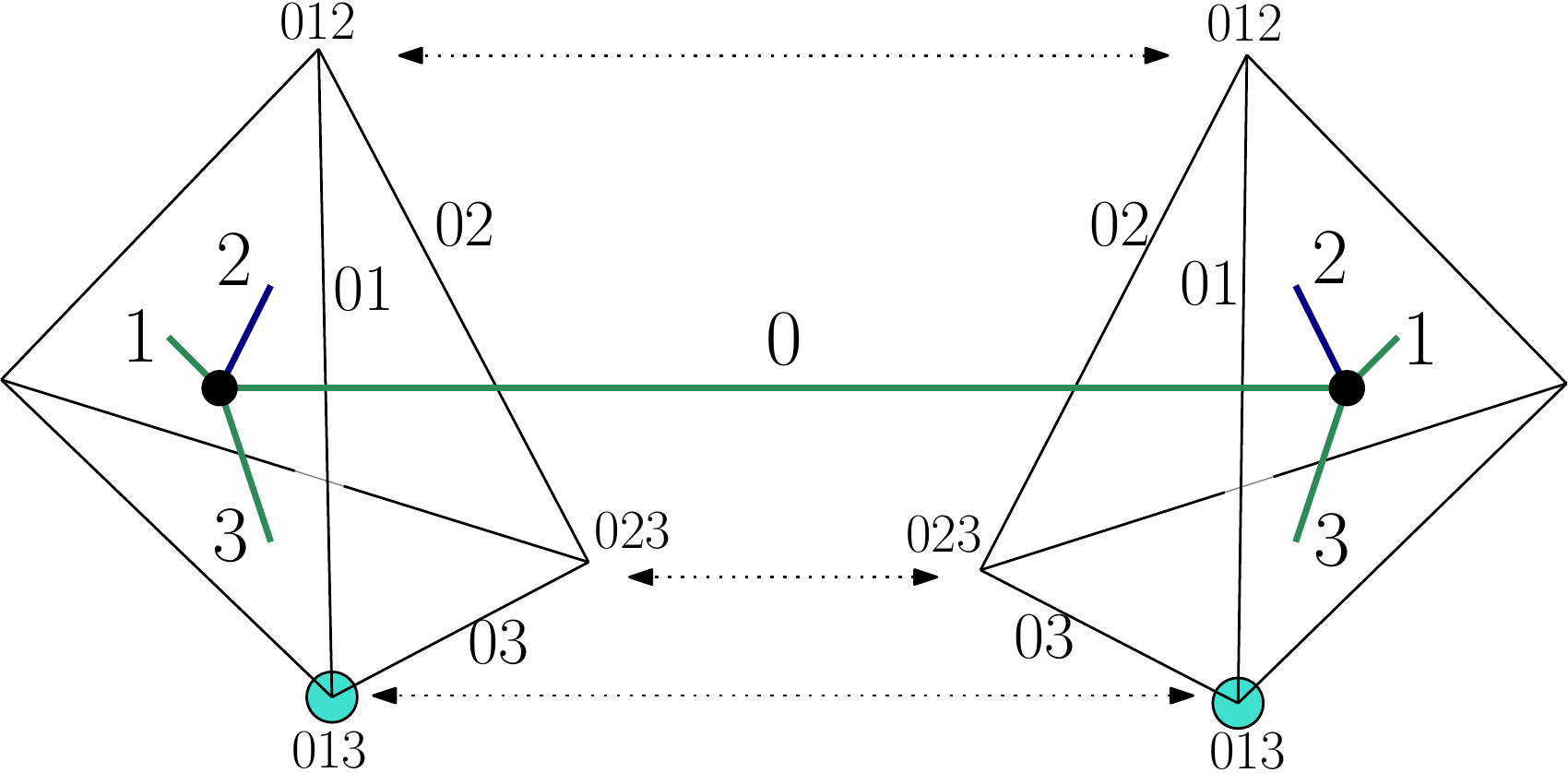}
	\caption{\label{fig:DualBubble} The vertex of colors $\{0,1,3\}$ is represented in $\graph(\trisp)$ as a connected component of the subgraph with the colors 0, 1, 3 (green edges).}
\end{figure}

\begin{proof}
	The first two items are trivial. %The correspondence between closed, connected colored triangulations and connected colored graphs is explained above and relies on the fact that the gluing of two colored $d$-simplices is entirely determined by the color of the $(d-1)$-simplex they share. Obviously, representing $d$-simplices by vertices and their connectivity by edges produces the 1-skeleton of the dual. The second item of the theorem is also obvious and just details the correspondence.
	We prove the third point. Each $k$-dimensional subsimplex of a $d$-simplex is identified by a $(d-k)$-tuple of colors. Say if $\sigma$ is a $d$-simplex, denote $\sigma(c_1, \dotsc, c_{d-k})$ the $k$-simplex with colors $c_1, \dotsc, c_{d-k}$. In $\graph(\trisp)$, $\sigma$ is represented as a vertex $v_{\sigma}$. Its boundary $(d-1)$-simplices $\sigma(c)$ are represented as half-edges of color $c=0, \dotsc, d$ incident to $v_{\sigma}$. It comes that $\sigma(c_1, \dotsc, c_{d-k})$ in $\sigma$ is identified by the $(d-k)$-tuple of half-edges which carry the colors $c_1,\dotsc, c_{d-k}$.
	
	The gluing rule for colored simplices is that when $\sigma$ and $\sigma'$ are glued along a $(d-1)$-simplex $\sigma(c)$ of color $c$, they identify two by two their subsimplices whose color labels contain $c$. Therefore, when $\sigma(c_1, \dotsc, c_{d-k})\subset \sigma$ is identified with $\sigma'(c_1, \dotsc, c_{d-k})\subset \sigma'$, it translates in $\graph(\trisp)$ into the fact that when $v_\sigma$ and $v_{\sigma'}$ are connected by an edge of color $c \in\{c_1, \dotsc, c_{d-k}\}$, the half-edges of colors $c_1, \dotsc, c_{d-k}$ incident to $v_{\sigma}$ and $v_{\sigma'}$ represent the same $k$-simplex of $T$.
	
	Denote $\graph(c_1, \dotsc, c_{d-k})$ the subgraph of $\graph(\trisp)$ which only retains the edges of colors $c_1, \dotsc, c_{d-k}$. A connected component of $G(c_1, \dotsc, c_{d-k})$ thus represents a $k$-simplex $\sigma(c_1, \dotsc, c_{d-k})$ of $\trisp$. Moreover, two connected components represent different $k$-simplices of $T$. Indeed, the only way for a $k$-simplex with colors $\{c_1, \dotsc, c_{d-k}\}$ to be shared by two $d$-simplices $\sigma, \sigma'$ is that they are glued along a $(d-1)$-simplex of color $c\in\{c_1, \dotsc, c_{d-k}\}$. This is equivalent to $v_{\sigma}$ and $v_{\sigma'}$ being connected by an edge of color $c$ in $\graph(\trisp)$.
\end{proof}

Consider labeled colored graphs, i.e. every white vertex receives a label from 1 to $n$, and the same for black vertices. Then colored graphs with $d+1$ colors are described by a vector of $d+1$ permutations $(\sigma_0, \dotsc, \sigma_{d}) \in \mathfrak{S}_n^{d+1}$. For every color $c\in[0..d]$, there is an edge of color $c$ between the white vertex of label $v$ and the black vertex of label $w$ if and only if $\sigma_c(v)=w$. For historical reasons, including Regge's discretization of the Einstein-Hilbert action on equilateral triangulations, we will attempt to classify colored triangulations at fixed number of $d$-simplices with respect to the number of $(d-2)$-simplices. On colored graphs, the latter are the bicolored cycles, which are in fact easily encoded through the permutations. A bicolored cycle with colors $\{a,b\}$ and of length $2p$ is a $p$-cycle of $\sigma_a^{-1} \sigma_b$. This implies that in a colored triangulation $\trisp_n$, the number of $(d-2)$-dimensional simplices $\Delta_{d-2}(\trisp_n)$ is found through the dual colored graph as
\begin{equation}
	\Delta_{d-2}(\trisp_n) = \sum_{0\leq a<b\leq d} \text{$\#$ cycles of $(\sigma_a^{-1}\sigma_b)$}.
\end{equation}

%The fundamental theorem of colored triangulations in topology is that they represent PL-pseudomanifolds. In two dimensions, there are only manifolds, but in three dimensions, if a connected component of a subgraph $H\subset G(T)$ with 3 colors has its canonical embedding which is a surface of non-zero genus, it means that it represents a vertex in $T$ whose neighborhood is not a 3-ball and has a conical singularity. The second important theorem for topology is that every manifold admits a representation as a colored triangulation (e.g. by barycentric subdivision of a non-colored one).

\subsection{Bijection with constellations} We recall duality for maps. The dual of a map $\m$ is a map $\m^*$ such that the vertices of $\m$ are the faces of $\m^*$ and the other way around. If there is an edge $e$ incident to the faces $f_1, f_2$ (they may be the same) in $\m$, then one draws an edge $e^*$ between the corresponding vertices $f_1^*, f_2^*$ of $\m^*$. The cyclic order of the edges incident to $f^*$ in $\m^*$ is the cyclic order of the corresponding edges around the face $f$ in $\m$. Duality is an involution. Notice that the dual of a $p$-angulation is a map whose vertices all have degree $p$, and the other way around, see Figure \ref{fig:QuadrangulationDual}.
\begin{figure}
	\includegraphics[scale=.65]{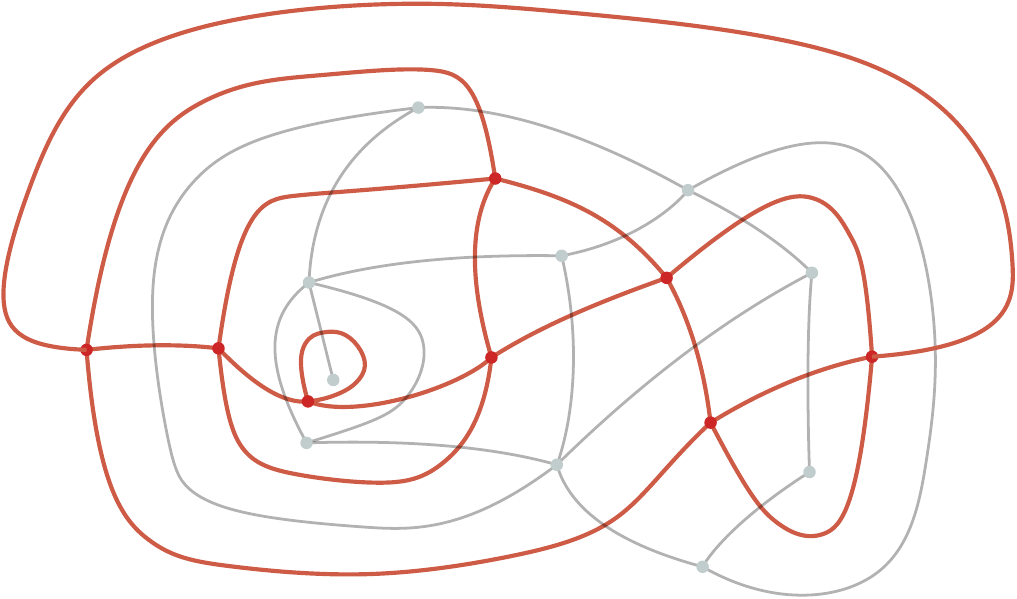}
	\caption{\label{fig:QuadrangulationDual}The dual map to a quadrangulation is a map whose vertices all have degree 4.}
\end{figure}

A color cycle is a cyclic order on the set of colors $[0..d]$. They correspond to permutations $\rho$ of $[0..d]$ made of a single cycle and defined up to inversion.

\begin{proposition}{}{}
	For each color cycle, there is a bijection between $(d+1)$-colored graphs and $d$-constellations.
\end{proposition}

\begin{proof}
	Consider the color cycle $(0, \rho(0), \dotsc, \rho^{d}(0))$ for a permutation $\rho$ of $[0..d]$ made of a single cycle. We define the map $\m_\rho(\graph)$ by the two-dimensional embedding of $\graph$ with this cyclic order around white vertices and its inverse around black vertices. It has as faces the bicolored cycles of colors $\{\rho^i(0), \rho^{i+1}(0)\}$ (the graph is $\graph$ itself).
	
	We then consider the map $\m^*_\rho(\graph)$ dual to $\m_\rho(\graph)$. Its vertices, dual to the faces of $\m_\rho(\graph)$, are thus labeled by pairs of colors $\{\rho^i(0), \rho^{i+1}(0)\}$. To simplify this vertex coloring, we use the color $i\in[0..d]$ instead. The map $\m^*_\rho(\graph)$ has white faces (dual to white vertices) and black faces (dual to black vertices), all of degree $d+1$. This is illustrated in Figure \ref{fig:Jacket}.
	
	To transform $\m^*_\rho(\graph)$ into a constellation, it is enough to remove its vertices of color $0$ (or any color for that matter) and chop the black faces by adding an edge between the vertices of colors $1$ and $d$, as follows
	\begin{equation} \label{PureConstellation}
		\includegraphics[scale=.4,valign=c]{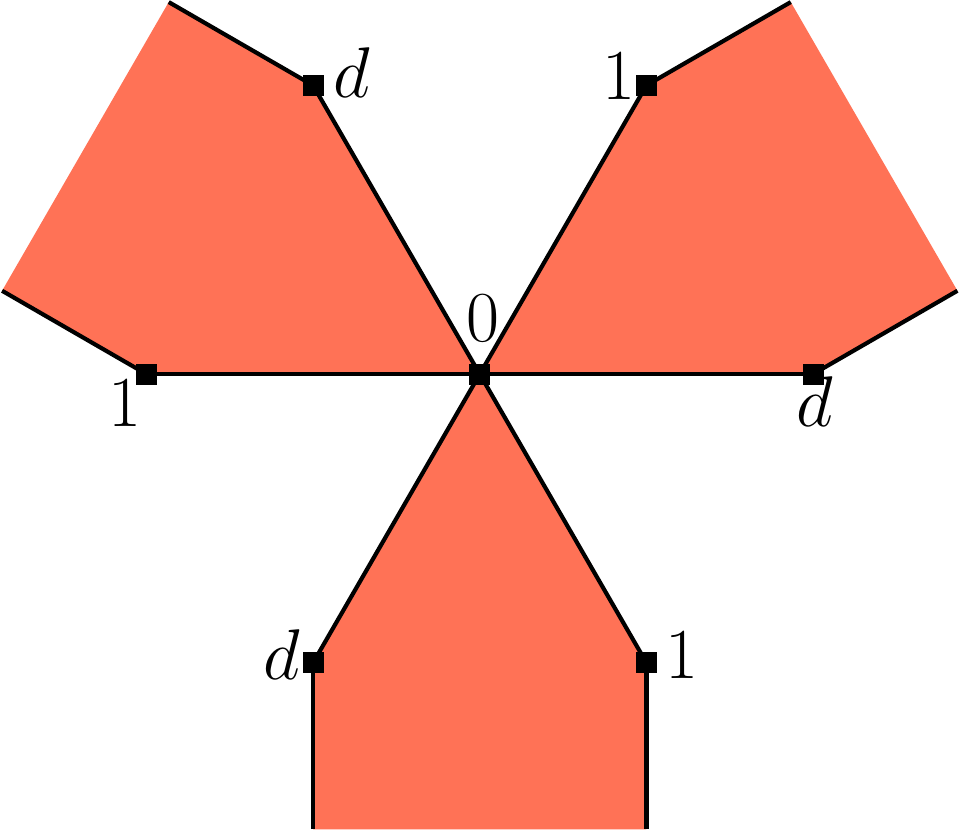} \qquad \to\qquad \includegraphics[scale=.4,valign=c]{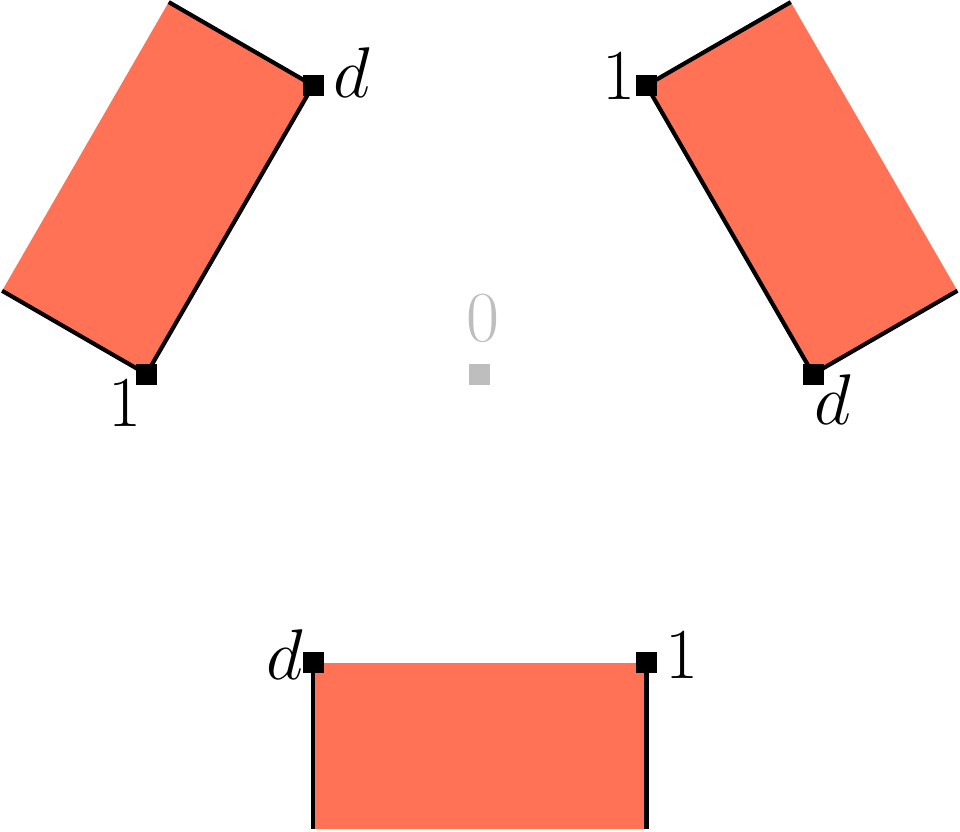}
	\end{equation}
	It is easily checked to be a $d$-constellation, which we denote $\bar{\m}_\rho(\graph)$.
\end{proof}
\begin{figure}
	\includegraphics[scale=.65]{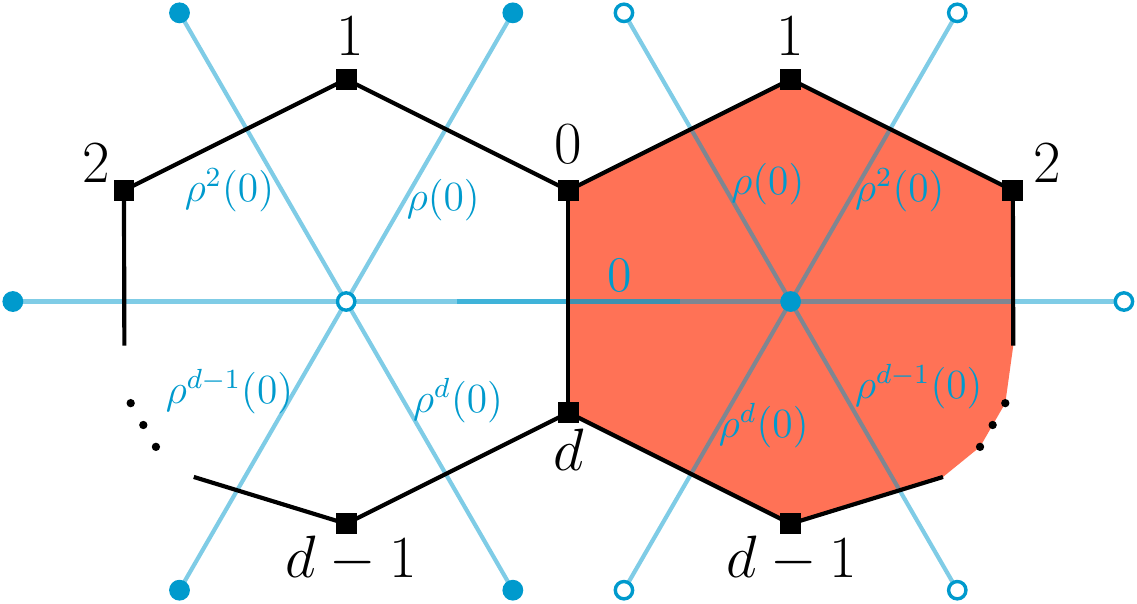}
	\caption{\label{fig:Jacket}A piece of a colored graph (in blue) embedded according to the color cycle $(0, \rho(0), \dotsc, \rho^d(0))$. The dual has black and white faces all of degree $d+1$.}
\end{figure}
As it is clear in this proof, the faces of $\bar{\m}_\rho(\graph)$ are the bicolored cycles with colors $\{\rho^k(0),\rho^{k+1}(0)\}$ of $\graph$, and thus a proper subset of all its bicolored cycles. This explains why maximizing the number of bicolored cycles does not correspond to singling out planar constellations.

\subsection{Combinatorics versus topology, Gurau's degree} While Euler's formula exist in higher dimensions, it is not as strong as in two dimensions. For instance, all orientable 3-manifolds have vanishing Euler's characteristic. So we look for another generalization of \eqref{Euler}. We have seen in Equation \eqref{EulerBound} that it can also be interpreted combinatorially as saying that there is a bound on the number of vertices which grows linearly with the number of faces. Gurau's theorem is precisely a combinatorial extension of this idea to any $d$-dimensional colored triangulations.

\begin{theorem}{}{Gurau}\cite{1/NExpansion}
	Gurau's degree defined as
	\begin{equation}
		\omega(\trisp) = d + \frac{d(d-1)}{4} \Delta_d(\trisp) - \Delta_{d-2}(\trisp)
	\end{equation}
	is a non-negative integer for any colored triangulation $\trisp$, where $\Delta_k(\trisp)$ is the number of $k$-dimensional simplices of $\trisp$.
	
	Moreover, the triangulations of vanishing Gurau's degree are $d$-spheres characterized as
	\begin{itemize}
		\item at $d=2$ the planar ones,
		\item at $d=3$ the melonic ones, defined below \cite{Melons}.
	\end{itemize}
\end{theorem}

It is easy to check that $\omega(\trisp)$ is equivalent to the genus in two dimensions. In higher dimensions, it is however \emph{not} a topological invariant. It is instead a genuine extension of the genus in purely combinatorial terms: it provides a bound on the number of $(d-2)$-simplices which grows linearly with the number of $d$-simplices,
\begin{equation} \label{GurauBound}
	\Delta_{d-2}(\trisp) \leq d + \frac{d(d-1)}{4} \Delta_d(\trisp).
\end{equation}
Gurau's degree thus measures the deficit of $(d-2)$-simplices with respect to their maximal value. This way, Gurau's degree can be used to classify colored triangulations and this classification has been performed by Gurau and Schaeffer in \cite{GurauSchaeffer}. For each value of $\omega$, it states that there is a finite number of ``core-graphs'' (called schemes) from which all graphs of degree $\omega$ can be recovered by simple substitutions.

In any dimensions, triangulations which maximize the number of $(d-2)$-simplices at fixed number of $d$-simplices are those of vanishing Gurau's degree. In two dimensions, they are the planar ones, which as we know are also all those homeomorphic to the sphere. For any $d\geq 3$ they are called melonic \cite{Melons}. Using the correspondence with colored graphs, melonic triangulations can be described as certain series-parallel graphs called melonic graphs, also in bijection with $(d+1)$-ary trees. A \emph{melonic dipole} is a pair of vertices connected by exactly $d$ edges. A melonic insertion of color $c$ is the move inserting a melonic dipole on an edge of color $c$,
\begin{equation} \label{MelonicInsertion}
	\includegraphics[scale=.5,valign=c]{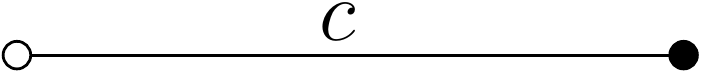} \quad \to \quad \includegraphics[scale=.5,valign=c]{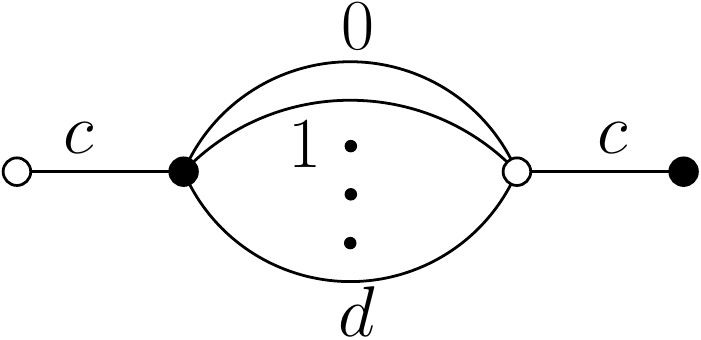}
\end{equation}
%It is a sequence of two series and $d$ parallel extensions.
A melonic graph is then a colored graph obtained from the 2-vertex graph by arbitrary sequences of melonic insertions. As metric spaces, melonic triangulations have been shown to converge to Aldous' continuous random tree \cite{MelonsBP}, in contrast with planar maps which converge to the Brownian sphere \cite{Miermont2013, LeGall2013}.

\begin{proof}[Proof of \Cref{thm:Gurau}, \cite{1/NExpansion}]
	The idea is that for each color cycle $\rho$, the constellation $\bar{\m}_\rho(\graph)$ has all the vertices and edges of $\graph$ but only a subset of its bicolored cycles (those with colors $\{\rho^i(0), \rho^{i+1}(0)\}$), as faces. By summing the genera of those constellations over all color cycles, one recovers a quantity which is symmetric in the colors, and also non-negative. By using Euler's relation for each $\bar{\m}_\rho(\graph)$, one finds that this is Gurau's degree,
	\begin{equation}
		\omega(\trisp) = \sum_\rho g(\bar{\m}_\rho(\graph(\trisp))).
	\end{equation}
\end{proof}

%%%%%%%%%%%%%%%%%%%
\section{Colored building blocks and their bubbles} \label{sec:Bubbles}

\subsection{Colored building blocks at $d=2$} Combinatorial maps can be defined as gluings of polygons along their edges. We have defined above colored triangulations in arbitrary dimension, which generalize the tricolored triangulations in two dimensions. But we have not defined yet a generalization of the polygons in higher dimensions. A key point is that polygons can be triangulated, so knowing triangulations in higher dimensions should enable us to construct generalizations of polygons. In fact, as we have seen, colored triangulations in two dimensions are equivalent to bipartite maps. A face of a map of degree $2n$ corresponds to a $2n$-gon made of colored triangles such that
\begin{itemize}
	\item the triangles are glued along their sides of colors 1 and 2,
	\item the sides of color 0 constitute the boundary.
\end{itemize}
In particular, the color 0 does not play a role in the building of the polygon from triangles. Notice that the boundary can be seen as a 2-colored triangulation, by removing the color 0 from the labels of the edges and vertices. Clearly the polygon is determined by its boundary triangulation and can be reconstructed by taking the cone over the boundary. Through coning, the vertices of colors $\{0,1\}$ and $\{0,2\}$ respectively give rise to edges of color 1 and 2 respectively, and the boundary edges, with colors 0, give rise to the colored triangles.

In the dual, tricolored graph representation, a $2n$-gon is represented as a bicolored cycle with colors $\{1,2\}$. One can then think of a tricolored graph as a collection of such cycles connected by edges of color 0. This also shows that restricting bipartite maps to some prescribed face degrees corresponds to sampling tricolored graphs according to the same prescribed degrees for their bicolored cycles with colors $\{1,2\}$.

\subsection{Colored building blocks and bubbles} We thus define building blocks in higher dimensions using the same idea: glue simplices along all their colors except 0. The latter color will lie on the boundary and it will be the one involved in gluing building blocks together. 

\begin{definition} [Colored building blocks and bubbles]
	A \emph{colored building block} (CBB) is a connected colored triangulation with a boundary, whose $(d-1)$-simplices of color 0 constitute the boundary.
	
	A $d$-\emph{bubble}, or simply bubble, is a $d$-colored graph with color set $[1..d]$.
\end{definition}
We will show below the following properties.
\begin{proposition}{}{CBB}
	{}
	\begin{itemize}
		\item There is a bijection between CBBs and $d$-colored triangulations, given by taking the boundary $\partial\CBB$ and whose inverse is
		\begin{equation}
			\CBB = \left(\partial \CBB \times [0,1]\right)\,/\, \left(\partial \CBB\times \{1\}\right),
		\end{equation}
		i.e. CBBs are cones over their boundaries.
		\item CBBs are described by bubbles. If $\CBB$ is a CBB, its bubble $\bb(\CBB)$ is
		\begin{equation}
			\begin{aligned}
				\bb(\CBB) &= \operatorname{col-1-sk}(\CBB^*)\setminus \{\text{Edges of color 0}\}\\
				&= \operatorname{col-1-sk}(\partial \CBB^*)
			\end{aligned}
		\end{equation}
		i.e. it can be described equivalently as the colored 1-skeleton of the dual $\CBB^*$ with the edges of color 0 removed, and also as the colored 1-skeleton of the dual to the boundary triangulation.
		\item Any closed colored triangulation can be obtained by gluing CBBs along their boundary subsimplices.
	\end{itemize}
\end{proposition}
The last point is trivial in terms of dual colored graphs, any colored graph with colors $[0..d]$ can be obtained from a collection of bubbles $\bb_1, \dotsc, \bb_N$ connected by edges of color 0. Here, ``connected by edges of color 0'' must follow the rules of colored graphs, i.e. respecting bipartiteness and so that every vertex is incident to exactly one edge of color 0.

\subsection{The case $d=3$} While two-dimensional CBBs are determined by a single integer (the size of their boundaries, $n$ for a $2n$-gon), this is not true in higher dimensions. However, CBBs are always determined by their boundary triangulations. In three dimensions, a CBB is a gluing of tetrahedra along all their triangles with colors 1, 2, 3, while triangles of color 0 constitute the boundary triangulation. From \Cref{thm:CBB}, there is a one-to-one correspondence 
\begin{itemize}
	\item between the boundary triangles and the tetrahedra, 
	\item between the boundary edges and the triangles which do not lie on the boundary, 
	\item between the boundary vertices and edges which do not lie on the boundary.
\end{itemize}
They are obtained by taking the cone over the boundary triangulation and inversely by projecting on the boundary. %with $\partial_0$.
\begin{figure}
	\includegraphics[scale=.5,valign=c]{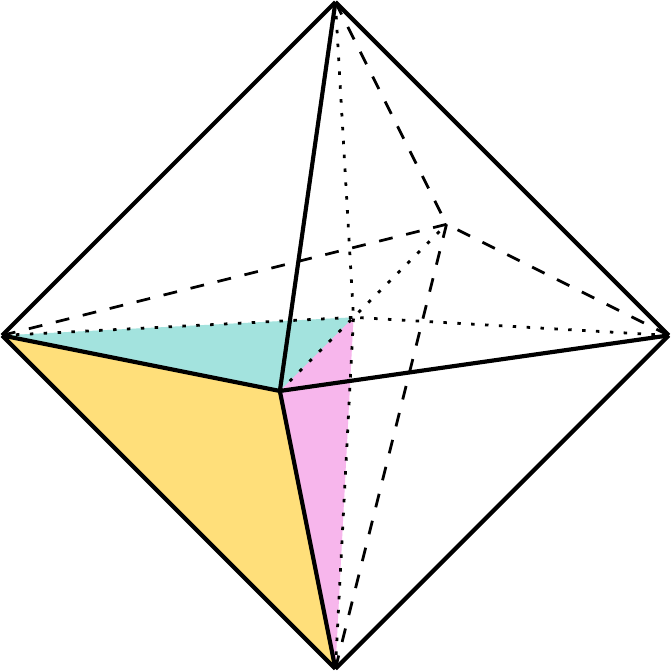}
	\hspace{2cm}
	\includegraphics[scale=.6,valign=c]{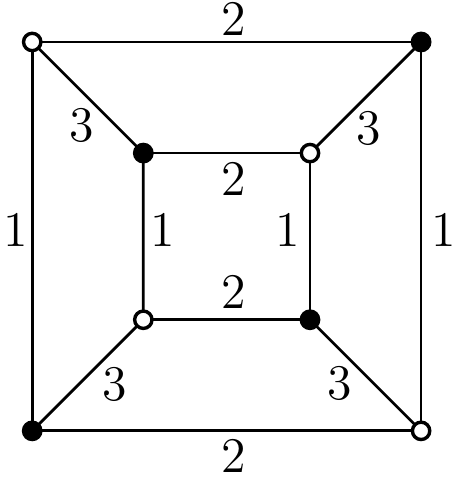}
	\caption{\label{fig:Octahedron} The octahedron is a CBB made of eight tetrahedra, and eight boundary triangles of color 0. Its bubble is the colored 1-skeleton of the dual with the color 0 removed, or equivalently the colored 1-skeleton of the dual to the boundary triangulation. The color 0 is only necessary for the gluing to other bubbles.}
\end{figure}

An example is given by the octehadron in Figure \ref{fig:Octahedron}, which is a colored gluing of eight tetrahedra. More generally, a CBB has a single interior vertex $\sigma(123)$, whose color label is $\{1, 2, 3\}$ (the only one not containing 0) and which is shared by all tetrahedra (it is the ``tip'' of the cone). Each tetrahedron contributes to a boundary triangle. Triangles of color $a=1, 2, 3$ are not contained in the boundary but intersect it at edges with colors $\{0,a\}$. Edges of color labels $\{a, b\}$ for $a, b =1, 2, 3$ do not lie on the boundary; they each have one vertex which is $\sigma(123)$ and another vertex with label $\{0, a, b\}$ on the boundary. Edges with color labels $\{0, a\}$ lie on the boundary and connect vertices with labels $\{0, a, b\}$ to $\{0, a, c\}$. We thus see that the boundary is a colored triangulation, whose set of colors is $\{1,2,3\}$ after removing the irrelevant label 0. 

In the dual picture, a CBB is represented by a bubble with colors $\{1,2,3\}$. This is a tricolored graph and by the above \Cref{thm:CBB}, it is the colored graph representation of the boundary triangulation. In the case of the octahedron, this gives the graph of the cube, as shown in Figure \ref{fig:Octahedron}.

\begin{figure}
	\includegraphics[scale=.55]{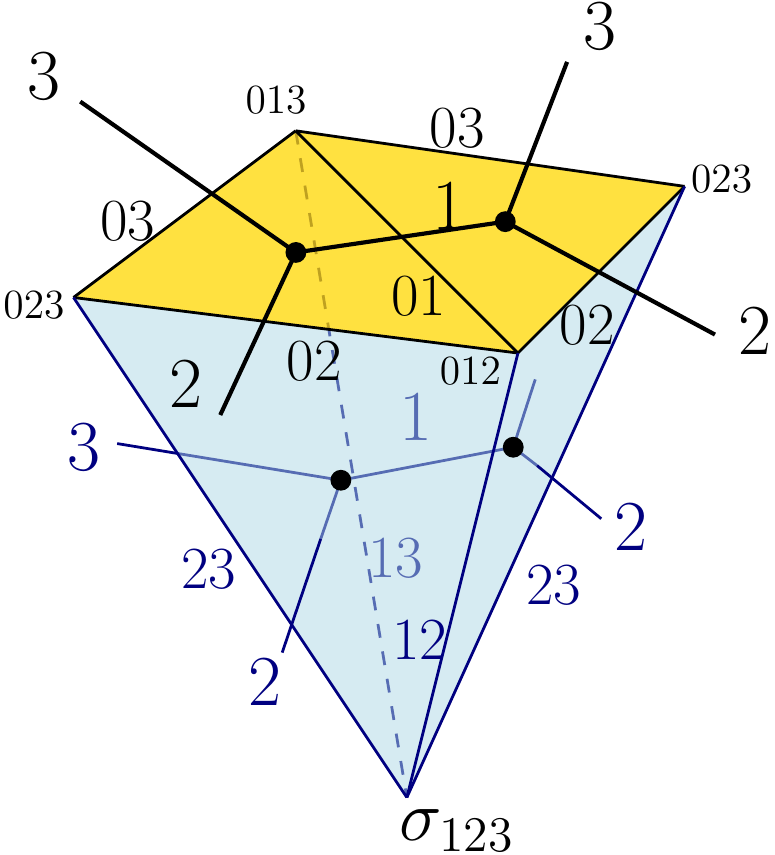}
	\caption{\label{fig:GluedTetDuals} A piece of a CBB $\CBB$ in dimension 3: two tetrahedra glued along their triangle of color 1. The boundary triangles (of color 0) are in yellow and the interior triangles are in blue. The CBB is the cone over its boundary, and both coning and projecting on the boundary are compatible with the colored gluing rule.	The 1-skeleton of the dual with the color 0 removed is depicted in the intoerior and is clearly the same as the 1-skeleton dual to the boundary.}
\end{figure}
\begin{proof}[Proof of \Cref{thm:CBB}]
	For the first point, one needs to show that the boundary triangulation of a CBB is a colored triangulation with one less color. This should be clear because by definition of the colored attaching map, the rules of colored gluings propagate to subsimplices of lower dimensions. It is also true the other way around: the cone over a $(d-1)$-dimensional colored triangulation is a $d$-dimensional colored triangulation with a boundary. In particular, the cone over a $(d-1)$-dimensional colored simplex is a $d$-dimensional colored simplex. Then one checks that the rules of the colored attaching map propagate correctly from the boundary to the bulk under coning, as shown in Figure \ref{fig:GluedTetDuals}. Therefore there is a one-to-one correspondence between the $k$-simplices of $\CBB$ and the $(k-1)$-simplices of its boundary, for $k=1, \dotsc, d$.
	
	To prove the second point, we consider 
	\begin{equation}
			\bb(\CBB) \coloneqq \operatorname{col-1-sk}(\CBB^*)\setminus \{\text{Edges of color 0}\},\quad \text{and} \quad \bb'(\CBB) = \operatorname{col-1-sk}(\partial \CBB^*)
	\end{equation}
	They are both $d$-colored graphs, i.e. bubbles and have the same set of vertices. It remains to prove that their edges have the same incidence on vertices, which follows from the coning above.
\end{proof}

Notice that this is consistent with \Cref{thm:ColoredGraphs}. There we proved that a connected component of the subgraph with all colors except 0 represents in a colored triangulation $\trisp$ a vertex. From the point of view of $\CBB$, this is the unique interior vertex $\sigma(1, \dotsc, d)$, and $\partial\CBB$ is its \emph{link}.

\subsection{CBBs with few simplices and their bubbles} It is simpler to consider the bubbles instead of the triangulations in order to find explicit CBBs (and more generally in order to prove properties of colored triangulations). Let us look at the smallest bubbles and their corresponding CBBs, see Figure \ref{fig:Bubbles}. The simplest one has two vertices connected by all edges with colors $c=1, \dotsc, d$. It encodes the gluing of two $d$-simplices along all their $(d-1)$-simplices except those of color 0, like in two dimensions two colored triangles glued along two edges. Topologically it is a $d$-ball with two boundary $(d-1)$-simplices.

Bubbles with four vertices, called \emph{quartic} bubbles, are characterized by their color sets $\colset\subset [1..d]$, whose colors are those of the parallel edges between a white and a black vertex. They are represented in Figure \ref{fig:Bubbles} with $\colset = \{c_1, \dotsc, c_q\}$, and denoted $\quart(\colset)$. By symmetry, we can use $\colset$ or its complement $[1..d]\setminus \colset$. By convention, we take $q \in\{1, \dotsc, (d-1)/2\}$ if $d$ is odd and $q\in\{1, \dotsc, d/2\}$ if $d$ is even. Notice that for $d=3$, only $q=1$ is possible, i.e. $\colset = \{c\}$ for $c=1, 2, 3$. Due to color relabeling, this leaves three distinct bubbles with four vertices at $d=3$. They are also homeomorphic to the $d$-ball.
\begin{figure}
	\includegraphics[scale=.35,valign=c]{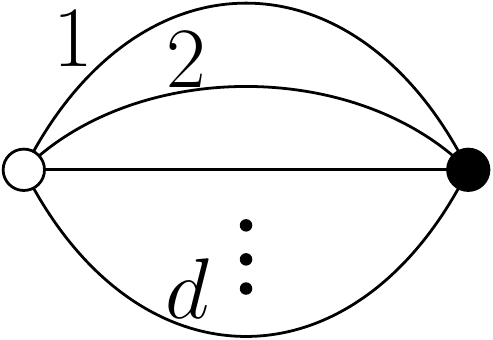}
	\hspace{1cm}
	$\quart(\colset) = $\includegraphics[scale=.35,valign=c]{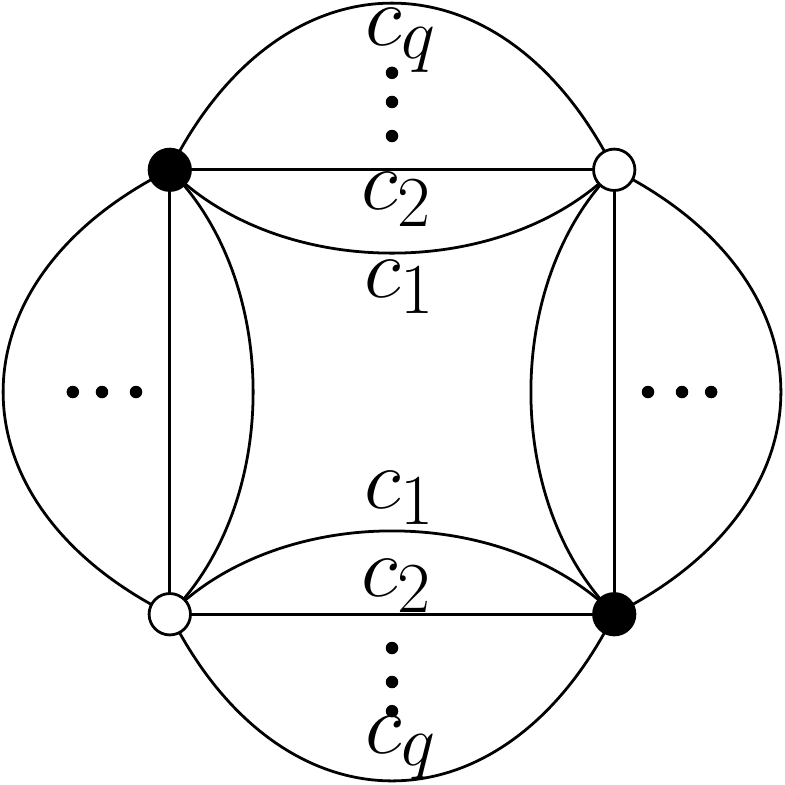}
	\hspace{1cm}
	$\quart(\{1\})_{|d=3} = $\includegraphics[scale=.35,valign=c]{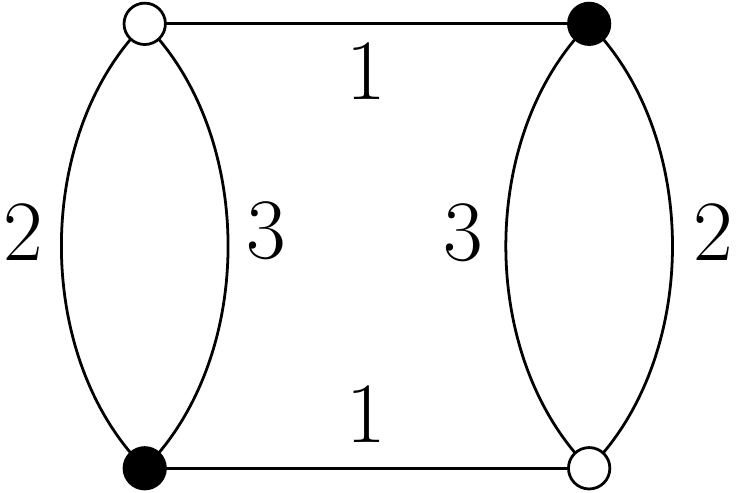}
	%	\hspace{1cm}
	%	\includegraphics[scale=.35,valign=c]{6VertexBubble2d.pdf}
	\caption{\label{fig:Bubbles}Examples of bubbles: the 2-vertex bubble; the 4-vertex bubbles characterized by $\colset=\{c_1, \dotsc, c_q\}$; only $q=1$ is possible at $d=3$ with 4 vertices, with three possible colorings.}
\end{figure}

%%%%%%%%%%%%%%%%%%%%%%
\section{3-balls and planar bubbles} 
At $d=2$, bubbles are bicolored cycles alternating the colors 1 and 2 and are thus characterized by a single integer, the length $2p$ of the cycle. However, at $d=3$, bubbles are dual to tricolored triangulations which cannot be characterized by just an integer anymore. In particular, there is not just one bubble at fixed number of vertices like for $d=2$. Up to color relabeling, the three-dimensional bubbles with six vertices are the following.
\begin{equation} \label{6VertexBubbles3d}
	\begin{array}{c} \includegraphics[scale=.3]{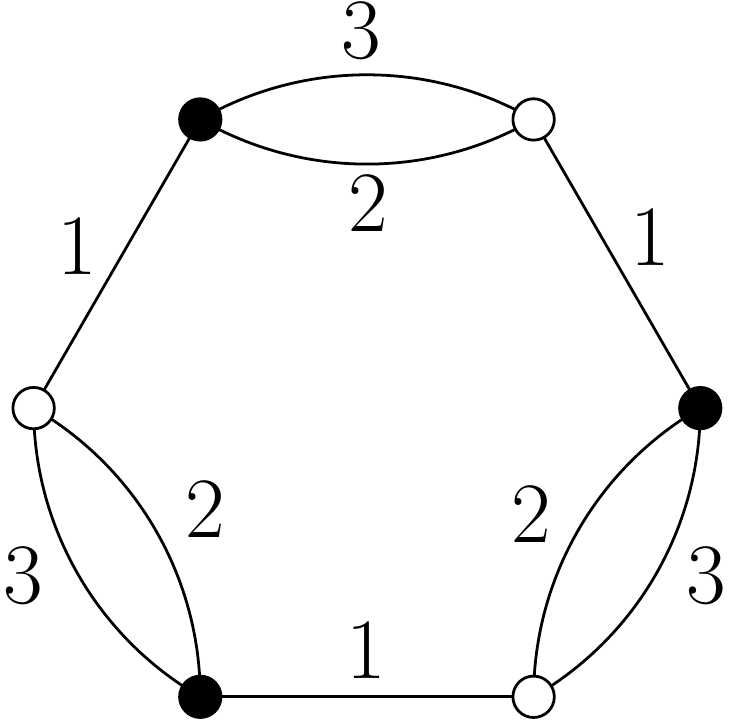} \end{array} \hspace{1.5cm}  \begin{array}{c} \includegraphics[scale=.3]{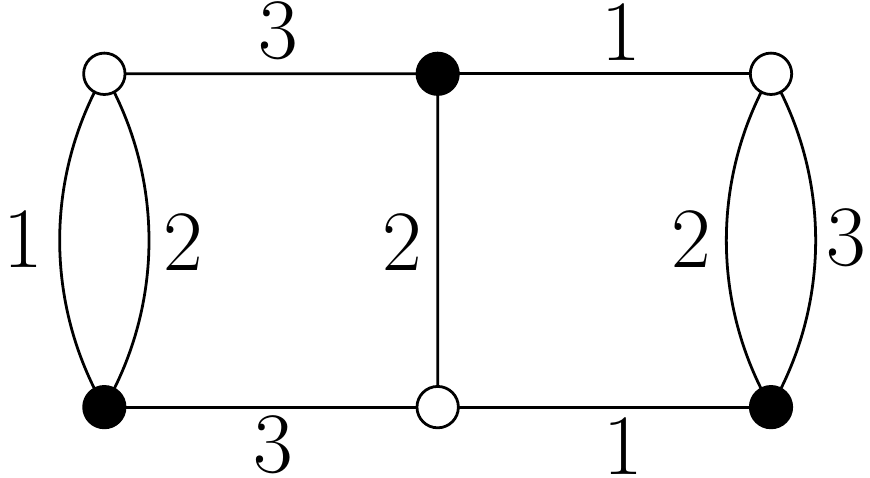} \end{array} \hspace{1.5cm} \begin{array}{c} \includegraphics[scale=.3]{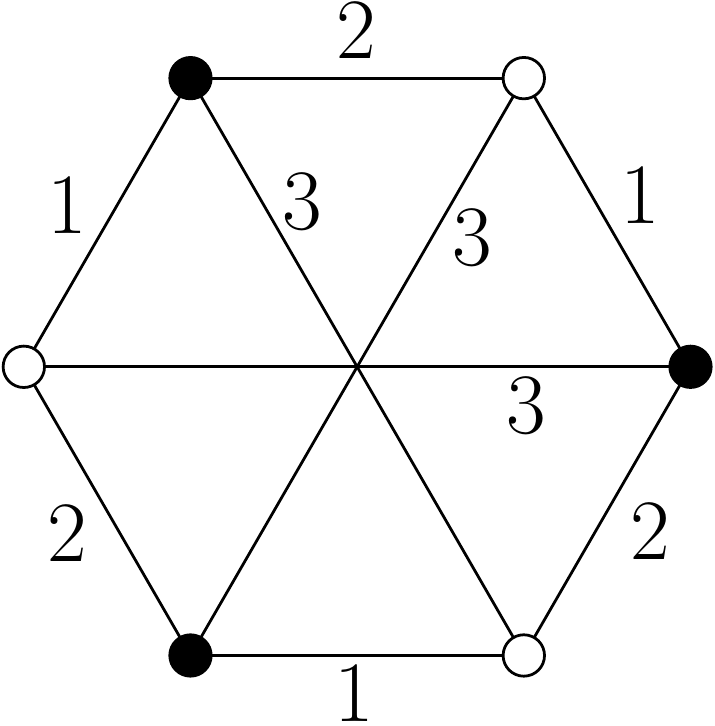} \end{array}
\end{equation}

Recall from \Cref{thm:CanonicalEmbedding} that tricolored graphs have a canonical embedding as cubic maps (here the set of colors is $\{1,2,3\}$ instead of $\{0,1,2\}$). In \eqref{6VertexBubbles3d} we have used that canonical embedding, with the cyclic order $(123)$ around white vertices and $(132)$ around black vertices, so that the bubbles can be seen as maps (dual to the colored triangulations). In particular, they are planar, except for the one with $K_{3,3}$ as underlying graph, which is a torus.

As a corollary of \Cref{thm:CBB}, one can characterize the CBBs which are homeomorphic to the 3-ball at $d=3$ in terms of their boundary triangulation. We say that a bubble has genus $g$ if that is its genus from Equation \eqref{EulerColoredGraph} or equivalently if its canonical embedding as a cubic map, \Cref{thm:CanonicalEmbedding}, has genus $g$.

\begin{corollary}{}{PlanarBubbles} 
	At $d=3$, a CBB is homeomorphic to the 3-ball if and only if its bubble is planar.
\end{corollary}

%%%%%%%%%%%%%%%%%%%%%%%
\chapter{Topological recursion for weighted Hurwitz numbers} \label{sec:TR}
A partition $\lambda=(\lambda_1, \lambda_2, \dotsc, \lambda_l)$ is a finite, non-increasing, sequence of positive integers, i.e. $\lambda_1\geq \lambda_2 \geq \dotsb\geq \lambda_l>0$. Its size is $|\lambda|\coloneqq \sum_{i=1}^l \lambda_i$ and its length is $\ell(\lambda)\coloneqq l$. We use the notation $\lambda \vdash n$ if $|\lambda|=n$ and say that it is a partition of $n$.

Partitions label the conjugacy classes of the symmetric group. If $\sigma\in\mathfrak{S}_n$ has a cycle decomposition with $m_i(\sigma)$ cycles of length $i>0$, its conjugacy class is the set $\{\tau \sigma\tau^{-1}\}_{\tau\in\mathfrak{S}_n}$ and it consists in permutations which have the same cycle type. Its conjugacy class is thus labeled by the partition with $m_i(\sigma)$ parts equal to $i$. As is well-known, partitions also label the irreducible representations of the symmetric group.

%%%%%%%%%%%%%%%%%
\section{Weighted Hurwitz numbers} \label{sec:WeightedHurwitzNumbers}

\subsection{$(m+1)$-Factorizations of permutations} Let $m\geq 0$. A classical problem in combinatorics is that of enumerating all sequences of permutations of fixed conjugacy classes which factorize the identity, i.e. set $\lambda^{(0)}, \lambda^{(1)}, \dotsc, \lambda^{(m)}$ some partitions of $n$ and count all solutions $(\sigma_0, \sigma_1, \dotsc, \sigma_m)\in \mathfrak{S}_n^{m+1}$ of
\begin{equation} \label{Factorization}
	\sigma_0 \sigma_1 \dotsm \sigma_m = \mathbb{1},
\end{equation}
such that $\sigma_i$ lies in the conjugacy class $\lambda^{(i)}$, for $i=0, \dotsc, m$.

This is a direct translation of a famous problem by Hurwitz in eumerative geometry, which is counting the number of $n$-sheeted branched coverings of the sphere with $m$ branch points and ramification profiles over the branch points given by the partitions $\lambda^{(0)}, \lambda^{(1)}, \dotsc, \lambda^{(m)}$. The genus $g$ of the covering is then given by the Riemann-Hurwitz formula,
\begin{equation}
	2-2g = 2n + \sum_{i=0}^{m} (\ell(\lambda^{(i)})-n).
\end{equation}
%where $\ell(\lambda)$ is the number of parts of the partition $\lambda$.

\subsection{Interpretation as constellations} \label{sec:EquivalenceFactorizationsConstellations}
\begin{proposition}{}{}
	There is a one-to-one correspondence between factorizations of the identity and labeled $m$-constellations of size $n$, where the parts of $\lambda^{(c)}$ are the degrees of the vertices of color $c=1, \dotsc, m$ and the parts of $\lambda^{(0)}$ are the degrees of the faces.
\end{proposition}

Through that bijection, the Riemann-Hurwitz formula is equivalent to Euler's formula for the genus of constellations.

\begin{proof}
	We first give a representation of factorizations of permutations as special $(m+1)$-constellations which we then show to be equivalent to $m$-constellations as defined above. To each $i\in[1..n]$, we associate a star-vertex, i.e. a special vertex with $m+1$ incident edges carrying both the label $i$ and the color labels $c=1..m$ clockwise around the star-vertex. Then to every cycle of $\sigma_c$ of length $k$, we associate a vertex of color $c$ and degree $k$. If the cycle reads $(i_1 \dotsb i_k)$, the edges with labels $i_1, \dotsc, i_k$ and color label $c$ are attached to that vertex, respecting the cyclic order to the cycle,
	\begin{equation}
		(i_1 \dotsb i_k)_c \quad \leftrightarrow\quad \includegraphics[scale=.6,valign=c]{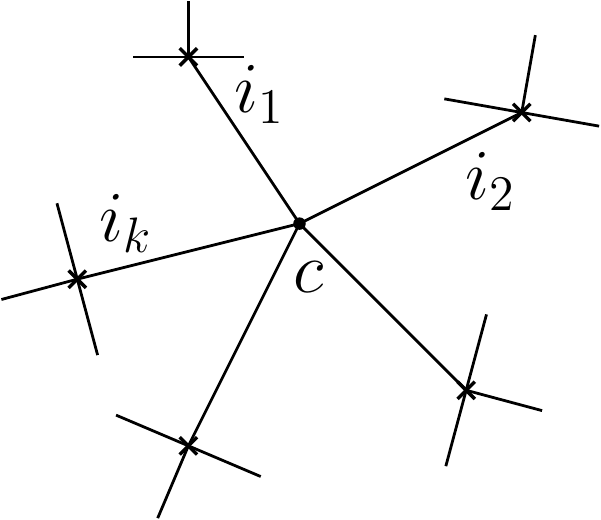}
	\end{equation}
	Here the star-vertices are represented as crosses and the subscript $c$ of the cycle indicates that it belongs to $\sigma_c$. One obtains a map. When one writes down the labels of the star-vertices encountered around each face, one finds that they are the cycles of the product $\sigma_0 \sigma_1 \dotsm \sigma_m$. Therefore a factorization of the identity is equivalent to such a map whose faces all go around exactly $m+1$ star-vertices.
	
	To turn those maps into constellations, one simply replaces star-vertices with black faces of degree $m+1$,
	\begin{equation} \label{StarVertexConstellations}
		\includegraphics[scale=.6,valign=c]{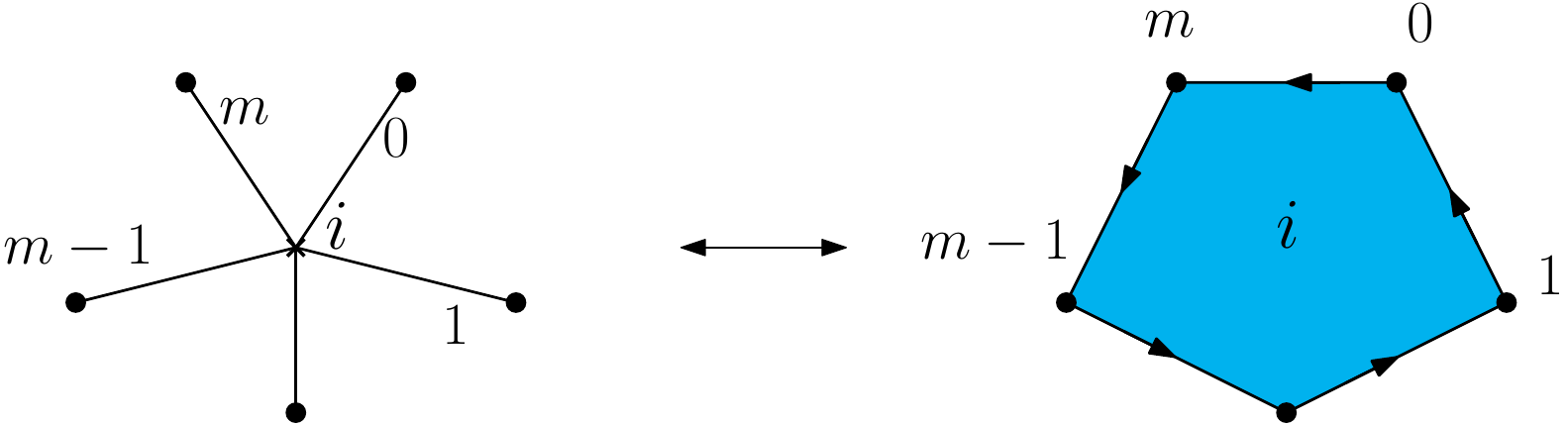}
	\end{equation}
	and the black face inherits the label $i\in[1..n]$ of the star-vertex. This turns the original faces into white faces of degree $m+1$, hence we find $(m+1)$-constellations whose faces all have degree $m+1$. The final step to recover $m$-constellations is to remove the vertices of color 0 and chop the black faces, which frees the white faces of the constraint on their degrees. This is exactly what we did in \eqref{PureConstellation} with $m=d$.
\end{proof}

After applying \eqref{PureConstellation} to turn an $(m+1)$-constellation with faces of degree $m+1$ only into a general $m$-constellation, the white faces turn out to correspond to the cycles $\sigma_0$, as illustrated in Figure \ref{fig:ConstellationPermutation}. In this sense, it is more natural to write the factorization as $\sigma_1\dotsm \sigma_m = \sigma^{-1}_0$, with permutations on the left corresponding to vertices and those on the right to faces.
\begin{figure}
	\includegraphics[scale=.5]{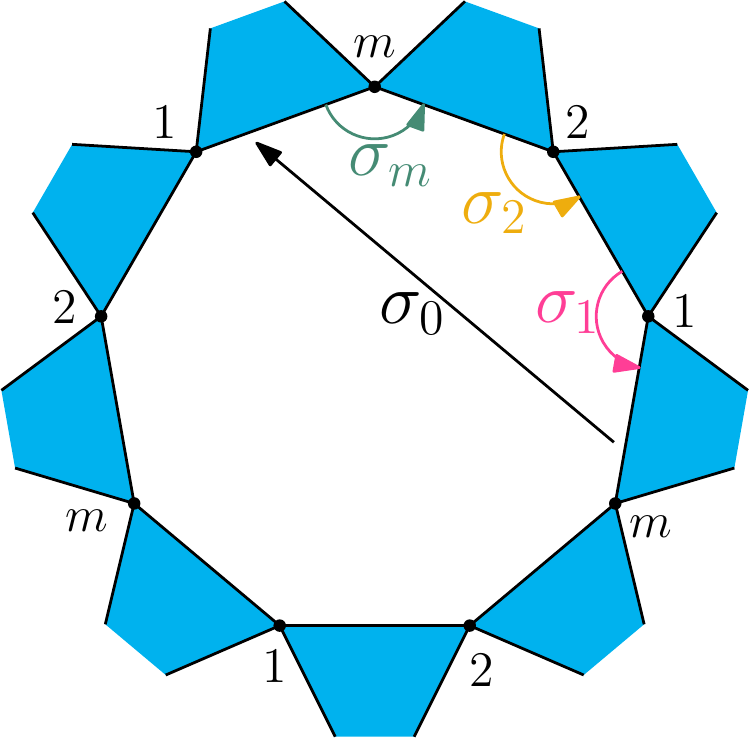}
	\caption{\label{fig:ConstellationPermutation} The permutations $\sigma_c$ act on the (labeled) black faces, so that the cycles of $\sigma_1\dotsm \sigma_m$ are those of $\sigma^{-1}_0$.}
\end{figure}

As for $(m-1)$-biconstellations (defined in \ref{sec:Biconstellations}), they also describe the same factorizations, since they are a simple redrawing of $m$-constellations. However, since black faces have been merged by the operation \eqref{VertexFaceBiconstellation}, the labels $1, \dotsc, n$ are not attached to them anymore. Instead, they are on \emph{left paths}, where a left path is an oriented sequence of $m$ edges from a vertex of color $m-1$ to the next vertex of color $m-1$ along a black face. The permutations $\sigma_c$ then acts on left paths and the factorization derives from the following picture
\begin{equation}
	\includegraphics[scale=1,valign=c]{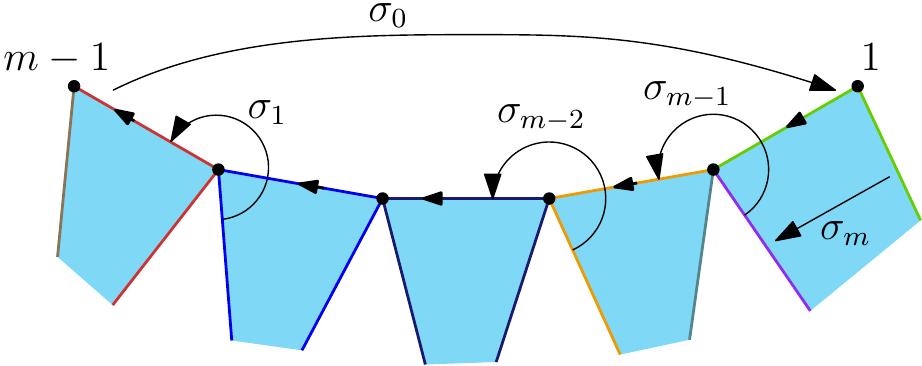}
\end{equation}
where the different colors of the edges represent different left paths. As for constellations, one can write the factorization for biconstellations as $\sigma_1 \dotsm \sigma_{m-1} = \sigma_0^{-1} \sigma_m^{-1}$ with the permutations giving rise to vertices on the left and permutations giving rise to faces on the right.

\subsection{Frobenius formula and Schur expansion} Denote $H^{(m+1)}(\lambda^{(0)}, \lambda^{(1)}, \dotsc, \lambda^{(m)})$ the number of $(m+1)$-factorizations \eqref{Factorization}. The enumeration problem is to extract information on those quantities and for that we will use generating series.

If $\ptimes = (p_1, p_2, \dotsc)$ is an infinite set of indeterminates and $\lambda = (\lambda_1 \geq \dotsb \geq \lambda_{\ell(\lambda)})$ a partition, let
\begin{equation} \label{Times}
	p_\lambda = p_{\lambda_1} \dotsm p_{\lambda_{\ell(\lambda)}},
\end{equation}
We will form a generating series for $(m+1)$-factorizations, by using a set of such variables to track the conjugacy class of every permutation. Let them be $\ptimes^{(0)}, \ptimes^{(1)}, \dotsc, \ptimes^{(m)}$, and let the series of factorizations be defined as
\begin{equation} \label{mFactorizationSeries}
	\tau^{\text{$(m+1)$-fact}}(\ptimes^{(0)}, \dotsc, \ptimes^{(m)};t) = \sum_{n\geq 0} \frac{t^n}{n!} \sum_{\mu^{(0)}, \dotsc, \mu^{(m)}\vdash n} H^{(m+1)}(\mu^{(0)}, \dotsc, \mu^{(m)}) \prod_{c=0}^{m} p^{(c)}_{\mu^{(c)}}.
\end{equation}

The set $\{p_\lambda\}_{\lambda}$ is a $\mathbb{Q}$-basis of the ring of of symmetric functions as defined in \cite{Macdonald1995}. It is thus a natural idea to ask whether changing the basis may be helpful in providing information on the numbers $H^{(m+1)}(\lambda^{(0)}, \lambda^{(1)}, \dotsc, \lambda^{(m)})$. We will see that it is indeed the case in Chapter \ref{sec:EnumerationIntegrability}. But for now we will focus on actually rewriting $\tau^{\text{$(m+1)$-fact}}(\ptimes^{(0)}, \dotsc, \ptimes^{(m)};t)$ in a different basis, namely the Schur basis.

As is standard in algebraic combinatorics, let $z_\lambda = \prod_{i=1}^{\ell(\lambda)} \lambda_i \prod_{r\geq 1} m_r(\lambda)!$, where $m_r(\lambda)$ is the number of parts of size $r$ in $\lambda$. The Schur function $s_\lambda(\ptimes)$ is 
\begin{equation} \label{Schur}
	s_\lambda(\ptimes) = \sum_\mu \frac{\chi_\lambda(\mu)}{z_\mu} p_\mu.
\end{equation}
where $\chi_\lambda$ is the character of the irreducible representation $\lambda$ of $\mathfrak{S}_n$, and $\chi_\lambda(\mu)$ its evaluation on the conjugacy class $\mu$. It corresponds to the usual Schur function $s_\lambda(x_1, x_2, \dotsc)$ when the variables $p_k$s are interpreted as the power-sums of auxiliary variables $x_1, x_2, \dotsc$ (an infinite number of them is needed for all $p_k$s to be independent)
\begin{equation}
	p_k(x_1, x_2, \dotsc) = \sum_{i\geq1} x_i^k.
\end{equation}
We keep the same notation $s_\lambda$ for both functions, whether it is seen as a function of the $p_k$s or of the auxiliary variables.

%with
%\begin{equation}
%	C^{(m)}(\mu^{(0)}, \dotsc, \mu^{(m)}) = [\mathbb{1}_{\mathfrak{S}_n}] C_{\mu^{(0)}} \dotsm C_{\mu^{(m)}}.
%\end{equation}
The factorization problem is clearly invariant under conjugation $\sigma^{(c)}\to \rho\sigma^{(c)}\rho^{-1}$ (it is called a \emph{central} problem), and it is thus amenable via character theory. Let $f^\lambda$ be the dimension of the irreducible representation indexed by $\lambda$. Frobenius' formula is
\begin{equation} \label{Frobenius}
	H^{(m+1)}(\mu^{(0)}, \dotsc, \mu^{(m)}) = \sum_{\alpha\vdash n} \frac{n!^m}{(f^\alpha)^{m-1}} \frac{\chi_\alpha(\mu^{(0)}) \dotsm \chi_\alpha(\mu^{(m)})}{z_{\mu^{(0)}} \dotsm z_{\mu^{(m)}}}
\end{equation}
and by the hook length formula, $\frac{n!^m}{(f^\alpha)^{m-1}} = n! \hook(\alpha)^{m-1}$ where $\hook(\alpha)$ is the product of the hook lengths over all boxes of $\alpha$. By plugging this formula into the definition \eqref{mFactorizationSeries}, we find
\begin{theorem}{}{SchurExpansionFactorizations}
	\begin{equation} \label{SchurExpansionFactorizations}
		\tau^{\text{$m$-fact}}(\ptimes^{(0)}, \dotsc, \ptimes^{(m)};t) = \sum_\lambda t^{|\lambda|} \operatorname{hook}(\lambda)^{m-1} s_\lambda(\ptimes^{(0)}) \dotsm s_\lambda(\ptimes^{(m)}),
	\end{equation}
\end{theorem}
%Here we have used the expression \eqref{Schur} of Schur functions and the hook length formula, $\operatorname{hook}(\lambda) = n!/\dim V_\lambda$. In terms of maps, this is the series of \emph{labeled $m$-constellations} (or $(m-1)$-biconstellations) where one keeps track of the degrees of the faces and of the vertices of \emph{every} color.

\subsection{The tau function of (bi)constellations} Us being able to track $m+1$ conjugacy classes in \Cref{thm:SchurExpansionFactorizations} is due to the fact that we used representation theory. However, in the context of maps it is in general too difficult to keep more than 2 conjugacy classes, say $\lambda \equiv \lambda^{(0)}$ and $\mu\equiv \lambda^{(1)}$, while only the numbers of parts of the others are kept. This corresponds to $H(\lambda, \mu, n_2, \dotsc, n_m)$, counting biconstellations with respect to the face degrees of black and white faces and with respect to the number of vertices $n_c$ of each color $c=2..m$. This reduction seems a condition to be able to write evolution equations {\it à la} Tutte, or deal with $b$-deformed maps where no representation theory is relevant.

We thus set $p^{(c)}_i = u_c$ for $c=2, \dotsc, m$, and $\ptimes\equiv \ptimes^{(0)}$, $\qtimes\equiv \ptimes^{(1)}$. We need $s_\lambda(1^{u_c})\equiv s_\lambda(\ptimes^{(c)} = (u_c,u_c,\dotsb))$. Notice that $s_\lambda(\ptimes^{(c)})$ is a polynomial, so it is enough to know $s_\lambda(1^{u_c})$ for integer values $u_c=N\in\mathbb{N}$. It is easily found since $u_c=N\in\mathbb{N}$ corresponds to setting the variables $x_i$s of the symmetric functions to $x_1=\dotsb=x_N=1$ and $x_i=0$ for $i>N$. This is called the principal specialization \cite{Macdonald1995} and it leads to
\begin{equation} \label{PrincipalSpecialization}
	s_\lambda(1^{u_c}) = \frac{1}{\operatorname{hook}(\lambda)} \prod_{\Box\in\lambda} (u_c+c(\Box))
\end{equation}
where $c(\square) = x-y$ is the content of the box of coordinates $(x,y)$ of $\lambda$ ($(0,0)$ being the bottom left corner). It comes	
\begin{equation}
	\begin{aligned}
		\tau^{\text{$(m-1)$-biconst}}(\ptimes, \qtimes, u_2, \dotsc, u_m;t) &= \sum_{n\geq 0} \frac{t^n}{n!} \sum_{\lambda, \mu\vdash n} H(\lambda, \mu, n_1, \dotsc, n_m) p_\lambda q_\mu \prod_{c=2}^m u_c^{n_c}\\
		&= \sum_{\lambda} t^{|\lambda|} s_\lambda(\ptimes) s_\lambda(\qtimes) \prod_{\Box\in\lambda} \prod_{c=2}^m (u_c+c(\Box)).
	\end{aligned}
\end{equation}

\subsection{Simple and monotone Hurwitz numbers} Instead of keeping $m$ fixed and working with arbitrary conjugacy classes, one can consider the factorization question for some given types of conjugacy classes with $m$ arbitrary. We give the Schur expansions without proof.
\begin{itemize}
	\item \emph{Double simple Hurwitz numbers} count the number of factorizations when $\sigma_2, \dotsc, \sigma_{m}$ are transpositions. Their series is
	\begin{equation}
		\tau^{\text{Double simple}}(\ptimes, \qtimes, w;t) = \sum_{\lambda} t^{|\lambda|} s_\lambda(\ptimes) s_\lambda(\qtimes) \prod_{\Box\in\lambda} e^{wc(\Box)}
	\end{equation}
	where $w$ counts the number of transpositions.
	\item \emph{Double monotone Hurwitz numbers} count the number of factorizations where $(\sigma_2, \dotsc, \sigma_{m})$ is a monotone run of transpositions, i.e. if $\sigma_c = (a_c, b_c)$ for $a_c<b_c$, then the constraints $b_1\leq \dotsb\leq b_{m-1}$ have to hold. Their series is
	\begin{equation}
		\tau^{\text{Double mon}}(\ptimes, \qtimes, v;t) = \sum_{\lambda} t^{|\lambda|} s_\lambda(\ptimes) s_\lambda(\qtimes) \prod_{\Box\in\lambda} \frac{1}{1+vc(\Box)},
	\end{equation}
	where $v$ counts the number of transpositions.
\end{itemize}
Below we introduce weighted Hurwitz numbers, which generalize all the above examples with two fixed conjugacy classes $\lambda, \mu$.

\subsection{Weighted Hurwitz numbers} Call a \emph{run of transpositions} of length $\ell$,  an $\ell$-tuple of transpositions in $\mathfrak{S}_n$ of the form
\begin{equation}
	\rho = ((a_1,b_1), \dots, (a_\ell,b_\ell)),
\end{equation}
and a \emph{monotone run of transpositions} if for $a_k<b_k$ for all $k$, then $b_1\leq \dotsb \leq b_\ell$.

Let $m,s\geq 0$ be integers and denote $\vec{\ell} = (\ell_0, \dotsc, \ell_{m-1})\in\mathbb{N}^m$ and $\vec{k} = (k_0, \dotsc, k_{s-1})\in\mathbb{N}^s$.
%Let $s \geq m>0$ be integers and $I=[0..m-1], J=[m..s -1]$. If $A$ is an ordered set, then denote $\ell_A \equiv (\ell_a)_{a\in A}$ a set of variables ordered according to $A$. 
\begin{definition}[Weighted Hurwitz numbers]
	Let $H(\lambda, \mu, \vec{\ell}, \vec{k}, j)$ be $(-1)^{\sum_{r=0}^{s-1} k_r}$ times the number of tuples 
	\begin{equation}
		%\mathfrak{h}(\lambda, \mu, \vec{\ell}, \vec{k}, j) = 
		(\sigma_{-2}, \sigma_{-1}, \sigma_0,\dotsc,\sigma_{m-1}, \rho_0, \dotsc, \rho_{s-1}, \tau),
	\end{equation}
	where
	\begin{itemize}
		\item $\sigma_c\in\mathfrak{S}_n$ for $c=-2, \dotsc, m-1$,
		\item $\sigma_{-2}$ is of cycle type $\lambda$, $\sigma_{-1}$ of cycle type $\mu$
		\item $\sigma_i$ has $n-\ell_i$ cycles for $i\in[0..m-1]$,
		\item $\rho_r$ for all $r\in [0..s-1]$ is a monotone run of transpositions in $\mathfrak{S}_n$ of length $k_r$,
		\item $\tau$ is a run of transpositions of length $j$.
		\item They form an ordered factorization
		\begin{equation}
			\sigma_{-2} \sigma_{-1} \underline{\tau} \sigma_0\dotsm \sigma_{m-1} \underline{\rho_{0}} \dotsm \underline{\rho_{s-1}} = \mathbb{1}_{\mathfrak{S}_n},
		\end{equation}
		where the underline notation $\underline{\rho}=(a_1,b_1) \dotsm (a_\ell,b_\ell)$ denotes the product of elements in the run.
	\end{itemize}
	The numbers $H(\lambda, \mu, \vec{\ell}, \vec{k}, j)$ are called weighted Hurwitz numbers.
\end{definition}

Biconstellations are recovered by taking $s=j=0$, simple Hurwitz numbers with $m=s=0$ and monotone Hurwitz numbers with $m=j=0$.

Let $\ptimes = (p_1, p_2, \dotsc)$ and $\qtimes = (q_1, q_2, \dotsc)$ be two infinite sets of indeterminates and denote $\vec{u} = (u_0, \dotsc, u_{m-1})$ and $\vec{v} = (v_0, \dotsc, v_{s-1})$ two other sets of indeterminates. We form the series of weighted Hurwitz numbers,
\begin{equation} \label{WeightedHurwitzSeries}
	\tau(\ptimes, \qtimes, \vec{u}, \vec{v}, w;t) = \sum_{n\geq 0} \frac{t^n}{n!} \sum_{\substack{\lambda,\mu\vdash n\\
			\vec{\ell}\in\mathbb{N}^m\\ \vec{k}\in\mathbb{N}^s, j\geq0}} H(\lambda,\mu,\vec{\ell},\vec{k},j) p_\lambda q_\mu \frac{w^{-j}}{j!} \prod_{i=0}^{m-1} u_i^{n-\ell_i} \prod_{r=0}^{s-1} v_r^{-n-k_r}.
\end{equation}

%\subsection{Summary of convention}
%We recap quickly and set the terminology as follows.
%\begin{itemize}
%	\item Factorizations: the cycle type of all permutations is fixed.
%	\item Weighted Hurwitz numbers: two cycle types are fixed (those of $\sigma_{-2}$ and $\sigma_{-1}$).
%	\item Single weighted Hurwitz numbers: one cycle type is fixed (say, that of $\sigma_{-2}$).
%\end{itemize}
%
%When $G$ is a polynomial, we usually consider
%\begin{itemize}
%	\item constellations to be counted with respect to a single cycle type, that of the faces, while the $u_c$s keep track of the number of cycles of $\sigma_0, \dotsc, \sigma_{m-1}$.
%	\item biconstellations to be counted with respect to two cycle types, those of the white and black faces.
%\end{itemize}

\section{Tau functions of weighted Hurwitz numbers: Schur expansions} \label{sec:TauFunctionsWeightedHurwitzNumbers}
\subsection{Hypergeometric tau functions} Let $G$ be a generic rational function with an extra exponential factor,
\begin{equation} \label{WeightFunction}
	G(z) = e^{\frac{z}{w}} \frac{\prod_{i=0}^{m-1} (u_i+z)}{\prod_{r=0}^{s-1}(v_r + z)}.
\end{equation}
with parameters $w,\vec{u},\vec{v}$. We form the so-called hypergeometric tau function of Orlov-Scherbin \cite{OrlovScherbin2000}
\begin{equation} \label{HypergeometricTauFunction}
	\tau_G(\ptimes, \qtimes, t) = \sum_{\lambda} t^{|\lambda|} s_\lambda(\ptimes) s_\lambda(\qtimes) \prod_{\square\in\lambda} G(c(\square))
\end{equation} 

\begin{theorem}{}{SchurExpansionWeightedHurwitz}\cite{Guay-PaquetHarnad2017}
	With $G(z)$ given in \eqref{WeightFunction},
	\begin{equation}
		\tau(\ptimes,\qtimes,t,\vec{u},\vec{v},w) = \tau_G(\ptimes, \qtimes, t)
	\end{equation}
\end{theorem}
In other words
\begin{equation}
	\left[\frac{t^n}{n!} p_\lambda q_\mu \frac{w^{-j}}{j!}\prod_{i=0}^{m-1} u_i^{n-\ell_i} \prod_{r=0}^{s-1} v_r^{-n-k_r}\right] \tau_G(\ptimes, \qtimes, t) = H(\lambda,\mu,\vec{\ell},\vec{k},j),
\end{equation}
where the brackets denote coefficient extraction. The proof can be performed in the group algebra $\mathbb{C}[\mathfrak{S}_n]$, using the Jucys-Murphy elements.

%\subsection{Constellations and polynomial $G$} The hypergeometric tau functions and the tau functions of factorizations intersect when $G$ is polynomial in $\tau_G$ and when $\tau^{\text{$m$-fact}}$ is specialized so that only two partitions are kept track of. In the latter case, the formula for the hypergeometric tau function specialized to $s=0$ and $w=\infty$ is 
%Here we have used a standard notation for Schur polynomials as symmetric polynomials of the variables $x_1, \dotsc, x_N$ and $1^N$ means the specialization to $x_i=1$ for $i=1..N$. Here $N$ is an integer, and the power-sums are specialized to $p_i(x_1, \dotsc, x_N)=N$. 
%\begin{equation}
%	s_\lambda(\forall i, p_i=u) = \frac{1}{\operatorname{hook}(\lambda)} \prod_{\Box\in\lambda} (u+c(\Box)).
%\end{equation}

\subsection{Single Hurwitz numbers} \label{sec:SingleHurwitz} For reference, we specialize the above series to \emph{single} Hurwitz numbers, where only the conjugacy class of a single permutation is kept track of, say with the variables $p_i$s. For $m$-constellations counted with respect to the faces of every degree and with a weight $u_c$ on vertices of color $c\in[0..m-1]$, we use the restriction $q_i = \delta_{i,1}$, $s=0$, $w=\infty$, with the specialization 
\begin{equation}
	s_\lambda(\qtimes=(1,0,\dotsc)) = \frac{1}{\operatorname{hook}(\lambda)},
\end{equation}
yielding
\begin{equation} \label{ConstellationsTauFunction}
	\tau^{\text{$m$-const.}}(\ptimes, t, u_0, \dotsc, u_{m-1}) = \sum_\lambda \frac{t^{|\lambda|}}{\operatorname{hook}(\lambda)} s_\lambda(\ptimes) \prod_{\Box\in\lambda} \prod_{c=0}^{m-1} (u_c + c(\Box)),
\end{equation}
and bipartite maps are obtained by setting $m=2$. Note that it can also be recovered from $\tau^{\text{$m$-fact}}$ \eqref{SchurExpansionFactorizations} by specializing $\ptimes^{(c)} = (u_c, u_c, \dotsc)$ for $c=1, \dotsc, m$.

For general maps, one can consider 2-constellations, where the vertices of one color are forced to have degree 2 only. So we take $m=2$, $q_i=\delta_{i,2}$ as well as $s=0$ and $w=\infty$ again. Denote $\theta_\lambda = s_\lambda(\qtimes=(0,1,0,\dotsc))$ then
\begin{equation} \label{MapsTauFunction}
	\tau^{\text{Maps}}(\ptimes,t,u) = \sum_{\lambda} t^{|\lambda|} s_\lambda(\ptimes)\theta_\lambda\prod_{\square
		\in \lambda}(u+c(\square)).
\end{equation}

\section{Genus expansion}
Consider the \emph{connected} weighted Hurwitz numbers, counting the tuples 
\begin{equation*}
	(\sigma_{-2}, \sigma_{-1}, \sigma_0,\dotsc,\sigma_{m-1}, \rho_0, \dotsc, \rho_{s-1}, \tau)
\end{equation*}
whose elements generate a transitive group on $\{1, \dotsc, n\}$. The \emph{genus} of such a tuple is the non-negative integer $g$ determined by $\lambda, \mu, \vec{\ell}, \vec{k}, j$ by
\begin{equation} \label{HurwitzGenus}
	2-2g = \ell(\lambda)+\ell(\mu)-\sum_{i=0}^{m-1}\ell_i-\sum_{r=0}^{s-1}k_r-j+n.
\end{equation}

It can be interpreted in two ways, either as an equivalence class of branched coverings of the sphere, or in terms of maps. 
\begin{itemize}
	\item In the first case, if a branched covering has monodromies given by the permutations $\sigma_{-2}, \sigma_{-1}, \sigma_0,\dotsc,\sigma_{m-1}, [\rho_0], \dotsc, [\rho_{s-1}], [\tau]$, where $[\rho]$ denotes the sequence of transpositions in the run $\rho$, then \eqref{HurwitzGenus} is the Riemann-Hurwitz formula for the genus of the covering surface.
	\item In the second case, it is possible to adapt the construction of Section \ref{sec:EquivalenceFactorizationsConstellations} to obtain ``weighted Hurwitz maps''\footnote{For each new run of transpositions, there is a new set of colors and the vertices of each of those colors are bivalent. By erasing them, we see that a transposition becomes a colored edge. For a monotone run, an additional monotonicity is added on the edge labels.}, in bijection with the tuples $(\sigma_{-2}, \sigma_{-1}, \sigma_0,\dotsc,\sigma_{m-1}, \rho_0, \dotsc, \rho_{s-1}, \tau)$. Then the constraint that its elements generate a transitive group is equivalent to the maps being connected, and \eqref{HurwitzGenus} is Euler's formula for the genus of those maps.
\end{itemize}
We will however not define weighted Hurwitz maps here and instead restrict our combinatorial analysis to biconstellations. For instance to prove the topological recursion for weighted Hurwtiz numbers, we do it for biconstellations, then use algebraic arguments to generalize the result.

Extracting from the tau function the contribution of genus $g$ can be done by introducing a variable $N$ which will track the contributions to the genus via the Riemann-Hurwitz-Euler formula. We define $G_N(z) = G(z/N)$, and
\begin{equation} \label{FreeEnergyFixedGenus}
	\begin{aligned}
		&F^G_g(\ptimes, \qtimes, t) = [N^{2-2g}]\log \tau_{G_N}(N\ptimes, N\qtimes,Nt),\\
		&= \sum_{n\geq 1}\frac{t^n}{n!} \sum_{\lambda,\mu\vdash n} \sum_{\substack{\ell_0,\dots,\ell_{m-1}\geq 0\\ k_0, \dotsc, k_{s-1}, j\geq 0\\ \ell(\lambda)+\ell(\mu)-\sum_{i=0}^{m-1}\ell_i-\sum_{r=0}^{s-1}k_r-j+n=2-2g}}
		H^{\circ}(\lambda, \mu, \vec{\ell}, \vec{k}, j)	p_\lambda q_\mu \frac{w^{-j}}{j!}\prod_{i=0}^{m-1}u_i^{n-\ell_i} \prod_{r=0}^{s-1}v_r^{-n-k_r} 
	\end{aligned}
\end{equation}
We introduce the following \emph{rooting operator}, parametrized by a formal variable $x$, 
\begin{equation}
	\nabla_x \coloneqq \sum_{i\geq 1} i x^{-i-1} \frac{\partial}{\partial p_i}.
\end{equation}
On biconstellations, $\nabla_x$ is interpreted as marking a white face of degree $mi$, together with a ``root'' vertex on its boundary, and changing its weight from $p_i$ to $x^{-i-1}$. Let us explain why this is the same as the rooted biconstellations we defined in Chapter \ref{sec:Definitions}. A rooted $m$-biconstellation with colors $c=0..m-1$ is an $m$-biconstellation with a marked left path, i.e. an oriented sequence of $m$ edges from a vertex of color $m-1$ to the next vertex of color $m-1$ along a black face. For $c\in[0..m-1]$, this left path determines a unique white face, as the face to the right of the edge of the path between the colors $c$ and $c+1 \mod m$, and a unique root corner of color $c$ on it. The other way around, still for fixed $c$, the operator $\nabla_x$ marks a corner of color $c$ on a face of degree $mi$ (there are $i$ corners of color $c$) and it determines a unique left path.
%which can naturally be interpreted in terms of connected ``weighted Hurwtiz maps'' of genus $g$. Nevertheless, instead of defining those, we will typically consider (bi)constellations only, which correspond to the case of polynomial $G$. The more general $G(z)$ can be treated either by extending the methods we used for constellations to weighted Hurwitz maps, or by extending the results obtained for constellations using algebraic properties of the generating series. This is the second option which was chosen in \cite{BCCGF22}.

The marked white face is then thought of as a \emph{boundary} component of the maps. In particular, let
\begin{equation}
	W_{0,1}(x) \coloneqq \nabla_x F^G_0(\ptimes, \qtimes, t)
\end{equation}
and
\begin{equation}
	W_{0,2}(x_1,x_2) \coloneqq \nabla_{x_1} \nabla_{x_2} F^G_0(\ptimes, \qtimes, t).
\end{equation}
The former is known as the \emph{disc} generating function because it counts maps of genus 0 with 1 boundary component, which is the topology of the disc. The latter is known as the \emph{cylinder} generating function because it counts maps of genus 0 with 2 boundary components.

More generally, the generating series with $n$ boundaries is
\begin{equation} \label{PlanarCorrelationFunctions}
	W_{0,n}(x_1, \dotsc, x_n) \coloneqq \nabla_{x_1} \dotsm \nabla_{x_n} F^G_0(\ptimes, \qtimes, t).
\end{equation}
In this framework, the non-root faces are called \emph{internal faces}.

%%%%%%%%%%%%%%
\section{Enumeration of planar weighted Hurwitz numbers}
In this section we will give an explicit algebraic parametrization of $W_{0,1}(x)$ and show that $W_{0,2}(x_1,x_2)$ takes on a universal form in that parametrization. 

\subsection{State of the art} 
\subsubsection{The case of constellations} Here $q_i=\delta_{i,1}$, $s=0$, $w=\infty$ and $u_c = u$ for all $c=0..m-1$. M. Bousquet-Mélou and G. Schaeffer \cite{BousquetSchaeffer2000} gave the number of planar, labeled $m$-constellations of size $n$ with $d_k$ faces of degree $mk$,
\begin{equation} \label{BMS}
	m \frac{((m-1)n-1)!}{((m-1)n-f+2)!} \prod_{k\geq 1} \left(k\binom{mk-1}{k}\right)^{d_k}
\end{equation}
where $f=\sum_{k\geq 1}d_k$ is the total number of faces. This is equivalent to $W_{0,n} (x_1, \dotsc, x_n)_{|\ptimes=0}$, i.e. planar constellations with no internal faces. They used a bijective approach, more precisely a bijection between those constellations and a family of trees which are then easily enumerated. This bijection is of the same type as those previously introduced by the French school of combinatorics, originated in the work of Cori and Vauquelin and modernized by Schaeffer, and which have been so successful in explaining the counting formulas of various families of maps. A similar bijection, that of the mobiles of Bouttier-Di Francesco-Guitter \cite{BouttierDiFrancescoGuitter2004}, was used to derive a compact formula for the series $W_{0,n}(x_1, \dotsc, x_n)$ by G.~Collet and É.~Fusy \cite{ColletFusy2014}.

The traditional approach to finding those counting formulas for maps is that of Tutte's equation. It is based on a decomposition of maps obtained by ``removing the root edge''. Although this is a robust approach, in particular in the ability to write a functional equation, it becomes more than cumbersome to solve for constellations with $m\geq 3$, as shown by W. Fang \cite{Fang:PhD}. His result is still an explicit algebraic parametrization for $W_{0,1}(x)$ with the restrictions given above, in particular a single set of degrees tracked by $\ptimes$, and all colors $c=0..m-1$ counted with the same variable.

\subsubsection{Moving onto biconstellations} With our co-authors of \cite{BCCGF22}, we initially thought of using the same mobile bijection as in \cite{ColletFusy2014}. Mobiles however deal not with rooted maps, but \emph{pointed} rooted maps, meaning that there is an additional marked vertex. In our case where we want to control the vertex colors, direct depointing (by integration) however requires controlling the color of the pointed vertex, which we were not able to do. A different solution to this depointing problem was nevertheless given in \cite{BouttierGuitterContinuedFractions}. In this article, not only solving the depointing problem, the authors introduce a slightly different bijective technology, that of \emph{slices}, which we decided to use. Introduced by J. Bouttier and E. Guitter in \cite{BouttierGuitterContinuedFractions, BouttierGuitter2014}, slices were generalized to constellations by M. Albenque and J. Bouttier in \cite{AlbenqueBouttier:fpsac12}.

In all those works however, there were a single set of tracked degrees $\ptimes$, i.e. $q_i=\delta_{i,1}$, and in addition $u_c=u$ for all $c$. The only enumerative result with two sets of tracked degrees $\ptimes, \qtimes$ is that of B. Eynard \cite{Eynard2002} on the 2-matrix model. He found an algebraic parametrization of $W_{0,1}(x)$ in the case of 1-biconstellations, or equivalently bipartite maps with controlled degrees of both white and black vertices (and none on the face degrees), using clever (a.k.a. mystifying) manipulations on the loop equations.

Fortunately for us, M. Albenque and J. Bouttier (again) were working on a derivation of Eynard's result using slices \cite{AlbenqueBouttier2022}. Although no preprint was available, the authors had made public their results, and M. Albenque was kind enough to share some details with us. We realized it is possible to extend \cite{AlbenqueBouttier2022} to biconstellations by marrying \cite{AlbenqueBouttier2022} and \cite{AlbenqueBouttier:fpsac12}, while also adding color-dependent weights $u_c$ on vertices.

Although our enumeration formulas are indeed more general, they are direct applications of the formalism developed by J.~Bouttier and E.~Guitter \cite{BouttierGuitterContinuedFractions, BouttierGuitter2014}, and M.~Albenque and J.~Bouttier \cite{AlbenqueBouttier:fpsac12, AlbenqueBouttier2022}. We will present our results for the disc and cylinder generating functions below, and present the slices, the slice decompositions and their generating series. However, we will refer to the above references for the details regarding the relations between the disc and cylinder functions and the slice series since they are essentially unchanged in our case.

\subsection{The disc and cylinder generating functions} We set $w=\infty$ to simplify the expressions below ($w\neq \infty$ can be found in the appendix of \cite{BCCGF22}). For the purpose of the topological recursion to come, we fix two arbitrary positive integers $D_1,D_2$, which will play the role of maximum allowed degrees for internal faces of each color, i.e. $p_i=0$ for $i>D_1$ and $q_i=0$ for $i>D_2$. Our first result is an algebraic parametrization of the function $W_{0,1}(x)$. In the case $G(z)=(u+z)$, what follows is Eynard's solution to the 2-matrix model in the planar sector~\cite{Eynard2002}.

To state the result, we need to introduce some quantities. First the formulas take a more universal form if we define $u_{m-1+r}\coloneqq v_{r}$ for $r=0..s-1$, because the $u_i$ and $v_r$ play almost similar roles, and $M=m+s$. To distinguish them when needed, we set $I=\{0,\dotsc, m-1\}$ and $J=\{m, \dotsc, m+s-1\}$. We denote $U=\{u_0, \dotsc, u_{m-1}, u_m, \dotsc, u_{M-1}\}$ and $U^{-1}=\{u_0^{-1}, \dotsc, u_{m-1}^{-1}, u_m^{-1}, \dotsc, u_{M-1}^{-1}\}$.

Moreover, if $f(z)=\sum_{k\geq k_0} f_k z^k$ is a Laurent series in $z$, we write
\begin{equation*}
	[f(z)]^{<} \coloneqq \sum_{k=k_0}^{-1} f_k z^k
\end{equation*}
which is a polynomial in $z^{-1}$. If $g(z)=\sum_{k\geq k_0} g_k z^{-k}$ is a Laurent series in $z^{-1}$, we write 
\begin{equation*}
	\{g(z)\}^{\geq} \coloneqq \sum_{k=k_0}^{0} g_k z^{-k},
\end{equation*}
which is a polynomial in $z$. %Finally we use the notation $\bar{x}\equiv x^{-1}$.

\begin{definition} \label{def:WkBk}
For $c\in I \cup J$ we define the Laurent polynomials
\begin{equation}
A^{(c)}(z) = \sum_{k=0}^{D_2} A_k^{(c)} z^k \, ,  \qquad 
B^{(c)}(z) = 1 + \sum_{k=1}^{D_1} B^{(c)}_k z^{-k},
\end{equation}
uniquely determined as elements of 
$\mathbb{Q}[\ptimes, \qtimes, U][[t]][z]$ 
and 
$\mathbb{Q}[\ptimes, \qtimes, U][[t]][z^{-1}]$ 
respectively, by the following system of equations:
\begin{align}
\label{eq:defWratconv2}
A^{(c)}(z) &= u_c + \sum_{s=1}^{D_2} q_{s} t^s   \biggl\{z^s \frac{\prod_{i\in I}  B^{(i)}(z)^s}{\prod_{j\in J}  B^{(j)}(z)^s}\frac{1}{B^{(c)}(z)}\biggr\}^{\geq},\\
\label{eq:defBratconv2}
B^{(c)}(z) &= 1 + \sum_{s=1}^{D_1} p_s \biggl[z^{-s} \frac{\prod_{i \in I} A^{(i)}(z)^s}{\prod_{j \in J} A^{(j)}(z)^s}\frac{1}{A^{(c)}(z)}\biggr]^{<},
\end{align}
\end{definition}
This is a rather compact form for a polynomial system of equations relating the series $A^{(c)}_k, B^{(c)}_k$s together, from which all coefficients of those series can be computed recursively, order by order in $t$. 
%In the case with no internal faces, notice that 
%\begin{equation}
%A_0^{(c)}(z) \coloneqq A^{(c)}(z)_{|\ptimes =0}  = 1 + u_c \sum_{k=1}^{D_2} q_k z^k.
%\end{equation}
%	For $c \in [[m]]$ we consider the%
%	%formal power series $Z^{(c)}\equiv Z^{(c)}(x)\in \mathbb{Q}[\mathbf{p},\mathbf{q},\mathbf{u},\bar{x}^{1/m}][[t]]-TODO$ defined by the equation
%	\begin{align}\label{eq:defZc}
%	Z^{(c)} 
%	= \bar{x}^{1/m} W^{(c)} \left(\frac{\prod_{i \in I} Z^{(i)}}{\prod_{j \in J} Z^{(j)}}\right).
%\end{align}
%Note that $Z^{(c)}(x) = u_{c} \bar{x}^{1/m}+O(t)$, and that any power of this function (positive or negative), defines a valid formal Laurent series in  
%$\mathbb{Q}[\mathbf{p},\mathbf{q},\mathbf{u},\bar{\mathbf{u}},x^{1/m},x^{-1/m}]((t)).$

Define the Laurent polynomial $H^{(c)}(z) \in \mathbb{Q}[\ptimes,\qtimes,U][[t]][z,z^{-1}]$ given by 
\begin{equation*}
	H^{(c)}(z) \coloneqq \left(A^{(c)}(z) B^{(c)}(z) - u_c \right).
\end{equation*}
It is symmetric in the variables $u_0, \dotsc, u_{m-1}$, symmetric in $u_m, \dotsc, u_{M-1}$, and moreover it is independent of the chosen value of $c \in I\cup J$. These properties are not trivial and are proved in \cite{BCCGF22}. We have now defined all quantities needed to introduce our spectral curve.
\begin{definition}\label{def:SpectralCurve}
	We consider the system of polynomial equations\footnote{When $\ptimes,\qtimes,U$ are fixed complex numbers and $|t|$ is small enough, all generating functions $A_k^{(c)}, B_k^{(c)}$ converge and these polynomial equations are defined over complex numbers. One can also consider these polynomial equations formally over the field $\mathbb{Q}(\ptimes,\qtimes,U, U^{-1})[[t]]$.}
	defined by
	\begin{align}\label{eq:spectralX}
		z X(z)&=  \frac{\prod_{i  \in I} A^{(i)}(z)}{\prod_{j \in J} A^{(j)}(z)}
		, \\ \label{eq:spectralY}
		X(z) Y(z) &= H^{(c)}(z),
	\end{align}
	where $c$ is any value in $I\cup J$.
\end{definition}
%Again in the case with no internal faces,
%\begin{equation}
%	X_0(z) \coloneqq X(z)_{|\ptimes=0} = \frac{1}{z} \frac{\prod_{i  \in I} (1 + u_i \sum_{k=1}^{D_2} q_k z^k)}{\prod_{j \in J} (1 + u_j \sum_{k=1}^{D_2} q_k z^k)}.
%\end{equation}

Note that the equation $X(Z)=x$ defines a unique power series $Z\in \mathbb{Q}[\ptimes,\qtimes,U, U^{-1},x^{-1}][[t]]$, satisfying
\begin{equation} \label{eq:ZasExcursions}
	Z(x) = \frac{1}{x} \frac{\prod_{i\in I} A^{(i)}(Z(x))}{\prod_{j\in J} A^{(j)}(Z(x))}.
\end{equation}
%It is an element of $\mathbb{Q}[\vec{u}, \vec{v},\bar{x}][[t]]$), with an expansion of the form	$Z = \bar{x} + O(t)$. which is nothing but the one we introduced in~\eqref{eq:ZasExcursionsrat}, with expansion~\eqref{eq:expansionZrat}. 
By substitution, it then defines a unique formal series $Y(Z(x))$ in $\mathbb{Q}[\ptimes,\qtimes,U,U^{-1}, x, x^{-1}]((t))$.

%The following theorem is proved in \cite{AlbenqueBouttier2022} in the case $M=m=1$. Their approach is in fact the main tool we use to prove the theorem below in the polynomial case, i.e.~for $M=m\geq 1$. In fact (see Section~\ref{sec:constellations}), the proof in this case follows relatively easily for a reader familiar with \cite{AlbenqueBouttier2022} and combinatorics of paths, once understood that the combinatorial objects in the polynomial case can be encoded into constellations with two sets of face weights (Proposition \ref{thm:FactorizationsConstellations}).
%This theorem extends Eynard's solution of the 2-matrix model in the planar sector, a.k.a. planar 1-biconstellations in this manuscript.
\begin{theorem}{}{Disc}
	The disc generating function $W_{0,1}(x)$ is, up to an explicit shift, given by the parametrisation $(Y(z),X(z))$ given above,
	\begin{equation*}
		W_{0,1}(x) + \sum_{k=1}^{D_1}p_kx^{k-1} = Y(Z(x)),
	\end{equation*}
	in $\mathbb{Q}[\ptimes,\qtimes,U,x,x^{-1}][[t]]$.
\end{theorem}

Once formulated in the ``change of variables'' $x \leftrightarrow Z$, the cylinder generating function has the universal expression already encountered in many other models of enumerative geometry. Indeed we have:
\begin{theorem}{}{Cylinder}
	The cylinder generating function $W_{0,2}$ is
	\begin{equation*}
		W_{0,2}(x_1,x_2)= \frac{Z'(x_1)Z'(x_2)}{(Z(x_1)-Z(x_2))^2} - \frac{1}{(x_1-x_2)^2},
	\end{equation*}
	in $\mathbb{Q}[\ptimes,\qtimes,U](x_1,x_2)[[t]]$. %It gives $\omega_{0,2}(z_1,z_2) = \frac{dz_1 dz_2}{(z_1-z_2)^2}$.
\end{theorem}

\subsection{Slices and elementary slices} We recall that a biconstellation is equipped with a canonical orientation, where the black faces are on the left of edges. A biconstellation with a virtual boundary is a map satisfying the same definition as biconstellations except for one face which is neither black nor white. In particular, its boundary is not necessarily an oriented cycle.

Here we consider ``oriented distances''. Given a pair of vertices $(v_1, v_2)$, there are oriented paths from $v_1$ to $v_2$ and the shortest paths are called \emph{geodesics}, and their number of edges is denoted $\ell_{v_2}(v_1)$.
\begin{definition}
A slice is a planar $m$-biconstellation with a virtual boundary $f^*$ which has three marked corners, $l, r, o$, which split the boundary of $f^*$ into three connected paths as follows:
\begin{itemize}
	\item The \emph{right boundary} is oriented from $r$ to $o$ (avoiding $l$) and is the unique geodesic from $r$ to $o$.
	\item The \emph{left boundary} is oriented from $l$ to $o$ (avoiding $r$) and is a geodesic from $r$ to $o$.
	\item The \emph{base} between $l$ and $r$.
	\item The left and right boundaries intersect only at the \emph{origin}, which is the vertex where $o$ sits.
\end{itemize}
If the base is oriented from $l$ to $r$ (respectively from $r$ to $l$) we say that it is a \emph{white} (respectively \emph{black}) slice, see Figure \ref{fig:Slice}. If the base consists in a single edge, we say that it is an \emph{elementary} slice.
\end{definition}
\begin{figure}
	\centering
	\includegraphics[scale=.9]{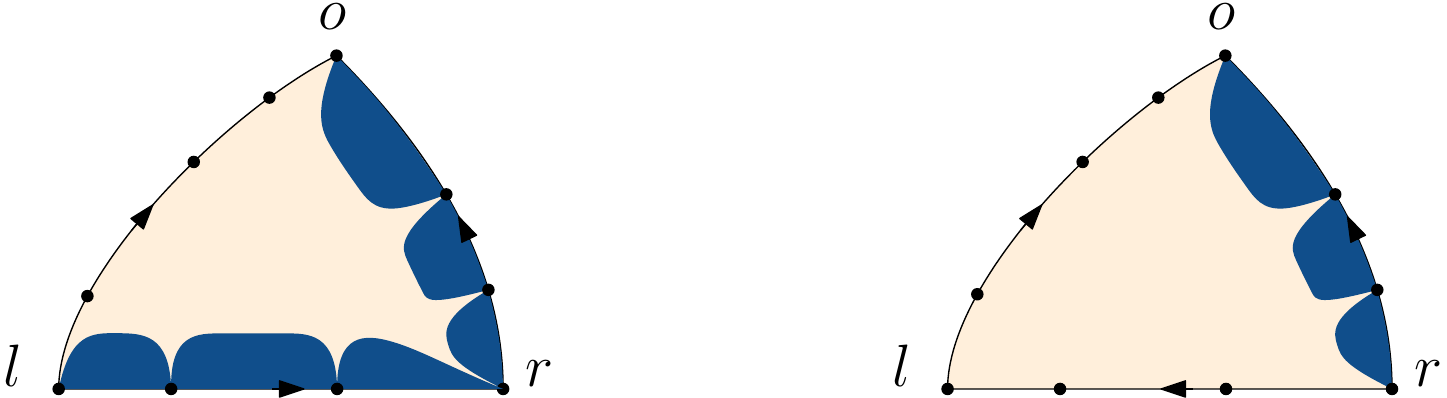}
	\caption{Some (schematic) white (left) and black (right) slices with their marked corners $o, l, r$. The virtual boundary is on the outside, the dark blue regions represent black faces incident to it, and the rest is in beige.}
	\label{fig:Slice}
\end{figure}

\subsubsection{The slice series} A slice is canonically pointed at its origin and vertices are labeled with $\ell_o(v)$ (abusing the notation by using the corner $o$ instead of the vertex as subscript). If we denote $\ell_o(l)$ and $\ell_o(r)$ the labels at $l$ and $r$, we call $\ell_o(r)-\ell_o(l)$ the \emph{increment} of white slices and $\ell_o(l)-\ell_o(r)$ the increment of black slices.

Let $A^{(c)}_{n,k}$ (respectively $B^{(c)}_{n,k}$) be the generating series of white (respectively black) slices of increment $k$, whose bases have length $n$ and where $l$ (respectively $r$) sits at a vertex of color $c$. As a special case, denote $A^{(c)}_k$ (respectively $B^{(c)}_k$) the generating series of \emph{elementary} white (respectively black) slices of increment $mk-1$ where $l$ (respectively $r$) sits at a vertex of color $c\in[0..m-1]$.
%Let $\mathcal{A}^{(c)}_{n,k}$ (respectively $\mathcal{B}^{(c)}_{n,k}$) be the set of white (respectively black) slices of increment $k$, whose bases have length $n$ and where $l$ (respectively $r$) sits at a vertex of color $c$. We denote $A^{(c)}_{n,k}$ their generating series (respectively $B^{(c)}_{n,k}$).
%Their generating series are defined as
%\begin{equation*}
%	A^{(c)}_{n,k} = \sum_{M\in \mathcal{A}^{(c)}_{n,k}} w(M),\qquad B^{(c)}_{n,k} = \sum_{M\in \mathcal{B}^{(c)}_{n,k}} w(M).
%\end{equation*}
%Let $\mathcal{A}^{(c)}_k$ (respectively $\mathcal{B}^{(c)}_k$) be the set of elementary white (respectively black) slices of increment $mk-1$ where $l$ (respectively $r$) sits at a vertex of color $c\in[0..m-1]$, with generating series $\tilde{A}^{(c)}_k$ (respectively $\tilde{B}^{(c)}_k$). The following proposition is essentially proved in \cite{AlbenqueBouttier2022} (we simply add the color-dependent weights here).
\begin{proposition}{}{CompositeSlices}
	Those series are given by
	\begin{equation} \label{CompositeSlices}
		A_{n,k}^{(c)} = [\alpha^k] \alpha^{-n} \prod_{r=1}^{n} A^{(c-r+1)}(\alpha^m),\qquad 
			%[\alpha^k] \prod_{r=0}^{n-1} \Bigl(\sum_{p=0}^{D_1} u^{-1}_c A_p^{(c-r)} \alpha^{mp-1}\Bigr),\\
		B_{n,k}^{(c)} = [\alpha^{-k}] \alpha^{n} \prod_{r=1}^{n} B^{(c-r+1)}(\alpha^m),
	\end{equation}
	where the polynomials $A^{(c)}(z), B_k^{(c)}(z)$ (in $z$ and $1/z$ respectively) are given by Definition \ref{def:WkBk}, which in the polynomial case reduces to
	\begin{equation} \label{eq:PolynomialAB}
		\begin{aligned}
			A^{(c)}(z) &= u_c + \sum_{s=1}^{D_2} q_s t^s \left\{z^s \frac{\prod_{c'=0}^{m-1} B^{(c')}(z)^s}{B^{(c)}(z)}\right\}^{\geq},\\
			B^{(c)}(z) &= 1 + \sum_{s=1}^{D_1} p_s \left[z^{-s} \frac{\prod_{c'=0}^{m-1} A^{(c')}(z)^s}{A^{(c)}(z)}\right]^{<}.
		\end{aligned}
	\end{equation}
\end{proposition}

\subsubsection{Weighted \L{}ukasiewicz paths}
We will see the generating series of slices as generating series of \L{}ukasiewicz paths appropriately weighted. Along an oriented edge of a slice
\begin{itemize}
	\item the color decreases by $1\mod m$,
	\item and the label $\ell_o(v)$ changes by $mk-1$ for $k\in\mathbb{N}$.
\end{itemize}
This explains why the following family of paths plays a role.

A white $m$-path starting at color $c\in[0..m-1]$ is a path on $\mathbb{Z}^2$ with steps of type $\omega_p^\circ = (1,mp-1)$ for $p\geq 0$, with associated weight $A^{(c+r-1)}_p$ if the $r$-th step is of type $\omega_p^{\circ}$. A black $m$-path is a path on $\mathbb{Z}^2$ with steps of type $\omega_p^\bullet = (1,1-mp)$ for $p\geq 0$, with associated weight $B^{(c-r+1)}_p$ if the $r$-th step is of type $\omega_p^{\bullet}$. We say that the $r$-th step has color $c-r+1\mod m$.

Let $\mathcal{L}^{(c)\circ}_{n,k}$ (resp. $\mathcal{L}^{(c)\bullet}_{n,k}$) be the set of $m$-paths which starts at $(0,0)$ with the color $c$ and ends at $(n,k)$ (respectively $(n,-k)$). Let $A^{(c)}(z) = \sum_{i\geq0} A_i^{(c)} z^i$ and $B(z) = 1 + \sum_{i\geq1} B_i^{(c)} z^{-i}$. Then the generating series of $\mathcal{L}^{(c)\circ}_{n,k}$ and $\mathcal{L}^{(c)\bullet}_{n,k}$ are, clearly,
\begin{equation} \label{eq:PathsSeries}
	\begin{aligned}
		L^{(c)\circ}_{n,k} &= [\alpha^k] \prod_{r=1}^{n} \Bigl(\sum_{i\geq 0} A^{(c-r+1)}_i \alpha^{mi-1}\Bigr) = [\alpha^k] \alpha^{-n} \prod_{r=1}^{n} A^{(c-r+1)}(\alpha^m),\\
		L^{(c)\bullet}_{n,k} &= [\alpha^{-k}] \prod_{r=1}^{n} \Bigl(\sum_{i\geq 0} B^{(c-r+1)}_i \alpha^{1-mi}\Bigr) = [\alpha^{-k}] \alpha^{n} \prod_{r=1}^{n} B^{(c-r+1)}(\alpha^m).
	\end{aligned}
\end{equation}

\begin{proof}[Proof of Proposition \ref{thm:CompositeSlices}]
	We follow the slice decompositions of~\cite{AlbenqueBouttier2022}, while keeping track of the colors of vertices.
	
	Let $M^{(c)}_{n,k}$ be a white slice of increment $k$, with $l$ sitting at a vertex of color $c$, and base of length $n$. Let us denote $v_1, \dotsc, v_n$ the sequence of vertices along the base from $l$ to $r$. Then the sequence of their labels can be encoded as a \L{}ukasiewicz path from $\mathcal{L}^{(c)\circ}_{n,k}$. If $M^{(c)}_{n,k}$ is a black slice, it works the same by writing the path from $r$ to $l$.
	
	Furthermore, one can apply the slice decomposition to $M^{(c)}_{n,k}$. It consists in drawing from each vertex of the base the leftmost geodesics to the origin. After cutting the geodesics every time two of them meet, one obtains a concatenation of $n$ elementary white slices, whose bases are the edges of the base between $v_i$ and $v_{i+1}$, $i=1, \dotsc, n-1$, and whose increments are $\ell_o(v_{i+1}) - \ell_o(v_i)$, see Figure \ref{fig:SliceDecomposition}. 
	
	In terms of generating series, one can thus decorate the steps of the \L{}ukasiewicz paths of $\mathcal{L}^{(c)\circ}_{n,k}$ (respectively $\mathcal{L}^{(c)\bullet}_{n,k}$) with the series $A^{(c)}_k$s (respectively $B^{(c)}_k$). Notice that the edge of the base between $v_i$ and $v_{i+1}$ has color $c-i+1\mod m$. Therefore, if the $i$-th step is of type $(1,mk_i-1)$, it receives the weight $A^{(c-i+1)}_{k_i}$. This gives for $k\in\mathbb{Z}$
	\begin{equation}
		A_{n,k}^{(c)} =  L^{(c)\circ}_{n,k}
		%\sum_{\mathcal{L}^\circ_{n,k}} \prod_{i=1}^n \tilde{A}^{(c-i+1)}_{k_i}
		,\qquad
		B_{n,k}^{(c)} =  L^{(c)\bullet}_{n,k}
		%\sum_{\mathcal{L}^\bullet_{n,k}} \prod_{i=1}^n \tilde{B}^{(c-i+1)}_{k_i}
		.
	\end{equation}
	Here it is important to notice that elementary slices whose bases have compatible colors can be glued together along their full boundaries, since the color along each oriented edge decreases by $1$ modulo $m$ (thus if the colors of two boundaries to be glued are compatible at the beginning, they are compatible throughout). We have thus proved \eqref{CompositeSlices}. 
	
	To prove the relations in Definition \ref{def:WkBk} for $A_k^{(c)}$, consider an elementary white slice. For $k=0$, it may be reduced to a single edge, where the corners $o$ and $r$ are the same, with weight $u_c$. If not, consider the black face to the left of its base. This face can have arbitrary degree $ms$ for $s\geq 1$. After removing the base, the white elementary slice becomes a black slice with base length $ms-1$. It is therefore found that
	\begin{equation*}
		A_k^{(c)} = u_c\delta_{k,0} + \sum_{s\geq 1} q_{s} t^{s} B^{(c-1)}_{ms-1,mk-1}.
	\end{equation*}
	This is illustrated in Figure \ref{fig:SliceDecomposition}. Similarly $B_k^{(c)} = \sum_{s\geq 1} p_{s} A^{(c-1)}_{ms-1,mk-1}$ for $k\geq 1$, while a black slice of increment $-1$ is by definition reduced to its base, implying $B_0^{(c)}=1$.
\end{proof}
\begin{figure}
	\centering
	\includegraphics[scale=.75]{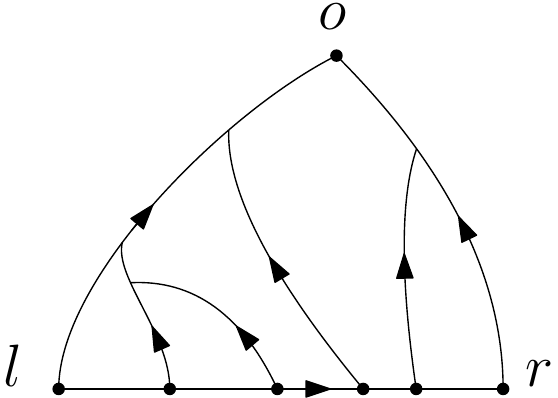}\hspace{2cm}
	\includegraphics[scale=.75]{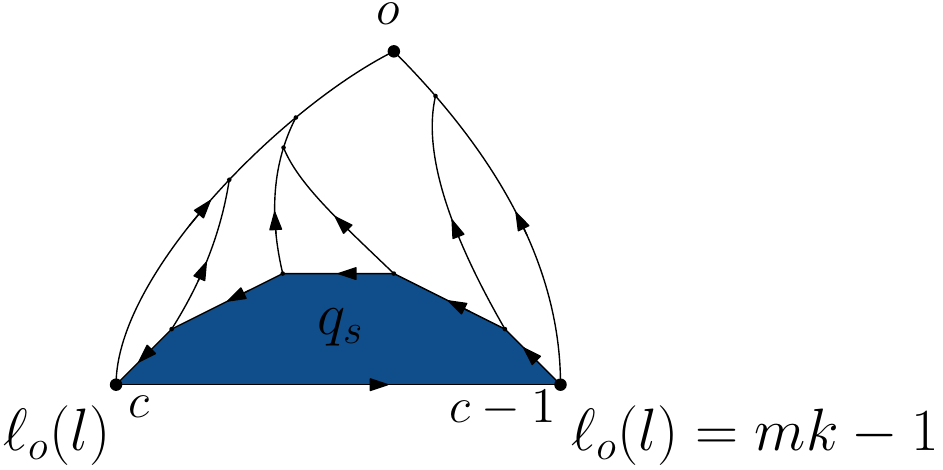}
	\caption{The slice decomposition of a white slice into elementary white slices (left). The slice decomposition applied to an elementary white slice after revealing the black face incident to the base (right).}
	\label{fig:SliceDecomposition}
\end{figure}

\subsection{Disc and cylinder functions from slices}
\subsubsection{The disc function} The proof of \Cref{thm:Disc} relies on a bijective argument originally due, in the case of maps, to \cite[Section 3.3]{BouttierGuitterContinuedFractions} which provides an expression for $W_{0,1}$. This argument was further used in \cite{AlbenqueBouttier:fpsac12} to derive $W_{0,1}$ in the case of constellations with $J=\emptyset$ and $q_i=\delta_{i,1}$ (in our notation). In our case, because there are two sets of controlled degrees, exploiting the bijection is more difficult and we are thankful to M.~Albenque for sharing some of the beautiful details of her derivation of $W_{0,1}(x)$ in \cite{AlbenqueBouttier2022}.

Consider an $m$-biconstellation, rooted on a white corner of color $c$, and also pointed, i.e. with a distinguished vertex (of any color). There is a leftmost, oriented path from the root vertex to the pointed vertex. We cut the map along that path and the root face, which shows the basic result relating biconstellations to slices.
\begin{theorem}{}{}[Adapted from \cite{BouttierGuitterContinuedFractions}]
Pointed, rooted biconstellations correspond to white slices of increment 0, where $l$ is the root corner and the origin is the pointed vertex.
\end{theorem}

Finding out the series of rooted, non-pointed maps is more difficult, but the plan is given as well in \cite{BouttierGuitterContinuedFractions}. Let  $\mathcal{M}$ be the set of connected, rooted planar $m$-biconstellations. Let $d\geq 0$ and $c\in[0..m-1]$. We introduce $\mathcal{M}^*_{d}$ the set of rooted, pointed $m$-biconstellations with pointed vertex $v^*$ and a root vertex $\vec{v}$ of color $c$, such that
\begin{itemize}
	\item $\ell_{v^*}(\vec{v}) = d$,
	\item all vertices $v$ on the root face satisfy $\ell_{v^*}(v) \geq d$.
\end{itemize}
Then
\begin{equation}
	\mathcal{M} = \mathcal{M}^*_{d=0} = \left(\cup_{d\geq 0} \mathcal{M}^*_d\right) \setminus \left(\cup_{d\geq 1} \mathcal{M}^*_d\right).
\end{equation}
The key point of \cite{BouttierGuitterContinuedFractions} is that the series of both sets into brackets can be evaluated independently, which is done in \cite{BCCGF22, AlbenqueBouttier2022} and leads to \Cref{thm:Disc}.

\subsubsection{Combinatorial origin of $Z(x)$} The ``change of variable'' $x\mapsto Z(x)$, \eqref{eq:ZasExcursions}, is the key of our parametrization and it clearly begs for a combinatorial interpretation, at least in the case of biconstellations again. An \emph{$m$-excursion} is an $m$-path starting at $(0,0)$, ending at $(n,0)$ for some $n$ and staying above~$0$. Our $m$-excursions will always be weighted like $m$-paths, with an extra weight $x^{-1/m}$ per step. That is to say, the weights are $x^{-1/m}A^{(c')}_k$ for each step at color $c'$ of type $(1,mk-1)$, $k\geq 0$.

Notice that if an $m$-path crosses the horizontal line of height $h$ at the color $c$, then it can only cross it again with the same color. Indeed, if the total height variation is zero after $p$ steps, one has $\sum_{j=1}^p (mk_j-1)=0$ for some $k_j\geq0$ which implies that $p$ is a multiple of $m$. Since colors decrease by 1 modulo $m$ along oriented edges, we conclude that it is always the same color which is encountered at a given height.

Let $Z^{(c)}$ be the generating series of $m$-excursions with an additional down step below 0. We can write schematically
\begin{equation*}
	Z^{(c)} = \begin{array}{c} \includegraphics[scale=.5]{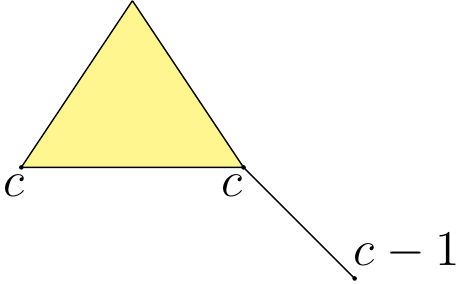} \end{array}
\end{equation*}
Let $Z$ then be the generating series of $m$-paths starting at $(0,0)$ with an arbitrary color $c$ and finishing at $(ms,-m)$ while staying above $-m$ until the last step, also counted with weight $x^{-1/m}A^{(c')}_k$ for each step at color $c'$ of type $(1,mk-1)$.% We will show that $\hat{Z} = \tilde{Z}$, where the latter is defined by \eqref{eq:Ztilde}.

We perform a first passage decomposition. Such a path starts with an excursion at color, say $c$, until it goes below 0 for the first time with a down step. The associated weight is $Z^{(c)}$. The down step goes from color $c$ to $c-1$. The path then follows an excursion starting at color $c-1$ until it goes below $-1$ for the first time and so on. There are exactly $m$ excursions of this type to reach height $-m$. This gives
\begin{equation*}
	Z = \begin{array}{c} \includegraphics[scale=.5]{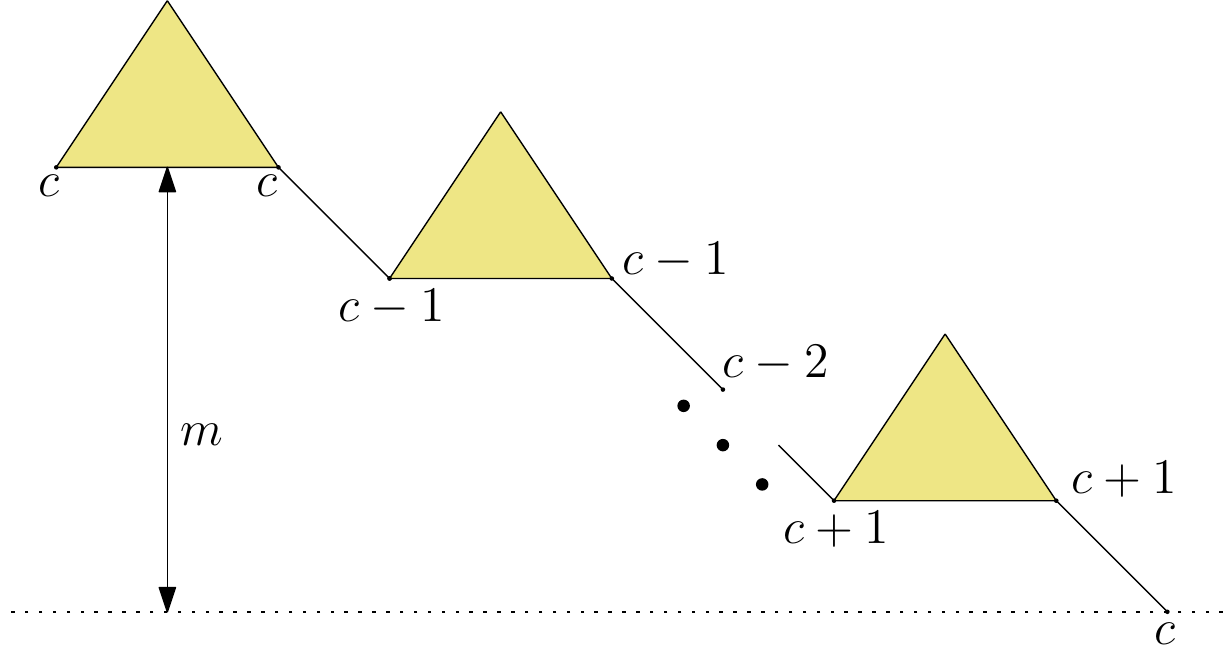}\end{array} =  \prod_{r=1}^{m} Z^{(c-r)}.
\end{equation*}
We then perform a first return decomposition to write an equation on $Z^{(c)}$. An excursion counted by $Z^{(c)}$ starts with an arbitrary step of type $(1,mk-1)$. If $k=0$, the $m$-excursion is over. If $k>0$, we perform again a first passage decomposition, decomposing the path into excursions above heights $c+mk-1, c+mk-2, \dots, c$. These $mk$ excursions can be arranged in consecutive groups of $m$ in which each group is formed by an excursion starting at color $i$ for all $i \in[0..m-1]$. The generating function for each group is precisely the object counted by $Z$ (up to circular permutation of excursions), hence
\begin{equation*}
	Z^{(c)}(x) = \includegraphics[scale=.5,valign=c]{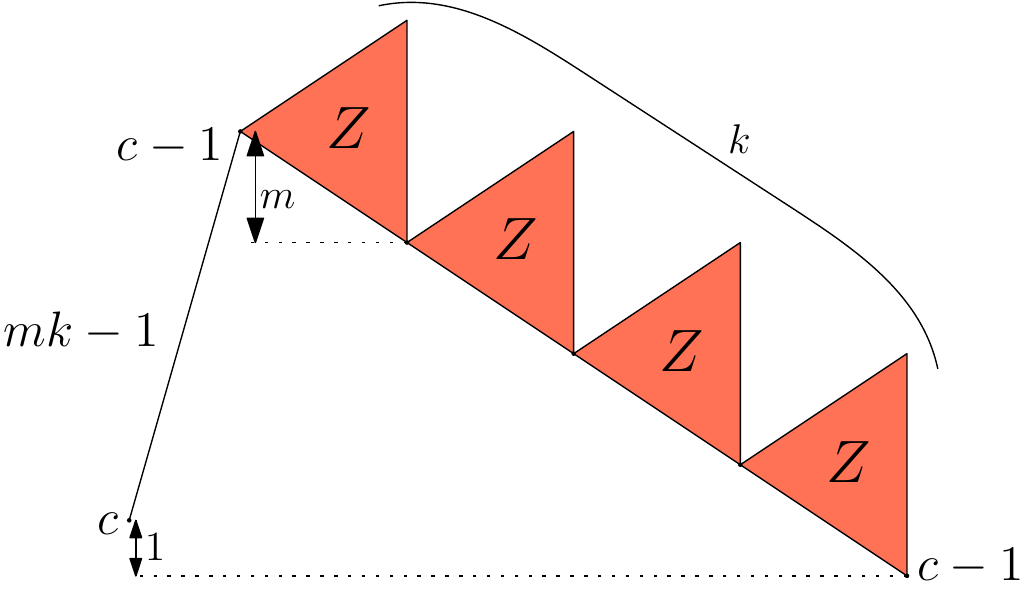} = x^{-1/m} \sum_{k=0}^{D_2} A_k^{(c)} 
	Z(x)^k = x^{-1/m} A^{(c)}(Z(x)).
\end{equation*}
We can in particular deduce that $A^{(c)}(Z(x))$ is the series of $m$-excursions with a final down step below 0 for which we ignore the weight $x^{-1/m}$ and that $Z(x)$ is determined by
\begin{equation} \label{eq:ZtildeSeries}
	Z(x) = \frac{1}{x} \prod_{c=0}^{m-1} A^{(c)}(Z(x)).
\end{equation}

\subsubsection{The cylinder function from slices} The magic of slices is even clearer for the decomposition of cylinders, because the pointing is canonical (although implicit). Let $p>0$. A $p$-\emph{annular} biconstellation is a rooted biconstellation with a marked black face of degree $mp$ such that any oriented cycle separating the root face from the marked face with the root face to its right has length larger than or equal to $mp$. A $p$-\emph{strict annular} biconstellation is a rooted biconstellation with a marked white face of degree $mp$ such that any oriented cycle separating the root face from the marked face with the root face to its left has length larger than $mp$.

The following proposition is originally due to Bouttier and Guitter \cite[Section 7]{BouttierGuitter2014} in the context of irreducible maps, and it was generalised to hypermaps in \cite{AlbenqueBouttier2022}, which we again follow.

\begin{proposition}{}{Annular}
	There is a bijection between slices of increment $-mp$ and $p$-annular maps. There is a bijection between slices of increment $mp$ and $p$-strict annular maps.
\end{proposition}
Those bijections have some very technical aspects to them and we will not describe them.

Consider a planar $m$-biconstellation $M$ with two distinct marked white faces $F_1, F_2$ of degrees $mf_1(M), mf_2(M)$, and a root corner of the same color in each of them. The series $W_{0,2}(x_1,x_2)$ is the generating series of those objects, counted with $x_1^{-f_1(M)-1} x_2^{-f_2(M)-1}$. We consider the unique shortest cycle separating $F_1$ and $F_2$ which has $F_2$ on its left and is as close as possible to it. Say it has length $mp$ for $p\geq 1$. Let us cut $M$ along that cycle. By definition, we obtain to its left a $p$-strict annular map and to its right a $p$-annular map whose bases have respective degrees $mf_1(M), mf_2(M)$.

Using Proposition \ref{thm:Annular}, we now have to enumerate slices of increment $-p$ with base of length $mf_1(M)$, and slices of increment $p$ with base of length $mf_2(M)$. The former correspond to $m$-paths starting at $(0,0)$ and ending at $(mf_1(M), -mp)$, while the latter end at $(mf_2(M), mp)$. They are respectively
\begin{align*}
	[\alpha^{-mp}] \prod_{r=1}^{mf_1} \Bigl(\sum_{i\geq 0} A_i^{(c-r+1)} \alpha^{mi-1}\Bigr) &= [\alpha^{-p}] \alpha^{-f_1} \prod_{c=0}^{m-1} A^{(c)}(\alpha)^{f_1}\\
	\text{and}\quad 
	[\alpha^{mp}] \prod_{r=1}^{mf_2} \Bigl(\sum_{i\geq 0} A_i^{(c-r+1)} \alpha^{mi-1}\Bigr) &= [\alpha^{p}] \alpha^{-f_2} \prod_{c=0}^{m-1} A^{(c)}(\alpha)^{f_2}.
\end{align*}
%Notice that they do not depend on the color chosen to start the paths at, as expected.

Given a $p$-annular constellation and a $p$-strict annular constellation, we can glue them along their marked faces, since they both have degree $mp$, one is white and the other black. There are moreover $p$ ways to do so (and not $mp$ because the colors of the vertices have to match for the gluing to be allowed). Therefore
\begin{equation}
	\begin{aligned}
		W_{0,2}(x_1, x_2) &= \sum_{f_1, f_2} x_1^{-f_1-1} x_2^{-f_2-1} \sum_{p\geq 1} p \left([\alpha^{-p}] \alpha^{-f_1} \prod_{c=0}^{m-1} A^{(c)}(\alpha)^{f_1}\right) \left( [z^{p}] z^{-f_2} \prod_{c=0}^{m-1} A^{(c)}(z)^{f_2}\right).\\
		&= x_1^{-1} x_2^{-1} \sum_{p\geq 1} p [z_1^{-p}] R(x_1,z_1)\ [z_2^p] R(x_2,z_2)
	\end{aligned}
\end{equation}
where $R(x,z) = \frac{1}{1-\frac{\prod_{c=0}^{m-1} A^{(c)}(z)}{xz}} \in \mathbb{Q}[\ptimes, \qtimes, U][[t]][z,z^{-1}][[x^{-1}]]$. A nice calculation allows for extracting $[z_1^{-p}] R(x_1,z_1)$ by isolating the ``small'' root of $1-\frac{\prod_{c=0}^{m-1} A^{(c)}(z)}{xz}$, i.e. the unique one which is a power series in $x^{-1}$. This then leads to \Cref{thm:Cylinder}.

\section{Topological recursion}
\subsection{The topological recursion} The topological recursion (TR) is a universal structure which has initially appeared in the context of matrix integrals and combinatorial maps. The first instance of the TR was devised by B.~Eynard in \cite{Eynard04} as an intrinsic solution to the loop equations/Tutte equations of those problems\footnote{We are extremely grateful to Bertrand Eynard, Nicolas Orantin and Gaëtan Borot for the time they took to explain the topological recursion to newbie me, as well as the time they devote to the community.}, see also \cite{AMM05, CE06, ChekhovEynardOrantin2006}. The foundations were then formalized by Eynard and Orantin in \cite{EO}. It has become a universal procedure which enables one to compute certain invariants, usually denoted by $\omega_{g,n}$, attached to the surface of genus $g$ with $n$ boundaries, from a small initial data consisting of an algebraic object called the \emph{spectral curve}. It has been gradually observed along the years that many natural problems in combinatorics or enumerative geometry were actually instances of the topological recursion, in the sense that the correlation functions $W_{g,n}$ naturally associated to these counting problems could be recovered by this procedure, with an appropriate spectral curve~\cite{EO, EynardOrantinWeilPetersson, BorotEynard2011, EMS11, Borot13, AC14, KZ15, DOPS17, AC18, ABCDGLW19, BCG21}. The TR is thus said in general to be ``counting surfaces'' \cite{Eynard:book}. It has grown into a vast field of research with deep mathematical connections in geometry and mathematical physics, related to the quantization of classical spectral curves \cite{EGMO21}, and of	Airy structures (which it inspired) \cite{KontsevichSoibelman2017, ABCO, HigherAiry} and to Givental's formula~\cite{DOSS}.

For pratical purposes, it is important to keep in mind that, although many examples of models giving rise to the TR have now been proved to exist, the proofs are often model-dependent. The question of the unification of the methods and results, and of giving conceptual explanations for TR to hold in vast generality is a subject of active interrogation.

The typical setting (not the most general though) is as follows. Let a spectral curve be $\mathcal{S}=(\mathbb{P}^1,V,x(z),y(z),\omega_{02})$ 
\begin{itemize}
	\item a branched covering $x:\mathbb{P}^1\to V$, where $V$ is an open subset of $\mathbb{P}^1$. The zeroes of $dx$ are the \emph{ramification points} and they are assumed to be simple.
	\item a meromorphic function $y$, such that the zeroes of $dy$ are different from those of $dx$.
	\item a symmetric bidifferential $\omega_{02}$ on $\mathbb{P}^1$ which behaves locally like
	\begin{equation}
		\omega_{02}(z_1,z_2) = \frac{dz_1 dz_2}{(z_1-z_2)^2} + \mathcal{O}(1)
	\end{equation}
\end{itemize}
We need to say a bit more about the spectral curve in order to write the TR. Let $b_1, \dotsc, b_R$ the ramification points of $x$, and denote $\sigma_i$, for $i=1,\dotsc, R$ the local involutions which satisfy
\begin{equation}
	\sigma_i(b_i) = b_i, \quad \sigma_i(z)\neq z \ \text{for $z\neq b_i$,}\quad x(\sigma_i(z))=x(z).
\end{equation}
The local kernel of the TR is
\begin{equation}
	K_i(z_1,z) \coloneqq \frac{1}{2}\frac{\int_{w=\sigma_i(z)}^{z}\omega_{0,2}(z_1,w)}{\omega_{0,1}(z)-\omega_{0,1}(\sigma_i(z))}.
\end{equation}

Let $\{\omega_{g,n}(z_1, \dotsc, z_n)\}_{\substack{2g+n-2\geq1\\ g\geq 0, n>0}}$ be a set of symmetric $n$-differentials on $\mathbb{P}^1$. We say that they satify the TR on $\mathcal{S}$ if
\begin{equation} \label{TR}
	\omega_{g,n}(z_1,L) = \sum\limits_{i=1}^{R} \Res_{z=b_i}\, K_i(z_1,z)\Big(\omega_{g-1,n+1}(z,\sigma_i(z),L)+\sum\limits_{\substack{h+h'=g \\ C\sqcup C'=L}}^{'}\omega_{h,1+|C|}(z,C)\omega_{h',1+|C'|}(\sigma_i(z),C')\Big),
\end{equation}
where $L=\{z_2,\dots,z_n\}$ and the notation for $\sum^{'}$ means that the terms $(h,C)=(0,\emptyset),\, (g,L)$ are excluded, so that no $\omega_{01}$ appears in the sum. It is true, though not obvious, that this formula produces symmetric differentials.

\subsection{Topological recursion and weighted Hurwitz numbers} The TR has been proved for various classes of weighted Hurwitz numbers, starting with the original root of it all, i.e. maps, but also simple Hurwitz numbers \cite{SimpleHurwitzTR} and bipartite maps \cite{KazarianZograf2015}. Recall that weighted Hurwitz numbers are determined by the choice of the function $G$, which we take of the form
\begin{equation*}
	G(z) = e^{w z} \frac{\prod_{i=0}^{m-1} (1+u_iz)}{\prod_{r=0}^{s-1}(1+v_r z)},
\end{equation*}
and their generating series is
\begin{equation*}
	\tau_G(\ptimes, \qtimes, t) = \sum_{\lambda} t^{|\lambda|} s_\lambda(\ptimes) s_\lambda(\qtimes) \prod_{\square\in\lambda} G(c(\square))
\end{equation*}
Maps correspond to $q_i = \delta_{i,2}$ and $G(z) = u+z$, bipartite maps to $q_i=\delta_{i,1}$ and $G(z) = (u+z)(v+z)$ and simple Hurwitz numbers to $q_i=\delta_{i,1}$ and $G(z) = e^{uz}$. %The interpretation in terms of factorizations of permutations for more general $G$s was described in Sections \ref{sec:WeightedHurwitzNumbers}, \ref{sec:TauFunctionsWeightedHurwitzNumbers} of Chapter \ref{sec:EnumerationIntegrability}.

What does it mean for those models to satisfy the TR? For this to make sense, we have to define the $\omega_{g,n}$s. Since we know how to extract the contributions of genus $g\geq 0$, see $F^G_g(\ptimes, \qtimes, t)$, Equation \eqref{FreeEnergyFixedGenus}, and we know how to create boundary faces (with the rooting operator $\nabla_x$), we define naturally
\begin{align}\label{eq:Wgn}
	W_{g,n}(x_1,\dotsc,x_n) \coloneqq  \Big( \nabla_{x_n} \dotsb \nabla_{x_1} F_g \Big) \Big|_{\substack{p_{i}=0, i> D_1\\q_{j}=0, j> D_2}}.
\end{align}
Combinatorially, $W_{g,n}$ is the generating function of weighted Hurwitz numbers on a surface of genus $g$ with $n$ boundaries, arbitrarily many internal faces, whose degrees are bounded by $D_1, D_2$ (depending on their color). We view it as a formal power series in $t$ whose coefficients are polynomials in the other variables, i.e.~an element of $\mathbb{Q}[\ptimes,\qtimes,\vec{u},\vec{v},w,x_1^{-1}, \dotsc, x_n^{-1}][[t]]$. It depends implicitly on the parameters of the function $G$ and on $D_1, D_2$, although we will not indicate it in the notation.

The TR then requires the functions $x(z), y(z)$. Although there is no general recipe to find those, and the details are model-dependent, experience shows that $y(z)$ is ``determined'' in some sense by $W_{0,1}(x(z))$ (they are sometimes equal), and 
\begin{equation}
	\omega_{0,2}(z_1,z_2) = \biggl(W_{0,2}(x(z_1),x(z_2)) + \frac{1}{(x(z_1)-x(z_2))^2}\biggr) dx(z_1) dx(z_2).
\end{equation}
Then for $2g+n-2\geq1$, the $n$-differentials which may satisfy the TR are expected to be
\begin{equation} \label{Omegagn}
	\omega_{g,n}(z_1, \dotsc, z_n) = W_{g,n}(x(z_1),\dotsc,x(z_n)) dx(z_1) \dotsm dx(z_n).
\end{equation}

The TR being satisfied in several instances of weighted Hurwitz numbers, that begs the question whether it might be satisfied in full generality. This program was started in \cite{AlexandrovChapuyEynardHarnadJ2018, AlexandrovChapuyEynardHarnad2020} where it was proved for \emph{constellations with no (white) internal faces}, i.e. setting $p_i = 0$ for all $i>0$ in \eqref{eq:Wgn}. This was re-proved in \cite{BychkovDuninBarkowskiKazarianShadrin2020}. Although those two groups used different ways of approaching the calculations, they both rely on the fact that $\tau_G$ is a KP tau function (this is explained in Chapter \ref{sec:EnumerationIntegrability}).

Let us explain why it was a little surprising in fact. In the original instances like maps, the TR was proved from the loop equations, which are generalizations of Tutte's equation to all genera and numbers of boundary components. In the case of maps and bipartite maps, those equations are \emph{quadratic} in the $W_{g,n}$s and this appeared to be directly the reason why the TR formula \eqref{TR} is quadratic. That was a problem for models like constellations where the loop equations are \emph{not} quadratic but of higher degree, see \cite{Fang:PhD} and it was not known whether the TR formula could hold in such models because of this. 

However, it was realized that only some properties derived from the loop equations were used to find \eqref{TR}, called linear and quadratic loop equations. As it turns out those linear and quadratic loop equations can actually be derived from the KP hierarchy, as shown by \cite{AlexandrovChapuyEynardHarnad2020}. Therefore, the (quadratic) TR could be obtained for constellations although their loop equations are not quadratic. %In a sense, the fact the TR is quadratic can be traced back in those models to the Plücker relations/KP hierarchy being quadratic.

The next part of the program was to prove the TR for constellations \emph{with} internal faces, and generalize to general $G(z)$. The group of \cite{BychkovDuninBarkowskiKazarianShadrin2020} did so in \cite{BychkovDuninBarkowskiKazarianShadrin2022} using techniques on the half-infinite wedge. We also proved the same result in \cite{BCCGF22}, using a different approach which we comment on a bit.

The functions $W_{0,1}$ and $W_{0,2}$ have to be provided for the TR to be proved and we derived them when $G(z)$ is polynomial in the previous section, in a purely combinatorial way (the general case can be deduced from the polynomial one using algebraic properties of the series, see \cite{BCCGF22}). In most works on the subject, the TR is proved by first obtaining strong structure results on the generating functions (showing that they have poles only at the branchpoints of the spectral curve, once expressed in the spectral variables), and on a number of additional equations (typically the linear and quadratic loop equations). In particular, the structure of generating functions is not deduced from TR, but proved before, or simultaneously.

In our work, we proceed differently. We use the idea of \emph{deformation of spectral curves}, taken from Eynard and Orantin's original paper~\cite{EO}. It proves the TR for arbitrary $p_1, \dotsc, p_{D_1}$ from the fact that it was already proved when $p_i=0$, and from our solution in genus $0$. In particular, we obtain the structure of the generating functions as a byproduct of TR, and not the converse. This is quite interesting from the combinatorial perspective since it means that the TR and the structure theorem is proved for objects with internal faces from the case with no internal faces, which is a non-trivial step.

\subsection{TR statement and structure corollary} 
\begin{theorem}{}{}
	Recall the series $X(z), Y(z)$ from Definition \ref{def:SpectralCurve}. 
	The $n$-differentials $\{\omega_{g,n}(z_1, \dotsc, z_n)\}$ defined in \eqref{Omegagn} satisfy the topological recursion \eqref{TR} with the spectral curve $\mathcal{S}_{\text{rational $G$}} \coloneqq (\mathbb{P}^1, \mathbb{P}^1, X(z), Y(z), \omega_{0,2} = \frac{dz_1 dz_2}{(z_1-z_2)^2})$.
\end{theorem}
Since the recourse to a spectral curve deformation is the main originality of our approach, let us sketch the way it works. We introduce a parameter $\alpha$ by the change $p_i\to \alpha p_i$ which will track the number of internal faces. Then we consider on one side the expansion of the $\omega_{g,n}$s of constellations defined by \eqref{Omegagn} in powers of $\alpha$, and on the other side the expansion of a family of $n$-differentials built directly from the spectral curve $\mathcal{S}_{\text{rational $G$}}$ via the TR formula. More explicitly,
\begin{itemize}
	\item let $\Gamma_x$ be the operator which transforms a boundary component into an internal face,
	\begin{equation}
		\Gamma_{x\to\ptimes} f = \sum_{k=1}^{D_1} \frac{p_k}{k}\, [x^{-k-1}] f(x).
	\end{equation}
	Clearly,
	\begin{equation}
		W_{g,n}(x_1, \dotsc, x_n) = \sum_{l\geq 0} \frac{\alpha^l}{l!} \Gamma_{x_{n+1}\to\ptimes} \dotsm \Gamma_{x_{n+l}\to\ptimes} W^0_{g,n+l}(x_1, \dotsc, x_n, x_{n+1}, \dotsc, x_{n+l})
	\end{equation}
	where $W^0_{g,n}(x_1, \dotsc, x_n)$ are the same functions at $\ptimes=0$. This leads to a ``Taylor expansion'' of the form $\omega_{g,n}(z_1, \dotsc, z_n) = \sum_{l\geq 0} \frac{\alpha^l}{l!} \omega_{g,n}^l(z_1, \dotsc, z_n)$.
	\item We define $\tilde{\omega}_{g,n}(z_1, \dotsc, z_n)$ for $n+2g-2\geq 1$ via the TR formula for $\mathcal{S}_{\text{rational $G$}}$. They admit an expansion in powers of $\alpha$ at fixed values of $X(z_1), \dotsc, X(z_n)$, $\tilde{\omega}_{g,n}(z_1, \dotsc, z_n) = \sum_{l\geq 0} \frac{\alpha^l}{l!} \tilde{\omega}_{g,n}^l(z_1, \dotsc, z_n)$. The crucial lemma is the proof that
	\begin{multline}\label{eq:chi:TR}
		\tilde{\omega}_{g,n}^{l}(z_1,L) = \sum\limits_{i=1}^{M\, D_2} \underset{z= \tilde{a}_i}{\Res}\, K_i(z_1,z) \Bigg(\sum\limits_{k=1}^{l}{l \choose k}\big(\tilde{\omega}^{k}_{0,1}(z)\, \tilde{\omega}^{l-k}_{g,n}(\sigma_i(z),L)+\tilde{\omega}^{k}_{0,1}(\sigma_i(z)\, \tilde{\omega}^{l-k}_{g,n}(z,L)\big)\\ 
		+\tilde{\omega}^{l}_{g-1,n+1}(z,\sigma_i(z),L) + \sum\limits_{\substack{h+h'=g\\ C\sqcup C'=L\\ k+k'=l}}^{'}{l \choose k} \tilde{\omega}_{h,1+|C|}^{k}(z,C) \tilde{\omega}_{h',1+|C'|}^{k'}(\sigma_i(z),C')\Bigg),
	\end{multline}
	where $L=\{z_2,\dots,z_n\}$. Here we hide details, like the meaning of $\tilde{a}_i$, the key point being a recurrence formula for the $\alpha$-expansion from the TR.
\end{itemize}
Finally one proves (by induction) that $\tilde{\omega}_{g,n}^{l}(z_1, \dotsc, z_n) = \omega_{g,n}^l(z_1, \dotsc, z_n)$.

Let $\mathbb{K}$ be the algebraic closure of $\mathbb{Q}(\ptimes,\qtimes,\vec{u})$ and $\mathbb{K}((t^*))$ be the (algebraically closed) field of Puiseux series (formal Laurent series in fractional powers of $t$) over this field. We denote $b_1(t), \dotsc, b_{MD_2}(t)\in \mathbb{K}((t^*))$ the ramification points, i.e. the zeroes of $dx(z)$.

\begin{corollary}{}{}
	\label{cor:structure}
	For $2g-2+n\geq 1$, $W_{g,n}(X(z_1),\dots, X(z_n))X'(z_1) \dots X'(z_n)$ can be written as a rational function. In each variable, the poles are located at the ramification points $b_i(t)$ and are of order at most $6g-4+2n$. More precisely, it belongs to the polynomial ring
	\begin{equation*}
		\mathbb{Q}\Big[\frac{1}{z_j-b_i(t)},\,b_i(t),\,\frac{1}{b_i(t)-b_{k}(t)},\, \frac{1}{Y'(b_i(t))},\, \frac{Y^{(k)}(b_i(t))}{Y'(b_i(t))},\,\frac{1}{X''(b_i(t))},\, \frac{X^{(k)}(b_i(t))}{X''(b_i(t))}\Big],
	\end{equation*}
	where for each $(g,n)$ a finite number of derivatives of $X$ and $Y$ at the ramification points contribute.
\end{corollary}

%%%%%%%%%%%%%%%%%
\chapter{Enumeration via integrability} \label{sec:EnumerationIntegrability}
\section{Some beautiful recursions for three families of (orientable) maps}
\subsection{Triangulations by number of edges and genus} Let $t^n_g$ be the number of rooted triangulations with $n$ edges and of genus $g$. Goulden and Jackson \cite{GouldenJackson2008} found a remarkable recurrence formula for those numbers
\begin{equation} \label{TriangulationRecurrence}
	(n+1) t^{n}_{g} = 4 n(3n-2) (3n-4) t^{n-1}_{g-1}+4\sum_{i+j=n-2\atop h+k=g}	(3i+2)(3j+2)t^{i}_{h} t^{j}_{k}.
\end{equation}
The coefficients are independent of $g$, meaning that the recurrence formula is equivalent to an \emph{ODE} on the series of the $t^n_g$s with respect to the parameter counting $n$. That ODE had also been written before in the physics literature \cite{KazakovKostovNekrasov1999}.

\subsection{Maps by number of edges and genus} Let $m^n_g\coloneqq m^n_g(u)$ be the generating series of rooted maps with $n$ edges and of genus $g$, with $u$ counting the number of vertices. Carrell and Chapuy first in \cite{CarrellChapuy2015}, then Kazarian and Zograf in \cite{KazarianZograf2015} found the following recurrence formula
\begin{multline} \label{MapRecurrence}
	(n+1) m^n_g = 2(1+u)(2n-1) m^{n-1}_g + \frac{1}{2}(2n-3)(2n-2)(2n-1) m^{n-2}_{g-1} \\+ 3 \sum_{i+j=n-2\atop h+k=g}	(2i+1)(2j+1) m^{i}_{h} m^{j}_{k}
\end{multline}
Again, the coefficients are independent of $g$ and the recursion is equivalent to an ODE.

\subsection{Bipartite maps by number of edges and genus} Let $b^n_g\coloneqq b^n_g(u,v)$ be the generating series of rooted bipartite maps with $n$ edges and of genus $g$, with $u$ counting the number of white vertices and $v$ the number of black vertices. Using the same technique as for \eqref{MapRecurrence}, Kazarian and Zograf \cite{KazarianZograf2015} found another recurrence formula of the same form
\begin{multline} \label{BipartiteRecurrence}
	(n+1) b^n_g = (2n-1)(1+u+v) b^n_{g-1} + (n-2)\bigl[4(u+v+uv)-(1+u+v)^2\bigr] b^n_{g-2} \\
	+ (n-1)^2(n-2) b^{n-1}_{g-2} + \sum_{i+j=n-2\atop h+k=g} 2(2+3i)j b^{i}_{h} b^{j}_{k}.
\end{multline}
Once again, it is equivalent to an ODE with respect to the parameter counting the number of edges.

Two key features of those recurrence formulas is that first they are simple: they are quadratic in the numbers of maps with polynomial coefficients in the other parameters, and second they are very efficient (it is even faster to work with the ODE for which methods exist which extract the coefficients in quasi-linear time \cite{aecf-2017-livre}).

\subsection{More recurrence formulas} More can be obtained, e.g. from bijections like maps by number of faces and vertices \cite{CarrellChapuy2015}. We list below other remarkable recurrence formulas.
\subsubsection{Constellations} Louf \cite{Louf2019} gave a similar formula for the number $C^{(m)}_{n,g}$ of rooted $m$-constellations with $n$ black faces and of genus $g$,
\begin{equation}
	\binom{n}{2} C^{(m)}_{n,g} = \sum_{{n_1+n_2=n\atop n_1, n_2\geq 1}\atop g = g_1+g_2+g^*} n_1 \binom{(m-1)n_2 + 2-2g_2}{2g^*+2} C^{(m)}_{n_1,g_1} C^{(m)}_{n_2,g_2}.
\end{equation}
\subsubsection{1-face maps} They are maps with a single face. Some beautiful recurrence formulas are known for them, which are linear with polynomial coefficients. They go as far back as the Harer-Zagier recursion for $\varepsilon_g^n$ the number of rooted 1-face maps with $n$ edges and of genus $g$ \cite{HarerZagier1986},
\begin{equation}
	(n+1) \varepsilon^n_g = 2(2n-1) \varepsilon_g^{n-1} + (n-1)(2n-1)(2n-3) \varepsilon^{n-2}_{g-1}.
\end{equation}
Adrianov \cite{Adrianov1997} gave an extension to the numbers $\beta^n_g$ of rooted bipartite 1-face maps with $n$ edges and of genus $g$,
\begin{equation}
	(n+1) \beta^n_g = 2(2n-1) \beta_g^{n-1} + (n-1)^2(n-2) \beta^{n-2}_{g-1}.
\end{equation}
Ledoux \cite{Ledoux2009} found a recurrence for \emph{non-oriented} 1-face maps,
\begin{multline} \label{Ledoux}
	(n+1)u^{n}_g = (8n-2) u^{n-1}_g-(4n-1) u^{n-1}_{g-1/2}
	+ n(2n-3)(10n-9) u^{n-2}_{g-1} - 8(2n-3) u^{n-2}_g 	\\
	-10(2n-3)(2n- 4)(2n-5) u^{n-3}_{g-1} 
	+ 5(2n-3)(2n-4)(2n-5) u^{n-3}_{g-3/2} \\
	+ 8(2n-3) u^{n-2}_{g-1/2} - 2(2n-3)(2n-4)(2n-5)(2n-6)(2n-7) u^{n-4}_{g-2},
\end{multline}
which has the same type of polynomial coefficients in $n$ (more on non-oriented maps in the next chapter).

\subsection{Robustness? Extensions?} In view of the above results, it is tempting to imagine some robustness of those recurrence formulas and the method(s) deriving it, if not some underlying universality. However, preliminary attempts to derive a recursion like \eqref{TriangulationRecurrence} for quadrangulations instead of triangulations have so far failed (not meaning that it is impossible, of course). Nevertheless, we were able to give in \cite{BCD-RecurrenceFormulas} some similar (albeit quite a bit longer) recurrence formulas for three families of \emph{non-oriented} maps, which are the same three families as above: \emph{triangulations}, \emph{general maps} and \emph{bipartite maps}.

\subsection{A mysterious origin} If multiple proofs of the Harer-Zagier recursion for 1-face maps exist (and the same proofs could probably be possible for Adrianov's and Ledoux' recursions), there is only one known way of deriving those for maps with multiple faces. It does not rely on Tutte's equations or bijective techniques, but on \emph{integrability} instead. The latter provides a set of partial differential equations(PDEs) on the generating series known as the \emph{KP hierarchy}.

The mystery relies on the fact that those PDEs are not understood from the more traditional approaches to maps. For instance, Tutte's equations determine all the relevant series, so in principle, one could derive the KP hierarchy from them, and yet it is not known how to do it. Moreover, the KP hierarchy has its own origin story, from the world of integrable systems and fluid dynamics, where there is no direct relation to maps. It can also be given a universal geometric origin in terms of infinite-dimensional Grassmannian, but this may be too general to understand the relation to maps. The KP hierarchy has an infinite set of solutions (see below) and series of maps are just a small subset. In fact, to point them out of all the solutions, an additional ingredient is needed, and (obviously) some Tutte's equations can be used. In this sense, integrability alone is not enough because it is not specific to the problem at hand, but it is nevertheless a key universal yet mystery ingredient in deriving all the above recurrence formulas.

On the bijective side, an explanation of the Harer-Zagier recursion for 1-face maps has been given by the Chapuy-F\'eray-Fusy's bijection with C-decorated trees \cite{ChapuyFerayFusy2013}. As for the recurrence formulas for maps with multiple faces, a bijection for planar maps has been found by Louf \cite{Louf2019}. Finally, we want to point out\footnote{And we thank M. Bousquet-M\'elou for actually pointing it out to us.} that Tutte's himself found a recursion of the same form, for \emph{properly $q$-colored planar maps} \cite{Bernardi2017}, using an equation {\it \`a la} Tutte. Whether it can be extended to all genera using integrability is an open problem.

\section{The KP hierarchy} \label{sec:KP}
\subsection{The partial differential equations of the hierarchy} The KP hierarchy is a set of partial differential equations on a formal series $\tau(\ptimes)\in\mathbb{C}[[p_1, p_2, \dotsc]]$, usually written as a generating sum known as Hirota bilinear equation,
\begin{equation} \label{KP}
	\sum_j h_j(-2\qtimes) h_{j+1}(\check{D}) e^{\sum_{r\geq 1} q_r D_r} \tau(\ptimes) \cdot \tau(\ptimes) = 0,
\end{equation}
where $\qtimes = (q_1, q_2, \dots)$ is a vector of free variables. Here we use the following notation:
\begin{itemize}%[itemsep=0pt, topsep=0pt,parsep=0pt, leftmargin=12pt]
	\item $h_j$ is the complete homogeneous symmetric function of degree $j$ \cite{Macdonald1995},
	\item $\check{D} = (k D_k)_{k\geq 1}$ where $D_r$ is the \emph{Hirota derivative} with respect to $p_r$ defined as the following bilinear mapping
	\begin{equation}
		\bigl(f(p_r),g(p_r)\bigr) \ \mapsto\ D_r f(p_r)\cdot g (p_r)\coloneqq \frac{\partial}{\partial s_r} f(p_r + s_r) g(p_r - s_r)_{|s_r=0}.
	\end{equation}
\end{itemize}

The PDEs are obtained by extracting the coefficient of $q_\mu$ in \eqref{KP} for arbitrary partitions $\mu$. For example, the ``first'' non-trivial equation of the hierarchy, known as the \emph{KP equation} is obtained by extracting the coefficient of $q_3$. It is usually written for $F(\ptimes) = \log\tau(\ptimes)$, as
\begin{equation}\label{eq:explicitKP}
	-F_{3,1}(\ptimes) + F_{2,2}(\ptimes) + \frac{1}{2} F_{1,1}(\ptimes)^2 + \frac{1}{12} F_{1,1,1,1}(\ptimes) = 0,
\end{equation}
with the notation $f_{i}(\ptimes) = \frac{\partial f(\ptimes)}{\partial p_i}$.

Solutions $\tau(\ptimes)$ of the hierarchy \eqref{KP} are called \emph{KP tau functions}.

%We present the KP hierarchy in its usual polynomial setting, where solutions are points in the orbit of the constant polynomial under $GL_\infty$. In combinatorics however, we need formal power series instead of polynomials. Solutions are thus not generated by $GL_\infty$ anymore, but the fermionic setting will still be relevant.
\subsection{Polynomial tau functions and the group $GL_\infty$} Denote 
\begin{equation}
	V = \bigoplus_{i\in\mathbb{Z}} \mathbb{C} e_i
\end{equation}
the infinite-dimensional complex vector space spanned by linearly independent vectors $e_i$s. In particular a vector has only a finite number of non-zero coordinates in that basis. Let $E_{ij}$ be the infinite-size matrix (up, down, left and right) with a 1 in position $(i,j)$ and 0s everywhere else. The endomorphisms of $V$ are infinite-size matrices with only a finite number of non-zero entries. They form the Lie algebra
\begin{equation}
	\mathfrak{gl}_\infty = \operatorname{span}\{ E_{ij}; i,j\in\mathbb{Z}\},
\end{equation}
whose Lie bracket is $[X,Y] = XY-YX$. It is the Lie algbra of the group 
\begin{equation}
	GL_\infty = \{A=(a_{ij})_{i,j\in\mathbb{Z}}; \text{$A$ is invertible and a finite number of $a_{ij}-\delta_{ij}$ are non-zero}\}
\end{equation}
We call non-trivial elements those $a_{ij}$ which are non-zero for $i\neq j$ or not equal to 1 for $i=j$.

Let $\ptimes = (p_1, p_2, \dotsc)$ an infinite sequence of indeterminates and $\mathbb{C}[p_1, p_2, \dotsc]$ their ring of polynomials. For each $n\in\mathbb{Z}$, it carries a representation of $GL_\infty$ as follows. For $A\in GL_\infty$ and partitions $\lambda, \mu$, denote
\begin{equation}
	A_{\lambda, \mu|n} = (a_{\lambda_i-i+n+1, \mu_j-j+n+1})_{i,j\geq 1}
\end{equation}
i.e. we pick up a subset of rows and columns of $A$ determined by $\lambda$ and $\mu$ and $n$. Then the representation $R_n(A)$ is defined by its action on Schur functions (the variables $p_1, p_2, \dotsc$ being seen as power-sums)
\begin{equation} \label{MinorExpansion}
	R_n(A) s_\lambda(\ptimes) = \sum_\mu \det (A_{\mu,\lambda|n})\ s_\mu(\ptimes),
\end{equation}
which is extended by linearity. Here the sum is finite. Indeed, all non-trivial elements fit in a square matrix $B=(a_{ij})_{-k\leq i,j\leq k}$ for some $k$. Among the columns of $A$ selected by $\lambda$, only a finite number of them go through this matrix, say $p\geq0$. For each column which does not go through $B$, $\mu$ must select the corresponding row. The freedom in $\mu$ is left to the choice of $p$ rows. If one of them does not go through $B$ however, the coefficients at the intersection with the columns which go through $B$ are all 0s, and so is the minor.

\begin{theorem}{}{PolynomialKP}\cite[Theorem 7.1]{Kac2013Bombay}
	$\tau\in\mathbb{C}[p_1, p_2, \dots]$ is a polynomial KP tau function if and only if there exists $A\in GL_\infty$ such that $\tau(\ptimes) = R_0(A) 1$. %In particular, the KP hierarchy \eqref{KP} is equivalent to the constraint \eqref{FermionicPlucker} via the boson-fermion correspondence.
\end{theorem}
The typical proof, as in \cite{Kac2013Bombay}, makes use of the so-called wedge representations of $GL_\infty$ on the half-infinite wedge, which are then shown to be equivalent to $R_n$, thereby proving the boson-fermion correspondence. All calculations can then be performed (easily, that is the point) on the half-infinite wedge and then translated over to $\mathbb{C}[p_1, p_2, \dotsc]$.

The corpus of references on the boson-fermion correspondence is huge, ranging from physics to pure mathematics. We found that Kac's books in general, or the literature from the Japanese school of Jimbo and Miwa, incorporate in general the right amount of proofs and details necessary to satisfy our mathematical curiosity.

\subsection{Combinatorial tau functions and the group $\overline{GL}_\infty$} In combinatorics however, we are interested in formal power series rather than polynomial solutions\footnote{Also in physics, one looks for non-polynomial solutions. In particular in the field where the KP hierarchy historically appeared, which is wave equations, the interesting solutions are \emph{solitons} which are exponentials.}. In this case, there is no theorem as above in the form of ``if and only if''. Nevertheless it is still possible to give sufficient conditions to find non-polynomial tau functions.

We found that the mathematical details of the level of \cite{Kac2013Bombay} become scarse when it comes to applications to formal power series. Even the combinatorial literature is definitely evasive on the nature of the extension of $GL_\infty$ required to work with formal power series. The paragraphs below are our tentative way of clarifying the framework for combinatorialists. There are obviously many mathematical subtleties which would require more discussion than it is reasonable here, like e.g. why the exponential map is well defined from $\overline{\mathfrak{gl}}_\infty$ to $\overline{GL}_\infty$\footnote{Hint: a kew word is \emph{locally nilpotent} \cite{JimboMiwa1983}.}.

We consider a group bigger than $GL_\infty$, which we denote $\overline{GL}_\infty$ (following \cite{Kac2013Bombay}, although we consider the transposed matrices)
\begin{multline}
	\overline{GL}_\infty = \{ A=(A_{ij})_{i,j\in\mathbb{Z}}; \text{$A$ is invertible and} \\ \text{only a finite number of $A_{ij}-\delta_{ij}$ are non-zero for $i\leq j$}\}
\end{multline}
In other words, those matrices only have a finite number of non-zero elements above the main diagonal. The Lie algebra $\overline{\mathfrak{gl}}_\infty$ is
\begin{equation}
	\overline{\mathfrak{gl}}_\infty = \{ X=(X_{ij})_{i,j\in\mathbb{Z}}; \text{only a finite number of $X_{ij}$ are non-zero for $i\leq j$}\}
\end{equation}
They both act by left multiplication on the space
\begin{equation}
	\overline{V} = \left\{\sum_{j\in\mathbb{Z}} v_j e_j; \text{$v_j=0$ for $j\ll 0$}\right\}
\end{equation}
in a well-defined manner. Moreover, the representation $R_n$ given in \eqref{MinorExpansion} can be extended to a representation of $\overline{GL}_\infty$ over $\mathbb{C}[[p_1, p_2, \dotsc]]$ (because all determinant minors are finite).

%We use the notation $A_{\lambda, \mu} = (A_{\lambda_i-i+1,\mu_j-j+1})$ (from the case of zero charge).
\begin{theorem}{}{TauFunctionSchurExpansion}
	Let $\tau(\ptimes)\in\mathbb{C}[[p_1, p_2, \dotsb]]$ and assume that there exists $A\in\overline{GL}_\infty$ such that $\tau(\ptimes) = R_0(A)1$, i.e.
	\begin{equation} \label{TauFunctionSchurExpansion}
		\tau(\ptimes) = \sum_\lambda \det(A_{\lambda, \emptyset}) s_\lambda(\ptimes).
	\end{equation}
	Then it is a KP tau function.
\end{theorem}

\subsection{Ways of recognizing a KP tau function} \label{sec:RecognizingKP} All cases I know from combinatorics rely on identifying the element $A\in\overline{GL}_\infty$ in some way (however it is my personal observation that the infinite-size matrix $A$ belongs to $\overline{GL}_\infty$, since this space is usually overlooked in the literature). Here I list three approaches.

\subsubsection{Algebraic approach} It is used to get the explicit expansion on Schur functions, like Frobenius' formula, then it remains to prove that the coefficients are determinant minors of the form $\det(A_{\lambda, \emptyset})$. For instance, Guay-Paquet and Harnad proved the following,
\begin{corollary}{}{GuayPaquetHarnad}\cite{Guay-PaquetHarnad2017}
	The generating series of weighted Hurwitz numbers \eqref{WeightedHurwitzSeries}, written as the so-called hypergeometric type \eqref{HypergeometricTauFunction},
	\begin{equation*}
		\tau_G(\ptimes, \qtimes, t) = \sum_{\lambda} t^{|\lambda|} s_\lambda(\ptimes) s_\lambda(\qtimes) \prod_{\square\in\lambda} G(c(\square))
	\end{equation*}
	is a KP tau function\footnote{It is with respect to both sets of times $\ptimes$ and $\qtimes$, and it satisfies in addition another hierarchy of partial differential equations which mix derivatives with respect to the two sets of times, and called the Toda hierarchy.}.
\end{corollary}
This is done by finding an explicit $A$ so that the coefficients $s_\lambda(\qtimes) \prod_{\square\in\lambda} G(c(\square))$ take the form $\det (A_{\lambda,\emptyset})$.
\subsubsection{Matrix integrals} It can be done for some models using matrix integrals \cite{MironovMorozovSemenoff1996, AlexandrovZabrodin2013} which provide $A$ directly through its wedge/fermionic representation. This works for instance for maps (the Hermitian 1-matrix model or GUE with a potential), bipartite maps (the matrix model for a complex rectangular matrix, or Wishart ensemble with a potential) and 1-biconstellations (generated by the Hermitian 2-matrix model).
\subsubsection{Via an evolution equation} In ``many'' cases I know, it is also possible to bypass the explicit Schur expansion and directly identify $A$ in the form $A = e^X$, by means of an \emph{evolution equation}. It is convenient to rescale $p_i\to t^i p_i$ for some formal parameter $t$ which counts the ``size'' of the object. In particular this gives $s_\lambda(\ptimes) \to t^{|\lambda|} s_\lambda(\ptimes)$ where $|\lambda| = \sum_i \lambda_i$ is the size of the partition $\lambda$. For a series $\tau\in\mathbb{Q}[p_1, p_2, \cdots][[t]]$, an evolution equation is an equation of the form
\begin{equation} \label{EvolutionEquation}
	\frac{\partial \tau(\ptimes, t)}{\partial t} = \hat{X} \tau(\ptimes, t),
\end{equation}
where $\hat{X}$ is a differential operator in the $p_i$s. Let us write $\tau$ as an exponential series,
\begin{equation}
	\tau(\ptimes, t) = \sum_{n\geq 0} \tau_n(\ptimes) \frac{t^n}{n!},
\end{equation}
then the evolution equation is equivalent to
\begin{equation} \label{EvolutionRecursion}
	\tau_{n+1}(\ptimes) = n![t^n]\hat{X} \tau(\ptimes, t) = n! \hat{X} \tau_n(\ptimes)
\end{equation}
where the last equality holds if $\hat{X}$ is independent of $t$, in which case we get a recursion on the coefficients $\tau_n(\ptimes)$. In the latter case, we can even write the solution to the evolution equation in a compact way which is useful to prove integrability. If $\tau(\ptimes,0) = 1$, then $\tau(\ptimes, t) = e^{t\hat{X}}1$. Depending on $\hat{X}$, one may identify $X\in\overline{\mathfrak{gl}}_\infty$ such that $e^{t\hat{X}} = R_0(e^{tX})$. If not, it may still be possible to prove that a Schur expansion of the form of \Cref{thm:TauFunctionSchurExpansion} is a solution of the evolution equation. Then one concludes with uniqueness of the formal power series solution of the evolution equation. This is what we did for non-oriented monotone Hurwitz numbers, see Chapter \ref{sec:Monotone}.

%%%%%%%%%%%%%%%%%%%%%%%%%%%%%%%%
\section{The case of weighted Hurwitz numbers} \label{sec:ConstellationsKPTauFunctions}
\subsection{1-constellations} A good example to become familiar with those notions and techniques is the case of 1-constellations, i.e. counting permutations according to their cyclic type\footnote{Obviously for 1-constellations, the factorization question is $\sigma_1\sigma_2=\mathbb{1}$, which requires the cycle types of both to be the same. This means that $\tau$ is a function of the products $p_i q_i$ and we set $q_i=1$ for simplicity.}. Let 
\begin{equation}
	\tau^{\text{1-const}}(\ptimes,t,u) = \sum_{n\geq 0} \frac{t^n}{n!} \sum_{\lambda\vdash n} |C_\lambda| u^{\ell(\lambda)} p_\lambda,
\end{equation}
where $|C_\lambda|$ is the number of permutations in the conjugacy class of cycle type $\lambda$. From the orbit-stabilizer theorem
\begin{equation}
	|C_\lambda| = \frac{n!}{z_\lambda},
\end{equation}
it comes that
\begin{equation}
	\tau^{\text{1-const}}(\ptimes,t,u) = e^{u\sum_{i\geq 0} \frac{p_i}{i}t^i} = \sum_{n\geq 0} h_k(up_1, up_2, \dotsc) t^k,
\end{equation}
so we know the tau function explicitly in this simple example, and we can analyze the tools at our disposal. We will show
\begin{proposition}{}{1Constellations}
	The above series of 1-constellations has the Schur expansion \eqref{TauFunctionSchurExpansion} with $A(t,u)=A_1, A_2$ given by
	\begin{align} \label{SchurExpansion1Constellations}
	A_1 &\equiv \bigl(h_{i-j}(\forall k\geq 1, q_k=t^k u)\bigr)_{0\leq j\leq i}\\
	A_2 &\equiv e^{t \sum_{i\in\mathbb{Z}} (i-1 + u) E_{i,i-1}}. \label{GL1Constellations}
	\end{align}
Both expressions imply that it is a KP tau function.
\end{proposition}

Notice that both the matrices $A_1$ and $A_2$ belong to $\overline{GL}_\infty$. Although they are distinct\footnote{The elements of $H(u)$ only depend on $i-j$, while 
	\begin{equation}
		e^{t\sum_{i\in\mathbb{Z}} (i-1+u) E_{i,i-1}} = \sum_{k\geq 0} \frac{t^k}{k!} \sum_{i\in\mathbb{Z}} \prod_{l=1}^k (i-l+u) E_{i,i-k} = \sum_{i\geq j} \frac{t^{i-j}}{(i-j)!} \prod_{l=1}^{i-j} (i-l+u) E_{ij}
	\end{equation}
	does not only depend on $i-j$.
}, they have the same determinant minors.

The Schur expansion is readily obtained via the Cauchy identity,
\begin{equation}
	e^{\sum_{i\geq 0} \frac{p_i q_i}{i}} = \sum_{\lambda} s_\lambda(\ptimes) s_\lambda(\qtimes).
\end{equation}
The principal specialization \eqref{PrincipalSpecialization} for $q_i = u$ for all $i\geq 1$ gives
\begin{equation}
	\tau^{\text{1-const}}(\ptimes, t, u) = \sum_\lambda \frac{t^{|\lambda|}s_\lambda(\ptimes)}{\hook(\lambda)}\prod_{\Box\in\lambda}(u+c(\Box)).
\end{equation}
Proving the KP hierarchy is then equivalent to proving that\footnote{There is no need to specialize to $q_k=u$ in fact, but also no benefits to keeping arbitrary $q_k$s.} $t^{|\lambda|} \prod_{\Box\in\lambda}(u+c(\Box))/\hook(\lambda)$ is a determinant of the form $\det A_{\lambda,\emptyset}$. Here it comes for instance from the Jacobi-Trudi formula for Schur functions \cite{Macdonald1995},
\begin{equation}
	s_\lambda(\qtimes) = \det (h_{\lambda_i-i+j}(\qtimes))_{1\le i,j\leq \ell(\lambda)}
\end{equation}
whose specialization to $q_i=u$ gives
\begin{equation}
	\frac{t^{|\lambda|}}{\hook(\lambda)}\prod_{\Box\in\lambda}(u+c(\Box)) = \det (A(t,u)_{\lambda, \emptyset}) \qquad \text{for $A(t,u)_{i,j} = h_{i-j}(q_k =t^k u)$.}
\end{equation}
which proves \eqref{SchurExpansion1Constellations}.

Now we prove \eqref{GL1Constellations} through an evolution equation. To write it, we can think of a permutation of cyclic type $\lambda$ as a graph on $n$ labeled vertices consisting in a disjoint union of cycles of size $\lambda_1, \dotsc, \lambda_{\ell(\lambda)}$. As we have seen in \eqref{EvolutionRecursion}, the operator $\frac{\partial}{\partial t}$ on exponential series adds the vertex of label $n+1$, so we need to understand how to build an object of size $+1$ from one of size $n$. If $\sigma$ is a set of disjoint cycles on $n$ vertices, the vertex of label $n+1$ can be added by
\begin{itemize}
	\item first choosing a cycle of size $k \geq 1$, and choosing one of $k$ slots to insert the new vertex, giving rise to a cycle of size $k+1$. On the generating series, this gives the operator $kp_{k+1} \frac{\partial}{\partial p_k}$.
	\item or adding a new connected component of size 1 with weight $u p_1$.
\end{itemize}
One gets
\begin{equation}
	\frac{\partial \tau^{\text{1-const}}(\ptimes, t, u)}{\partial t} = \underbrace{\left(\sum_{k\geq1} p_{k+1} p_k^* + up_1\right)}_{\hat{X}(u)}\tau^{\text{1-const}}(\ptimes, t, u).
\end{equation}
where $p_k^* \equiv \frac{k\partial}{\partial p_k}$. The operator $\hat{X}(u)$ can be proved, via the boson-fermion correspondence, to be a representation of\footnote{There is a famous construction of the Virasoro algebra of charge 1 \cite{Kac2013Bombay}, known as Sugawara's construction, which gives
	\begin{equation}
		\hat{L}_n = \begin{cases} \sum_{k\geq1} p_{k+n} p_k^* + \frac{1}{2} \sum_{k=1}^{n-1} p_k^* p_{n-k}^* \quad \text{for $n>0$},\\
			\sum_{k\geq1} p_{k-n} p_k^* + \frac{1}{2} \sum_{k=1}^{-n-1} p_k p_{-n-k} \quad \text{for $n<0$} \end{cases}
	\end{equation}
	Then our operator is $X(u) = L_{-1} + u p_{1}$.}
\begin{equation}
	X(u) = \sum_{i\in\mathbb{Z}} (i-1 + u) E_{i,i-1}
\end{equation}
of $\overline{\mathfrak{gl}}_\infty$. It comes	$\tau^{\text{1-const}}(\ptimes, t, u) = e^{t\hat{X}(u)}1 = R_0(e^{tX(u)})1$.

\subsection{Weighted Hurwitz numbers} \Cref{thm:SchurExpansionWeightedHurwitz} of Guay-Paquet and Harnad provides an explicit Schur expansion of the generating series of weighted Hurwitz numbers $\tau(\ptimes, \qtimes, t, \vec{u}, \vec{v}, w)$. That type of Schur expansion, known as hypergeometric tau functions,
\begin{equation*}
	\tau_G(\ptimes, \qtimes, t) = \sum_\lambda s_\lambda(\ptimes) s_\lambda(\qtimes) \prod_{\Box\in\lambda} G(c(\Box))
\end{equation*}
has been studied by Orlov and Scherbin in \cite{OrlovScherbin2000} where two matrices $A$ are provided. 
\begin{theorem}{}{}\cite{OrlovScherbin2000}
	One has
	\begin{align*}
		\tau_G (\ptimes, \qtimes, t) &= D_G R_0\bigl(e^{\sum_{k\geq 1} \frac{q_k}{k} t^k \sum_{i\in\mathbb{Z}} E_{i,i-k}}\bigr) 1\\
		&= R_0\bigl(e^{\sum_{k\geq 1}\frac{q_k}{k} t^k \sum_{i\in\mathbb{Z}} G(i-1) \dotsm G(i-k) E_{i,i-k}}\bigr)1.
	\end{align*}
	where $D_G$ acts diagonally on the Schur basis of $\mathbb{Q}[[p_1, p_2, \dotsc]]$\footnote{In the first case, it is not an element of $\overline{GL}_\infty$ because of $D_G$. Diagonal matrices with an infinite number of elements different from 1 can lead to divergences. They can be corrected by introducing the normal ordering. This is a more general setup than using elements of $\overline{GL}_\infty$ which is often used in mathematical physics \cite{AlexandrovZabrodin2013}.}.
\end{theorem}

\section{Recurrence formulas from the KP equation} \label{sec:Elimination} In order to derive the recurrence formulas \eqref{TriangulationRecurrence}, \eqref{MapRecurrence}, \eqref{BipartiteRecurrence}, we ultimately use the KP equation
\begin{equation*}
	-F_{3,1}(\ptimes) + F_{2,2}(\ptimes) + \frac{1}{2} F_{1,1}(\ptimes)^2 + \frac{1}{12} F_{1,1,1,1}(\ptimes) = 0,
\end{equation*}
with $F(\ptimes) = \log \tau(\ptimes)$ and $f_i = \frac{\partial f(\ptimes)}{\partial p_i}$.

However, we face two issues:
\begin{itemize}
	\item The KP equation (or even the hierarchy) does not determine a unique solution.
	\item We must get rid of the partial derivatives.
\end{itemize}
Both issues will be solved thanks to a set of additional equations, which are model dependent. At first, we introduce them to solve the first issue (then they will also be used to solve the second issue). In the three cases of triangulations, general maps and bipartite maps, those equations are the so-called Virasoro constraints. Let as in Section \ref{sec:SingleHurwitz}
\begin{equation} \label{MapsTauFunctions}
	\begin{aligned}
		\tau^{\text{Maps}}(t;\ptimes,u) &= \sum_{n \geq 0} t^n\sum_{\lambda \vdash n} 
		s_\lambda(\ptimes)\theta_\lambda\prod_{\square
			\in \lambda}(u-c(\square)), \\
		\tau^{\text{Bip}}(t;\ptimes,u,v) &= \sum_{n \geq 0} t^n\sum_{\lambda \vdash n} 
		\frac{s_\lambda(\ptimes)}{\operatorname{hook}(\lambda)}\prod_{\square
			\in \lambda}(u-c(\square))(v-c(\square)),
	\end{aligned}
\end{equation}
They satisfy the partial differential equations
\begin{equation}
	\begin{aligned}
		L^{\text{Maps}}_i \tau^{\text{Maps}}(t;\ptimes,u) &= 0, \quad \text{for $i\geq -1$},\\
		L^{\text{Bip}}_i \tau^{\text{Bip}}(t;\ptimes,u,v) &= 0, \quad \text{for $i\geq 0$.}
	\end{aligned}
\end{equation}
with
\begin{equation} \label{VirasoroConstraints}
	\begin{aligned}
		L^{\text{Maps}}_i &= \frac{p_{i+2}^*}{t^2} - \bigg(\sum_{\substack{m,n\geq1\\m+n=i}}p_m^*p_n^*+ \sum_{n \geq 1} p_n p^*_{n+i} + (i+1+2u)p_i^* 
		+ \delta_{i,-1}up_1+u^2\delta_{i,0}\bigg),\\
		L^{\text{Bip}}_i &= \frac{p_{i+1}^*}{t} -\bigg(\sum_{\substack{m,n\geq1\\m+n=i}}p_m^*p_n^*+ \sum_{n \geq 1} p_n p^*_{n+i} + (i+u+v)p_i^*+uv\delta_{i,0}\bigg).
	\end{aligned}
\end{equation}
and $p_k^* \equiv \frac{k\partial}{\partial p_k}$, which determine the tau functions uniquely (however we do not know how to get the recurrence formulas directly from them). Those constraints\footnote{They are called Virasoro constraints because the satisfy $[L_n,L_m] = (n-m) L_{n+m}$ (up to shifts of the indices). The appearance of that algebra is well documented in the physics literature since the 80s in the context of matrix integrals. It reveals deep connections to conformal field theories (or vertex operator algebras) and to the Airy structures introduced by Kontsevich and Soibelman \cite{KontsevichSoibelman2017}.} can be typically found via Tutte's classical approach of removing the root edge ($i$ above is the root face degree in that context).

To solve the second issue, we use the elimination technique introduced by Goulden and Jackson \cite{GouldenJackson2008}. First we specialize our tau functions by evaluating the times $p_i$s on some parameter so that the genus can be controlled. We introduce the specialization operators
\begin{equation}
	\theta_T : p_i \mapsto z \delta_{i,3}, \quad \theta_{M} : p_i \mapsto z.
\end{equation}
In order to eliminate the partial derivatives in the KP equation, we act with those specialization operators on the Virasoro constraints. We have three cases to consider and denote
\begin{equation}
	F^{\text{Maps}}(t;\ptimes,u) = \log \tau^{\text{Maps}}(t;\ptimes,u),\qquad F^{\text{Bip}}(t;\ptimes,u,v) = \log \tau^{\text{Bip}}(t;\ptimes,u,v).
\end{equation}
Then
\begin{itemize}
	\item $\Phi^{\text{Tri}}(t,z,u)\coloneqq \theta_T F^{\text{Maps}}(t;\ptimes,u)$ the series of triangulations counted with $t$ on edges, $z$ on triangles and $u$ on vertices
	\item $\Phi^{\text{Maps}}(t,z,u)\coloneqq \theta_M F^{\text{Maps}}(t;\ptimes,u)$ the series of maps counted with $t$ on edges, $z$ on faces and $u$ on vertices
	\item $\Phi^{\text{Bip}}(t,z,u,v)\coloneqq \theta_M F^{\text{Bip}}(t;\ptimes,u,v)$ the series of bipartite maps counted with $t$ on edges, $z$ on faces, and $u$ on black vertices and $v$ on white vertices.
\end{itemize}
%and
%\begin{equation}
%	\begin{aligned}
	%	F^{\text{Tri}}_{i,j,\dotsc}(t,z,u) = \theta_T F^{\text{Maps}}_{i,j,\dotsc}(t;\ptimes,u)\\
	%	\Phi^{\text{Maps}}_{i,j,\dotsc}(t,z,u) = \theta_M F^{\text{Maps}}_{i,j,\dotsc}(t;\ptimes,u)\\
	%	\Phi^{\text{Bip}}_{i,j,\dotsc}(t,z,u) = \theta_M F^{\text{Bip}}_{i,j,\dotsc}(t;\ptimes,u)
	%	\end{aligned}
%\end{equation}

The specialized KP equation reads for triangulations
\begin{equation}
	-\theta_T \bigl(F^{\text{Maps}}_{3,1}\bigr) + \theta_T \bigl(F^{\text{Maps}}_{2,2}\bigr) + \frac{1}{2} \theta_T \bigl(F^{\text{Maps}}_{1,1}\bigr)^2 + \frac{1}{12} \theta_T \bigl(F^{\text{Maps}}_{1,1,1,1}\bigr) = 0,
\end{equation}
which is an equation in $\mathbb{Q}[u,z][[t]]$, and same for maps and bipartite maps.

The following lemma claims that all those terms are polynomials in $\Phi^{\text{Tri}}(t,z,u)$ and its $t$-derivatives.
\begin{lemma} \label{thm:Elimination} \cite{KazarianZograf2015, GouldenJackson2008}
	The functions $\theta_T \bigl(F^{\text{Maps}}_{3,1}\bigr)$, $\theta_T \bigl(F^{\text{Maps}}_{2,2}\bigr)$, $\theta_T \bigl(F^{\text{Maps}}_{1,1}\bigr)$, $\theta_T \bigl(F^{\text{Maps}}_{1,1,1,1}\bigr)$ are polynomials in $t,z,u$ and the derivatives $\frac{\partial^k\Phi^{\text{Tri}}}{\partial t^k}$ for some finite order $k$. Same for maps and bipartite maps.
\end{lemma}

\begin{proof}[Illustration of the lemma] Let us find a polynomial expression for $\theta_M F^{\text{Maps}}_1(t,z,u)$ (the lemma follows from the same mechanism). Use the constraint
	\begin{equation}
		L_{-1} = - \frac{\partial}{\partial p_1} + t^2 \sum_{n\geq 2} (n-1) p_n \frac{\partial}{\partial p_{n-1}} + t^2 up_1
	\end{equation}
	which gives
	\begin{equation} \label{SpecializationConstraint}
		\theta_M F^{\text{Maps}}_1(t,z,u) = t^2 z \sum_{n\geq 2} (n-1) \theta_M F^{\text{Maps}}_{n-1}(t,z,u) + t^2uz
	\end{equation}
	It might seem helpless because of the sum over $n$, but this sum can in fact be expressed via the homogeneity equation
	\begin{equation}
		t\frac{\partial \tau^{\text{Maps}}}{\partial t} = \sum_{n\geq 1} np_n \frac{\partial\tau^{\text{Maps}}}{\partial p_n}.
	\end{equation}
	Its specialization through $\theta_{M}$ gives
	\begin{equation}
		t\frac{\partial\Phi^{\text{Maps}}(t,z,u)}{\partial t} = z\sum_{n\geq 1} n \theta_M F^{\text{Maps}}_n(t,z,u),
	\end{equation}
	and plugging this into \eqref{SpecializationConstraint} gives
	\begin{equation}
		\theta_M F^{\text{Maps}}_1(t,z,u) = t^3 \frac{\partial\Phi^{\text{Maps}}(t,z,u)}{\partial t} + t^2uz
	\end{equation}
\end{proof}

%%%%%%%%%%%%
\chapter{Non-oriented maps and the BKP hierarchy} \label{sec:NonOrientedMaps}

\section{Non-oriented maps}
We recall that orientable maps can be described as gluings of polygons (Section \ref{sec:PolygonGluings} of Chapter \ref{sec:Definitions}), as properly embedded graphs (Section \ref{sec:EmbeddedGraphs} of Chapter \ref{sec:Definitions}) and as factorizations of permutations (Chapter \ref{sec:TR}). Non-oriented maps are generalizations of orientabled maps, and the three definitions can indeed be generalized.

\subsection{Non-oriented maps as gluings of polygons} We must first supplement the definition of Chapter \ref{sec:Definitions} Section \ref{sec:PolygonGluings} with the place where orientability appears. Not discussed there, is the fact that there are two ways of identifying two sides of polygons together,
\begin{equation}
	\includegraphics[scale=.6,valign=c]{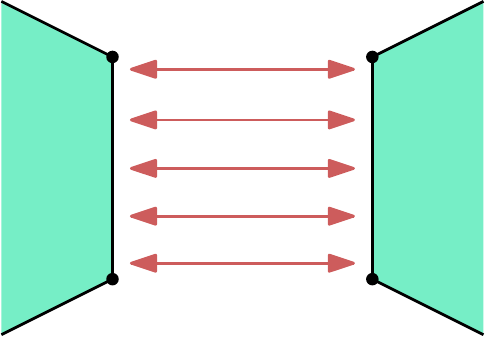} \hspace{2cm} \includegraphics[scale=.6,valign=c]{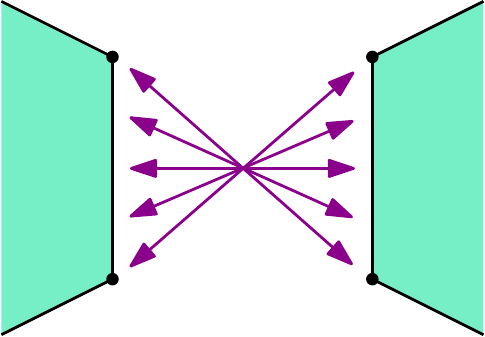}
\end{equation}
Let us orient every polygon, so that their sides inherit an orientation. Then two sides can be identified in two ways: by matching their orientations or with opposite orientations. If it is possible to adapt the orientations of the polygons so that the sides which are identified always have opposite orientations, then the map is orientable, and the other way around.

For instance, take a single oriented $n$-gon and equip its sides with their inherited orientation. There are $\operatorname{Cat}_n$ ways to identify the sides 2 by 2 to create a planar map, homeomorphic to the sphere, as illustrated in the left of Figure \ref{fig:OrientedPolygon} and they can all be realized by identifying the sides with opposite orientations. However, if the gluing of two sides is switched to matching orientations, as in the right of Figure \ref{fig:OrientedPolygon}, the topology becomes that of the (non-orientable) projective plane.
\begin{figure}
	\includegraphics[scale=.5,valign=c]{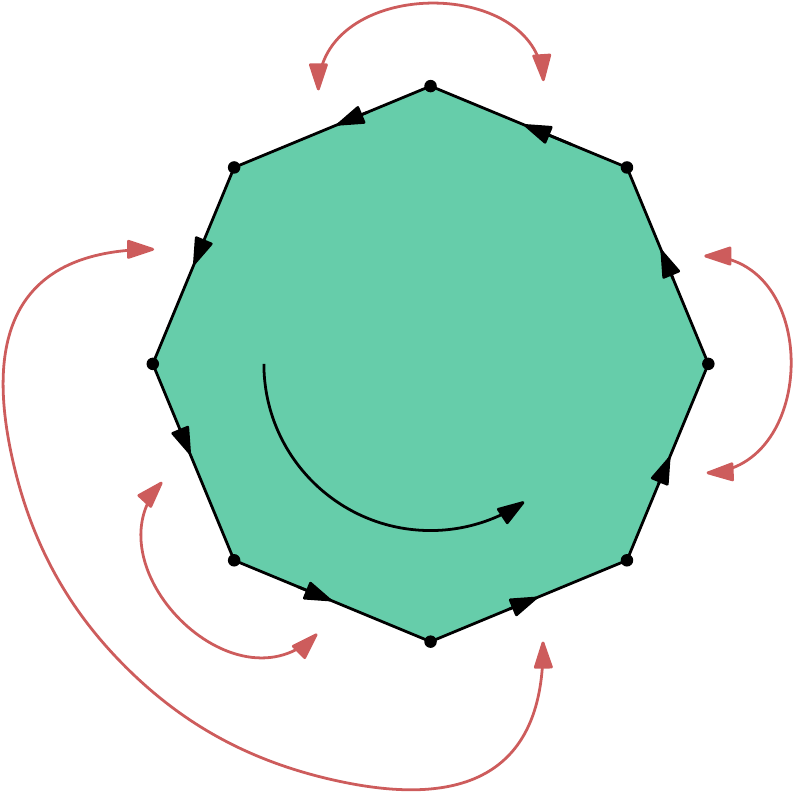} \hspace{2cm} \includegraphics[scale=.5,valign=c]{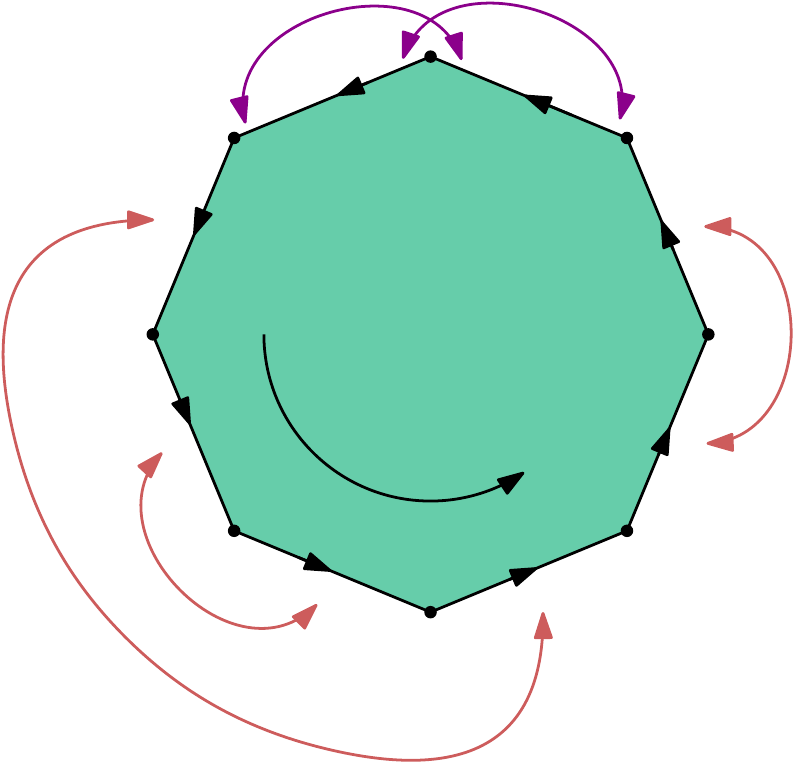}
	\caption{\label{fig:OrientedPolygon}}
\end{figure}

\subsection{Non-oriented maps as embedded graphs} The second definition is as graphs which are embedded in non-oriented surfaces such that the graph complement is a disjoint union of discs (corresponding to the above polygons). 

When drawing them in the plane, it is not enough to specify a rotation system as in the orientable case. A common way of incorporating non-orientability is by having \emph{twisted edges}, i.e. edges such that a face along it switches sides. The name ``twisted'' comes from the ribbon representation where the embedded graph is thickened to a ribbon. In the non-oriented case, ribbons are allowed to get twisted. Notice that the representation is dual to the first definition with gluings of polygons, in the sense that the dual to an edge coming from two sides with matching orientations is a twisted edge.

An illustration of a graph with three embeddings of different topologies is given in Figure \ref{fig:ThreeTopologies}. It has one vertex and two edges, but can give rise to a sphere, a torus and a projective plane.
\begin{figure}
	\includegraphics[scale=.6,valign=c]{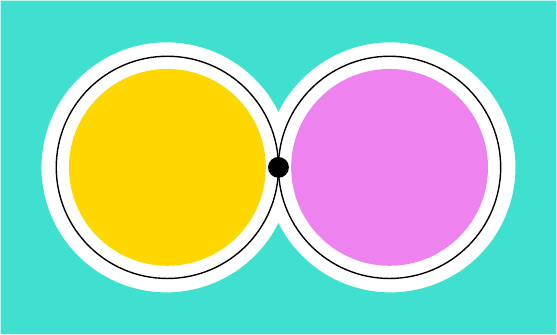} \hspace{1cm} \includegraphics[scale=.6,valign=c]{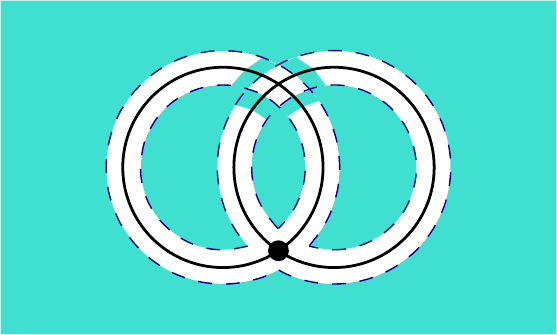} \hspace{1cm} \includegraphics[scale=.6,valign=c]{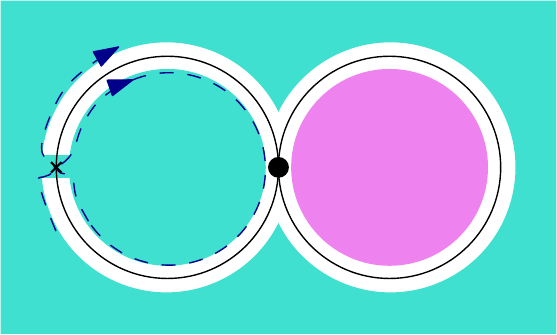}
	\caption{\label{fig:ThreeTopologies}}
\end{figure}

%%%%%%%%%%%%%%%%%%%%%%%%%%%%
\section{Recurrence formulas for non-oriented maps} 
Here we present the recurrence formulas we found for non-oriented maps which extend those of Chapter \ref{sec:EnumerationIntegrability} for orientable maps. Contrary to the orientable case however, we found two types of formulas. To understand why, we need to explain a bit their origin. 

\subsection{From the BKP equation} In the orientable case, they were derived from the KP equation, which is satisfied by generating series of several families of maps because they are tau functions of the KP hierarchy. In the non-oriented case, the latter is replaced with a different hierarchy, the BKP hierarchy of Kac and van de Leur \cite{KacVandeLeur1998}. An important difference between the KP and BKP hierarchy is that the function $F(\ptimes)$ which is a solution of this hierarchy also involves a so-called \emph{charge parameter} $N$, which in our context will always be a variable marking faces or vertices of a certain kind. The first BKP equation reads
\begin{equation}
	\label{BKP1Log}
	F_{2^2}(N,\ptimes) - F_{3,1}(N,\ptimes) + \tfrac{1}{2} F_{1^2}(N,\ptimes)^2 + \tfrac{1}{12} F_{1^4}(N,\ptimes)=
	S_2(N) \frac{\tau(N\!-\!2,\ptimes) \tau(N\!+\!2,\ptimes)}{\tau(N,\ptimes)^2},
\end{equation}
where (as usual) $\tau(N,\ptimes)=e^{F(N,\ptimes)}$, $f_i(\ptimes) = \frac{\partial f(\ptimes)}{\partial p_i}$ and where $S_2(N)$ is a model-dependant normalizing factor that will always be an explicit rational function in our case. In~\cite{Carrell2014}, Carrell used the fact that the generating function of non-oriented maps satisfies this equation, together with the elimination techniques developed by Goulden and Jackson in the orientable case, to obtain a functional equation for the case of triangulations (this technique leads to an explicit recurrence).

In \cite{BCD-RecurrenceFormulas}, we used the same method again to derive recurrence formulas for maps and bipartite maps in addition to triangulations, thereby really extending the orientable case to the non-oriented case. All these recurrences are larger than those of the orientable case, but incredibly short compared to any alternative, and it is not unreasonable to believe that they could have a combinatorial interpretation. For example, we obtain in \cite{BCD-RecurrenceFormulas} the following recurrence formula. In this chapter, the symbol $\hsum$ denotes a sum over elements of $\tfrac{1}{2}\mathbb{N}$.
\begin{theorem}{}{NonOrientedMapsCC}
	The number $\mathfrak{h}_n^g$ of rooted maps of genus $g$ with $n$ edges, orientable or not, can be computed from the following recurrence formula:
	\begin{multline} \label{eq:recurrenceCC}
		\mathfrak{h}_n^g
		=
		\tfrac{2}{(n+1)(n-2)} \Bigg(
		n(2n-1)( 2 \mathfrak{h}_{n-1}^{g} + \mathfrak{h}_{n-1}^{g-1/2})
		+   \tfrac{(2 n-3) (2 n-2) (2 n-1) (2 n)}{2} \mathfrak{h}_{n-2}^{g-1}
		\\ 
		+12 \mkern-10mu\hsum_{g_1=0..g\atop g_1+g_2=g} \sum_{n_1=0..n \atop n_1+n_2=n} 
		\mkern-10mu\tfrac{(2 n_2-1)(2 n_1-1)n_1}{2}        \mathfrak{h}_{n_2-1}^{g_2}          
		\mathfrak{h}_{n_1-1}^{g_1}   
		-\mkern-10mu
		\hsum_{g_1=0..g\atop g_1+g_2=g} \sum_{n_1=0..n-1\atop n_1+n_2=n}
		\hsum_{g_0=0..g_1\atop g_1-g_0\in \mathbb{N}} \mkern-10mu {\textstyle {n_1+2-2 g_0 \choose n_1-2 g_1}} 
		2^{2 (1+g_1-g_0)}
		\mathfrak{h}_{n_1}^{g_0}
		\\
		\Big(
		\tfrac{(2 n_2-1) (2 n_2-2) (2 n_2-3)}{2} \mathfrak{h}_{n_2-2}^{g_2-1}
		-\delta_{(n_2,g_2)\neq(n,g)}\tfrac{n_2+1}{4} {\mathfrak{h}}_{n_2}^{g_2} 
		+\tfrac{2 n_2-1}{2} ( 2 \mathfrak{h}_{n_2-1}^{g_2} + \mathfrak{h}_{n_2-1}^{g_2-1/2} )\\
		+6 
		\hsum_{g_3=0..g_2\atop g_3+g_4=g_2}
		\sum_{n_3=0..n_2\atop n_3+n_4=n_2}
		\tfrac{(2 n_3-1)(2 n_4-1) }{4}\mathfrak{h}_{n_3-1}^{g_3} \mathfrak{h}_{n_4-1}^{g_4}
		\Big)\vspace{3cm} \Bigg)	
	\end{multline}
	for $n>2$, with the initial conditions 
	$\mathfrak{h}_0^0=1$, $\mathfrak{h}_1^0=2$, $\mathfrak{h}_2^0=9$, $\mathfrak{h}_{1}^{1/2}=1$, $\mathfrak{h}_2^{1/2}=10$, $\mathfrak{h}_2^1=5$, and $\mathfrak{h}_n^g=0$ if $n<2g$.
\end{theorem}
We obtained similar theorems for other models, in particular one with control on vertices and faces, one for bipartite maps and one for triangulations which is implicit in Carrel's work. As already observed in the oriented case (in the literature, not in this manuscript), there are several ways of performing the elimination technique of Goulden and Jackson, one due to Carrell and Chapuy which goes through Tutte's bijection between general maps and bipartite quadrangulations before performing the elimination, and one due to Kazarian and Zograf which we explained in Chapter \ref{sec:EnumerationIntegrability}, Section \ref{sec:Elimination}. However in the non-oriented case we found that they do not lead to the same recurrence formulas. \Cref{thm:NonOrientedMapsCC} is obtained via the Carrell-Chapuy elimination.

The crucial fact that the BKP equation~\eqref{BKP1Log} involves not only the function $F(N, \ptimes)$ but also its shifts $F(N+2,\ptimes)$ and $F(N-2,\ptimes)$ has an important effect on the recurrence formulas we obtain. Indeed, the functional equations corresponding to those recurrences involve derivatives with respect to the $p_i$s but also shifts of $N$, and are therefore \emph{not} ODEs in their main variables. In return, the recurrences obtained do not have polynomial coefficients (for example \eqref{eq:recurrenceCC} contains binomial coefficients, which are not polynomials in the summation variables). This is a deep structural difference between the recurrence formulas of Chapter \ref{sec:EnumerationIntegrability} and recurrences such as~\eqref{eq:recurrenceCC}.

\subsection{ODEs (without shifts) from multiple BKP equations} It is natural to ask if one could instead obtain  formulas in which the shifts are not involved, i.e. true polynomial recurrence formulas, corresponding to nonlinear ODEs with polynomial coefficients for the associated generating functions. Although the shifts are intrinsic to the BKP hierarchy, we will see that the answer to this question is \emph{yes}. To do this, we will use combine three equations of the BKP hierarchy. Using additional derivations and manipulations, we were able to eliminate the shifts from equations, and obtain an ODE at fixed $N$. The counterpart is that it contains higher derivatives compared to those involving shifts. It is not obvious, but it will be true, that a finite number of Virasoro constraints will still be sufficient to perform the elimination of variables in this context.

Due to the use of higher BKP equations and additional manipulations involved, the equations thus obtained are bigger than the previous ones. They are however not gigantic, e.g. the main ODE for maps fits in slightly more than a page in $\backslash$tiny LaTex print. We will not reproduce it here, but state it in a non-explicit form. The reader eager to see those recurrence formulas at work may access them and use them to compute numbers of maps, in the accompanying Maple worksheet~\cite{us:Maple}. A typical statement we obtain from these methods is the following.
\begin{theorem}{}{rec-maps-nonshifted}
	The number $\mathfrak{h}_n^g$ of rooted maps of genus $g$ with $n$ edges, orientable or not, 
	is solution of an explicit recurrence relation of the form
	\begin{align}\label{eq:rec-maps-nonshifted}
		\mathfrak{h}_n^g = 
		\sum_{a=0}^{K_1}\hsum_{b=0}^{K_2} \sum_{k=1}^{K_3}
		\sum_{n_1,\dots,n_k \geq 1 \atop n_1+\cdots + n_k =n-a}
		\hsum_{g_1,\dots,g_k \geq 0 \atop  g_1+\dots + g_k =g-b}
		P_{a,b,k}(n_1,\dots,n_k) \mathfrak{h}_{n_1}^{g_1} \mathfrak{h}_{n_2}^{g_2}\dots \mathfrak{t}_{n_k}^{g_k},
	\end{align}
	where the $P_{a,b,k}$ are rational functions with $P_{0,0,1}=0$, and $K_1,K_2,K_3 <\infty$.
\end{theorem}
We obtained similar theorems for bipartite maps, and for triangulations \cite{BCD-RecurrenceFormulas}. Moreover, we will in fact prove a version of Theorem~\ref{thm:rec-maps-nonshifted} with control on the number of faces. From it, we obtain a closed recurrence formula enumerating non-oriented 1-face maps, which is precisely the one of Ledoux \cite{Ledoux2009} given in \eqref{Ledoux} which he obtained using matrix integrals.
%\begin{theorem}{}{}
%	The number $\mathfrak{u}_{n}^g$ of rooted non-oriented maps
%	of genus $g$ with $n$ edges and only one face (or equivalently
%	with $n$ edges and only one vertex) is given by the recursion
%	\begin{multline} \label{eq:Ledoux}
%		(n+1)\mathfrak{u}_{n}^g = (8n-2) \mathfrak{u}_{n-1}^g-(4n-1) \mathfrak{u}_{n-1}^{g-1/2}
%		+ n(2n-3)(10n-9)
%		\mathfrak{u}_{n-2}^{g-1} - 8(2n-3) \mathfrak{u}_{n-2}^g 
%		\\
%		-10(2n-3)(2n- 4)(2n-5) \mathfrak{u}_{n-3}^{g-1} 
%		+ 5(2n-3)(2n-4)(2n-5) \mathfrak{u}_{n-3}^{g-3/2} \\
%		+ 8(2n-3) \mathfrak{u}_{n-2}^{g-1/2} - 2(2n-3)(2n-4)(2n-5)(2n-6)(2n-7) \mathfrak{u}_{n-4}^{g-2} 
%	\end{multline}
%	with the convention that $\mathfrak{u}_{n}^{g} = 0$ for $g<0$ and $g
%	>\frac{n}{2}$ and with the initial condition
%	$\mathfrak{u}_{1}^{1/2}=1$,
%	$\mathfrak{u}_{2}^{1}=\mathfrak{u}_{2}^{1/2} = 5$,
%	$\mathfrak{u}_{3}^{3/2}=41$, $\mathfrak{u}_{3}^{1}=52$, $\mathfrak{u}_{3}^{1/2}=22$.
%\end{theorem}

It is remarkable to see that it is, in fact, the shadow of bigger nonlinear recurrence giving access to an arbitrary number of faces. That phenomenon was already observed in the oriented case as the famous Harer-Zagier recurrence for 1-face maps is itself a special case of the recurrence of~\cite{CarrellChapuy2015}. The Harer-Zagier recurrence has a nice analogue in the bipartite case due to Adrianov~\cite{Adrianov1997}, and it is natural to ask if our non-shifted recursion in the bipartite case implies an non-oriented version of Adrianov's result. The answer is yes.
\begin{theorem}{}{NonOrientedAdrianov}
	The number $\mathfrak{b}_{n}^{i,j}$ of
	rooted one-face maps with $n$ edges, $i$ white and $j$ black
	vertices, orientable or not, is given by the recursion:
	\begin{align}
		(n\splus{}1) \mathfrak{b}_{n}^{i,j} = &\hphantom{+} 
		(4n\sminus{}1)(\mathfrak{b}_{n\sminus{}1}^{i-1,j}\splus{}\mathfrak{b}_{n\sminus{}1}^{i,j-1}\sminus{}\mathfrak{b}_{n\sminus{}1}^{i,j})
		\splus{}(5n^3 \sminus{} 16n^2 \splus{} 13n\sminus{} 1) \mathfrak{b}_{n\sminus{}2}^{i,j}
		\nonumber\\&
		\splus{}(2n \sminus{} 3)(4\mathfrak{b}_{n\sminus{}2}^{i-1,j} \splus{} 4\mathfrak{b}_{n\sminus{}2}^{i,j-1} \sminus{}3\mathfrak{b}_{n\sminus{}2}^{i-2,j} \splus{} 3\mathfrak{b}_{n\sminus{}2}^{i,j-2}\sminus{} 2 \mathfrak{b}_{n\sminus{}2}^{i-1,j-1}) 
		\nonumber\\&
		\splus{}(10n^3 \sminus{} 68n^2 \splus{} 150n \sminus{}107)(\mathfrak{b}_{n\sminus{}3}^{i,j}\sminus{}\mathfrak{b}_{n\sminus{}3}^{i-1,j}\sminus{}\mathfrak{b}_{n\sminus{}3}^{i,j-1})
		\nonumber\\&
		\splus{}(4n	\sminus{}11)(\mathfrak{b}_{n\sminus{}3}^{i-3,j}
		\splus{}\mathfrak{b}_{n\sminus{}3}^{i,j-3}\sminus{}\mathfrak{b}_{n\sminus{}3}^{i-2,j-1}
		\sminus{}\mathfrak{b}_{n\sminus{}3}^{i-1,j-2} \sminus{}
		\mathfrak{b}_{n\sminus{}3}^{i-2,j}\sminus{}\mathfrak{b}_{n\sminus{}3}^{i,j-2} \splus{} 2
		\mathfrak{b}_{n\sminus{}3}^{i-1,j-1})\nonumber\\&
		\splus{}(4\sminus{}n)((2n \sminus{} 7)^2(n \sminus{} 2)^2 \mathfrak{b}_{n\sminus{}4}^{i,j}\splus{}(5n^2 \sminus{} 32n \splus{}
		53)(\mathfrak{b}_{n\sminus{}4}^{i-2,j}\splus{}\mathfrak{b}_{n\sminus{}4}^{i,j-2}\sminus{}2
		\mathfrak{b}_{n\sminus{}4}^{i-1,j-1})\nonumber\\&
		\splus{}\mathfrak{b}_{n\sminus{}4}^{i-4,j}\splus{}\mathfrak{b}_{n\sminus{}4}^{i,j-4}\sminus{}4 \mathfrak{b}_{n\sminus{}4}^{i-3,j-1}\splus{}4\mathfrak{b}_{n\sminus{}4}^{i-1,j-3}\splus{}6 \mathfrak{b}_{n\sminus{}4}^{i-2,j-2})) \label{eq:adrianovnori}
	\end{align}
	with the convention that $\mathfrak{b}_{n}^{i,j} =0$ for
	$i+j>n+1$, and $\mathfrak{b}_{n}^{i,0}=\mathfrak{b}_{n}^{0,j}=0$ and the initial conditions $\mathfrak{b}_{1}^{1,1} =
	\mathfrak{b}_{2}^{2,1} = \mathfrak{b}_{2}^{1,2} =
	\mathfrak{b}_{2}^{1,1}= 1$,
	$\mathfrak{b}_{3}^{3,1}=\mathfrak{b}_{3}^{1,3}=1$,
	$\mathfrak{b}_{3}^{2,2}=\mathfrak{b}_{3}^{2,1}=\mathfrak{b}_{3}^{1,2}=3$,
	$\mathfrak{b}_{3}^{1,1}=4$.
\end{theorem}
To conclude this introduction to the BKP recurrences, it is natural to ask whether our approach which eliminates the shifts are specific to the case of maps or apply to general solutions of the BKP hierarchy. The latter is in fact true: any function $F(N,\ptimes)$ which solves the BKP hierarchy is in fact solution of an explicit (but big) PDE involving only the function $F(N,\ptimes)$ and its $\ptimes$-derivatives, with no shifts (\cite[Appendix]{BCD-RecurrenceFormulas}). We are not aware of any in-depth study of such ``fixed charge'' BKP equations, which might be worth considering in the future.

\section{Zonal expansions of factorizations of matchings}
\subsection{Matchings} Although permutations can be used to describe non-oriented maps, it is argued (and showed!) by Goulden and Jackson in \cite{GouldenJackson1996a} that \emph{matchings} are more suited for enumeration. We recall that labeled orientable maps on $n$ edges are 2-constellations with vertices of one color constrained to be bivalent and thus encoded by triples of permutations satisfying $\sigma\alpha=\phi \in\mathfrak{S}_{2n}$ where the $2n$ labels are assigned to the half-edges and 
\begin{itemize}
	\item $\sigma$ describes the half-edges meeting at each vertex in a fixed cyclic order,
	\item $\phi$ describes the half-edges along the faces in a fixed cyclic order, and
	\item $\alpha$ only has cycles of length 2 and simply pairs the half-edges which belong to the same edges.
\end{itemize}

In order to account for the possibility of twisting the edges, one must now label not the half-edges but \emph{half-sides} (called side-ends in \cite{GouldenJackson1996a}), where a half-side is a side of a half-edge. In terms of ribbons, one cuts each ribbon in two half-ribbons, and each of them along their length, which gives four half-sides. The map is then determined by
\begin{itemize}
	\item a matching which indicates the pairs of half-sides forming half-edges,
	\item a matching which indicates the pairs of half-sides forming sides,
	\item a matching which indicates the pairs of half-sides which form corners at vertices.
\end{itemize}
This can be generalized to provide an extension of general factorizations to matchings \cite{BenDali2021}.

\subsection{Zonal expansions} One application of the encoding of maps as factorizations of permutations is the use of the representation theory of the symmetric group to expand the generating series of maps on Schur functions, see \Cref{thm:SchurExpansionFactorizations}. Now for non-oriented maps, there is a similar algebraic frame, where the symmetric group is replaced with a double coset $\mathfrak{B}_n\diagdown \mathfrak{S}_{2n}\diagup \mathfrak{B}_n$ \cite{HanlonStanleyStembridge1992}, where the hyperoctahedral group $\mathfrak{B}_n$ is the subgroup of $\mathfrak{S}_{2n}$ which fixes any arbitrarily chosen matching.

It is then still possible to make use of the representation theory of $\mathfrak{S}_{2n}$. By averaging over the hyperoctahedral group, one can find ``characters'' of the double coset, denoted $\phi_\lambda(\mu)$ \cite{HanlonStanleyStembridge1992}. To draw the parallel with Frobenius' formula \eqref{Frobenius}, we consider instead of maps\footnote{Both cases can be treated similarly \cite{GouldenJackson1996a}.}, the number of 3-factorizations $\tilde{C}^{(3)}_n(\lambda, \mu,\nu)$, i.e. labeled non-oriented bipartite maps with $n$ edges, where $\lambda, \mu, \nu$ are respectively the set of degrees of the faces, black vertices and white vertices. It satisfies an extension of Frobenius' formula \eqref{Frobenius} involving those characters of the double coset instead of the $\chi_\lambda(\mu)$s from the symmetric group,
\begin{equation}
	\tilde{C}^{(3)}_n(\lambda, \mu,\nu) = \frac{(2n)!}{(2^n n!)^3} \sum_{\alpha\vdash n} \frac{1}{\hook(2\alpha)} \phi_\alpha(\lambda) \phi_\alpha(\mu) \phi_\alpha(\nu).
\end{equation}
Going to the generating series, instead of Equation \eqref{Schur} which makes Schur functions appear, one finds \emph{zonal} polynomials
\begin{equation} \label{Zonal}
	Z_\lambda(\ptimes) = \frac{1}{2^n n!} \sum_{\mu\vdash n} \phi_\lambda(\mu) p_\mu.
\end{equation}

\begin{theorem}{}{NonOrientedBipartite}\cite{GouldenJackson1996a}
	A generating series of non-oriented 3-factorizations is
	\begin{equation} \label{NonOrientedBipartite}
		\begin{aligned}
			\tilde{\tau}^{\text{3-fact.}}(\ptimes, \qtimes, \vec{r};t) &\coloneqq \sum_{n\geq 0} \frac{t^n}{(2n)!} \sum_{\lambda, \mu, \nu\vdash n} \tilde{C}^{(3)}_n(\lambda, \mu,\nu) p_\lambda q_\mu r_\nu\\
			&= \sum_{\alpha} \frac{t^{|\alpha|}}{\hook(2\alpha)} Z_\alpha(\ptimes) Z_\alpha(\qtimes) Z_\alpha(\vec{r}).
		\end{aligned}
	\end{equation}
\end{theorem}
By generalizing the above construction with more matchings, it is possible to introduce a notion of non-oriented factorizations for which the above theorem holds with more than three partitions \cite{BenDali2021}.

\subsection{Non-oriented weighted Hurwitz numbers} A specialization of \cite{ChapuyDolega2020} gives a zonal expansion for the generating series of non-oriented constellations (which we do not define here) which is then extended to non-oriented weighted Hurwitz numbers, whose generating series is shown to be
\begin{equation} \label{NonOrientedWeightedHurwitz}
	\tilde{\tau}_G(\ptimes, \qtimes, t) = \sum_{\lambda} \frac{t^{|\lambda|}}{\hook(2\lambda)} Z_\lambda(\ptimes) Z_\lambda(\qtimes) \prod_{\Box\in\lambda} G(c_1(\Box))
\end{equation}
where $c_1(x,y) = 2x-y-1$ is a deformed content\footnote{This is consistent with \eqref{NonOrientedBipartite}. Let $Z_\lambda(1^u)\equiv Z_\lambda(\ptimes = (u, u, \dotsc))$. As in the Schur case, this is a polynomial in $u$, so it is enough to know it for integer values. If $u=N\in\mathbb{N}$, this is the principal specialization of symmetric functions, where the variables $x_1=\dotsb=x_N=1$ and $x_i=0$ for $i>N$. The principal specialization gives $Z_\lambda(1^N) = \prod_{\square \in \lambda}(N+c_1(\square))$.}.

%%%%%%%%%%%%%%%%%%%%
\section{The BKP hierarchy}
It has been known since \cite{VandeLeur2001} that general and bipartite non-oriented maps ($b=1$) satisfy not the KP hierarchy but a generalization of it called the BKP hierarchy, see \cite{KacVandeLeur1998} for the full formalization. It is a $B$-type analogue of the KP hierarchy, where $GL_\infty$ is replaced with a larger group\footnote{There is another BKP hierarchy, this time based on the orthogonal subgroup of $GL_\infty$. Both those BKP hierarchies are sometimes distinguished as large/small or charged/neutral. Here we only consider the large (or charged) one.} of type $B$ (for polynomial tau functions).

The BKP hierarchy is a set of coupled partial derivative equations for a \emph{sequence of series} $(\tau_N(\ptimes))_{N\in\mathbb{Z}}$ with each $\tau_N(\ptimes)\in\mathbb{C}[[p_1, p_2, \dotsb]]$. It is generated by the following bilinear equation,
\begin{multline} \label{BKP}
	\sum_{j\geq k+1} h_j(2\qtimes) h_{j-k-1}(-\check{D}) U(\qtimes) \tau_{N-2}(\ptimes) \cdot \tau_{N+k}(\ptimes) \\
	+ \sum_{j\geq 0} h_j(-2\qtimes) h_{j+k-1}(\check{D}) U(\qtimes) \tau_N(\ptimes) \cdot \tau_{N+k-2}(\ptimes) = \frac{1-(-1)^k}{2} U(\qtimes) \tau_{N-1}(\ptimes)\cdot \tau_{N+k-1}(\ptimes)
\end{multline}
for $k\in\mathbb{N}, k\geq 1$ and $\qtimes = (q_1, q_2, \dots)$ a vector of formal indeterminates. We recall the notation from Section \ref{sec:KP}, Chapter \ref{sec:EnumerationIntegrability},
\begin{itemize}%[itemsep=0pt, topsep=0pt,parsep=0pt, leftmargin=12pt]
	\item $h_j$ is the complete homogeneous symmetric function of degree $j$,
	\item $\check{D} = (k D_k)_{k\geq 1}$ where $D_r$ is the \emph{Hirota derivative} with respect to $p_r$ defined as the following bilinear mapping
	\begin{equation*}
		\bigl(f(p_r),g(p_r)\bigr) \ \mapsto\ D_r f(p_r)\cdot g (p_r)\coloneqq \frac{\partial}{\partial s_r} f(p_r + s_r) g(p_r - s_r)_{|s_r=0}.
	\end{equation*}
\end{itemize}
For example, the ``first'' equation of this hierarchy, often called \emph{the BKP-equation} is obtained by taking $k=2$ and extracting the coefficient of $q_3$. It is given by
\begin{align}\label{eq:explicitBKP}
	-F_{N|3,1}(\ptimes) + F_{N|2,2}(\ptimes) + \frac{1}{2} F_{N|1,1}(\ptimes)^2 + \frac{1}{12} F_{N|1,1,1,1}(\ptimes)=
	\frac{\tau_{N-2}(\ptimes) \tau_{N+2}(\ptimes)}{\tau_N(\ptimes)^2},
\end{align}
where $F_N(\ptimes) \coloneqq \log \tau_N(\ptimes)$ and $f_{|i} \equiv \frac{\partial f(\ptimes)}{\partial p_i}$. For the KP hierarchy, the RHS would be replaced with 0.%We will write explicitly some of the PDEs below.

\subsection{Formal-$N$ BKP hierarchy} We will need to amend the BKP hierarchy in three directions. In combinatorics we need to extend it to formal power series obviously. We also need to introduce an $N$-dependent normalization of the tau functions, $\beta_N$. It will appear explicitly in our equations (while a normalization in the KP hierarchy would drop out).

The other, more subtle, extension we need consists in allowing tau functions where the parameter $N$ is \emph{formal rather than an integer}. Indeed, in our case $N$ will always be some weight marking vertices of maps and therefore it is more natural for it to be a formal variable.

Not only that, but in the case of non-oriented monotone Hurwitz numbers which we study in the next chapter, our main function will in fact \emph{not} be defined at integer values of $N$. Although we have not discussed the non-oriented extension of the hypergeometric tau functions, the problem is just like setting $v_r=N$ with $N$ non-negative integer in the hypergeometric tau function \eqref{HypergeometricTauFunction}. We will hit a divergence for partitions with at least $N$ parts because then there are boxes whose content is $-N$. Moreover, we think it is more natural for non-oriented weighted Hurwitz numbers to have the parameters of the weight function $G$ as formal rather requiring evaluations at integers. This in turn requires some care since, as we said, we need to renormalize our series for them to be tau functions and those renormalizations are $N$-dependent and defined a priori for integer, non-formal, $N$.
%
%In the context of matrix integrals, some of these subtleties are hidden in the fact that the underlying specialization of the variables $p_i$s to powersums of finitely many eigenvalues realize in some sense this truncation explicitly, see~\eqref{eq:matrixTrunc} in the next section. The next definition enables us to work at the formal level and without specializations.
The formal definition we introduced in \cite{BonzomChapuyDolega2021} is the following. It relies on truncations and rational functions.
\begin{definition}[Formal-$N$ BKP tau function]\label{def:formalBKP}
	Let $\tau(N,\ptimes) \in \mathbb{Q}(N)[[p_1, p_2,\dotsc]]$ be a formal power series,
	\begin{equation}
		\tau(N,\ptimes) = \sum_\mu \tau_\mu(N) s_\mu(\ptimes)
	\end{equation}
	with coefficients which are rational functions of the formal parameter $N$. For each $L\in \mathbb{N}$, consider the following truncation of its Schur expansion, in which we only keep Schur functions for partitions with at most $L$ parts,
	\begin{equation}
		\tau^{\leq L}(N,\ptimes) := \sum_{\substack{\mu\\ \ell(\mu)\leq L}}  \tau_\mu(N) s_\mu(\ptimes).
	\end{equation}
	\begin{itemize}
		\item Assume that there exists a sequence $(\alpha_K)_{K\geq 1}$ tending to infinity, such that for each integer value $K\in \mathbb{N}$ the coefficients of $\tau^{\leq \alpha_K}(N,\ptimes)$ have no poles at $N=i$ for any integer $i\geq K$.
		\item Assume that there exists a sequence $(\beta_N)$ such that the sequence $(\beta_N \tau^{\leq \alpha_N}(N,\ptimes))$ is a tau function of the BKP hierarchy (which makes sense because $\tau^{\leq \alpha_N}(N,\ptimes)$ can be evaluated on $N\in\mathbb{N}$),
		\item Assume that $(\beta_N)$ is such that for all $N, k\in\mathbb{N}$,
		\begin{equation}\label{eq:2factorial}
			\frac{\beta_{N-1}\beta_{N+k-1}}{\beta_{N}\beta_{N+k-2}} =R_k(N), \quad \text{and}\quad \frac{\beta_{N-2}\beta_{N+k}}{\beta_{N}\beta_{N+k-2}} =S_k(N)
		\end{equation}
		for some rational functions $R_k(N), S_k(N)\in \mathbb{Q}(N)$.
	\end{itemize} 
	Then we say that $(\tau(N,\ptimes), \beta_N)$ is a \emph{formal-$N$ BKP tau function}.
\end{definition}
The first two conditions ensure that it makes sense to talk about ``$\beta_N \tau(N,\ptimes)$'' being a BKP tau function. The last condition\footnote{It can be relaxed a bit, like saying that $\frac{\beta_{N-1}\beta_{N+k-1}}{\beta_{N}\beta_{N+k-2}} =R_k(N)$ holds for odd $k$ instead of all $k$ is enough.} is crucial to extend the equations of the BKP hierarchy from integers $N$ to formal $N$.
\begin{proposition}{}{FormalBKP}[BKP hierarchy for formal $N$]
	Let $(\tau(N,\ptimes), \beta_N)$ be a formal-$N$ BKP tau function. Then for $k\in\mathbb{N}, k\geq 1$ the following bilinear identity holds in  $\mathbb{C}(N)[\ptimes,\qtimes][[t]]$,
	\begin{multline}
		\frac{(-1)^k-1}{2} R_k(N) U(\qtimes) \tau(N-1,\ptimes)\cdot \tau(N+k-1, \ptimes) 
		\\
		+ S_k(N) \sum_{j\geq k+1} h_j(2\qtimes) h_{j-k-1}(-\check{D}) U(\qtimes) \tau(N-2, \ptimes) \cdot \tau(N+k, \ptimes) \\
		+ \sum_{j\geq0} h_j(-2\qtimes) h_{j+k-1}(\check{D}) U(\qtimes) \tau(N, \ptimes) \cdot \tau(N+k-2, \ptimes) = 0
	\end{multline}
\end{proposition}
%Here it is crucial to note that for fixed $k$, the quantities $\frac{\beta_{N-1}\beta_{N+k-1}}{\beta_{N}\beta_{N+k-2}}, \frac{\beta_{N-2}\beta_{N+k}}{\beta_{N}\beta_{N+k-2}}$ are elements of $\mathbb{Q}(N)$. %A sufficient condition for~\eqref{eq:2factorial} to hold is that $\beta_N\beta_{N+2}/\beta_{N+1}^2$ is a rational function in $N$. This will be the case for the main function studied in this paper, but not for the cases mentioned in Appendix~\ref{sec:otherModels}.
\begin{proof}
	By definition of a formal-$N$ BKP tau function we know that	$(\beta_N \tau^{\leq\alpha_N}(N,\ptimes))$ is a BKP tau function, i.e. satisfies \eqref{BKP}. Now, fix two partitions $\lambda$ and $\mu$, and consider the coefficient of $[p_\lambda q_\mu]$ in \eqref{BKP}. This coefficient is a certain polynomial identity relating a finite number of coefficients of the functions	$\tau^{\leq\alpha_m}(m,\ptimes)$ for $m\in \{N, N-1,N-2, N+k-2, N+k-1, N+k\}$.
	
	We recall that the coefficients of $\tau(N,\ptimes)$ are rational functions in $N$. In particular, by taking $N$ large enough (i.e. $N>N_0(\lambda, \mu,k)$) we can see that this polynomial identity is an	identity between rational functions in $N$, which holds true
	for infinitely many values of $N$. Therefore it in fact holds by replacing each coefficient involved by the corresponding rational function. This is precisely what~\Cref{thm:FormalBKP} is saying. 
\end{proof}

\subsection{Schur expansion with Pfaffian minors}
We now need some ways to be able to prove when a series is a BKP tau function. We will state a formal-$N$ BKP equivalent of \Cref{thm:TauFunctionSchurExpansion} without proof. Just like the matrix elements of the representation $R_0$ of $GL_\infty$ are \emph{determinant minors}, the matrix elements the bigger $B$-type group for the BKP hierarchy are \emph{Pfaffian minors}. We recall that the \emph{Pfaffian} of a skew-symmetric matrix $A$ of even size $n$ is the quantity $\Pf(A)$ defined by
\begin{equation}
	\label{eq:DefPf}
	\Pf(A) \coloneqq \frac{1}{2^{n/2} ( n/2)!}\sum_{\sigma \in \mathfrak{S}_n}\sgn(\sigma)\prod_{i=1}^{ n/2}a_{\sigma(2i-1),\sigma(2i)}.
\end{equation}
\begin{theorem}{}{PfaffianCoefficients}
	Let $\tau(N,\ptimes)=\sum_\lambda b_\lambda(N) s_\lambda(\ptimes)$ be a formal power series in $\mathbb{C}(N)[[\ptimes]]$ and assume that there exists an infinite skew-symmetric matrix $B=(B_{i,j})_{i,j\in \mathbb{Z}}$ such that for every integer $N\geq0$ and partition $\lambda$ such that $\ell(\lambda) \leq N$,
	\begin{equation}\label{eq:pfaffianCoeff}
		\beta_N b_\lambda (N)= \begin{cases} \Pf(B_{\lambda_i+N-i,\lambda_j+N-j})_{1 \leq i,j \leq N}\hfill\text{ for $N$ even,}\\
			\Pf(B_{\lambda_i+N-i,\lambda_j+N-j})_{1 \leq i,j \leq N+1}\ \text{ for $N$ odd,} 
		\end{cases}
	\end{equation}
	where $(\beta_N)_{N\geq 1}$ is a sequence satisfying~\eqref{eq:2factorial}.	Then $(\tau(N,\ptimes), \beta_N)$ is a formal-$N$ BKP tau function. 
\end{theorem}

\section{Maps and bipartite maps} 
In the present chapter we have so far introduced:
\begin{itemize}
	\item generating series of non-oriented maps and weighted Hurwitz number \eqref{NonOrientedWeightedHurwitz}, which are naturally expanded on zonal polynomials via the representation theorey of matchings,
	\item the notion of BKP tau functions which can be found as series in the Schur basis with Pfaffian coefficients.
\end{itemize}
Both objects turn out to be related to objects of type $B$, which we have already claimed for the BKP hierarchy, while zonal polynomials are also related to orthogonal groups (more precisely to the Gelfand pair $(GL_N, O_N)$, Part 7 in Macdonald's \cite{Macdonald1995}).

However, there is clearly a \emph{tension} between those two types of objects, since series of non-oriented maps are naturally expanded in the zonal basis, while it is the Schur basis for BKP tau functions. By ``naturally expanded'', we mean that they have \underline{\emph{nice}} coefficients in those bases, as products of a function of the cell contents in \eqref{NonOrientedWeightedHurwitz} and as Pfaffians minors in \Cref{thm:PfaffianCoefficients}. 

It would thus be quite remarkable if some generating series of non-oriented maps were BKP tau functions as they would have nice expansions in both bases. From the literature, we gathered that this is indeed true for general and bipartite non-oriented maps, i.e. the zonal versions of \eqref{MapsTauFunctions}

\begin{theorem}{}{MapsBKP}[Based on \cite{VandeLeur2001}]
	The generating series of non-oriented general, and bipartite maps\footnote{General maps are obtained from \eqref{NonOrientedBipartite} by setting $r_i = N$ and $q_i=\delta_{i,2}$ for all $i$, with $\theta^{(1)}_\lambda \equiv Z_\lambda(\delta_{i,2})$, and bipartite maps by setting $r_i=N$ and $q_i = N+\delta$.}
	\begin{equation} \label{MapsBKPTauFunctions}
		\begin{aligned}
			\tilde{\tau}^{\text{maps}}(t;\ptimes,N) &= \sum_{n \geq 0} t^{2n}\sum_{\lambda \vdash n} 
			Z_\lambda(\ptimes)\theta^{(1)}_\lambda\prod_{\square \in \lambda}(N+c_1(\square)), \\
			\tilde{\tau}^{\text{bip}}(t;\ptimes,N,N+\delta) &= \sum_{n \geq 0} t^n\sum_{\lambda \vdash n} 
			\frac{Z_\lambda(\ptimes)}{\operatorname{hook}(2\lambda)}\prod_{\square \in \lambda}(N+c_1(\square))(N+\delta+c_1(\square)),
		\end{aligned}
	\end{equation}
	are formal-$N$ BKP tau functions after rescaling $\ptimes \to2\ptimes$ with the following Schur expansions
	\begin{multline} \label{GeneralMapsBKP}
		\tilde{\tau}^{\text{maps}}(t;2\ptimes,N) \\
		= \beta^{\text{Maps}}_N \sum_{\lambda} \Pf\biggl( \int_{\mathbb{R}^2} dx dy\ e^{-\frac{x^2+y^2}{4t^2}} x^{\lambda_i+N-i} \sgn(x-y) y^{\lambda_j+N-j}\biggr)_{1\leq i,j\leq N} s_\lambda(\ptimes)
	\end{multline}
	and 
	\begin{multline*}
		\tilde{\tau}^{\text{bip}}(t;2\ptimes,N,N+\delta) = \\
		\beta^{\text{Bip}}_N(\delta) \sum_{\lambda} \Pf\biggl(\int_{\mathbb{R}_+^2} dx dy\ e^{-\frac{x+y}{2t}} x^{\lambda_i+N+\frac{\delta-1}{2}-i} \sgn(x-y) y^{\lambda_j+N+\frac{\delta-1}{2}-j}\biggr)_{1\leq i,j\leq N} s_\lambda(\ptimes)
	\end{multline*}
	for $N$ even and similar expansions for $N$ odd.
\end{theorem}
van de Leur's techniques make use of matrix integrals as generating series of non-oriented maps, which results in the above Pfaffian expansions. Proving that they are formal-$N$ tau functions then requires to check that $\beta^{\text{Maps}}_N$ and $\beta^{\text{Bip}}_N(\delta)$ satisfy the conditions \eqref{eq:2factorial}. This can be checked because they are in fact well-known Selberg integrals \cite{BonzomChapuyDolega2021}.

\section{Recurrence formulas from the BKP hierarchy}
Exaclty as we did in Chapter \ref{sec:EnumerationIntegrability}, Section \ref{sec:Elimination}, we introduce the Virasoro constraints satisfied non-oriented maps and bipartite maps. They allow for two things: determining a unique solution to the BKP equation, and eliminating the partial derivatives in the latter. However, this does not eliminate the shifts on the RHS of the BKP equation. Doing so requires (or at least that is how we did it in \cite{BCD-RecurrenceFormulas}) adding two more equations from the BKP hierarchy.

\subsection{The Virasoro constraints}
%We proved in \cite{BonzomChapuyDolega2021} a lemma which extracts constraints from the evolution equation of Chapuy-Do\l\k{e}ga \cite{ChapuyDolega2020}. Under mild assumptions, we proved that if the evolution operator $H$, such that $\frac{\partial \tau}{\partial t} = H\tau$, writes $H = \sum_i p_i L_i$, for some differential operators $L_i$, then the latter annihilates $\tau$, i.e. $L_i \tau=0$. It applies generally to $b$-deformed constellations. For non-oriented maps and bipartite maps, the constraints read

\begin{proposition}{}{}\cite{VandeLeur2001,BonzomChapuyDolega2021}
	The generating series $\tilde{\tau}^{\text{maps}}(t;\ptimes,u)$ and $\tilde{\tau}^{\text{bip}}(t;\ptimes,u,v)$ satisfy
	\begin{equation}
		\begin{aligned}
			\tilde{L}^{\text{maps}}_i \tilde{\tau}^{\text{maps}}(t;\ptimes,u) &= 0, \quad \text{for $i\geq -1$},\\
			\tilde{L}^{\text{bip}}_i \tilde{\tau}^{\text{bip}}(t;\ptimes,u,v) &= 0, \quad \text{for $i\geq 0$.}
		\end{aligned}
	\end{equation}
	with
	\begin{equation}
		\begin{aligned}
			\tilde{L}^{\text{maps}}_i &= \frac{p_{i+2}^*}{t^2} - \bigg(2\sum_{\substack{a,b\geq1\\a+b=i}}p_a^*p_b^*+ \sum_{a \geq 1} p_a p^*_{a+i} + ((i+1)+2u)p_i^* 
			+ \tfrac{\delta_{i,-1}up_1+u(u+1)\delta_{i,0}}{2}\bigg),\\
			\tilde{L}^{\text{bip}}_i &= \frac{p_{i+1}^*}{t}
			-\bigg(2\sum_{\substack{a,b\geq1\\a+b=i}}p_a^*p_b^*+ \sum_{a \geq 1} p_a p^*_{a+i} + (i+u+v)p_i^*+\frac{uv\delta_{i,0}}{2}\bigg).
		\end{aligned}
	\end{equation}
\end{proposition}
They are the non-oriented versions of \eqref{VirasoroConstraints}. %Remarkably, even for general $b$, the constraints still satisfy the Virasoro algebra $[L_n, L_m]= (n-m)L_{n+m}$ (up to shifts of the indices).

\subsection{Recurrence formulas from the BKP equation}
Next we re-introduce the specialization operators
\begin{equation*}
	\theta_T : p_i \mapsto z \delta_{i,3}, \quad \theta_{M} : p_i \mapsto z.
\end{equation*}
and introduce 
\begin{equation}
	\tilde{F}^{\text{Maps}}(t;\ptimes,u) = \log \tilde{\tau}^{\text{Maps}}(t;\ptimes,u),\qquad \tilde{F}^{\text{Bip}}(t;\ptimes,u,v) = \log \tilde{\tau}^{\text{Bip}}(t;\ptimes,u,v).
\end{equation}
Then
\begin{itemize}
	\item $\Theta^{\text{Tri}}(t,z,u)\coloneqq \theta_T \tilde{F}^{\text{Maps}}(t;\ptimes,u)$ the series of triangulations counted with $t$ on edges, $z$ on triangles and $u$ on vertices
	\item $\Theta^{\text{Maps}}(t,z,u)\coloneqq \theta_M \tilde{F}^{\text{Maps}}(t;\ptimes,u)$ the series of maps counted with $t$ on edges, $z$ on faces and $u$ on vertices
	\item $\Theta^{\text{Bip}}(t,z,u,v)\coloneqq \theta_M \tilde{F}^{\text{Bip}}(t;\ptimes,u,v)$ the series of bipartite maps counted with $t$ on edges, $z$ on faces, and $u$ on black vertices and $v$ on white vertices.
\end{itemize}

The specialized BKP equation reads for maps
\begin{multline}
	-\theta_M \bigl(\tilde{F}^{\text{Maps}}_{3,1}\bigr)(t,z,u) + \theta_M \bigl(\tilde{F}^{\text{Maps}}_{2,2}\bigr)(t,z,u) + \frac{1}{2} \theta_M \bigl(\tilde{F}^{\text{Maps}}_{1,1}\bigr)^2(t,z,u) + \frac{1}{12} \theta_M \bigl(\tilde{F}^{\text{Maps}}_{1,1,1,1}\bigr)(t,z,u) \\
	= S_2(N) e^{\theta_M \tilde{F}^{\text{Maps}}(t,z,u-2) + \theta_M \tilde{F}^{\text{Maps}}(t,z,u+2) - 2\theta_M \tilde{F}^{\text{Maps}}(t,z,u)}
\end{multline}
where $S_2(N) = t^4 u(u-1)$, which holds in $\mathbb{Q}[u,z][[t]]$, and same for maps and bipartite maps (in the latter case, the shifts are on both $u$ and $v$).

The equivalent of Lemma \ref{thm:Elimination} is
\begin{lemma} \cite{BCD-RecurrenceFormulas}
	The functions $\theta_M \bigl(\tilde{F}^{\text{Maps}}_{3,1}\bigr)$, $\theta_M \bigl(\tilde{F}^{\text{Maps}}_{2,2}\bigr)$, $\theta_M \bigl(\tilde{F}^{\text{Maps}}_{1,1}\bigr)$ and $\theta_M \bigl(\tilde{F}^{\text{Maps}}_{1,1,1,1}\bigr)$ are polynomials in $u,z,t$ and a finite number of $\frac{\partial^k \Theta^{\text{Maps}}(t,z,u)}{\partial t^k}$. The same holds for bipartite maps and triangulations.
\end{lemma}
Let us give the more precise version contained in \cite{BCD-RecurrenceFormulas}.
\begin{proposition}{}{relGlambda}
	For $i\geq -1$ and $n_1,n_2,n_3\geq 0$, one has the recurrence relation 
	\begin{multline} \label{eq:relGlambda}
		\frac{(i+2)\theta_M F^{\text{Maps}}_{i+2,3^{n_3},2^{n_2},1^{n_1}}}{t^2} =
		2\sum_{\substack{a+b=i\\a,b\geq 1}} \sum_{\substack{l_i=0\\
				i=1,2,3}}^{n_i} ab\binom{n_1}{l_1} \binom{n_2}{l_2}
		\binom{n_3}{l_3} \theta_M F^{\text{Maps}}_{a,3^{l_3},2^{l_2},1^{l_1}} \theta_M F^{\text{Maps}}_{b,3^{n_3-l_3},
			2^{n_2-l_2}, 1^{n_1-l_1}}\\
		+2\sum_{\substack{a+b=i\\a,b\geq 1}} ab\theta_M F^{\text{Maps}}_{a,b,3^{n_3},2^{n_2},1^{n_1}}+\sum_{j=1}^3n_j(i+j)\theta_M F^{\text{Maps}}_{i+j,3^{n_3-\delta_{3,j}}, 2^{n_2-\delta_{2,j}}, 1^{n_1-\delta_{1,j}}}+ t\frac{\partial}{\partial t} \theta_M F^{\text{Maps}}_{3^{n_3},2^{n_2},1^{n_1}} \\
		-(n_1+2n_2+3n_3) \theta_M F^{\text{Maps}}_{3^{n_3},2^{n_2},1^{n_1}} - z \sum_{a=1}^i
		a\theta_M F^{\text{Maps}}_{a,3^{n_3},2^{n_2},1^{n_1}} + \delta_{i \neq -1}(2u+i+1)i \theta_M F^{\text{Maps}}_{i,3^{n_3},2^{n_2},1^{n_1}}\\
		+ \Bigl(\delta_{i,-1}(\delta_{n_1,1}+z\delta_{n_1,0}) + (u+1) \delta_{i,0}\delta_{n_1,0}\Bigr)\frac{u}{2}\delta_{n_2,0}\delta_{n_3,0},
	\end{multline}
	with the initial condition that $\theta_M F^{\text{Maps}}_{3^{0}2^{0}1^{0}}=\theta_M F^{\text{Maps}}_{\emptyset}=\Theta(t,z,u)$.
	
	Consequently, for any integer vector 
	$\lambda$ of the form $\lambda=[\ell,3^{n_3},2^{n_2},1^{n_1}]$
	with $\ell \leq 9$ and of size $|\lambda| = \ell + n_1 +2n_2 +3n_3$, there exists a polynomial
	$P_{\lambda}$ 
	in $|\lambda|$ variables,
	with coefficients in $\mathbb{Q}[t,u,z]$,
	such that
	\begin{equation} \label{eq:Flambdadiff}
		\theta_M F^{\text{Maps}}_\lambda = P_\lambda\left(\frac{\partial}{\partial t}\Theta(t,z,u), \dotsc, \frac{\partial^{|\lambda|}}{\partial t^{|\lambda|}}\Theta(t,z,u)\right).
	\end{equation}
	This polynomial is linear for $\ell\leq 3$ and $|\lambda|\leq 5$, and
	quadratic for $4 \leq \ell\leq 6$ and $|\lambda|\leq 6$.
\end{proposition}

Note that the size of all vectors indexing $\theta_M F^{\text{Maps}}$s appearing in the RHS of
\eqref{eq:relGlambda} are strictly smaller than the size $i+2 + n_1+2n_2+3n_3$ in the LHS.
In order to be able to iterate this equation on all these terms, we need all vectors appearing in the RHS to have at most one part larger than $3$. The only term on the RHS that could have two parts larger than $3$ is
of the form $\theta_M F^{\text{Maps}}_{a,b,3^{n_3},2^{n_2},1^{n_1}}$. Since $a+b=i$, this does not
happen unless $i+2> 9$. We thus obtain~\eqref{eq:Flambdadiff}.

It is now almost immediate to see that applying the operator $\theta_M$ to the BKP equation~\eqref{BKP1Log} produces a functional equation on $\Theta^{\text{Maps}}(t,z,u)$. To eliminate the exponential on the RHS, we take the derivative with respect to $t$ and combine it with the BKP equation.
\begin{theorem}{}{}
	The generating function $\Theta^{\text{Maps}}(t,z,u)$ satisfies a
	functional equation of the following form:
	\begin{multline*}
		\Big(\frac{\partial}{\partial t}\big(\Theta(t,z,u+2)+\Theta(t,z,u-2)-2\Theta(t,z,u)\big)\Big)
		P\left(\frac{\partial}{\partial t}\Theta(t,z,u), \dotsc, \frac{\partial^5}{\partial t^5}\Theta(t,z,u)\right)\\
		= Q\left(\frac{\partial}{\partial t}\Theta(t,z,u), \dotsc, \frac{\partial^5}{\partial t^5}\Theta(t,z,u)\right),
	\end{multline*}
	where $P$ and $Q$ are quadratic polynomials with coefficients in $\mathbb{Q}[t,u,z]$.
\end{theorem}

That functional equation can be transformed into a recurrence to compute the coefficients. It has the following relatively compact form.
\begin{theorem}{}{MapsRecurrence}
	The generating polynomial $H_n^g\equiv H_n^g(u,z) = \sum_{i+j=n+2-2g}H_n^{i,j}u^iz^j$ of rooted non-oriented maps of genus $g$ with $n$ edges, with weight $u$ per vertex and $z$ per face, can be computed from the following recurrence formula,
	\begin{multline*} 	
		H_n^g+\tfrac{3 u^2 \partial^2}{n(n+1) \partial u^2}H_n^g
		=
		\tfrac{1}{n(n+1)} \times \Bigg(n\bigg(
		2(2n-1)( (4u+z) H_{n-1}^{g} -2 H_{n-1}^{g-1/2})
		+   4(2 n-3)\big(3uz H_{n-2}^{g}\\
		+  (2 n-1)(n-1)H_{n-2}^{g-1}     \big)
		+6\hsum_{g_1=0..g\atop g_1+g_2=g} \sum_{n_1=0..n \atop n_1+n_2=n} 
		(2n_1 - 1)(2n_2 - 1)H_{n_2-1}^{g_2}H_{n_1-1}^{g_1}   \bigg)
		\\
		-
		\hsum_{g_1=0..g\atop g_1+g_2=g} \sum_{n_1=1..n\atop n_1+n_2=n}
		\hsum_{g_0=0..g_1\atop g_1-g_0\in \mathbb{N}}
		\sum_{p+j=n_1+2-2g_0}
		2^{2 (1+g_1-g_0)} {p \choose 2(1+g_1-g_0)}u^{n_1-2g_1-j}z^j H_{n_1}^{p,j}   
		\\
		\Big(
		-\frac{n_2+1}{2}H_{n_2}^{g_2}
		+   
		(2n_2-1) \big((4u+z)H_{n_2-1}^{g_2}-2H_{n_2-1}^{g_2-1/2}\big)+
		2(2n_2 - 3)\big((2n_2 - 1)(n_2 - 1)H_{n_2-2}^{g_2-1}\\
		+3uz H_{n_2-2}^{g_2}\big)
		+\frac{3}{2}\delta_{g_0\neq g}\delta_{n_1,n}u(\delta_{g_1,g}u- \delta_{g_1,g-1/2})
		+\delta_{n_1,n-1}uz(\delta_{g_1,g}(4u + z) - 2\delta_{g_1,g-1/2})\\
		+3uz\delta_{n_1,n-2}(\delta_{g_1,g}uz + 2\delta_{g_1,g-1}) 
		+ 3\hsum_{g_3=0..g_2\atop g_3+g_4=g_2} \sum_{n_3=0..n_2 \atop n_3+n_4=n_2} 
		(2n_3-1)(2n_4-1)H_{n_4-1}^{g_4}H_{n_3-1}^{g_3}   
		\Big)\vspace{3cm} \Bigg)
	\end{multline*}
	for $n>2$, with the initial conditions $H_0^g=0$, $H_1^0=uz(u+z)$, $H_2^0=uz(2u^2+5uz+2z^2)$, $H_{1}^{1/2}=uz$, $H_2^{1/2}=5uz(u+z)$, $H_2^1=5uz$, and $H_n^g=0$ if $n<2g$.
\end{theorem}
Similar formulas are given in \cite{BCD-RecurrenceFormulas} for bipartite maps and triangulations.

\subsection{Removing the shifts using three BKP equations} Recall that a formal-$N$ BKP tau function has the charge parameter which we here denote $N$ ($u$ in the above recurrence formulas). To save space we write here $\tau(N) \equiv \tau(N,\ptimes)$ and $F_{i,j,\dotsc}(N) \equiv \frac{\partial}{\partial p_i \partial p_j \dotsm} \log \tau(N)$. The ``first'' three equations of the hierarchy are
\begin{equation}
	%\label{BKP1Log}
	F_{2^2}(N) \sminus{}F_{3,1}(N) \splus{} \tfrac{1}{2} F_{1^2}(N)^2 \splus{} \tfrac{1}{12} F_{1^4}(N)=
	S_2(N) \tau(N\!\sminus{}\!2) \tau(N\!\splus{}\!2)\tau(N)^{-2},
\end{equation}
\begin{multline}
	\label{BKP2Log}
	\sminus{}2F_{4,1}(N) \splus{} 2F_{3,2}(N) \splus{} 2F_{2,1}(N)
	F_{1^2}(N) \splus{} \tfrac{1}{3} F_{2,1^3}(N) \\
	= S_2(N)\frac{\tau(N\!\sminus{}\!2) \tau(N\!\splus{}\!2)}{\tau(N)^2}(F_1(N\!\splus{}\!2)-F_1(N\!\sminus{}\!2)).
\end{multline}
\begin{multline}
	\label{BKP3Log}
	\sminus{}6F_{5,1}(N) \splus{} 4F_{4,2}(N) \splus{} 2F_{3^2}(N)\splus{}4F_{3,1}(N) F_{1^2}(N) \splus{} \tfrac{2}{3}F_{3,1^3}(N)
	\splus{}4F_{2,1}(N)^2 \\
	\splus{} 2F_{2^2}(N) F_{1^2}(N)\splus{} F_{2^2,1^2}(N) \splus \tfrac{1}{3}F_{1^2}(N)^3 \splus \tfrac{1}{6}F_{1^4}(N) F_{1^2}(N) \splus \tfrac{1}{180}F_{1^6}(N) \\
	= S_2(N)
	\frac{\tau(N\!\sminus{}\!2) \tau(N\!\splus{}\!2)}{\tau(N)^2}\bigl(F_{1^2}(N\!\splus{}\!2)+F_{1^2}(N\!\sminus{}\!2)
	+2F_{2}(N\!\splus{}\!2)-2F_{2}(N\!\sminus{}\!2)\\+(F_1(N\!\splus{}\!2)-F_1(N\!\sminus{}\!2))^2\bigr),
\end{multline}

As it turns out, we can combine them to obtain one PDE at fixed $N$
\begin{theorem}{}{}\cite{BCD-RecurrenceFormulas}
	Set $F\equiv F(N)  = \log \tau(N) \in\mathbb{Q}(N)[[p_1, p_2, \dotsc]]$ for a formal-$N$ BKP tau function. Then
	\begin{multline}
		\label{eq:FixedCharge}
		2 F_{1^3}\operatorname{KP1}^3 =  (\operatorname{KP3}_1-2 \operatorname{KP2}_2)\operatorname{KP1}^2 -
		(\operatorname{KP3}-3\operatorname{KP1}_{1^2}) \operatorname{KP1} \operatorname{KP1}_1  \\+2 (\operatorname{KP1}_2-\operatorname{KP2}_1)\operatorname{KP1}\operatorname{KP2}+ 2 \operatorname{KP2}^2 \operatorname{KP1}_1 
		-2 \operatorname{KP1}_1^3 - \operatorname{KP1}^2 \operatorname{KP1}_{1^3},
	\end{multline}
	where
	\begin{align*}
		\operatorname{KP1} &= -F_{3,1} + F_{2^2} + \frac{1}{2}F_{1^2}^2 +\frac{1}{12}F_{1^4},\\
		\operatorname{KP2} &= -2F_{4,1} + 2F_{3,2} + 2F_{2,1}F_{1^2}
		+\frac{1}{3} F_{2,1^3},\\
		\operatorname{KP3}&=-6F_{5,1}+ 4F_{4,2}+2F_{3^2}+
		4F_{3,1}F_{1^2}
		+\frac{2}{3}F_{3,1^3}+4F_{2,1}^2+2F_{2^2}F_{1^2} +
		F_{2^2,1^2}\\
		&+ \frac{1}{3}F_{1^2}^3 + \frac{1}{6} F_{1^4}F_{1^2} + \frac{1}{180} F_{1^6},
	\end{align*}
\end{theorem}
Using the same elimination technique as before, namely \Cref{thm:relGlambda}, an ODE is derived and extracting the coefficients lead to polynomial recurrence formulas, given explicitly in \cite{us:Maple}. As it is possible to do that while controlling the number of faces, a special recurrence can be obtained for 1-face maps, which turns out to be exactly Ledoux' recurrence \cite{Ledoux2009}. In the bipartite case, we obtain \Cref{thm:NonOrientedAdrianov} which is a non-oriented version of Adrianov's recurrence formula \cite{Adrianov1997}. 

As for triangulations, it does not make sense to restrict to a single face, but instead it does to restrict to a single vertex. By duality, they correspond to cubic 1-face maps. There already exist explicit and simple formulas in this case obtained from bijective methods~\cite{BernardiChapuy2011} so we did not try to work out this case from the BKP equations.

%%%%%%%%%%%%%%%%%%%%%%%
\chapter{A remarkable case: that of monotone Hurwitz numbers} \label{sec:Monotone}

\section{The Goulden-Jackson $b$-conjecture} \Cref{thm:NonOrientedBipartite} was a starting point for an important conjecture, known as the $b$-conjecture or the Matching-Jack conjecture, proposed in \cite{GouldenJackson1996}. It notices that the zonal expansion of \eqref{NonOrientedBipartite} and the Schur expansion of \eqref{SchurExpansionFactorizations} are two special cases of the more general series
\begin{equation} \label{bDeformedBipartite}
	\tau^{(b)}(\ptimes, \qtimes, \vec{r};t) = \sum_{\lambda} \frac{t^{|\lambda|}}{j_\lambda} J^{(1+b)}_\lambda(\ptimes) J^{(1+b)}_\lambda(\qtimes) J^{(1+b)}_\lambda(\vec{r}),
\end{equation}
where $J^{(1+b)}_\lambda$ is the \emph{Jack polynomial}. Jack polynomials form a 1-parameter basis of symmetric functions, which coincide with Schur functions (up to normalization) at $b=0$ and with zonal polynomials at $b=1$. We refer to \cite{Stanley1989} for definitions and properties.

The conjecture of Goulden and Jackson \cite{GouldenJackson1996} is that when $\tau^{(b)}(\ptimes, \qtimes, \vec{r};t)$ is expanded on the power sum basis, $p_\lambda q_\mu r_\nu$, the coefficients are polynomials in $b$ with non-negative integer coefficients. In other words, that series counts things, which are presumably some ``$b$-weighted'' matchings and possibly interpretable as non-oriented bipartite maps. Ben Dali showed in \cite{BenDali2021} that if true, this can be extended to an arbitrary number of cycle types.

A difficulty in proving the $b$-conjecture is that there is \emph{no representation theory} to provide a sort of Frobenius' formula for $b\neq 0,1$ (and $b=-1/2$ which is a dual value to $b=1$ \cite{Macdonald1995}). An important step forward was made by Chapuy and Do\l\k{e}ga in \cite{ChapuyDolega2020} who proved a version of the conjecture for $b$-weighted Hurwitz numbers. Let us explain what it means and the underlying reasoning. Because of the lack of representation theory for general $b$, other approaches are required. A natural one for combinatorists is to write equations {\it \`a la} Tutte, some kind of evolution equations like \eqref{EvolutionEquation}, and prove that such an equation is satisfied by both $\tau^{(b)}(\ptimes, \qtimes, \vec{r};t)$ and some series of maps or matchings. One hopefully concludes with uniqueness.

However, it has been so far impossible to write equations {\it \`a la} Tutte for series with more than two cycle types. This is why Chapuy and Do\l\k{e}ga focused on weighted Hurwitz numbers. They first consider the polynomial case, i.e. that of biconstellations and introduce a notion of $b$-weighted biconstellations for which they are able to write an evolution equation which determines the series, denoted $\tau^{(b)}_{\text{$m$-biconst}}(\ptimes, \qtimes; t, \vec{u})$. They are then also able to prove that
\begin{equation} \label{JackExpansionConstellations}
	\tau^{(b)}_{\text{$m$-biconst}}(\ptimes, \qtimes; t, \vec{u}) = \sum_{\lambda} \frac{t^{|\lambda|}}{j_\lambda} J^{(1+b)}_\lambda(\ptimes) J^{(1+b)}_\lambda(\qtimes) \prod_{\Box\in\lambda} \prod_{c=0}^{m-1} (u_c + c_b(\Box)),
\end{equation} 
because the RHS satisfies the same evolution equation. Here $c_b(\Box)$ is a $b$-deformed notion of content of a box, $c_b(x,y) = c(x,y) + b(x-1)$. Chapuy-Do\l\k{e}ga then extends their approach to non-polynomial $G$ like in \eqref{WeightFunction}
\begin{equation} \label{bDeformedTauG}
	\tau^{(b)}_G(\ptimes, \qtimes, t) = \sum_\lambda \frac{t^{|\lambda|}}{j_\lambda} J^{(1+b)}_\lambda(\ptimes) J^{(1+b)}_\lambda(\qtimes) \prod_{\Box\in\lambda} G(c_b(\Box)),
\end{equation}
which are shown to provide a $b$-deformed version of weighted Hurwitz numbers.

%%%%%%%%%%%
\section{Monotone Hurwitz numbers}
Since the zonal expansions of the series \eqref{MapsBKPTauFunctions} of non-oriented maps and bipartite maps can be recovered from \eqref{JackExpansionConstellations} at $b=1$ (when Jack polynomials reduce to zonal polynomials), the question naturally arises as to whether \Cref{thm:MapsBKP}, stating that they are formal-$N$ BKP tau functions, can be generalized to those non-oriented constellations at $b=1$. Testing the BKP equations has not led us to any change of variables with which the BKP equation could be satisfied. However we found one other case where we could prove the formal-$N$ BKP hierarchy, that of non-oriented monotone Hurwitz numbers. They are $b$-deformations of ordinary monotone Hurwitz numbers with the generating series
\begin{equation} \label{bMonotoneSeries}
	\tau^{\text{mon}}_b (t;\ptimes,u) = \sum_{n \geq 0} t^n\sum_{\lambda \vdash n} \frac{J^{(b)}_\lambda(\ptimes)}{j^{(b)}_\lambda}\prod_{\square
		\in \lambda}\frac{1}{u^{-1}+c_b(\square)}, 
\end{equation}
This series is understood as an element of $\mathbb{Q}(b,u)[\ptimes][[t]]$, and it also belongs to $\mathbb{Q}(b)[\ptimes][[u,t]]$. Of special interest in this section (because of the BKP hierarchy) is the case $b=1$,
\begin{equation} \label{NonOrientedMonotoneSeries}
	\tilde{\tau}^{\text{mon}} (t;\ptimes,u) = \sum_{n \geq 0} t^n\sum_{\lambda \vdash n} \frac{Z_\lambda(\ptimes)}{\hook(2\lambda)}\prod_{\square
		\in \lambda}\frac{1}{u^{-1}+c_1(\square)}, 
\end{equation}
A combinatorial model of ``$b$-deformed Hurwitz maps'' is given in \cite{BonzomChapuyDolega2021} but we will not use it and therefore we chose not to include it here. We will instead ``prove'' (as in ``sketch the proof'') of the following.
\begin{theorem}{}{MonotoneBKP}
	$(\beta_N, \tilde{\tau}^{\text{mon}}(t;2\ptimes,\frac{1}{2N}))$ is a formal-$N$ BKP tau function for $\beta_N=\tfrac{1}{\prod_{k=1}^{N-1}(2k)!}$.
\end{theorem}

\subsection{Evolution equation and Virasoro constraints} Recall that an evolution equation read $\frac{\partial\tau}{\partial t} = \hat{X}\tau$ where $\hat{X}$ is a differential operator with respect to the $p_i$s. For constellations \cite{Fang:PhD, ChapuyDolega2020}, it based on the \emph{cut-and-join} idea. When an edge is added at the rooted corner, it can either split the root face into a smaller root face and an internal face, or join the root face with an internal face thus creating a bigger root face\footnote{Notice that since the degree of the root face has to be tracked, a catalytic variable is used to derive $\hat{X}$ from that reasoning.}. For constellations, there is one action of the cut-and-join operator for every color, and therefore if one thinks of monotone Hurwitz numbers as counting maps with an arbitrary number of colors, one finds an evolution equation of infinite order, because it features an infinite number of actions of the cut-and-join operator.

Remarkably we found in \cite{BonzomChapuyDolega2021} that $\tau^{\text{mon}}_b (t;\ptimes,u)$ also satisfies another evolution equation, which is of finite order\footnote{It is perhaps similar to the way the functional equation for the generating series of rooted plane trees with vertices of unbounded degrees may seem to be a ``polynomial of infinite degree'' if it is written like $T = 1 + z\sum_{k\geq 1} T^k$ where each degree contributes to $T^k$, and this gives a series. However, it can be clearly rewritten as a quadratic equation.}. We recall the notation $p_i^* = \frac{i\partial}{\partial p_i}$.
\begin{theorem}{}{MonotoneEvolutionEquation}
	The function $\tau^{\text{mon}}_b (t;\ptimes,u)$ is	uniquely determined by
	\begin{equation}\label{MonotoneEvolutionEquation}
		\frac{\partial}{\partial t} \tau^{\text{mon}}_b (t;\ptimes,u) = H_b(t,u) \tau^{\text{mon}}_b (t;\ptimes,u),
	\end{equation}
	where
	\begin{equation}
		H_b(t,u) = \frac{u p_1}{1+b} - \frac{u}{t}\Bigl((1+b)\sum_{m,n\geq 1}p_{m+n}p_m^*p_n^*+ \sum_{n,m \geq 1}
		p_n p_m p^*_{n+m}+b\sum_{n\geq 1}(n-1) p_{n}p_n^*\Bigr).
	\end{equation}
\end{theorem}
The special case $b=0$ of the above equation is known as the cut-and-join equation for the monotone Hurwitz numbers first proved in \cite{GouldenGuayPaquetNovak2013} and in \cite{Dunin-BarkowskiKramerPopolitovShadrin2019}. Our proof, when reduced to this case gives yet another proof. It is based on the theory of Jack polynomials and we skip it.

In Chapter \ref{sec:EnumerationIntegrability}, Section \ref{sec:RecognizingKP}, we explained that evolution equations can sometimes be used to prove the KP hierarchy, when $e^{t\hat{X}}$ is the representation $R_0$ of an element of $\overline{GL}_\infty$. Similar arguments can in principle be used for the BKP hierarchy, but it does not seem like integrating \eqref{MonotoneEvolutionEquation} produces an operator living in a representation of the bigger group of the BKP hierarchy\footnote{That is perhaps not too surprising because the evolution equation is (obviously) with respect to $t$, while the operators from the bigger group of the BKP hierarchy rather describe an \emph{evolution with respect to the charge}, so $u^{-1}$ here.}. Nevertheless, we are able to make use of \eqref{MonotoneEvolutionEquation} at $b=1$ by showing that a series whose Schur expansion has explicit Pfaffian minors as coefficients is a solution, thereby proving the BKP hierarchy at $b=1$ by uniqueness of the combinatorial solution.

\subsection{Schur expansion for $b=1$ and orthogonal group characters}\label{sec:schurExpansion} In the remaining of this chapter, we set $b=1$. The explicit Schur expansion we will give is directly related to the characters of the irreducible representations of the orthogonal group. % See Theorem~\ref{thm:schurExpansion} and Theorem~\ref{thm:OliveiraNovaes} which establishes the closely related conjecture of Oliveira and Novaes.

\subsubsection{Irreducible representations of orthogonal and special orthogonal groups.} Here we give formulas for the dimensions of the irreducible representations of the orthogonal group $O(2n)$ without details. We use \cite{Weyl1939,FultonHarris1991} as a black box. Those representations are indexed by partitions $\lambda$ with $\ell(\lambda) \leq n$ and we denote $o_\lambda(1^{2n})$ the corresponding dimension. The dimension of the irreducible representation of $SO(2n)$ with highest weight $\lambda$ is given by the following formula
\begin{equation} \label{eq:dimWeyl}
	so_\lambda(1^{2n}) = \prod_{1\leq i < j	\leq n}\frac{(\rho_i-\rho_j)(\rho_i+\rho_j+2n)}{(-i+j)(-i-j+2n)},
\end{equation}
where $\rho_i := \lambda_i-i$ and 
\begin{equation} \label{eq:dimOrthogonal}
	o_\lambda(1^{2n}) = \begin{cases} so_\lambda(1^{2n}) &\text{
			for $\ell(\lambda) < n$,}\\ 2so_\lambda(1^{2n}) &\text{
			for $\ell(\lambda) = n$.}\end{cases}
\end{equation}
A different formula for $o_\lambda(1^{n})$, valid regardless of the parity of $n$, was given by El Samra and King \cite{ElSamraKing1979}
\begin{equation} \label{eq:dimWeyl2}
	o_\lambda(1^{n}) = \frac{1}{\hook(\lambda)} \prod_{\substack{(x,y)\in \lambda\\ x \leq y}}(n+\lambda_x+\lambda_y-x-y) \prod_{\substack{(x,y)	\in \lambda\\ x >y}}(n-\lambda^t_x-\lambda^t_y+x+y-2).
\end{equation}

\subsubsection{Explicit expansion in scaled Schur functions} From \eqref{eq:dimWeyl2}, the quantity $o_\lambda(1^{n}) \in \mathbb{Q}[n]$ can be considered as a polynomial in $n$ of degree $|\lambda|$, which allows us to extend the definition of this ``dimension'' to the case where $n$ is a formal variable. Below we will use this convention with $n=u^{-1}$, so that $o_\lambda(1^{u^{-1}})\in \mathbb{Q}[u^{-1}]$ and therefore $1/o_\lambda(1^{u^{-1}})$ has a valid power series expansion at $u=0$, i.e. $1/o_\lambda(1^{u^{-1}}) \in \mathbb{Q}[[u]]$.
\begin{theorem}{}{MonotoneSchurExpansion}
	The series \eqref{NonOrientedMonotoneSeries} has the following expansion in Schur functions over the power sums $\ptimes/2$,
	\begin{equation}\label{SchurExpansion}
		\tilde{\tau}^{\text{mon}}(t;\ptimes,u) = \sum_{\lambda} \frac{t^{|\lambda|} s_\lambda(\ptimes/2)}{\hook(\lambda)^2 o_\lambda(1^{u^{-1}})}.
	\end{equation}
	This identity holds in $\mathbb{Q}(u)[\ptimes][[t]]$ and in $\mathbb{Q}[\ptimes][[u,t]]$.
\end{theorem}

The proof of Theorem~\ref{thm:MonotoneSchurExpansion} consists in showing that the RHS of~\eqref{SchurExpansion} satisfies the evolution equation~\eqref{MonotoneEvolutionEquation}. This is however quite technical\ldots We still sketch the proof because some steps are conceptually interesting (at least the first one is, while the others are kind of classical).

The first step is to transform the evolution equation into a set of differential constraints $L_n \tilde{\tau}^{\text{mon}}(t;\ptimes,u)=0$, $n\geq 1$. The advantage of doing so is that those constraints live on the Lie algebra version of the representation $R_0$\footnote{It is a representation $r_0$ of $\overline{\mathfrak{gl}}_\infty$ which satisfies $e^{r_0(X)} = R_0(e^X)$.}. This implies that those operators are simple enough on Schur functions.

Let $a_\lambda(n) \coloneqq \frac{1}{\hook(\lambda)^2 o_\lambda(1^{2n})}$. Then the constraints give recursion relations on the $a_\lambda(n)$s which we need to prove,
\begin{equation}
	-\frac{t}{2}\delta_{r,1} a_\lambda(n) + t^r\sum_{j=1}^{k}\Bigl(n+\lambda_j-j+r\Bigr)a_{\lambda+r\epsilon_j}(n) = 0
\end{equation}
where $\lambda+r\epsilon_j$ is the partition with a ribbon of size $r$ added to $\lambda$ such that its lowest cell is on the row $j$ (if possible). This is proved as an identity between rational functions of $\rho_i\coloneqq \lambda_i-i$. By showing that the LHS has no poles (including at $\infty$), it is found to be a constant, which in turn is found to be zero.

\begin{remark}
	Oliveira and Novaes conjectured \cite[Conjecture 1]{OliveiraNovaes2021} an equality relating $o_\lambda(1^n)$ to the principal specialization $Z_\lambda(1^n)$. This is in fact just the projection of the equality between the zonal expansion \eqref{NonOrientedMonotoneSeries} and the Schur expansion of \Cref{thm:MonotoneSchurExpansion} on an arbitary Schur function via the Hall scalar product.
\end{remark}

\subsection{Formal-$N$ BKP tau function} The last step in our analysis of $\tilde{\tau}^{\text{mon}}$ is to show that it is a formal-$N$ BKP tau function. From \Cref{thm:PfaffianCoefficients}, we know that it is enough to prove that the coefficients $a_\lambda(n)$ in the Schur expansion are Pfaffian minors.
\begin{theorem}{}{MonotonePfaffianCoefficients}
	Let $n$ be an integer, and suppose that $\ell(\lambda) \leq n$.	Then
	\begin{align}\label{eq:alambdaPfaffian}
		a_\lambda(n) &= \frac{1}{\hook_\lambda^2\cdot o_\lambda(1^{2n})} = \begin{cases} \displaystyle\prod_{k=1}^{n-1}(2k)! \cdot \Pf(a_{\lambda_i+n-i,\lambda_j+n-j})_{1 \leq i,j \leq n}\ \text{ for $n$ even,}\\
			\displaystyle\prod_{k=1}^{n-1}(2k)! \cdot \Pf(a_{\lambda_i+n-i,\lambda_j+n-j})_{1 \leq i,j \leq n+1}\ \text{ for $n$ odd,} \end{cases}\\
		\mbox{where }\ \ \ a_{i,j} &= \begin{cases} \frac{i-j}{4(i+j)\ i!^2 j!^2} &\mbox{ for }\  i,j\geq 1,\\
			\frac{1}{2i!^2} &\mbox{ for }\  j\in \{-1,0\}, i > 0,\\
			1 &\mbox{ for }\ j=-1, i=0,\\
			\frac{-1}{2j!^2} &\mbox{ for }\ i\in \{-1,0\}, j > 0,\\
			-1 &\mbox{ for }\ i=-1, j=0,\\
			{\scriptstyle{0}} &\mbox{ for }\  i,j \in \{-1,0\},i=j,\end{cases} \nonumber
	\end{align}
	and $\lambda_{n+1}=0$ by convention.
\end{theorem}
It relies on the classical (very nice nonetheless) Schur's Pfaffian identity.
\begin{lemma}
	\label{lem:SchurPfaffian}
	Let $n$ be an integer and $x_1, \dotsc, x_n$ be $n$
	real variables. Set $x_{n+1}:=0$ by convention, and
	$x_{ij} := \frac{x_i-x_j}{x_i+x_j}$ for $i,j=1,
	\dotsc, n+1$, with $x_{ij}=\begin{cases} 1 \text{ if }
		x_i=x_j=0 \text{ and } i\neq j,\\ 0 \text{ if }
		x_i=x_j=0 \text{ and } i=j.\end{cases}$. Then
	\begin{equation}
		\label{eq:SchurPfaffian}
		\prod_{1 \leq i < j \leq n}\frac{x_i-x_j}{x_i+x_j} = \begin{cases} \Pf( x_{ij})_{1\leq i,j\leq n} \qquad & \text{for $n$ even,}\\
			\Pf( x_{ij})_{1\leq i,j\leq n+1} & \text{for $n$ odd.} \end{cases}
	\end{equation}
\end{lemma}
It is applied to \eqref{eq:dimOrthogonal} (with some algebraic manipulations) to get \Cref{thm:MonotonePfaffianCoefficients}. As a consequence of \Cref{thm:MonotoneSchurExpansion}, \Cref{thm:PfaffianCoefficients} and \Cref{thm:MonotonePfaffianCoefficients}, we find \Cref{thm:MonotoneBKP}.

%%%%%%%%%%%%%%%%%%%%%%%
\chapter{Colored triangulations from prescribed building blocks} \label{sec:PrescribedBubbles}
In this chapter we introduce the main questions we will tackle in the coming chapters, and the key notion of \emph{boundary bubbles} which is used to formulate the Tutte-like equations for colored graphs.

\section{Colored graphs from prescribed bubbles} 
\subsection{Rooted graphs} For enumeration purposes, it is simpler to work with rooted graphs. \emph{A root in a colored graph is a marked white vertex}, or equivalently a marked edge of any color (the root being the white vertex it is connected to). Note that a root can be used to obtain a labeling of the vertices of the graph, by complementing it e.g. with any algorithm which uses the colors to move into the graph. This shows that rooting removes the symmetries. For colored triangulations, the root is a marked white $d$-simplex, or equivalently a marked $(d-1)$-simplex of any color.

Instead of considering the whole set of colored triangulations, we focus on those which are built by gluing some prescribed CBBs. Let $\{\CBB_1, \dotsc, \CBB_N\}$ be a finite set of CBBs and $n_1, \dotsc, n_N$ some positive integers. Let $\cT_{n_1, \dotsc, n_N}(\CBB_1, \dotsc, \CBB_N)$ be the set of connected, rooted colored triangulations built from $n_i$ copies of the CBB $\CBB_i$ for $i=1, \dotsc, N$. Similarly, let $\{\bb_1, \dotsc, \bb_N\}$ be the set of bubbles corresponding to $\{\CBB_1, \dotsc, \CBB_N\}$ and $\cG_{n_1, \dotsc, n_N}(\bb_1, \dotsc, \bb_N)$ be the set of connected, rooted colored graphs with $n_i$ copies of the bubble $\bb_i$, for $i=1, \dotsc, N$.

\subsection{Grading with respect to the number of $(d-2)$-simplices}  For $\trisp\in \cT_{n_1, \dotsc, n_N}(\CBB_1, \dotsc, \CBB_N)$ we further denote $\Delta_{ab}(\trisp)$ the number of $(d-2)$ simplices of $\trisp$ labeled with the pair of colors $\{a, b\}$ and given some integers $\{\delta_{ab}\}_{a<b}$,
\begin{equation}
	\cT_{n_1, \dotsc, n_N}(\CBB_1, \dotsc, \CBB_N | \{\delta_{ab}\}) = \left\{\trisp \in \cT_{n_1, \dotsc, n_N}(\CBB_1, \dotsc, \CBB_N), \forall\ a<b \quad \Delta_{ab}(\trisp) = \delta_{ab} \right\}.
\end{equation}
For $\graph\in \cG_{n_1, \dotsc, n_N}(\bb_1, \dotsc, \bb_N)$, we denote $C_{ab}(\graph)$ the number of bicolored cycles with colors $\{a, b\}$ and we introduce the subset of rooted colored graphs built from the bubbles $\{\bb_i\}$ with a prescribed number of bicolored cycles
\begin{equation}
	\cG_{n_1, \dotsc, n_N}(\bb_1, \dotsc, \bb_N | \{\delta_{ab}\}) = \left\{\graph \in \cG_{n_1, \dotsc, n_N}(\bb_1, \dotsc, \bb_N), \forall\ a<b \quad C_{ab}(\graph) = \delta_{ab} \right\}.
\end{equation}

From \Cref{thm:ColoredGraphs} (third point), we know that
\begin{equation}
	\Delta_{ab}(\trisp) = C_{ab}(\graph(\trisp)),
\end{equation}
and from \Cref{thm:CBB} we know that sampling triangulations by CBBs is the same as sampling colored graphs by bubbles. This shows that 
\begin{equation}
	\cT_{n_1, \dotsc, n_N}(\CBB_1, \dotsc, \CBB_N | \{\delta_{ab}\}) \cong \cG_{n_1, \dotsc, n_N}(\bb_1, \dotsc, \bb_N | \{\delta_{ab}\}).
\end{equation}
This can obviously be extended to $k$-simplices of colors $\{c_1, \dotsc, c_{d-k}\}$ and connected components of $\graph(c_1, \dotsc, c_{d-k})$, but as already explained, we will be exclusively interested in the number of $(d-2)$-simplices. 

The reason why \Cref{thm:Gurau} and the Gurau-Schaeffer classification \cite{GurauSchaeffer} of colored graphs with respect to Gurau's degree are not enough to classify the colored triangulations of $\cT_{n_1, \dotsc, n_N}(\CBB_1, \dotsc, \CBB_N)$ is that in general this set does not contain triangulations of vanishing Gurau's degree and consequently \cite{GurauSchaeffer} does not apply. Indeed, triangulations of vanishing Gurau's degree are melonic and it is easy to see by removing the color 0 that all their CBBs are melonic too (their colored graphs are melonic with $d$ colors).

Colored triangulations which are built from non-melonic CBBs cannot have vanishing Gurau's degree, because they cannot grow as many $(d-2)$-simplices. In fact, the Gurau-Schaeffer classification shows that for non-melonic CBBs there is only a finite number of colored triangulations at fixed value of Gurau's degree. This means that no notion of large scale, continuous limit can be reached, and universality cannot be studied using Gurau's degree in general.

If $\graph\in\cG_{n_1, \dotsc, n_N}(\bb_1, \dotsc, \bb_N)$, we denote $C(\graph)$ the total number of bicolored cycles of colors $\{a,b\}$ for $0\leq a<b\leq d$. Fixing the bubbles $\bb_i$ and their numbers $n_i$ actually fixes the number of bicolored cycles with colors $\{a, b\}$ for $1\leq a<b\leq d$, i.e. those which do not have the color 0. Let us denote $C(\bb_i)$ the total number of bicolored cycles of $\bb_i$ and 
\begin{equation}
	C_0(\graph) = \sum_{a=1} C_{0a}(\graph)
\end{equation}
the total number of bicolored cycles with colors $\{0, a\}$. Therefore
\begin{equation}
	C(\graph) = \sum_{0\leq a<b\leq d} C_{ab}(\graph) = C_0(\graph) + \sum_{i=1}^N n_i\, C(\bb_i).
\end{equation}
Since each $C(\bb_i)$ is fixed in $\cG_{n_1, \dotsc, n_N}(\bb_1, \dotsc, \bb_N)$, the classification with respect to $C(\graph)$ is equivalent to the classification with respect to $C_0(\graph)$. This establishes the main question below.

%%%%%%%%%%%%%%%%%%%%%%%%%
\section{Main questions} \label{sec:MainQuestion} 
\subsection{The linear growth hypothesis} Here we fix the bubbles $\bb_1, \dotsc, \bb_N$. Let $v_i$ be the number of vertices of $\bb_i$, so that $V\equiv \sum_{i=1}^N v_i n_i$ is the number of vertices of any $\graph\in\cG_{n_1, \dotsc, n_N}(\bb_1, \dotsc, \bb_N)$. Denote 
\begin{equation}
	C_{n_1 \dotsc n_N}(\bb_1, \dotsc, \bb_N) = \max_{\graph \in \cG_{n_1, \dotsc, n_N}(\bb_1, \dotsc, \bb_N)} C_0(\graph)
\end{equation}
the maximal number of bicolored cycles with colors $\{0, a\}$, $a=1, \dotsc, d$ at fixed numbers $n_1, \dotsc, n_N$ of bubbles, and 
\begin{equation} \label{Gmax}
	\cG^{\max}_{n_1, \dotsc, n_N}(\bb_1, \dotsc, \bb_N) = \Bigl\{\graph \in \cG_{n_1, \dotsc, n_N}(\bb_1, \dotsc, \bb_N),\ C_0(\graph) = C_{n_1, \dotsc, n_N}(\bb_1, \dotsc, \bb_N)\Bigr\}
\end{equation}
the subset of graphs which have this maximal number of bicolored cycles. The main question is two-fold.
\begin{itemize}
	\item Find $C_{n_1 \dotsc n_N}(\bb_1, \dotsc, \bb_N)$. This is equivalent to a sharp bound on $C_0(\graph)$. As it turns out, we have always found in practice that it \emph{grows linearly with the size of the graph}. Since the size of the graph is a function of every $n_i$, $i=1, \dotsc, N$, independently, there is a natural \emph{asymptotic} linear growth hypothesis,
	\begin{equation}
		C_{n_1 \dotsc n_N}(\bb_1, \dotsc, \bb_N) \thicksim_{n_i\to\infty} \alpha(\bb_i) n_i,
	\end{equation}
	for some constant $\alpha(\bb_i)$. We have always found in practice a much stronger version: for the same constants $\alpha(\bb_1), \dotsc, \alpha(\bb_N)$,
	\begin{equation} \label{LinearGrowth}
		C_{n_1 \dotsc n_N}(\bb_1, \dotsc, \bb_N) = d + \sum_{i=1}^N \alpha(\bb_i) n_i,
	\end{equation}
	(the shift by $d$ is found by setting all $n_i=0$). By comparison with Gurau's theorem in \Cref{thm:Gurau}, the constants satisfy
	\begin{equation} \label{GurauValue}
		\alpha(\bb_i) \leq \frac{d(d-1)}{4} v_i - \sum_{1\leq a<b\leq d} C_{ab}(\bb_i)
	\end{equation}
	where $C_{ab}(\bb_i)$ is the number of bicolored cycles with colors $\{a,b\}$ of $\bb_i$, and with equality if and only if either $\bb_i$ is melonic at $d\geq 3$, or $d=2$. This is the linear growth hypothesis, which we have been able to prove for some set of bubbles $(\bb_1, \dotsc, \bb_N)$.
	\item The second point is much simpler to state, but not to solve: Characterize the graphs which maximize the number of bicolored cycles, i.e. $\cG^{\max}_{n_1, \dotsc, n_N}(\bb_1, \dotsc, \bb_N)$. We were able to solve this question to some extent for a few families, as presented in the following chapters.
\end{itemize}

At $d=2$, let us apply Euler's relation \eqref{EulerColoredGraph} for 3-colored graphs to the present setting. Bicolored cycles have the pairs of colors $\{0,1\}$, $\{0,2\}$ and $\{1,2\}$. Each bubble $\bb_i$ contributes to one bicolored cycles with colors $\{1,2\}$, so $C(\graph) = C_0(\graph) + \sum_{i=1}^N n_i$. The total number of vertices is $\sum_{i=1}^N v_i n_i$, then \eqref{EulerColoredGraph} gives
\begin{equation} \label{EulerPrescribedBubbles}
	C_{0}(\graph) - \frac{1}{2}\sum_{i=1}^N (v_i-2) n_i = 2 - 2g(\graph),
\end{equation}
where $g(\graph)$ was defined as the genus of the associated bipartite map. This gives
\begin{equation} \label{MaxFacesPrescribedBubbles}
	C_{n_1 \dotsc n_N}(\bb_1, \dotsc, \bb_N) = 2 + \frac{1}{2}\sum_{i=1}^N (v_i-2) n_i
\end{equation}
i.e. \eqref{LinearGrowth} is satisfied with $\alpha(\bb_i)=(v_i-2)/2$, and in agreement with Gurau's value \eqref{GurauValue}, no surprises there. Notice the coefficient $1/2$ which is independent of the bubble. A major change in higher dimensions is that it is not the case anymore.
%Because of the special role played by the color 0, we will, from here on out, draw {\bf edges of color 0 with dashed lines and edges of colors $1, \dotsc, d$ with solid lines}.
\subsection{Root bubble} In order to study the structure of a graph $\graph\in\cG^{\max}_{n_1, \dotsc, n_N}(\bb_1, \dotsc, \bb_N)$ around a given bubble $\bb\subset \graph$, it is convenient to mark this bubble, then called the \emph{root bubble}, and to label its vertices\footnote{It can also be achieved by rooting the bubble and using the colors to induce a labeling of the vertices with any prescribed algorithm.}. If $\bb$ is a bubble, we denote $\vec{\bb}$ a labeled version.

We also allow $\bb$ to differ $\bb_1, \dotsc, \bb_N$ and denote $\cG_{n_1, \dotsc, n_N}(\vec{\bb};\bb_1, \dotsc, \bb_N)$ the set of graphs made out of $n_i$ copies of $\bb_i$ for $i=1, \dotsc, N$ and containing one marked copy of $\vec{\bb}$. There is an obvious inclusion\footnote{If $\bb$ is the same as one of the $\bb_i$s, we still think of it as a different one.}
\begin{equation} \label{RootBubbleInclusion}
	\cG_{n_1, \dotsc, n_N}(\vec{\bb};\bb_1, \dotsc, \bb_N) \hookrightarrow \cG_{1,n_1, \dotsc, n_N}(\bb,\bb_1, \dotsc, \bb_N),
\end{equation}
where one uses for instance the vertex with label 1 to root the graph, then removing the other labels.
%Then $\cG(\vec{\bb};\bb_1, \dotsc, \bb_N) = \displaystyle{\bigcup_{n_1, \dotsc, n_N\geq 0}} \cG_{n_1, \dotsc, n_N}(\vec{\bb};\bb_1, \dotsc, \bb_N)$. 
In the case of maps, marking a bubble $\vec{\bb}$ with labeled vertices means marking a face of a given degree with labeled vertices.

In order to discuss the main questions established above, we also pull apart those graphs which maximize the number of bicolored cycles at fixed $n_1, \dotsc, n_N$. Let $C_{n_1, \dotsc, n_N}(\vec{\bb};\bb_1, \dotsc, \bb_N)$ be the maximal value of the number of bicolored cycles among $\graph\in\cG_{n_1, \dotsc, n_N}(\vec{\bb};\bb_1, \dotsc, \bb_N)$, then denote (it is the same as \eqref{Gmax} with a root bubble)
\begin{equation} 
	\cG^{\max}_{n_1, \dotsc, n_N}(\vec{\bb};\bb_1, \dotsc, \bb_N) = \Bigl\{\graph \in \cG_{n_1, \dotsc, n_N}(\vec{\bb};\bb_1, \dotsc, \bb_N),\ C_0(\graph) = C_{n_1, \dotsc, n_N}(\vec{\bb};\bb_1, \dotsc, \bb_N)\Bigr\}.
\end{equation}

Also as in the case of rooted graphs, we expect the same linear growth of the number of bicolored cycles at large size. This is because the difference between $\cG^{\max}_{n_1, \dotsc, n_N}(\vec{\bb};\bb_1, \dotsc, \bb_N)$ and $\cG^{\max}_{n_1, \dotsc, n_N}(\bb_1, \dotsc, \bb_N)$ is ``local'', i.e. the number of bicolored cycles with colors $\{0,c\}$ which go through $\vec{\bb}$ remains finite (as long as $\vec{\bb}$ is) when the number of vertices becomes large. Thus we expect
\begin{equation}
	C_{n_1, \dotsc, n_N}(\vec{\bb};\bb_1, \dotsc, \bb_N) \thicksim_{n_i\to\infty} C_{n_1, \dotsc, n_N}(\bb_1, \dotsc, \bb_N).%\leq C(\vec{\bb}) + \sum_{i=1}^N \alpha(\bb_i) n_i.
\end{equation}
If the stronger linear growth hypothesis \eqref{LinearGrowth} holds for the set $\{\bb, \bb_1, \dotsc, \bb_N\}$, then from the inclusion \eqref{RootBubbleInclusion}
\begin{equation} \label{LinearGrowthRootBubble}
	C_{n_1, \dotsc, n_N}(\vec{\bb};\bb_1, \dotsc, \bb_N) = d + \alpha(\bb) + \sum_{i=1}^N \alpha(\bb_i) n_i.
\end{equation}
It also makes clear that $\alpha(\bb)$ can be found by maximizing the number of bicolored cycles at $n_i=0$ for all $i=1, \dotsc, N$. This corresponds to graphs with a single bubble, or dually triangulations with a single CBB. Studying those objects will be the specific topic of Chapter \ref{sec:1CBB} so we refrain to add further comments at this point.

\section{Generating series}
\subsection{Definitions} It is well-known since Tutte that it is possible to use a decomposition at the root edge in order to write equations on the generating series of maps (this is the cut-and-join idea mentioned in the previous chapters). As also well-known, this requires to track the degree of the root face. The same type of decomposition can be used with generating series of colored graphs. This requires to keep track of the root bubble. We use formal variables $p_1, \dotsc, p_N$ to track the bubbles and $x$ to track the bicolored cycles and $t$ for the number of edges of color 0. Introduce for any $\vec{\bb}$ the generating series 
\begin{equation} \label{GF}
	\langle \vec{\bb}\rangle \equiv \sum_{n_1, \dotsc, n_N\geq 0} t^{\frac{v(\bb) + \sum_{i=1}^N v_in_i}{2}}\Biggl(\sum_{\graph\in\cG_{n_1, \dotsc, n_N}(\vec{\bb};\bb_1, \dotsc, \bb_N)} x^{C_0(\graph)}\Biggr) \prod_{i=1}^N p_i^{n_i}.
	%t^{E_0(\graph)}
\end{equation}
%where $E_0(\graph)$ is the number of edges of color 0. 
and
\begin{equation}
	\langle \vec{\bb}\rangle_{\max} = \sum_{n_1, \dotsc, n_N\geq 0} t^{\frac{v(\bb) + \sum_{i=1}^N v_in_i}{2}} |\cG^{\max}_{n_1, \dotsc, n_N}(\vec{\bb};\bb_1, \dotsc, \bb_N)| x^{C_{n_1, \dotsc, n_N}(\vec{\bb};\bb_1, \dotsc, \bb_N)} \prod_{i=1}^N p_i^{n_i}
\end{equation}
by restricting the sum of \eqref{GF} to $\cG^{\max}_{n_1, \dotsc, n_N}(\vec{\bb};\bb_1, \dotsc, \bb_N)$. Assuming the linear growth \eqref{LinearGrowthRootBubble} of $C_{n_1, \dotsc, n_N}(\vec{\bb};\bb_1, \dotsc, \bb_N)$, we can factor out the $x$-dependence by rescaling the variables $p_i\to p_i x^{-\alpha(\bb_i)}$,
\begin{equation}
	\langle \vec{\bb}\rangle_{\max|p_i\to p_i x^{-\alpha(\bb_i)}} = x^{d+\alpha(\bb)} \sum_{n_1, \dotsc, n_N\geq 0} t^{\frac{v(\bb) + \sum_{i=1}^N v_in_i}{2}} |\cG^{\max}_{n_1, \dotsc, n_N}(\vec{\bb};\bb_1, \dotsc, \bb_N)| \prod_{i=1}^N p_i^{n_i},
\end{equation}
where $v_i$ is the number of vertices of $\bb_i$.

At $d=2$, $\alpha(\bb) = \frac{v(\bb)-2}{2}$ and the above rescaling has the effect of giving the same power of $x$ to all maps of the same genus,
\begin{equation}
	x^{C_0(\graph)} \prod_{i=1}^N p_i^{n_i}{}_{|p_i\to p_i x^{-\alpha(\bb_i)}} = x^{1 + \frac{v(\vec{\bb})}{2}-2g(\graph)} \prod_{i=1}^N p_i^{n_i}.
\end{equation}
For general $d$ one can perform this rescaling already in $\langle \vec{\bb}\rangle$ and get a $1/x$-expansion. For instance at $d=2$ again
\begin{equation}
	\langle\vec{\bb}\rangle_{|p_i\to p_i x^{-\frac{v_i-2}{2}}} = \sum_{g\geq 0} x^{1 + \frac{v(\vec{\bb})}{2} - 2g} \sum_{n_1, \dotsc, n_N\geq 0} t^{\frac{v(\bb) + \sum_{i=1}^N v_in_i}{2}} |\cG_{n_1, \dotsc, n_N}^{(g)}(\vec{\bb};\bb_1, \dotsc, \bb_N)| \prod_{i=1}^N p_i^{n_i}
\end{equation}
where $\cG_{n_1, \dotsc, n_N}^{(g)}(\vec{\bb};\bb_1, \dotsc, \bb_N)$ is the restriction of $\cG_{n_1, \dotsc, n_N}(\vec{\bb};\bb_1, \dotsc, \bb_N)$ to genus $g$.

\subsection{Bubble subgraphs and bubble trees} We generalize the notion of root bubble to root subgraphs. Consider some bubbles $\vec{\bb}_1', \dotsc, \vec{\bb}_{N'}'$ (with labeled vertices and some can be the same bubbles) and a connected graph $\vec{H}$ whose bubbles are $\vec{\bb}_1', \dotsc, \vec{\bb}_{N'}'$ connected by some edges of color 0, but some vertices are allowed to have no incident edge of color 0. We call that type of graphs \emph{bubble subgraphs}. As a special case of bubble subgraphs, \emph{bubble trees} are those where the edges of color 0 are cut-edges.

We define $\cG(\vec{H};\bb_1, \dotsc, \bb_N)$ as the set of colored graphs built from $n_i$ copies of $\bb_i$ for $i=1..N$ and one copy of $\vec{H}$, by adding edges of color 0 until all vertices are incident to one. Then we define the same series as above,
\begin{equation} \label{SubgraphExpectation}
	\langle \vec{H}\rangle \equiv \sum_{\graph\in\cG(\vec{H};\bb_1, \dotsc, \bb_N)} t^{E_0(\graph)} x^{C_0(\graph)} \prod_{i=1}^N p_i^{n_i(\graph)}
\end{equation}
Here $E_0(\graph), C_0(\graph)$ are respectively the numbers of edges of color 0 and of bicolored cycles with colors $\{0,c\}$ in $\graph$ which are not contained in $\vec{H}$, and $n_i(\graph)$ the number of copies of $\bb_i$ in $\graph$ (which are not in $\vec{H}$ even though some $\bb'_i$s may be the same as some $\bb_j$s).

\section{Boundary bubbles and the Tutte/Schwinger-Dyson equations} 
There are two types of edges of color 0: those connecting two vertices of the same bubble, and those connecting two vertices of different bubbles. They give rise to the notions of \emph{self-contractions} and \emph{bridge-contractions} which generate a dynamics on bubbles, captured by the notion of \emph{boundary bubble}.
\subsection{Boundary bubbles} If a bubble subgraph $\vec{H}$ has no edges of color 0, then it is a bubble. Otherwise, we ask if there is a bubble $\bb_{\vec{H}}$ such that $\langle \vec{H}\rangle = \langle \bb_{\vec{H}}\rangle$. It exists, although it might not be connected and is called the boundary bubble of $\vec{H}$ and denoted $\partial \vec{H}$.
\begin{definition-lemma}
	The contraction of an edge $e$ of color 0 is the following move
	\begin{equation}
		\includegraphics[scale=.5,valign=c]{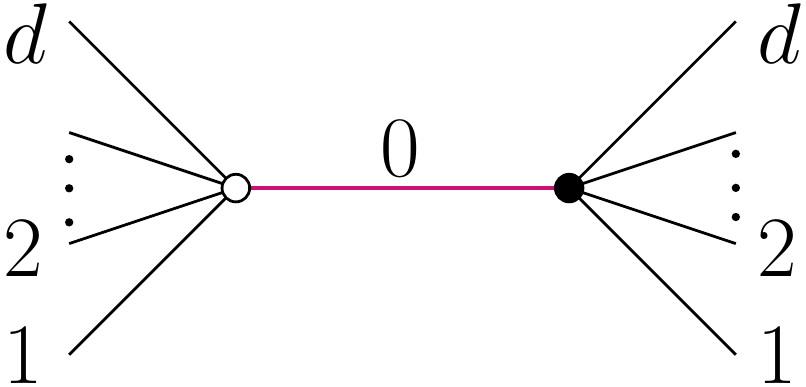} \quad\to\quad \includegraphics[scale=.5,valign=c]{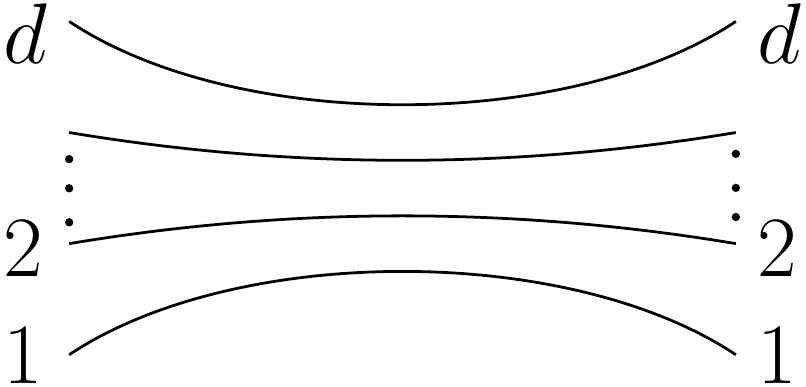}
	\end{equation}
	which we denote by the operator $\partial_e$. By applying $\partial_e$ to all edges of color 0, and eliminating possible closed loops, one obtains a disjoint union of bubbles called the boundary bubble. It is independent of the order of the contractions, and denoted
	\begin{equation}
	\partial \vec{H} \equiv \prod_{\text{$e$ of color 0}} \partial_e \vec{H}.
	\end{equation}
	We have
	\begin{equation} \label{BoundaryExpectation}
		\langle \vec{H}\rangle = \langle \partial \vec{H}\rangle.
	\end{equation}
\end{definition-lemma}
In other words, the series $\langle \vec{H}\rangle$ only sees the boundary. Notice that $\partial \vec{H}$ is obtained by keeping the vertices which have no incident edge of color 0 and drawing an edge of color $c\in[1..d]$ between two of them if there is a sequence of edges with colors $\{0,c\}$ between them in $\vec{H}$.
\begin{remark}
	If the bubbles incident to the white and black vertices on the left are distinct, and if one of them is dual to a $d$-ball, then the contraction does not change the topology of the dual triangulation.
\end{remark}

%\begin{proof}
%	The fact that $\partial \vec{H}$ is independent of the order of the contractions is left to the reader.
%	
%	There is an obvious bijection between $\cG(\vec{H};\bb_1, \dotsc, \bb_N)$ and $\cG(\partial\vec{H};\bb_1, \dotsc, \bb_N)$ obtained by interchanging $\vec{H}$ and $\partial \vec{H}$ in each graph. In the definition of the series, mapping $\graph\in \cG(\vec{H};\bb_1, \dotsc, \bb_N)$ to its image $\graph_\partial\in\cG(\partial\vec{H};\bb_1, \dotsc, \bb_N)$ does not change $n_i(\graph)$s. It only remains to check that they have the same number of bicolored cycles with colors $\{0,c\}$ which are not contained in $\vec{H}$, and that is clear by the definition of $\partial \vec{H}$.
%\end{proof}

\subsection{Tutte-like decomposition} 
Let $\graph\in\cG(\vec{\bb};\bb_1, \dotsc, \bb_N)$ and let $v\in\vec{\bb}\subset\graph$ be the root vertex, e.g. the white vertex with label 1. There is an edge of color 0 incident to $v$ in $\graph$. The decomposition consists in looking at which black vertex $\bar{v}$ it is connected to. The vertex $\bar{v}$ can belong to
\begin{itemize}
	\item $\vec{\bb}$ itself, with weight $t$, and $x$ to the power the number $C(v, \bar{v})$ of edges with colors in $[1..d]$ between $v$ and $\bar{v}$,
	\item or to a copy of $\bb_i$ for $i\in[1..N]$, with weight $t p_i/|\operatorname{Aut}(\vec{\bb}_i)|$, where $|\operatorname{Aut}(\vec{\bb})|$ is the order of the symmetry group of $\vec{\bb}$.
\end{itemize}
We obtain
\begin{equation} \label{SDEwithEdges}
	\begin{aligned}
		\langle \vec{\bb}\rangle %&= \left\langle \includegraphics[scale=.5,valign=c]{RootedBubble.pdf} \right\rangle \\
		= t\sum_{\bar{v}\in\vec{B}} x^{C(v,\bar{v})} \left\langle \includegraphics[scale=.5,valign=c]{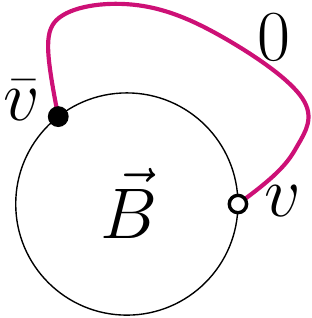} \right\rangle + t\sum_{i=1}^N \frac{p_i}{|\operatorname{Aut}(\vec{\bb}_i)|} \sum_{\bar{v}\in\bb_i} \left\langle \includegraphics[scale=.5,valign=c]{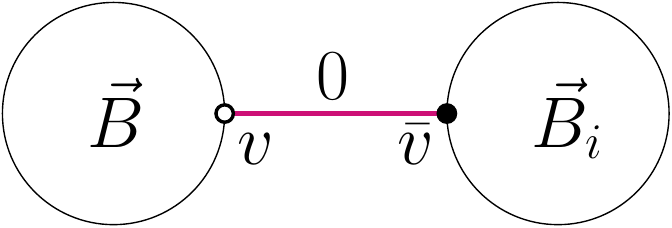} \right\rangle
	\end{aligned}
\end{equation}
where in the last term an arbitrary labeling of $\bb_i$ is used to perform the sum over all its black vertices. Equivalently, one can remove $1/|\operatorname{Aut}(\vec{\bb}_i)|$ and sum over non-equivalent vertices.

\subsection{Self-contractions and bridge-contractions} In the second line we have series of subgraphs $\vec{H}$ as defined in \eqref{SubgraphExpectation}, and we can use the boundary bubbles to recast them as depending on bubbles instead of subgraphs. Denote
\begin{equation}
	\vec{\bb}_{/(v, \bar{v})} = \partial \includegraphics[scale=.5,valign=c]{RootedBubbleContraction.pdf}, \qquad \vec{\bb}_v*\vec{\bb}_{i,\bar{v}} = \partial \includegraphics[scale=.5,valign=c]{RootedBubbleGluing.pdf}
\end{equation}
where the first operation is called a \emph{self-contraction} and the second a \emph{bridge-contraction} (because the edge of color 0 is a bridge, i.e. a cut-edge). Then the previous equation gives rise the following theorem, initially derived as the Schwinger-Dyson equations of random tensor models.

\begin{theorem}{}{} \cite{SDE-Gurau}
	\begin{equation} \label{SDE}
		\langle \vec{\bb}\rangle = t\sum_{\bar{v}\in\vec{B}} x^{C(v,\bar{v})} \langle \vec{\bb}_{/(v, \bar{v})}\rangle + t\sum_{i=1}^N \frac{p_i}{|\operatorname{Aut}(\vec{\bb}_i)|} \sum_{\bar{v}\in\bb_i} \langle \vec{\bb}_v*\vec{\bb}_{i,\bar{v}}\rangle
	\end{equation}
\end{theorem} 

At $d=2$, one recovers the Tutte equations, because $\vec{\bb}$ is a cycle with colors $\{1,2\}$ and $2p$ vertices, labeled, say increasingly along the cycle. Then, if $\bar{v}$ is a the $k$-th black vertex, $\vec{\bb}_{/(v, \bar{v})}$ is a disjoint union of two cycles, one with $2(k-1)$ vertices, the other with $2p-2k$ vertices, $k=1, \dotsc, p$. Moreover, if $\bb_i$ is a cycle with $2i$ vertices, then $\vec{\bb}_v*\vec{\bb}_{i,\bar{v}}$ is a cycle with $2p+2i-2$ vertices, and there are $i$ possible black vertices to create it. At $d=3$ however, this dynamics is much more complicated because the bubbles are (dual to) maps instead of cycles.

\section{Boundary dynamics} 
The boundary dynamics of the model $\cG(\bb_1, \dotsc, \bb_N)$ is the set $\mathcal{B}(\bb_1, \dotsc, \bb_N)$ of bubbles which appear on the RHS of \eqref{SDE} by starting to apply it with each $\vec{\bb} = \vec{\bb}_i$ for $i=1..N$, and then recursively with $\vec{\bb}$ equal to every bubble which appears on the RHS. If $\vec{\bb}\in\mathcal{B}(\bb_1, \dotsc, \bb_N)$, by definition there exists a sequence
\begin{equation}
		\vec{\bb} = \vec{\bb}^{(0)}\ \underset{\partial_{e_1}}{\leftarrow}\ \vec{\bb}^{(1)}\ \leftarrow \dotsb \underset{\partial_{e_k}}{\leftarrow} \ \vec{\bb}^{(k)} = \vec{\bb}_i
\end{equation}
starting with some $\vec{\bb}_i$ and where $\vec{\bb}^{(j-1)}$ is obtained from $\vec{\bb}^{(j)}$ by applying either a self-contraction or a bridge-contraction with another bubble from $\{\bb_1, \dotsc, \bb_N\}$ on an edge $e_j$ of color 0.

Since the order of the contractions does not matter, we can wait to perform them ``at the end'', i.e. let $\vec{H}_{e_1, \dotsc, e_k}$ be the bubble subgraph obtained by adding the edges $e_1, \dotsc, e_k$ and the additional bubbles in case of bridge-contractions. Clearly $\vec{H}$ is a bubble subgraph from the model $\cG(\bb_1, \dotsc, \bb_N)$, and $\vec{\bb} = \partial \vec{H}_{e_1, \dotsc, e_k}$.
\begin{proposition}{}{BoundariesOfBubbleGraphs}
	$\mathcal{B}(\bb_1, \dotsc, \bb_N)$ is the set of boundaries of bubble subgraphs of $\cG(\bb_1, \dotsc, \bb_N)$.
\end{proposition}

\subsection{Boundary dynamics in a quartic model} If the allowed bubbles are $\bb_i = \quart(\{i\})$, i.e. the 4-vertex bubbles with $\colset=\{i\}$ for $i=1, \dotsc, d$, see Figure \ref{fig:Bubbles}, then it is called a quartic model and we have
\begin{proposition}{}{BoundariesQuarticModel}\cite{StuffedColoredMaps} 
	$\mathcal{B}(\quart(\{1\}), \dotsc, \quart(\{d\}))$ is the set of all bubbles.
\end{proposition}
There is in addition an explicit construction of $\vec{H}$ such that $\partial \vec{H} = \vec{\bb}$ for any $\vec{\bb}$. It means that any bubble can be obtained as a boundary bubble from that quartic model. It is especially useful because that quartic model is fairly well-understood. In particular, there is a simple bijection between its colored graphs and maps with colored edges. By complementing that bijection with the explicit construction for the bubble subgraphs $\vec{H}$, we have obtained in \cite{StuffedColoredMaps} a general bijection for colored graphs, which is detailed in Section \ref{sec:Bijection} of Chapter \ref{sec:Universality2}.

Obviously the real difficulty here is to control the exponent of $x$ associated with a sequence of self-contractions and bridge-contractions, which we do not know how to do in general. Still, we were able in \cite{StuffedColoredMaps} to make use of this bijection to characterize $\cG^{\max}_{n_1, \dotsc, n_N}(\bb_1, \dotsc, \bb_N)$ for new sets of bubbles. Chapter \ref{sec:1CBB} will be devoted to that question in the case of self-contractions only, equivalently triangulations with a single CBB, for some families of CBBs. 
\chapter{One-CBB triangulations and pairings} \label{sec:1CBB}
%%%%%%%%%%%%%%%%%%%%%%%%
Here we continue using $\vec{\bb}$ for a bubble $\bb$ with labeled vertices. %, and use $\vec{\CBB}$ for the CBB with labeled $(d-1)$-simplices of color 0.
	
\section{Some enumerative results}
\subsection{Pairings} A particular set of triangulations consists of those with a single CBB. They generalize 1-face maps \cite{ChapuyFerayFusy2013} to higher dimensions. In terms of simplices, they are formed by a perfect matching, or a \emph{pairing}, of its boundary $(d-1)$-simplices which identifies them two by two to form a closed space (recall that the identification is uniquely determined by the colors).

In terms of colored graphs, they have a single bubble $\bb$, and are obtained by adding edges of color 0 between its black and white vertices. If $v$ is a white vertex and $\bar{v}$ a black vertex and if they are connected by an edge of color 0, we denote as in Figure \ref{fig:Pairing}
\begin{equation}
	\bar{v} = \pi(v)
\end{equation}
so that $\pi$ is a map from the white to the black vertices. For a fixed labeling of the vertices of both colors from $1$ to $n$, $\pi$ is thus a permutation on $[1..n]$. Conversely $\pi$ determines a unique way of adding to $\bb$ the edges of color 0 so that one obtains a $(d+1)$-colored graph $\graph_\pi\in\cG(\vec{\bb})$. Here $\cG(\vec{\bb})$ is the set of connected $(d+1)$-colored graphs made out of a single bubble with labeled vertices $\vec{\bb}$. 
\begin{figure}
	\includegraphics[scale=.5]{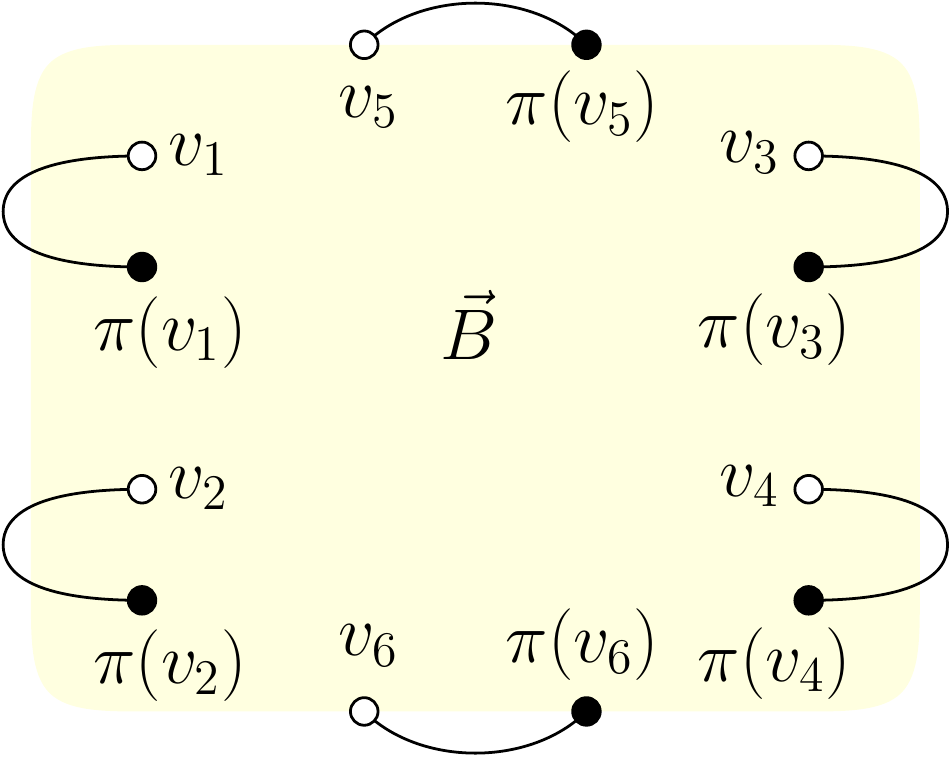}
	\caption{\label{fig:Pairing}A pairing is equivalent to a graph $\graph\in\cG(\vec{\bb})$.}
\end{figure}

Recall that the Tutte/Schwinger-Dyson equations revealed two types of operations which build the boundary bubbles: self-contractions and bridge-contractions. The 1-bubble colored graphs we now study correspond to studying self-contractions, since they have all their edges of color 0 on the same bubble.
%\begin{definition} [Pairing of a CBB/bubble]
%	A pairing (or perfect matching) $\pi$ of a labeled CBB $\CBB$ is a partition of its boundary $(d-1)$-simplices in bipartite pairs. In the dual picture, a pairing of $\vec{\bb}$ is a bijection between its white vertices and its black vertices. 
%\end{definition}

\subsection{Enumerative results}
We denote $\cG^{\max}(\vec{\bb})$ the set of graphs $\graph\in\cG(\vec{\bb})$ which maximize $C_0(\graph)$ the total number of cycles of colors of $\{0,c\}$ for $c=1, \dotsc, d$, see Section \ref{sec:MainQuestion} of Chapter \ref{sec:PrescribedBubbles}. Not much is known about that set outside of simple families of bubbles. However, a broad set of behaviors has already been observed, i.e. families for which $|\cG^{\max}(\vec{\bb})|$ has exponential growth or polynomial growth \cite{Meanders}. In general however it seems hopeless to look for a formula for $|\cG^{\max}(\vec{\bb})|$. In particular, $\vec{\bb}$ is defined by $d$ permutations and the simple families mentioned above are characterized instead by a smaller set of data.

Following the classical strategy of decomposing an object into easier or smaller objects, an interesting approach concerns bubbles obtained from sequences of bridge-contractions, or equivalently which are trees of bubbles. A natural question for them is whether there is a formula relating the $|\cG^{\max}(\vec{\bb})|$ to the values $|\cG^{\max}(\vec{\bb}')|$ of the bubbles $\vec{\bb}'$ which appear in the tree of bubbles. For instance, for the bubbles which are boundary bubbles of trees of quartic bubbles, it is possible to recast them as trees of other bubbles called necklaces, so that $|\cG^{\max}(\vec{\bb})|$ is just a product $\prod_{\vec{\bb}'} |\cG^{\max}(\vec{\bb}')|$ of the values $|\cG^{\max}(\vec{\bb}')|$ for every necklace, which are Catalan numbers. So in that case, $|\cG^{\max}(\vec{\bb})|$ is a product of Catalan numbers \cite{Enhancing, GM}.

Below we study another example from \cite{Meanders}, for a family where $|\cG^{\max}(\vec{\bb})|$ counts meander systems, and there is a factorization over some irreducible components up to some choices of trees.

\section{2-cylic bubbles and meanders}
In \cite{Meanders} we found a family of bubbles $B_\sigma$ for $\sigma\in\mathfrak{S}_n$ where $2n$ is the number of vertices, for which $\cG^{\max}(\vec{\bb}_\sigma)$ is a set of meanders characterized by $\sigma$. We reproduce the main results here. Throughout this section $d=4$, but it can be generalized to ``higher-dimensional meanders'' for any $d=4+2d'$.

\subsection{2-cyclic bubbles} The family of bubbles we are interested in are those which have 
\begin{itemize}
	\item a single bicolored cycle with colors $\{1,2\}$,
	\item a single bicolored cycle with colors $\{3,4\}$.
\end{itemize}
\begin{figure}
	\includegraphics[scale=.45,valign=c]{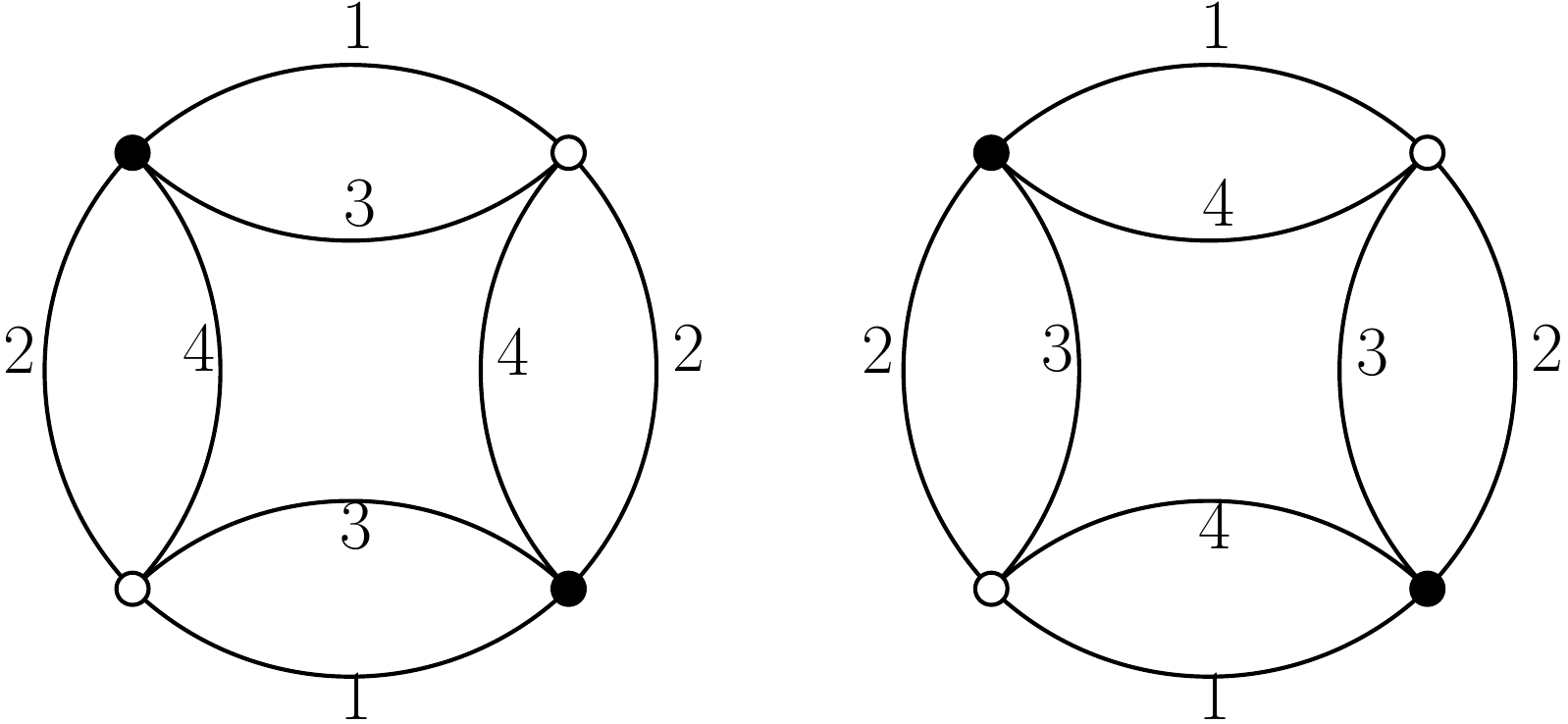} \hspace{2cm} \includegraphics[scale=.6,valign=c]{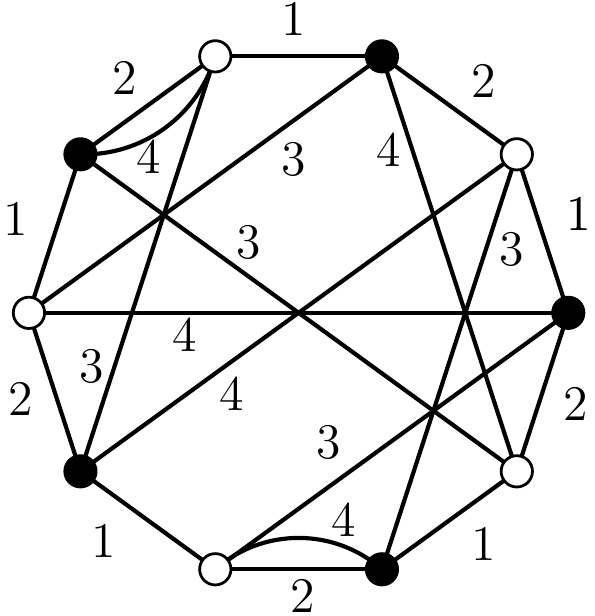}
	\caption{\label{fig:2CyclicBubbles}There are two 2-cyclic bubbles with 4 vertices. On the right is one with 10 vertices.}
\end{figure}
We call them 2-cyclic bubbles. %and denote $\mathcal{C}_n$ the set of 2-cyclic bubbles on $2n$ (unlabeled) vertices, and $\mathcal{C}_n^\circ$ the set of 2-cyclic bubbles on $2n$ vertices with a marked white vertex. Examples are provided in Figure \ref{fig:2CyclicBubbles}.
They are conveniently described with permutations, provided vertices are labeled. We use special, cycle-induced labelings, shown in Figure \ref{fig:InducedLabeling}.
\begin{itemize}
	\item Given a marked white vertex labeled $1_\circ$, the labeling induced by the cycle of colors $(1,2)$ (note that there is an order between the two colors), denoted $(1_\bullet, 1_\circ, \dotsc, n_\bullet, n_\circ)$, has an edge of color $1$ connecting $j_\circ$ to $j_\bullet$ and an edge of color $2$ connects $j_\circ$ to $(j+1)_\bullet$ (modulo $n$),
	\item Given a possibly different marked white vertex $1'_\circ$, the labeling induced by the cycle of colors $(3,4)$, denoted $(1'_\bullet, 1'_\circ, \dotsc, n'_\bullet, n'_\circ)$, has an edge of color $3$ connecting $j'_\circ$ to $j'_\bullet$ and an edge of color $4$ connects $j'_\circ$ to $(j+1)'_\bullet$ (modulo $n$).
\end{itemize}
\begin{figure}
	\includegraphics[scale=.5]{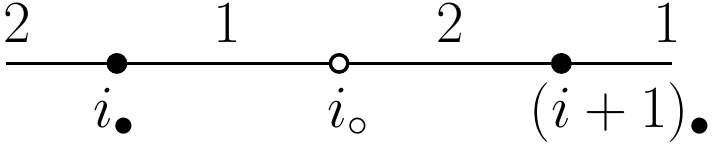}\hspace{2cm} \includegraphics[scale=.5]{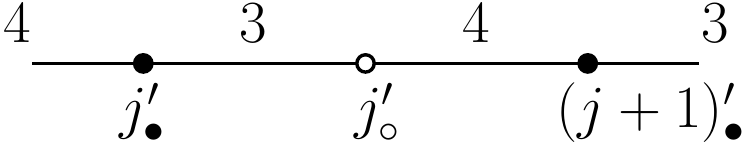}
	\caption{\label{fig:InducedLabeling}The labeling induced by the cycles with colors $(1,2)$ and $(3,4)$.}
\end{figure}
For simplicity, we set $1'_\circ =1_\circ$, so that $\mathcal{C}_n^\circ$ the set of 2-cyclic bubbles on $2n$ vertices with a marked white vertex is $\mathcal{C}_n^\circ \cong \mathfrak{S}_{n-1}\times\mathfrak{S}_{n}$. 

Indeed, after choosing the marked vertex $1_\circ$, every other white vertex receives two labels, say $i'_\circ$ from the cycle with colors $(3,4)$ and $j_{i\circ}$ from the cycle with colors $(1,2)$ (and $i'_\bullet, j_{i\bullet}$ for black vertices), $i=1,\dotsc,n$. A 2-cyclic bubble is thus described by a pair of permutations $(\sigma_\circ, \sigma_\bullet)$ defined as
\begin{equation}
	\bigl(\sigma_\circ(i)\bigr)_\circ = j_{i\circ},\qquad \text{and} \qquad \bigl(\sigma_\bullet(i)\bigr)_\bullet = j_{i\bullet}.
\end{equation}
i.e. they map the labels induced by the cycle with colors $(3,4)$ to those induced by the cycle with colors $(1,2)$. This looks locally like
\begin{equation}
	\includegraphics[scale=.65,valign=c]{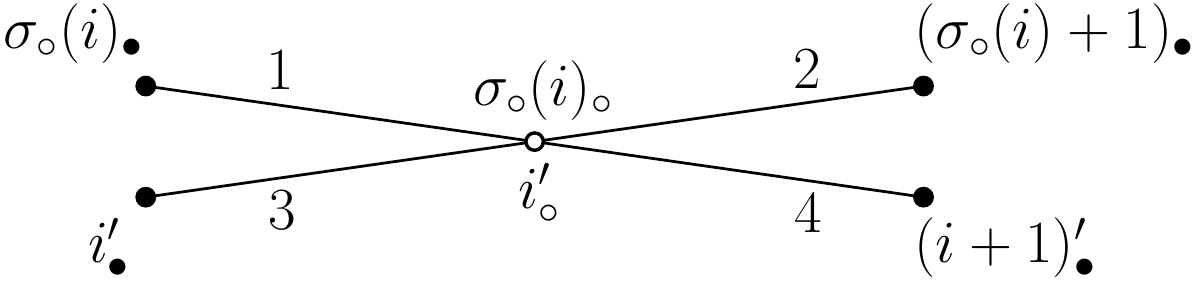}\qquad \includegraphics[scale=.65,valign=c]{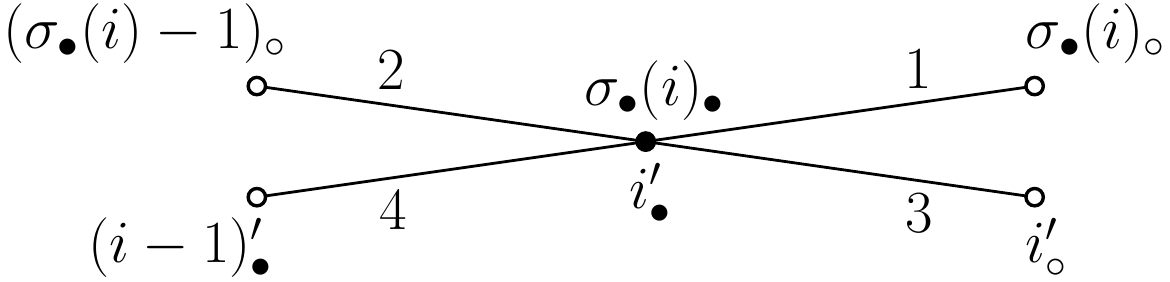}
\end{equation}
The other way around, every pair $(\sigma_\circ, \sigma_\bullet)\in\mathfrak{S}_n^2$ with $\sigma_\circ(1)=1$ can be used to construct a 2-cyclic bubble by reversing the process. It is denoted $\vec{B}_{\sigma_\circ, \sigma_\bullet}$.
	
Changing the reference vertex $1_\circ$ is an action of the cyclic group $\mathbb{Z}_n$ by conjugation. Let $\Delta_p:i\in[1..n]\mapsto (i+p)\mod n$ be the shift modulo $n$. The action of $\mathbb{Z}_n$ is
\begin{equation} \label{LeftRightActions}
	(\sigma_\circ, \sigma_\bullet) \mapsto (\Delta_p\circ \sigma_\circ \circ \Delta_{-p}, \Delta_p\circ \sigma_\bullet \circ \Delta_{-p}), \qquad p=0,\dotsc,n-1,
\end{equation}
and obviously $\cG(\vec{B}_{\sigma_\circ, \sigma_\bullet}) = \cG(\vec{B}_{\Delta_p\circ \sigma_\circ \circ \Delta_{-p}, \Delta_p\circ \sigma_\bullet \circ \Delta_{-p}})$.

With our parametrization, the (equivalence classes under $\mathbb{Z}_n$ of) $(\sigma_\circ, \sigma_\bullet) = (\id, \id)$ and $(\sigma_\circ, \sigma_\bullet) = (\id, \Delta_1)$ are given in Figure \ref{fig:Examples2Cyclic}.
\begin{figure}
	\includegraphics[scale=.65]{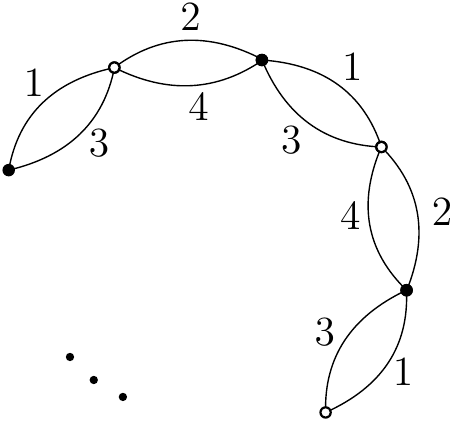} \hspace{2cm} \includegraphics[scale=.65]{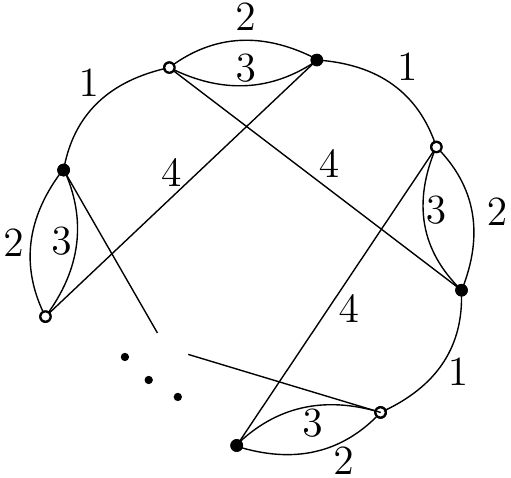}
	\caption{\label{fig:Examples2Cyclic} Two examples of 2-cyclic bubbles, corresponding to $(\sigma_\circ, \sigma_\bullet) = (\id, \id)$ and $(\sigma_\circ, \sigma_\bullet) = (\id, \Delta_1)$.}
\end{figure}

\subsection{The 2 genera} Let $\graph\in \cG^{\max}(\vec{\bb})$ with $\vec{\bb}\in \mathcal{C}_n^\circ$. We will here express the total number of bicolored cycles $C_0(\graph) = \sum_{c=1}^4 C_{0c}(\graph)$, in terms of topological quantities. Recall that $\graph$ is defined by adding edges of color 0 between black and white vertices, i.e. a pairing $\pi$. We first represent the tricolored graph with colors $\{0,1,2\}$ as a map $\m_{12}$, using the canonical embedding of \Cref{thm:CanonicalEmbedding}, Chapter \ref{sec:Definitions}. Since $\vec{\bb}$ has a single cycle with colors $\{1,2\}$, the map $\m_{12}$ is connected and its genus is given by Equation \eqref{EulerColoredGraph},
\begin{equation}
	2-2g(\m_{12}) = C_{01}+C_{02} + 1 - n.
\end{equation}
One can follow that reasoning for the tricolored graphs with colors $\{0,3,4\}$ also. One gets a map $\m_{34}$ with a face made of the cycle with colors $\{3,4\}$, and whose other faces are the bicolored cycles with colors $\{0,3\}$ and $\{3,4\}$. From Euler's relation \eqref{EulerColoredGraph}, one has $2-2g(\m_{34}) = C_{03}+C_{04} + 1 - n$, and therefore
\begin{equation}
	C_0(\graph) = 2 - 2g(\m_{12}) - 2g(\m_{34}) + 2n.
\end{equation}
Our question of maximizing $C_0(\graph)$ thus reduces to finding the pairings such that both $\m_{12}$ and $\m_{34}$ are planar, if any. We define
\begin{equation}
	\cG^{\max}(\vec{\bb}_{\sigma_\circ, \sigma_\bullet}) = \{\pi \in \cG(\vec{\bb}_{\sigma_\circ, \sigma_\bullet}) ; g(\m_{12}) = g(\m_{34}) = 0\}
\end{equation}
This slightly contradicts our general definition of $\cG^{\max}(\vec{\bb})$ which by definition can never be empty. Here $\cG^{\max}(\vec{\bb}_{\sigma_\circ, \sigma_\bullet})$ can be empty and \emph{we will be working with this convention for the rest of this chapter}. In that case, there are still pairings which maximize the number of bicolored cycles; they just have at least one non-zero genus. In fact, as it clearly appears below, almost all those sets are empty. Unfortunately, we have not been able to characterize them.

\subsection{Meander systems} %The classical informal picture of a meander is as follows. Consider a river, oriented from west to east, with $2n$ bridges. A meander is a closed, self-avoiding road which crosses all the bridges. A meander system with $k$ roads is a set of $k$ non-intersecting meanders.
\begin{definition}
	A \emph{meander} of order $n$ is a closed, planar, self-avoiding curve (the road) which crosses an infinite oriented horizontal line (the river) exactly $2n$ times (the bridges), up to isotopy. The number of meanders of order $n$ is denoted $M_n$.
	
	A \emph{meander system} of order $n$ with $k$ components is a set of $k$ non-crossing meanders all intersecting the same horizontal line, exactly $2n$ times in total. The number of meander systems with $k$ components and of order $n$ is denoted $M_n^{(k)}$.
\end{definition}
Another representation of the problem of counting the number of meanders is the problem of calculating the entropy associated to compact foldings of a polymer on the plane \cite{ArchStat, GeometricallyConstrainedSystems}.

We will use the classical representation, where the river is oriented from west to east, has $2n$ marked vertices labeled $1_\bullet, 1_\circ, \dotsc, n_\bullet, n_\circ$, and the segments of the roads, above and under the river, are represented as semi-circular arches (caps and cups) whose feet are the vertices. An example is provided in the Figure \ref{fig:MeandricSysExample}. 

An \emph{arch configuration} is a set of semi-circular arches between white and black vertices, all in the upper half-plane or all in the lower half-plane, and it is planar if the arches do not cross. An arch configuration is thus a pairing $\pi$ between the black and white vertices and a pairing is said to be planar if its arch configuration is. A meander system therefore consists of an upper pairing $\pi_N$ and a lower pairing $\pi_S$ both of which planar. We will denote $\Pl\mathfrak{S}_n$ the set of planar pairings.
\begin{figure}
	\includegraphics[scale=.7,valign=c]{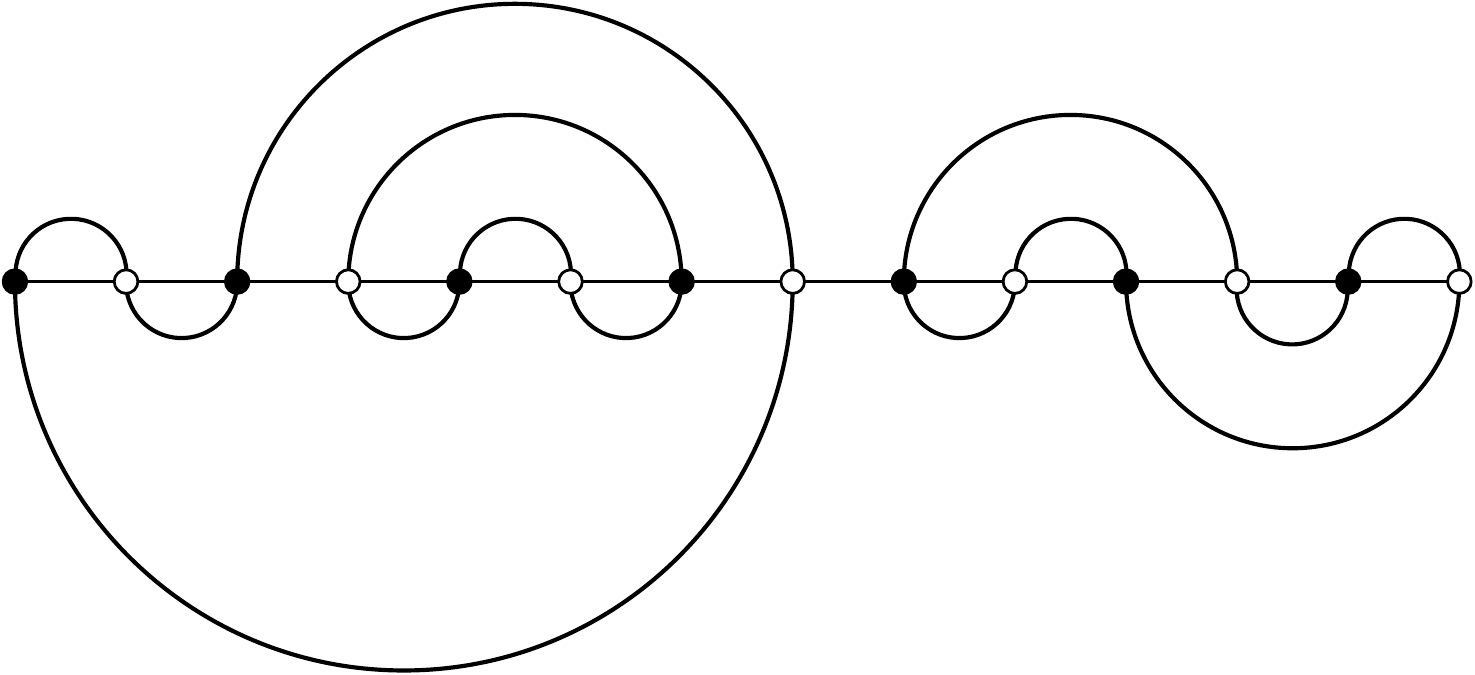}
	\caption{\label{fig:MeandricSysExample} A meander system with three components.}
\end{figure}

\subsection{Pairings of 2-cyclic bubbles and meander systems} Maximizing the number of bicolored cycles $C_0(\graph)$ is equivalent to having $\m_{12}$ and $\m_{34}$ planar at the same time, where $\graph$ is determined by a pairing $\pi$ mapping its white vertices to its black vertices. At fixed $\graph$, we can regoup $\m_{12}$ and $\m_{34}$ into a single meander system. 

We use the canonical embedding of $\m_{12}$ and represent the cyclic order of colors $(012)$ of \Cref{thm:CanonicalEmbedding} in the counter-clockwise direction. We draw the circle with colors $\{1,2\}$ in the plane, and the edges of color 0 on the outside region as follows in the picture. Equivalently we can cut the circle open between $n_\circ$ and $1_\bullet$ and have the pairing $\pi$ as an upper arch configuration. In addition, the colors are not needed since they can be recovered. This gives (representing a single edge of color 0)
\begin{equation}
	\m_{12} = \includegraphics[scale=.7,valign=c]{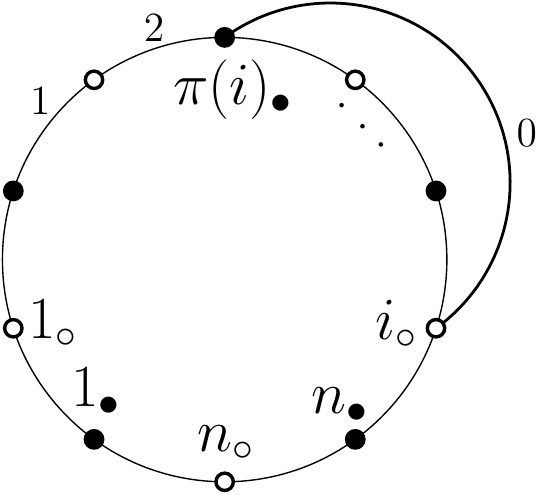} \simeq \includegraphics[scale=.7,valign=c]{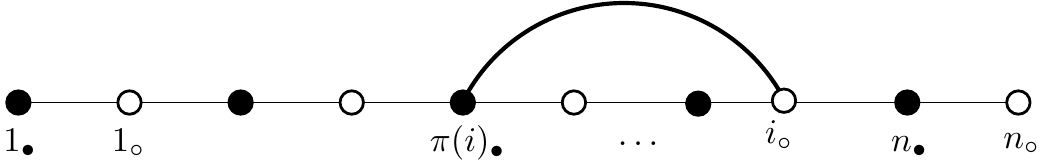}
\end{equation}

We do the same with the map $\m_{34}$. We draw the circle with colors $\{3,4\}$ in the plane, with vertices labeled $1'_\bullet, 1'_\circ, \dotsc, n'_\bullet, n'_\circ$. This time, we draw the edges of color 0 on the inside region of the circle. To do so, we must find the action of the pairing $\pi$ on primed labels, i.e. to which primed label of a black vertex is the primed label of a white vertex connected to by an edge of color 0? A moment of reflexion shows that it is $\sigma_\bullet^{-1}\circ \pi\circ \sigma_\circ$. Then we can open up the circle between $1'_\bullet$ and $n'_\circ$ and represent the edges of color 0 as a lower arch configuration, and forget about the colors,
\begin{equation}
	\m_{34} = \includegraphics[scale=.7,valign=c]{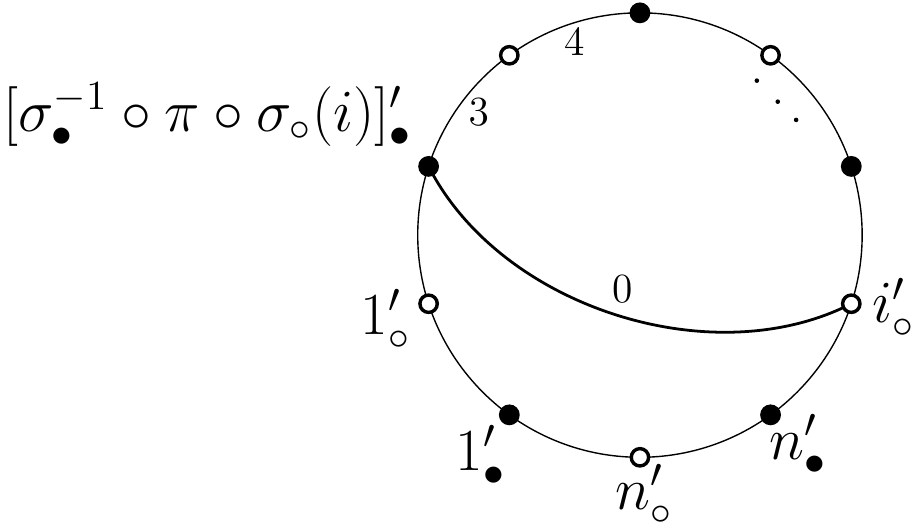} \simeq \includegraphics[scale=.7,valign=c]{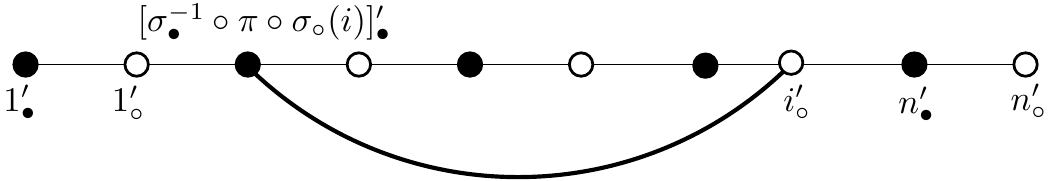}
\end{equation}

We have a representation of both $\m_{12}$ and $\m_{34}$ as an upper and a lower arch configuration, whose genera are $g(\m_{12})$ and $g(\m_{34})$. By forgetting the primes on $\m_{34}$, we identify the two horizontal lines and obtain a single object which has an upper and a lower arch configuration,
\begin{equation}
	\graph \simeq \includegraphics[scale=.7,valign=c]{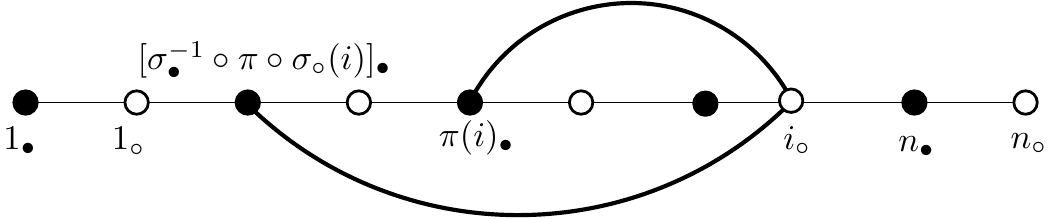}
\end{equation}
Therefore we have proved the following.
\begin{theorem}{}{MeanderSystems}
	Let $\mathcal{M}_{\sigma_\circ, \sigma_\bullet}$ be the set of meander systems such that if $\pi\in\Pl\mathfrak{S}_n$ is the upper arch configuration, then $\sigma_\bullet^{-1}\circ \pi\circ \sigma_\circ$ is planar too and represents the lower arch configuration. Then
	\begin{equation}
		|\cG^{\max}(\vec{\bb}_{\sigma_\circ, \sigma_\bullet})| = |\mathcal{M}_{\sigma_\circ, \sigma_\bullet}|
	\end{equation}	
\end{theorem}
Moreover, if $\sigma_\circ\in\mathfrak{S}_{n-1}$ is fixed and $\sigma_\bullet$ covers $\mathfrak{S}_n$, then there is a one-to-one correspondence between meander systems of order $n$ and elements of $\{\mathcal{M}_{\sigma_\circ, \sigma_\bullet}\}_{\sigma_\bullet\in\mathfrak{S}_n}$.

We notice in \cite{Meanders} that more can be said when $\sigma_\circ=\id$. In that case, the above correspondence is finer, since it becomes possible to track the components of the meander systems as the cycles of $\sigma_\bullet$.
\begin{proposition}{}{MeanderComponents}
	If $\sigma\in\mathfrak{S}_n$, denote $\ell(\sigma)$ its number of cycles. Then
	\begin{equation}
		\sum_{\substack{\sigma_\bullet\in\mathfrak{S}_n\\ \ell(\sigma_\bullet)=k}} |\cG^{\max}(\vec{\bb}_{\id, \sigma_\bullet})| = M_n^{(k)}.
	\end{equation}
\end{proposition}

\section{2-cyclic bubbles and factorization on stabilized-interval-free permutations}
In this section we set $\sigma_\circ=\id$ and denote $\vec{\bb}_{\sigma} \equiv \vec{\bb}_{\id, \sigma}$.

\subsection{Decomposition on stabilized-interval-free (SIF) permutations} 
\begin{definition}
	We say that $\sigma\in\mathfrak{S}_n$ is SIF if it does not stabilize any proper interval of $[1..n]$, i.e.
	\begin{equation}
		\forall\ a\leq b \in [1..n] \qquad \sigma([a..b]) \neq [a..b],
	\end{equation}
	except $[a..b]=[1..n]$. We denote $\SIF_n\subset \mathfrak{S}_n$ the set of SIF permutations.
\end{definition}

A closely related notion is that of connected (a.k.a. indecomposable) permutations. We say that $\sigma\in\mathfrak{S}_n$ is connected if 
\begin{equation}
	\forall i\in[2..n-1] \qquad \sigma([1..i])\neq [1..i].
\end{equation}
Clearly every permutation has a unique decomposition into connected blocks, i.e. there exists $1=i_1< i_2<\dotsb< i_p<i_{p+1}=n$ and connected permutations $\sigma_1, \dotsc, \sigma_p$, with $\sigma_j\in\mathfrak{S}_{i_{j+1}-i_j+1}$ defined by
\begin{equation}
	\forall k\in[1..i_{j+1}-i_j]\qquad \sigma_j(k) = \sigma(k+i_j-1)
\end{equation}
for $j=1,\dots, p$. They can be found by identifying the connected block which contains 1 and goes to $i_2-1$, and so on, as illustrated in Figure \ref{fig:ConnectedBlocks}.
\begin{figure}
	\includegraphics[scale=.5,valign=c]{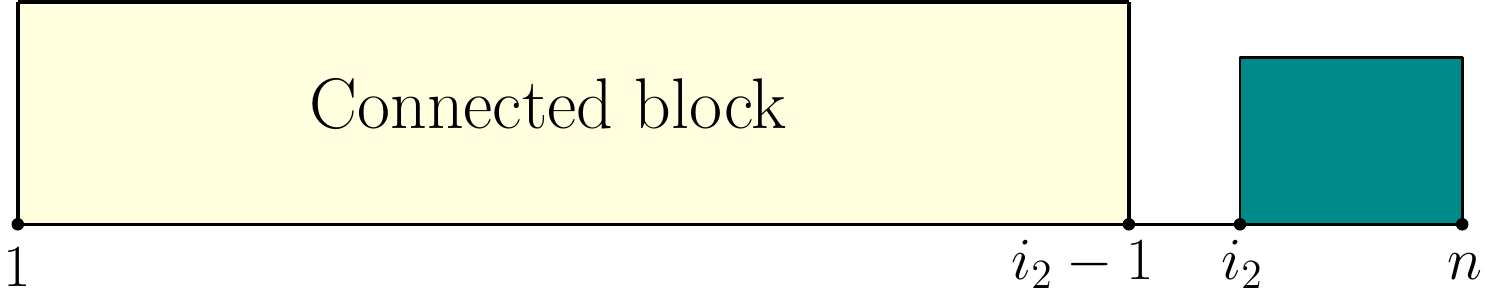} $\to$ \includegraphics[scale=.5,valign=c]{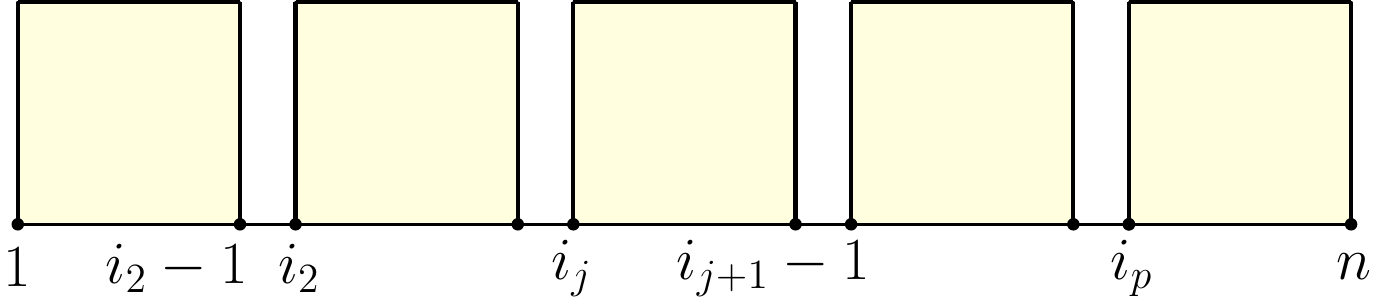}
	\caption{\label{fig:ConnectedBlocks} On the left we find the connected block on $[1..i_2-1]$, and the block on $[i_2..n]$ can be anything. On the right, we have proceeded recursively to get the full decomposition onto connected blocks.}
\end{figure}

We give two ways of decomposing a permutation into SIF blocks.
\begin{enumerate}
	\item For $\sigma\in\mathfrak{S}_n$ there is a unique SIF block which acts on 1. It can be found by considering the first connected block $\sigma_1$, and removing from $[i_1..i_2-1]$ the maximal, proper intervals (if any) stabilized by $\sigma_1$. The action of $\sigma_1$ on the remaining elements $\{j_1, \dotsc, j_m\}\subset [i_1..i_2-1]$ is SIF. Then $\sigma$ can be recovered by adding permutations $\rho_1, \dots, \rho_{m_1}$ between $j_l$ and $j_{l+1}$ for $l=1, \dotsc, m-1$, and a permutation $\rho_m$ on $[i_2..n]$, such that the sizes add up to $n$. Those permutations can be arbitrary. This decomposition is illustrated in Figure \ref{fig:SIFBlocks}. The vertical lines connected by a horizontal bar form a SIF permutation. This leads to the (unique) decomposition of a permutation on SIF blocks \cite{Callan04}. The latter are irreducible blocks in the sense of Beissinger \cite{Beissinger1985}. This translates directly into an equation on the formal generating series (see \cite{Beissinger1985})\footnote{ It can be used to extract the asymptotic series. We thank O. Bodini for pointing out theorems from Bender related to such asymptotics, and for discussions on that topic. Let us simply mention that the number of connected permutations is asymptotically $n!$, so permutations are almost surely connected, and the number of SIF permutations goes like $n!/e$.}
	\begin{equation} \label{SIFEquation}
		S(t)\coloneqq \sum_{n\geq 0} n! t^n = I(tS(t)),
	\end{equation}
	where $I(t) = \sum_{n\geq 0} I_n t^n$, $I_n$ being the number of SIF permutations of size $n$. Remarkably, we will re-encounter this equation several times in this manuscript, for irreducible meander systems and non-seperable planar maps playing the roles of the irreducible objects.
	\begin{figure}
		\includegraphics[scale=.4,valign=c]{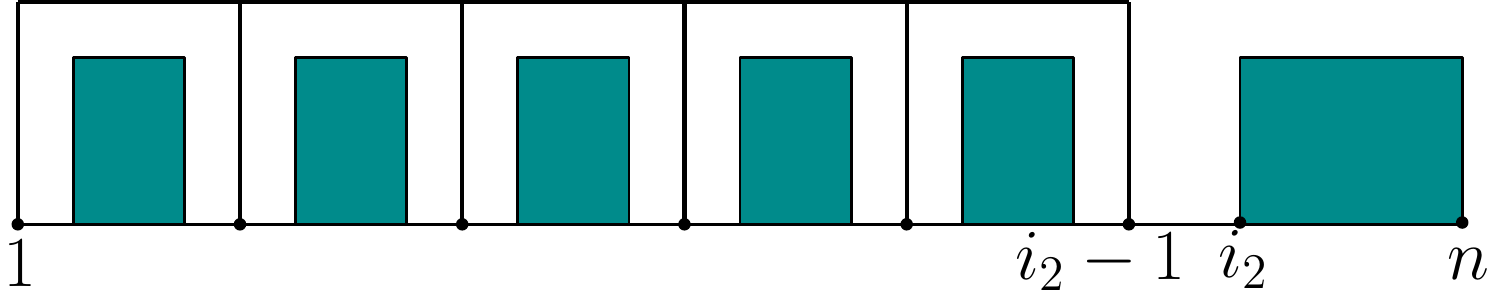}% $\to$ \includegraphics[scale=.4,valign=c]{SIFBlocksFullDecomposition.pdf}
		\caption{\label{fig:SIFBlocks} A decomposition of a permutation on $[1..n]$ as a SIF block containing 1, and arbitrary permutations between its elements, and a final one. One can proceed recursively to get the full decomposition onto SIF blocks.}
	\end{figure}
	\item We give another algorithm to find the SIF decomposition, which we will be used later. Denote $\Delta_k^{(n)}:i\in[1..n]\mapsto i+k\mod n$ the cyclic shift of $k$ on $[1..n]$. Another way to find SIF blocks is as connected blocks up to cyclic shifts. The algorithm starts with the connected block decomposition $(\sigma_1, \dotsc, \sigma_p)$ of $\sigma$. Then at every step of the algorithm:
	\begin{itemize}
		\item The input is a connected permutation $\sigma_{J}$ where $J=(j_1, \dotsc, j_m)$ (denotes a path from the root in a rooted plane tree, see below) acting on $[1..n_J]$ for some $n_J\leq n$.
		\item Test whether the cyclically-shifted permutation $\Delta_{-k}^{(n_J)} \circ \sigma_J\circ \Delta_k^{(n_J)}$ is also connected, by increasing values of $k$. 
		\item If $\Delta_{-k}^{(n_J)} \circ \sigma_J\circ \Delta_k^{(n_J)}$ is connected for all $k=1, \dotsc, n_J$, then $\sigma_J$ is SIF and it is stored as such.
		\item If a non-connected permutation is obtained, one performs its connected block decomposition, into say $p_J$ blocks, as shown in Figure \ref{fig:ConnectedSIF}. It gives rise to connected permutations $\sigma_{J_1}, \dotsc, \sigma_{J_{p_J}}$ where $J_l = (j_1, \dotsc, j_m, l)$ for $l=1, \dotsc, p_J$. They will then be considered as inputs for new steps of the algorithm.
	\end{itemize}
	Since the connected block decomposition reduces the size of the permutations, the algorithm stops after a finite number of steps and all final blocks are SIF. It can obviously be schematically pictured as a rooted plane tree as shown in Figure \ref{fig:SIFAlgorithm}. It also echoes the invariance of $\cG^{\max}(\vec{\bb}_\sigma)$ under the action of $\mathbb{Z}_n$.
	\begin{figure}
		\includegraphics[scale=.6,valign=c]{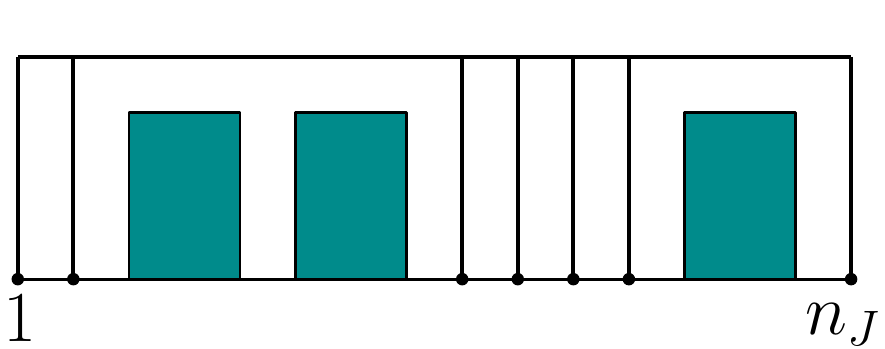} $\quad\overset{\Delta^{(n_J)}_{-2}\circ\sigma_1\circ \Delta^{(n_J)}_{2}}{\to}\qquad$ \includegraphics[scale=.6,valign=c]{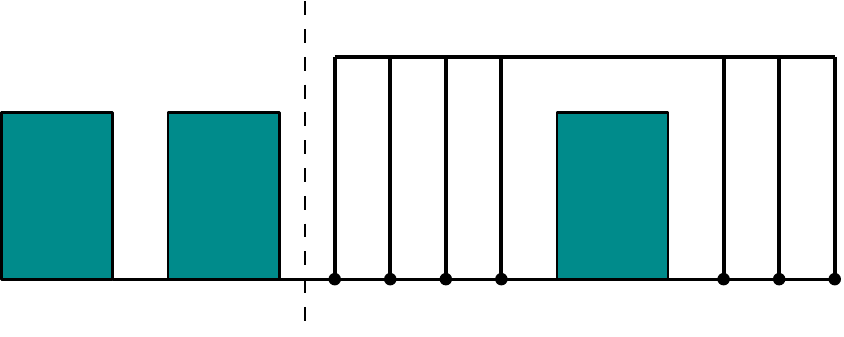}
		\caption{\label{fig:ConnectedSIF} On the left we have a connected block on $[1..n_J]$. After conjugation by $\Delta^{(n_J)}_2$, a non-connected permutation is found, and one writes its connected block decomposition, then applies this procedure recursively on all connected blocks until they are all SIF.}
	\end{figure}
	\begin{figure}
		\includegraphics[scale=.45]{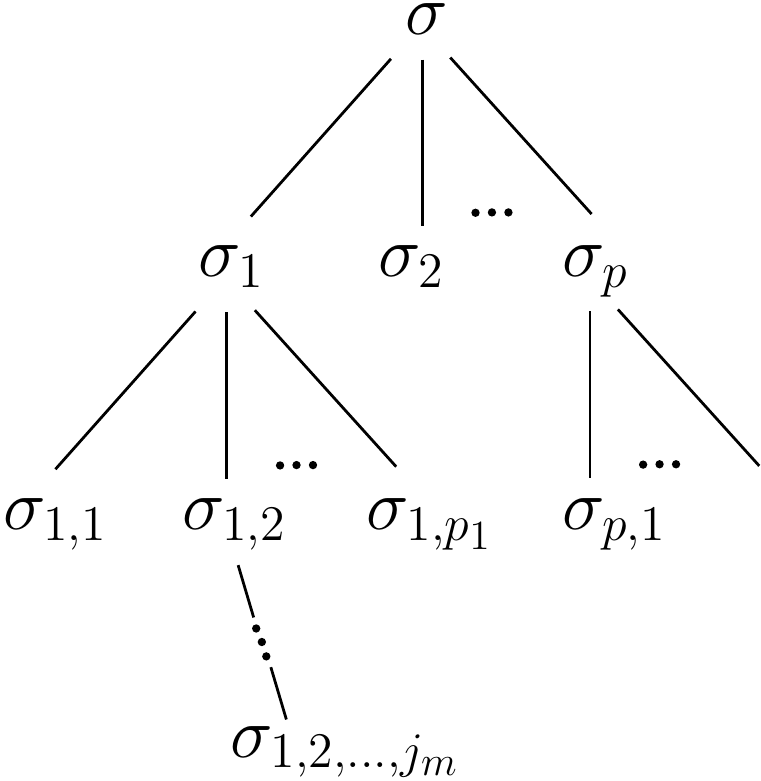}
		\caption{\label{fig:SIFAlgorithm} The algorithm we propose to find the SIF decomposition of $\sigma$ has the structure of a rooted plane tree. At every permutation, the conjugation by a shift of increasing value is used to find a non-connected permutation and its children are those connected blocks.}
	\end{figure}
\end{enumerate}
% It implies that if we find a decomposition of $\cG^{\max}(\vec{\bb}_\sigma)$ onto $\cG^{\max}(\vec{\bb}_{\sigma_1}), \cG^{\max}(\vec{\bb}_{\sigma_2}), \dotsc$ where the connected blocks, it automatically will be onto SIF permutations, thanks to our second algorithm above.

\subsection{2-cyclic SIF bubbles} If $\sigma\in\mathfrak{S}_n$, denote $\kappa_1, \dotsc, \kappa_p$ its SIF blocks (ordered e.g. by increasing order of their smallest elements), such that $\kappa_i$ acts on $\{j^{(i)}_1, \dotsc, j^{(i)}_{\ell_i}\}$ and in particular has size $\ell_i$, for $i=1, \dotsc, p$. Denote $\tilde{\kappa}_i$ the version of $\kappa_i$ ``rescaled'' so that its acts on $[1..\ell_i]$ (its $m$-th element is shifted by $m-j^{(i)}_m$) and is thus an element of $\mathfrak{S}_{\ell_i}$. Then
\begin{proposition}{}{}
	$\vec{B}_\sigma$ can be obtained as the boundary bubble of a tree on the bubbles $\vec{\bb}_{\tilde{\kappa}_i}$ for $i=1, \dotsc, p$.
\end{proposition}
Here a tree means that all edges of color 0 are cut-edges.
\begin{proof}
	Assume that $\sigma$ stabilizes the interval $[i..i+k]$ for $k>0$, then $\vec{\bb}_\sigma$ has a 4-edge-cut with an edge of every color as follows,
	\begin{equation}
		\vec{\bb}_\sigma = \includegraphics[scale=.4,valign=c]{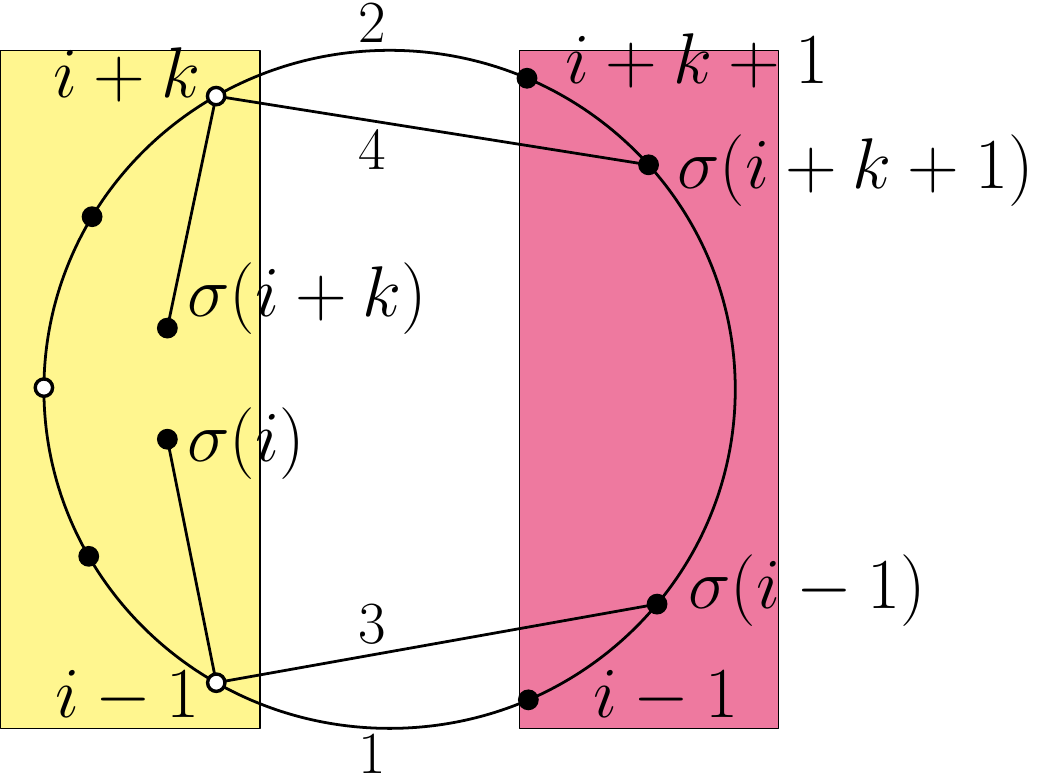}
	\end{equation}
	where the edges of color 1 and 3 (respectively 2 and 4) on the left hand side are incident to the white vertex of label $i-1$ (respectively $i+k$). It is a boundary bubble as follows
	\begin{equation}
		\vec{\bb}_\sigma = \partial \includegraphics[scale=.4,valign=c]{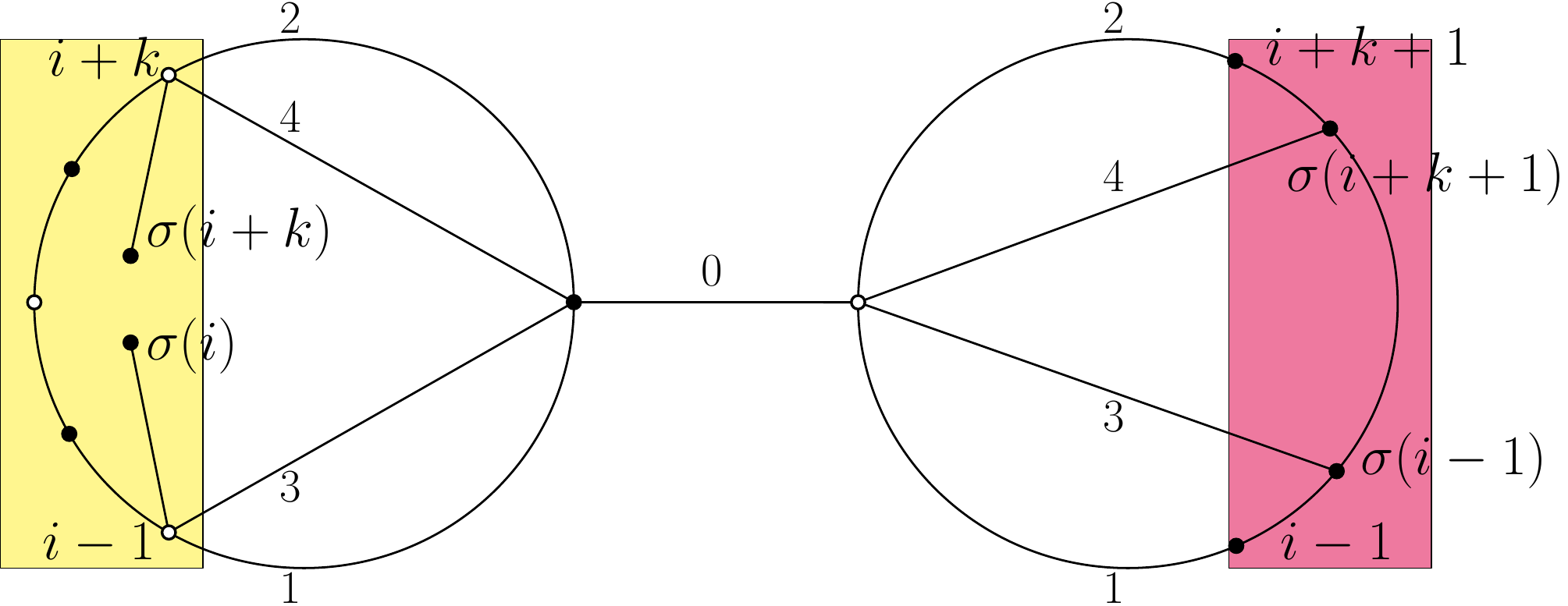}
	\end{equation}
	and this produces recursively a tree of the bubbles $\vec{\bb}_{\tilde{\kappa}_i}$ for $i=1, \dotsc, p$.
\end{proof}

\subsection{Factorization onto SIF permutations and their two-type trees} The SIF decomposition of $\sigma\in\mathfrak{S}_n$ gives rise to a non-crossing partition on $n$ elements. Due to the cyclic symmetry, it is convenient to represent non-crossing partitions on the circle with $n$ vertices. A \emph{non-crossing partition} is a set of convex polygons inside the circle on the $n$ vertices, which do not cross. Each polygon is colored ($1$-gons are vertices and $2$-gons are edges) and they are separated by connected non-colored regions. An example is given in Figure \ref{fig:NCP}.
\begin{figure}
	\includegraphics[scale=.7,valign=c]{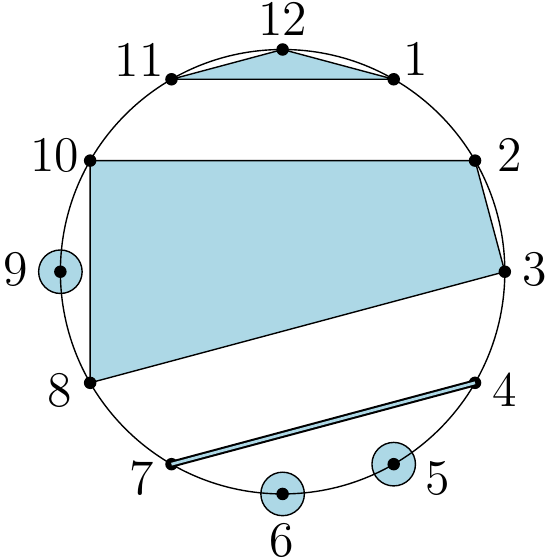} \hspace{1cm} \includegraphics[scale=.7,valign=c]{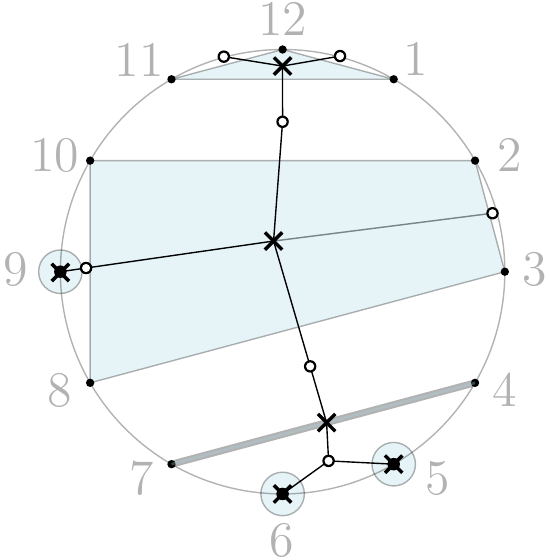}
	\caption{\label{fig:NCP}On the left is a non-crossing partition on 12 labeled elements, with blue-colored polygons. On the right is its dual two-type tree.}
\end{figure}

If $\kappa_1, \dotsc, \kappa_p$ are the SIF blocks of $\sigma\in\mathfrak{S}_n$, such that $\kappa_i$ has support $S_i \coloneqq\{j^{(i)}_1, \dotsc, j^{(i)}_{\ell_i}\}\subset [1..n]$, then one draws a convex polygon for each $S_i$. This determines a unique non-crossing partition.

We consider the \emph{two-type tree} $t(\sigma)$ dual to the non-crossing partition associated to $\sigma$. It has cross-vertices ``$\times$'' dual to the polygons and white vertices ``$\circ$'' dual to non-colored regions, and edges between them when a region is adjacent to a polygon, see Figure \ref{fig:NCP}.
\begin{theorem}{}{SIFFactorization}\cite{Meanders}
	\begin{equation}
		|\cG^{\max}(\vec{\bb}_\sigma)| = \prod_{\circ\in\tree(\sigma)} \operatorname{Cat}_{\operatorname{deg}(\circ)}\ \prod_{\times\in\tree(\sigma)} |\cG^{\max}(\vec{\bb}_{\sigma(\times)})|
	\end{equation}
\end{theorem}
Before sketching the proof, let us translate our algorithm which finds the SIF block decomposition of $\sigma\in\mathfrak{S}_n$ in the tree $t(\sigma)$. Notice that every arc between $i$ and $i+1$ in the non-crossing partition corresponds to a corner at a white vertex of $t(\sigma)$. Consider the corner corresponding to the arc $(n,1)$ and the white vertex incident to that corner.
\begin{itemize}
	\item If it is not a leaf, then $n$ and $1$ are in different stabilized interval, meaning that $\sigma$ is not connected. Its connected blocks correspond to the subtrees attached to that white vertex. We separate them by replacing the white vertex with a leaf for each subtree, thereby obtaining trees $t_1, \dotsc, t_p$ which are the trees $(t(\sigma_j))_{j=1, \dotsc, p}$ of the connected blocks $\sigma_1, \dotsc, \sigma_p$ of $\sigma$.
	\item If the white vertex is a leaf, then $\sigma$ is connected. The action of the shift $\Delta_k^{(n)}$ corresponds to moving on to the next arcs $(k, k+1)$ by increasing values of $k$ until the corresponding corner sits at a white vertex which is not a leaf. Then one performs the previous step.
\end{itemize}
Those steps are then repeated on each subtree. The algorithm eventually splits $t(\sigma)$ at every white vertex which is not a leaf, until they all are. The SIF blocks are then identified as the cross-vertices.

\begin{proof}
	The key equation to prove is the following for the factorization on connected blocks,
	\begin{equation} \label{ConnectedFactorization}
		|\cG^{\max}(\vec{\bb}_\sigma)| = \operatorname{Cat}_p \prod_{j=1}^p |\cG^{\max}(\vec{\bb}_{\sigma_j})|.
	\end{equation}
	Assuming it holds, one can then apply the above algorithm to each $\sigma_j$, using the cyclic invariance of $\cG^{\max}(\vec{\bb}_{\sigma_j})$, until one has only SIF blocks. This produces a Catalan number every time one uses a decomposition into connected blocks, so at every white vertex of $t(\sigma)$.
	
	To prove \eqref{ConnectedFactorization}, we use the representation of $\cG^{\max}(\vec{\bb}_\sigma)$ as meander systems, as well as several lemmas. We recall that $\mathcal{M}_\sigma$ is the set of meander systems such that if $\pi\in\mathfrak{S}_n$ encodes the upper arch configuration (going from $i_\circ$ to $\pi(i)_\bullet$), then $\sigma^{-1}\circ\pi$ encodes the lower arch configuration. The first lemma is that if $\sigma$ has a decomposition $\sigma_1, \dotsc, \sigma_p$ into connected blocks, then there is an injective map
	\begin{equation}
		\left(\mathcal{M}_{\sigma_1} \times \dotsm \times \mathcal{M}_{\sigma_p}\right) \times \Pl\mathfrak{S}_p \hookrightarrow \mathcal{M}_\sigma
	\end{equation}
	which we explain ($\Pl\mathfrak{S}_n$ is the set of planar pairings). Obviously, given meander systems from $\mathcal{M}_{\sigma_1}, \dotsc, \mathcal{M}_{\sigma_p}$, we can concatenate them to obtain a meander system from $\mathcal{M}_\sigma$, see Figure \ref{fig:ConnBlocks}. Let us explain the additional action of $\Pl\mathfrak{S}_p$. Given $\rho\in\Pl\mathfrak{S}_p$, we can cut the upper and lower arches which touch the last white vertex of each region, and re-arrange them using $\rho$. This always produces a new meander system (i.e. no arch crossings). This is detailed in Figure \ref{fig:NewMeander}.
	
	The next lemma shows that all meander systems are of this type: given the decomposition in connected blocks, only the last white vertex of each region can be joined to other regions with upper and lower arches. We do not prove this here (see \cite{Meanders}). It relies on the inversion property of connected permutations (obviously): If $\sigma\in\mathfrak{S}_n$ is connected, then 
	\begin{equation}
		\forall k\in[1,n-1]\qquad \exists\ l>k, \quad \sigma(l)\leq k.
	\end{equation}
\end{proof}
\begin{figure}
	\subfloat[A collection of five meander systems ordered and concatenated. We only draw explicitly the arches with the last white vertex of each block as a foot.]{\includegraphics[scale=.6]{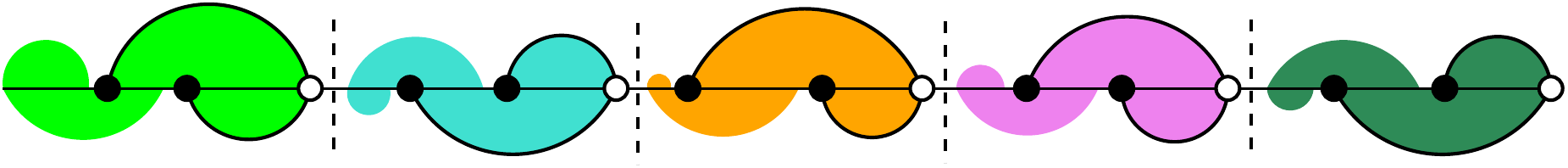}\label{fig:ConnBlocks}}\\
	\subfloat[A planar permutation $\rho$ on the set of blocks of the above meander system. We can consider the white vertices to correspond to the last white vertex of each block.]{\includegraphics[scale=.6]{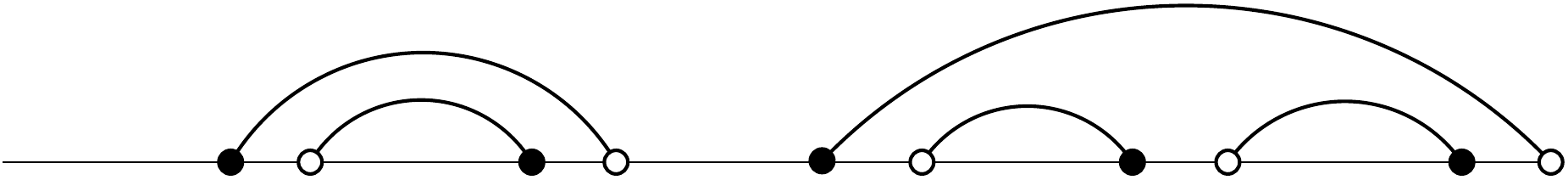}\label{fig:PlanarPerm}}\\
	\subfloat[We use $\rho$ to re-arrange the arches touching the last white vertex of each block. Clearly the use of a \emph{planar} permutation prevents the new arches to cross each other. Since they only have the last white vertex of each region as white foot, they do not intersect the arches contained in each region.]{\includegraphics[scale=.6]{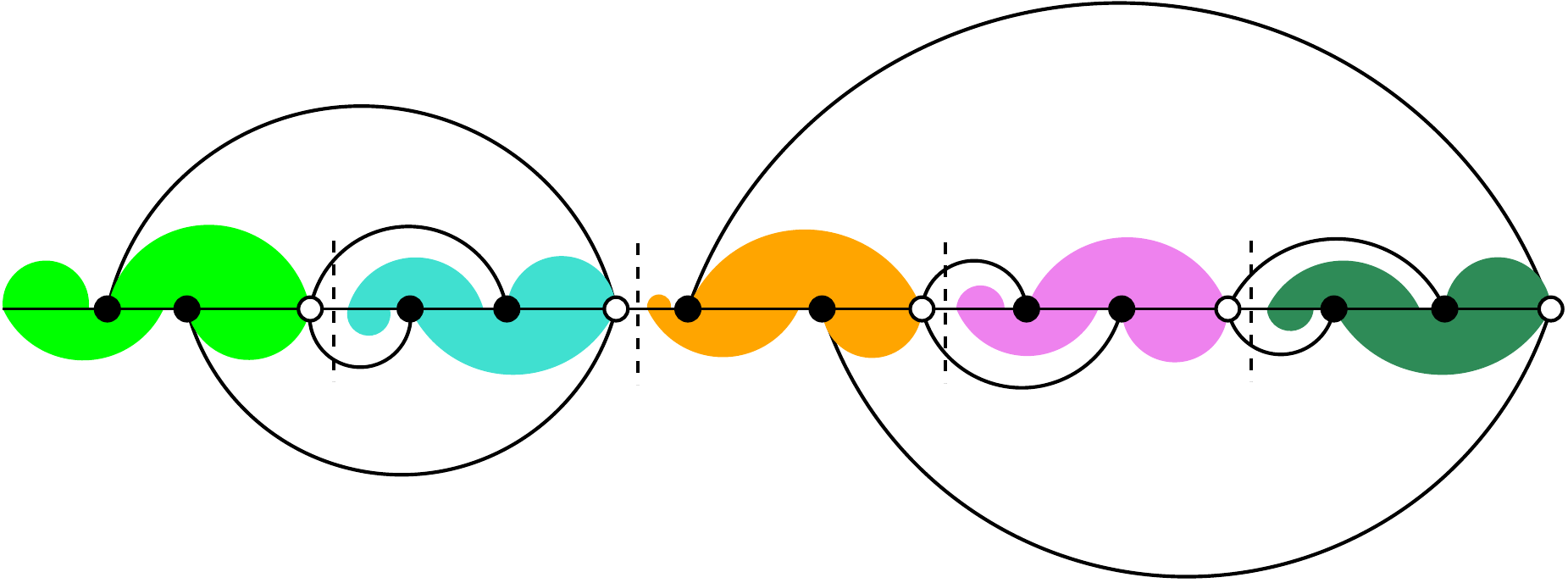}\label{fig:NewMeander}}
	\caption{\label{fig:Injectivity} From a set of meander systems and a planar permutation on them, we get a new meander system.}
\end{figure}

Since $|\cG^{\max}(\vec{\bb}_{\sigma})|$ counts meander systems as seen in \Cref{thm:MeanderSystems}, and we have been able to factorize it onto its SIF blocks, it is natural to ask whether taking $\sigma$ to be SIF restricts the meander systems which are counted to a remarkable family. This is indeed the case.
\begin{definition} \label{def:IrredMeanders}
	We say that a meander system is 1-reducible if a single cut on the horizontal line can produce two disconnected systems, and we say that it is 1-irreducible otherwise. A meander system is said to be 2-reducible if it becomes disconnected after exactly two cuts of the horizontal line, and 2-irreducible otherwise.
\end{definition}

Those notions were introduced in \cite{LandoZvonkin1993} (see also \cite{ArchStat}). A 2-reducible meander system thus has the following structure
\begin{equation}
	\includegraphics[scale=.25,valign=c]{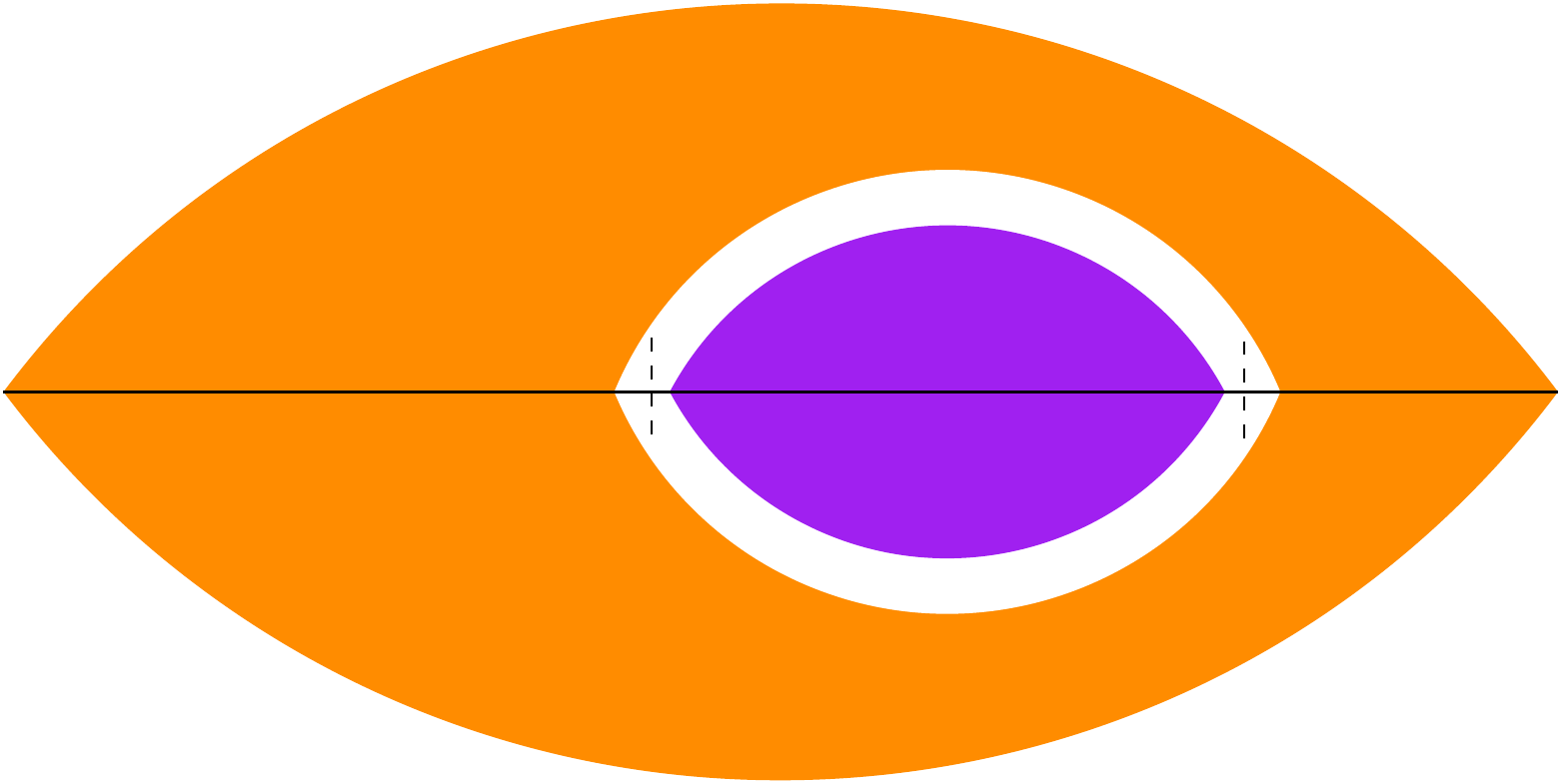}
\end{equation}
i.e. it has a sub-system totally restricted to a connected inner region and another sub-system which avoids entirely this region. The two vertical dashed lines indicate the two cuts which disconnect the two meander sub-systems. \Cref{thm:MeanderSystems} has the following refinement.
\begin{proposition}{}{}\cite{Meanders}
	Let $\SIF_n$ be the set of SIF permutations on $[1..n]$. The set $\{\mathcal{M}_\sigma\}_{\sigma\in\SIF_n}$ is the set of 2-irreducible meander systems.
\end{proposition}
Meander systems can be decomposed into 2-irreducible blocks using the same two algorithms we gave for the SIF decomposition of permutations. For instance, there is a unique 2-irreducible block which goes through $1_\bullet$ and a meander system can be reconstructed from this unique 2-irreducible block by adding arbitrary meander systems between its vertices and a final one, see Figure \ref{fig:2IrreducibleMeanderDecomposition}. If $m(t) \coloneqq \sum_{n\geq 0} \operatorname{Cat}_n^2 t^n$ is the generating series of all meander systems, then
\begin{equation} \label{2IrreducibleDecomposition}
	m(t) = I(t m(t)^2),
\end{equation}
where $I(t) = \sum_{n\geq 0} I_n t^n$, $I_n$ being the number of 2-irreducible meander systems on $2n$ vertices.
\begin{figure}
	\includegraphics[scale=.4]{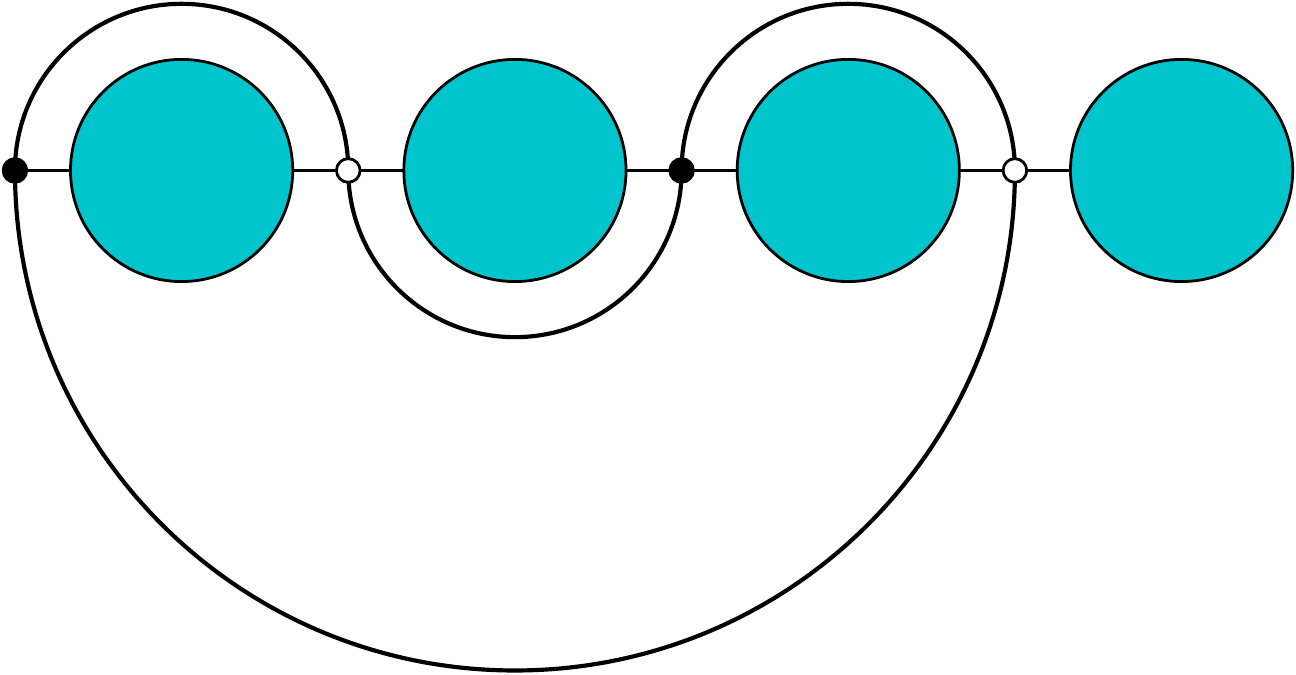}
	\caption{\label{fig:2IrreducibleMeanderDecomposition} An arbitrary meander system can be decomposed via its 2-irreducible meander going through $1_\bullet$ and insertions of other arbitrary systems inbetween the vertices. This is the same 2-irreducible meander going through $1_\bullet$ as in Figure \ref{fig:MeandricSysExample}.}
\end{figure}

\section{Examples of counting}
\subsection{Examples} We recall $\Delta_l^{(n)}: i\in[1..n]\mapsto i+l\mod n$ and consider $\vec{B}_{\Delta_l^{(n)}}$ for arbitrary $l$ and $n$.% In this section all permutations have size $n$ so we drop $n$ from the notation $\Delta_k\equiv \Delta_k^{(n)}$.
\begin{proposition}{}{SIFVEV}
	\begin{subnumcases}
		{|\cG^{\max}(\vec{\bb}_{\Delta_l^{(n)}})| =} \operatorname{Cat}_n &\qquad \text{for $l=0\mod n$,}\label{SIF0}\\
		\operatorname{Motzkin}_n &\qquad  \text{for $l=\pm1\mod n$,}\label{SIF1}\\
		n &\qquad \text{else.} \label{SIFelse}
	\end{subnumcases}
	where $\operatorname{Cat}_n$ and $\operatorname{Motzkin}_n$ are the Catalan and Motzkin numbers.
\end{proposition}

Notice that when $l$ and $n$ are coprime, then $\Delta_l^{(n)}$ has a single cycle, and then $|\cG^{\max}(\vec{\bb}_{\Delta_l^{(n)}})|$ counts a number of meanders (as opposed to generic meander systems).

\begin{proof}
	Equation \eqref{SIF0} corresponds to $t(\sigma)$ a tree with a one white vertex of degree $n$ connected to $n$ cross-vertices of degree 2 and $n$ leaves. Then one concludes with \Cref{thm:SIFFactorization}. Equivalently, $\sigma=\id$ so in terms of meander systems, the lower arch is the same as the upper arch, and they are $\operatorname{Cat}_n$ planar configurations.
	
	Let us prove Equation \eqref{SIF1} for $l=-1$. The relevant meanders are such that there is an upper arch connecting $i_\circ$ to $j_\bullet$ if and only if there is a lower arch connecting $i_\circ$ to $(j+1)_\bullet$. We aim at a recursion on $n$ and for the time of the proof we switch to the notation $|\cG^{\max}(\vec{\bb}_{\Delta_{-1}^{(n)}})| = m_n$.
	
	Let $k\in[2..n]$ and denote $m_{n,k}$ the number of contributing meanders with an upper arch between $1_\circ$ and $k_\bullet$. They also have a lower arch connecting $1_\circ$ to $(k+1)_\bullet$,
	\begin{equation}
		\includegraphics[scale=.65,valign=c]{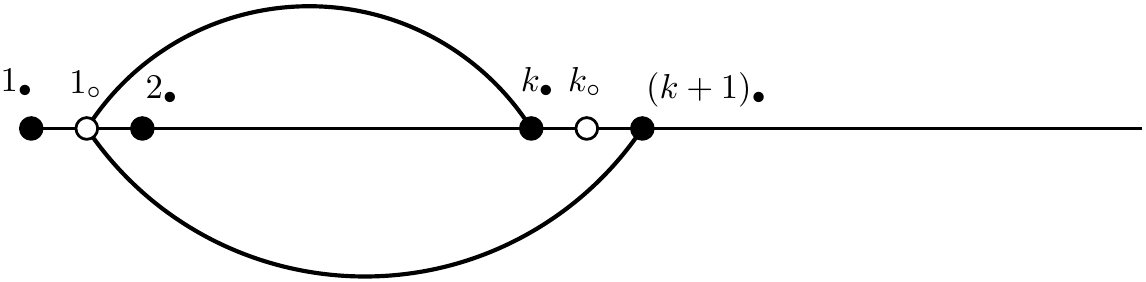}
	\end{equation}
	It can then be shown that $k_\circ$ is necessarily connected to $1_\bullet$ via an upper arch and to $2_\bullet$ via a lower arch,
	\begin{equation*}
		\includegraphics[scale=.65,valign=c]{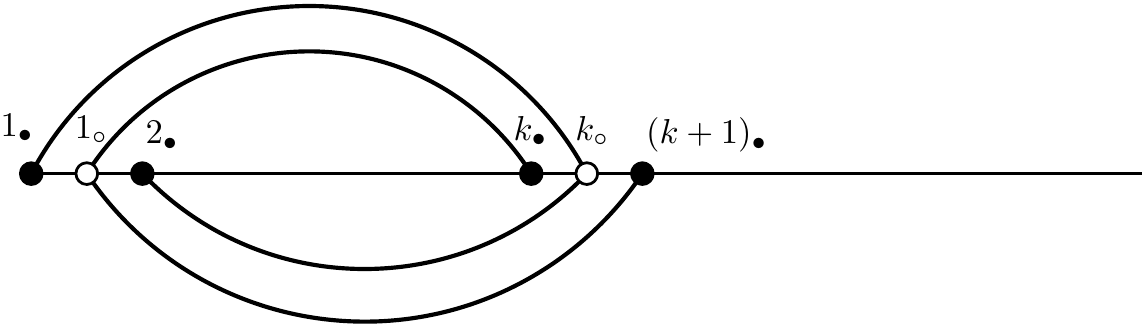}
	\end{equation*}
	This separates an inner region from an outer region and gives rise to a recursion. For $k=1$, one has $m_{n,1}=m_{n-1}$, and thus by summing over the position $k$,
	\begin{equation}
		m_n = \sum_{k=1}^n m_{n,k} = m_{n-1} + \sum_{k=2}^n m_{k-2}\,m_{n-k} %= m_{n-1} + \sum_{p=0}^{n-2} m_p\,m_{n-2-p}.
	\end{equation}
	Together with the initial conditions $m_0=m_1=1$, this recursion defines the Motzkin numbers and $m_n =\operatorname{Motzkin}_n$. We give a bijective interpretation below.
	%we sketch some elements of the proof. Let $l\in[2,n-2]$.
	
	As for Equation \eqref{SIFelse}, the strategy is to prove that choosing $\pi(1)_\bullet\in [1,n]$ completely determines a meander system, and there are therefore $n$ possibilities. For instance, if $\pi(1)_\bullet = k_\bullet$ with $k\geq4$ and $k+p\leq n+3$, then this fills the entire upper region between $2_\circ$ and $k_\bullet$ as follows
	\begin{equation*}
		\includegraphics[scale=.7,valign=c]{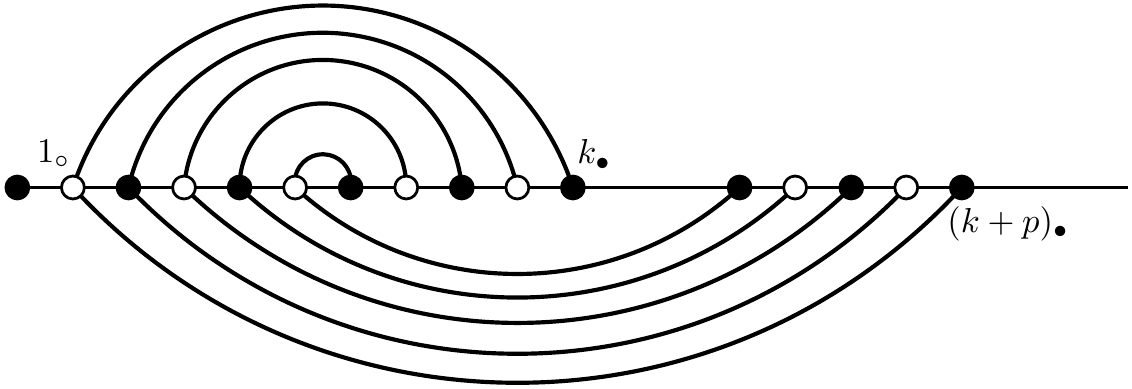}
	\end{equation*}
	It is then possible to shift cyclically that configuration so that we are once again in position of applying the same constraint (similar ones exist for all $k$), and so on until a unique meander system is fully built.
\end{proof}

\subsection{Bijection with Motzkin paths} It is well-known that planar arch configurations are one-to-one with Dyck paths. A \emph{Dyck path} of order $2n$ is a $2n$-step path in the upper half-plane which starts at $(0,0)$, ends at $(2n,0)$ and for which only two types of steps are allowed, the north-east step $(+1,+1)$ and the south-east step $(+1,-1)$. Given an arch configuration over $2n$ vertices on a horizontal line, oriented west to east, we list between each pair of consecutive vertices the number of arches which pass. This produces a list of $(2n-1)$ positive integers $(h_1,\dotsc,h_{2n-1})$ which are interpreted as the heights of a Dyck path after the step $1,\dotsc,2n-1$.

\begin{figure}
	\subfloat[]{\includegraphics[scale=.65]{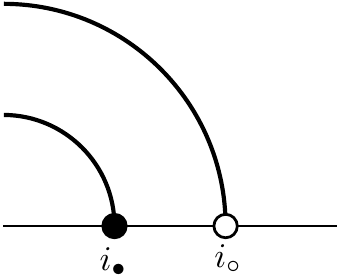}\label{fig:TwoDown}}\qquad
	\subfloat[]{\includegraphics[scale=.65]{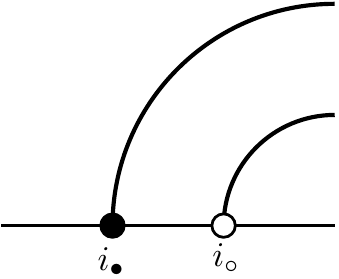}\label{fig:TwoUp}}\qquad
	\subfloat[]{\includegraphics[scale=.65]{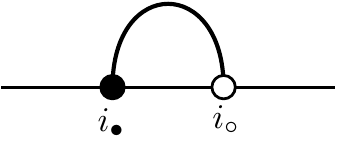}\label{fig:OneUpOneDown}}\qquad
	\subfloat[]{\includegraphics[scale=.65]{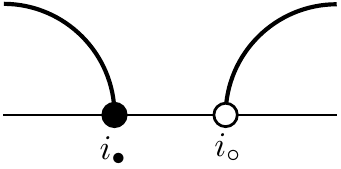}\label{fig:OneDownOneUp}}
	\caption{The four possible patterns at a pair of vertices in the upper half-plane of a meandric system.\label{fig:FourPatternsDyck}}
\end{figure}
Motzkin numbers are known to count \emph{Motzkin paths}. A Motzkin path of length $n$ is a path of $n$ steps in the upper half-plane which starts at $(0,0)$ and ends at $(n,0)$ with three types of steps, north-east $(+1, +1)$, south-east $(+1, -1)$ or east $(+1, 0)$, i.e. the horizontal step. Since $\cM_{\Delta_{-1}^{(n)}}$ is also enumerated by Motzkin numbers, it is natural to look for a bijection with the set of Motzkin paths. Consider a meander in $\mathcal{M}_{\Delta_{-1}^{(n)}}$. Looking at a pair $(i_\bullet, i_\circ)$ of vertices, there are four patterns which can arise in the upper half-plane, displayed in the Figure \ref{fig:FourPatternsDyck}. In terms of Dyck paths, they represent the four possible combinations of two consecutive steps: Figure \ref{fig:TwoDown} is two steps down, \ref{fig:TwoUp} two steps up, \ref{fig:OneUpOneDown} one up and one down, and finally \ref{fig:OneDownOneUp} one down and one up. 

However, meanders in $\mathcal{M}_{\Delta_{-1}^{(n)}}$ can not contain the pattern \ref{fig:OneDownOneUp}\footnote{Indeed, up to a cyclic permutation of the vertex labels, we can assume that we are looking at the leftmost pair of vertices $(1_\bullet, 1_\circ)$. Then, they are either connected together like in the Figure \ref{fig:OneUpOneDown}, or the white vertex $1_\circ$ is connected to some $j_\bullet$, $j>1$. In this case, we know from the proof of \Cref{thm:SIFVEV} that $1_\bullet$ is connected by an upper arch to $j_\circ$. Therefore when the labels are cyclically shifted, two arches always connect the pair $(i_\bullet,i_\circ)$ to the two vertices of a pair $((j+i-1)_\bullet, (j+i-1)_\circ) \mod n$. Therefore these two arches always point in the same direction, like in the Figures \ref{fig:TwoDown}, \ref{fig:TwoUp}.}. As a consequence, only the three other patterns in the upper-half plane are allowed. To find Motzkin paths, it is sufficient to just associate with the Figure \ref{fig:TwoDown} the south-east step, with \ref{fig:TwoUp} the north-east step and with \ref{fig:OneUpOneDown} the horizontal step. %The other way around it is straightforward to show that a Motzkin path gives rise to a single meander in $\mathcal{M}_{\Delta{-1}}$.

%%%%%%%%%%%%%%%%%%%%%%%%%%%%%%%%
\chapter{Universality results in dimension \texorpdfstring{$d=3$}{d=3}} \label{sec:Universality3D}
\section{The maximal 2-cut family} \label{sec:Max2Cut} 
%Here we introduce a family of graphs based on 2-edge insertions which will play a crucial role at various stages in this manuscript. In this section we will arrive at a lower bound on the coefficients $\alpha(\bb)$ of the linear growth hypothesis. Later in Chapter \ref{sec:Universality} we will meet again the graphs from this family as those solving the main questions of Section \ref{sec:MainQuestion} for some sets of bubbles.

%\subsection{Pairings} Let $\vec{\bb}$ be a bubble with labeled vertices. A \emph{pairing} is a partition of its vertices into pairs formed of a black and a white vertex. By associating an edge of color 0 to partner vertices, we see that a pairing is equivalent to a colored graph made of a single bubble, as illustrated in Figure \ref{fig:Pairing}. This is the set $\cG(\vec{\bb})$ and we use the notation $\graph$ or $\pi$ for its elements. The pairs are written $\{v, \pi(v)\}$ for $v\in\vec{\bb}$, so that $\pi$ is a fixed-point-free involution on the set of vertices of $\vec{\bb}$. Equivalently, it is possible to think of it as a map from the set of white vertices to the set of black vertices.

\subsection{Bonds and the maximal 2-cut property} We recall that an edge-cut in a connected graph is a set of edges whose removal disconnects the graph. In addition we will only consider \emph{bonds}, i.e. edge-cuts such that no proper subset is an edge-cut. In the remaining of the article, a $k$-bond means a bond with $k$ edges. Moreover, we say that a $k$-bond is incident on a bubble $\bb$ if it is formed by edges of color 0 which all have one end in $\bb$ and the other not in $\bb$.

Let $\graph$ be a colored graph, a bubble $\bb\subset\graph$ and let $\mathcal{E}(\graph;\bb)$ be the set of edges of color 0 of $\graph$ which have one end in $\bb$ and the other not in $\bb$. There is a unique partition of $\mathcal{E}(\graph; \bb)$ into bonds incident on $\bb$. Indeed, removing all edges of $\mathcal{E}(\graph; \bb)$ turns $\graph$ into $\bb$, decorated with edges of color 0 between some of its vertices, together with $L(\graph;\bb)$ connected components $\graph_1, \dotsc, \graph_{L(\graph;\bb)}$. The set of edges of color 0 which connects $\bb$ to $\graph_l$ in $\graph$ is $\mathcal{E}_l(\graph;\bb) \subset \mathcal{E}(\graph; \bb)$ and
\begin{equation}
	\mathcal{E}(\graph; \bb) = \dot\bigcup_{l=1}^{L(\graph; \bb)} \mathcal{E}_l(\graph; \bb),
\end{equation}
which is a disjoint union. If $k_l$ is the number of edges in $\mathcal{E}_l(\graph; \bb)$, then those edges form a $k_l$-bond incident on $\bb$. This is illustrated in Figure \ref{fig:Bond}.
\begin{figure}
	\includegraphics[scale=.5,valign=c]{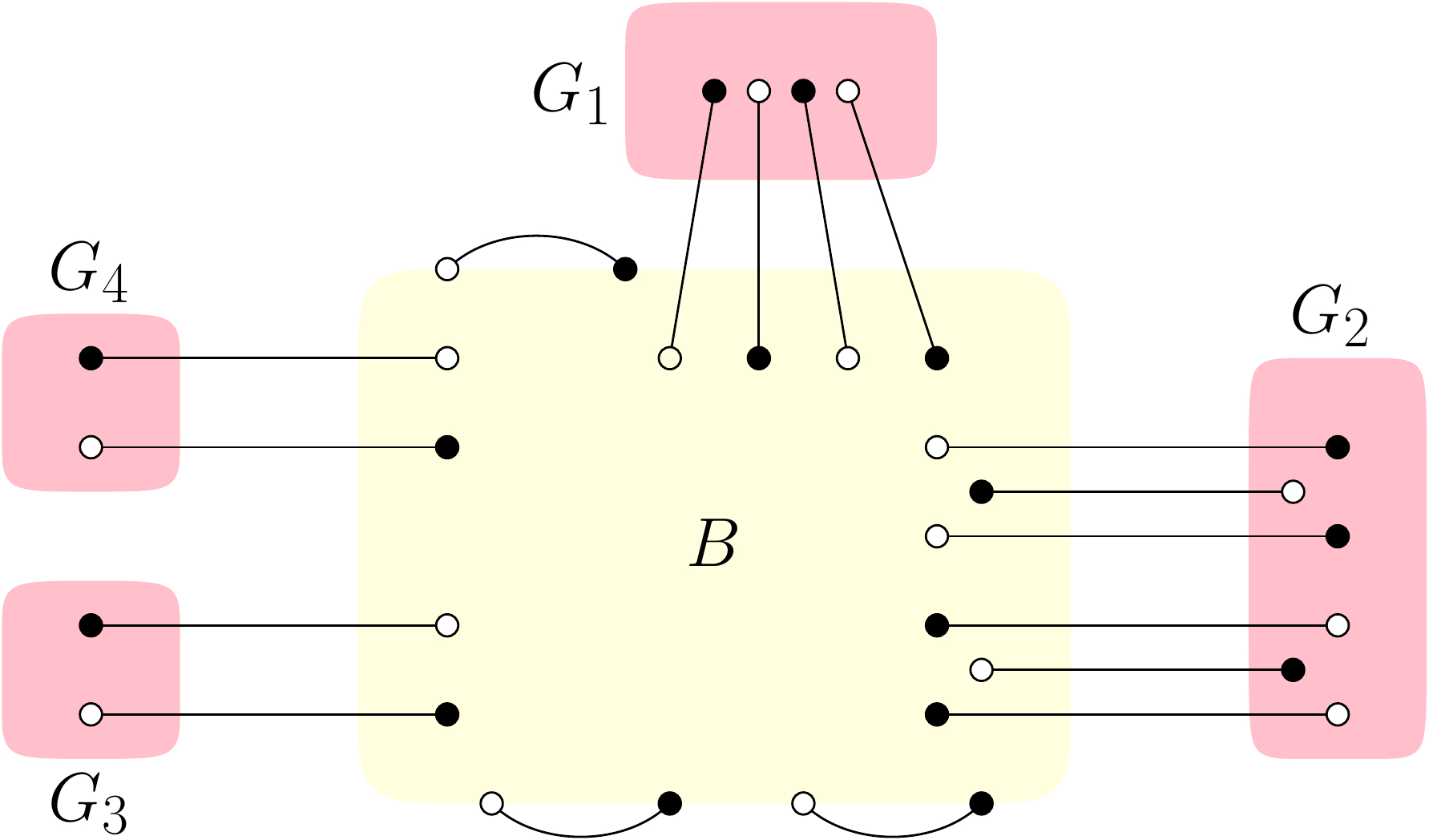}
	\caption{\label{fig:Bond}A graph $\graph$ partitioned as $\bb$ together with edges of color 0 between some of its vertices, and subgraphs $\graph_1, \dotsc, \graph_4$ connected via a 4-bond, a 6-bond, and two 2-bonds respectively.}
\end{figure}

%Our main theorem is that in order to maximize the number of bicolored cycles, a planar bubble can only be incident to specific 2-bonds. We call this the maximal 2-cut property.
\begin{definition} [Maximal 2-cut property]
	Let $\graph\in\cG_{n_1,\dotsc, n_N}(\vec{\bb};\bb_1, \dotsc, \bb_N)$. We say that $\vec{\bb}\subset \graph$ satisfies the maximal 2-cut property if there exists a pairing $\pi \in \cG^{\max}(\vec{\bb})$ such that for any pair of vertices $\{v, \pi(v)\}$, $v\in \bb$, there is 
	\begin{itemize}
		\item either an edge of color 0 between them,
		\item or two edges of color 0 forming a 2-bond incident on $\bb$, i.e. 
		\begin{equation}
			\forall l =1, \dotsc, L(\graph;\bb) \hspace{1cm}
			\lvert \mathcal{E}_l(\graph; \bb) \rvert = 2.
		\end{equation}
	\end{itemize}
	This is illustrated here,
	\begin{equation}
		\includegraphics[scale=.5,valign=c]{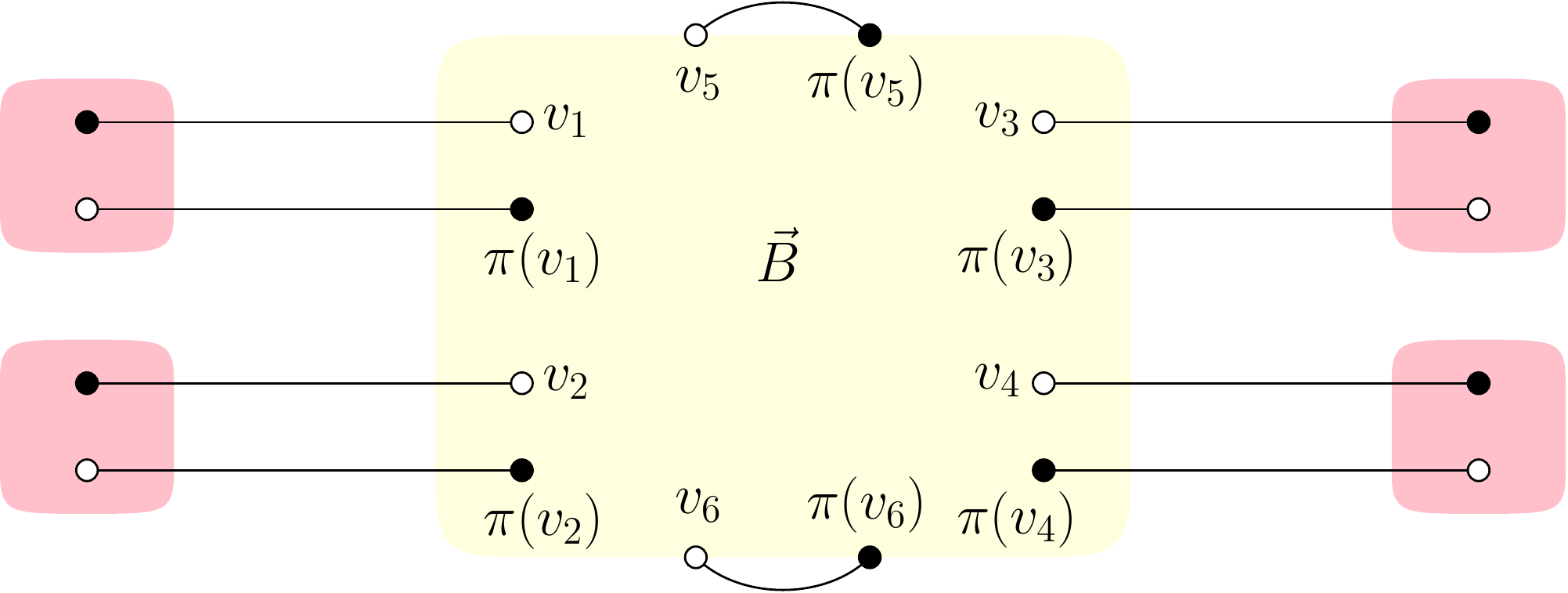}
	\end{equation}
\end{definition}

The number of faces of a graph in which $\vec{\bb}$ satisfies the maximal 2-cut property is independent of the choice of the pairing $\pi\in\cG^{\max}(\vec{\bb})$. More precisely, if one performs all the following moves, called edge flips of the 2-bonds incident on $\vec{\bb}$
\begin{equation}
	\includegraphics[scale=.5,valign=c]{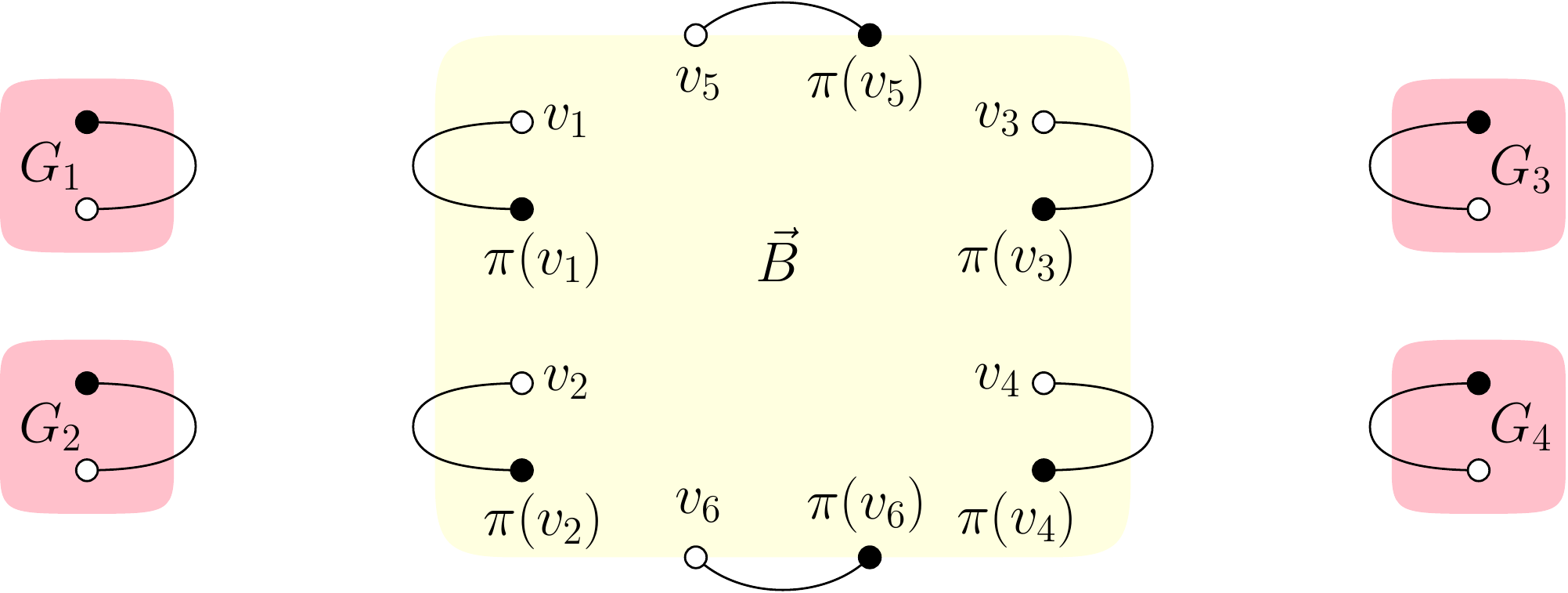}
\end{equation}
and denote $\graph_j$ the connected component which was attached to $\{v_j, \pi(v_j)\}$, then%They are all rooted at their new edges. Then
\begin{equation} \label{BicoloredCyclesMax2Cut}
	C_0(\graph) = C(\vec{\bb}) + \sum_j (C_0(\graph_j) -d),
\end{equation}
where $C(\vec{\bb})$ is the number of bicolored cycles $C_0(\graph_\pi)$ for any $\pi\in\cG^{\max}(\vec{\bb})$. %The $-d$ is due to the fact for every flip, all $d$ bicolored cycles of colors $\{0,c\}$ go along both edges of the 2-cut.

Indeed, if $\graph$ is a connected colored graph with colors $\{0, \dotsc, d\}$ which has a 2-bond formed by two edges $e, e'$ of color 0, then flipping $e$ with $e'$ gives two connected colored graphs $\graph_L, \graph_R$
\begin{equation}
	\graph = \includegraphics[scale=.4,valign=c]{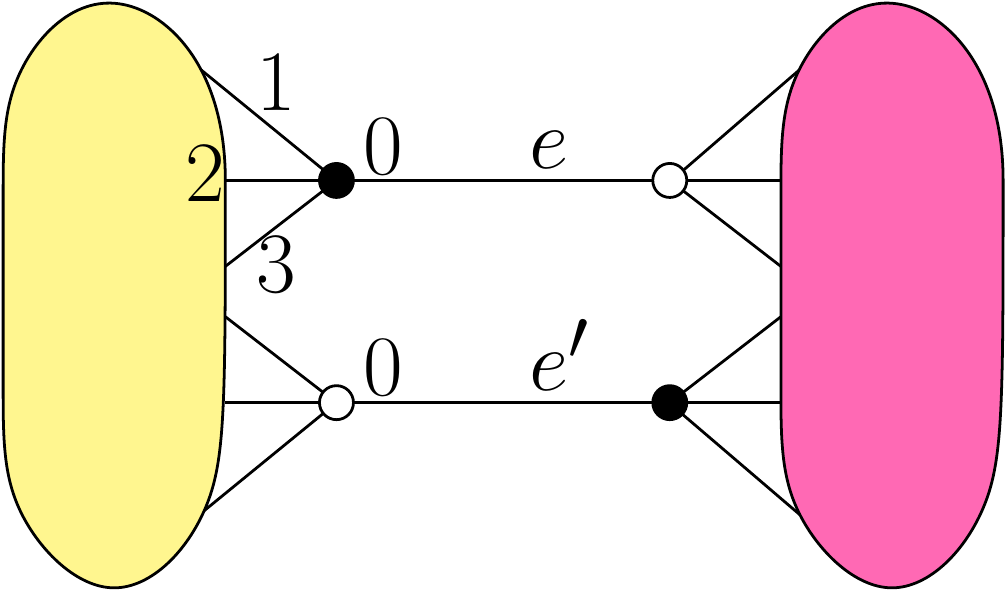} \quad \underset{\text{Flip}}{\to} \quad \graph_L = \includegraphics[scale=.4,valign=c]{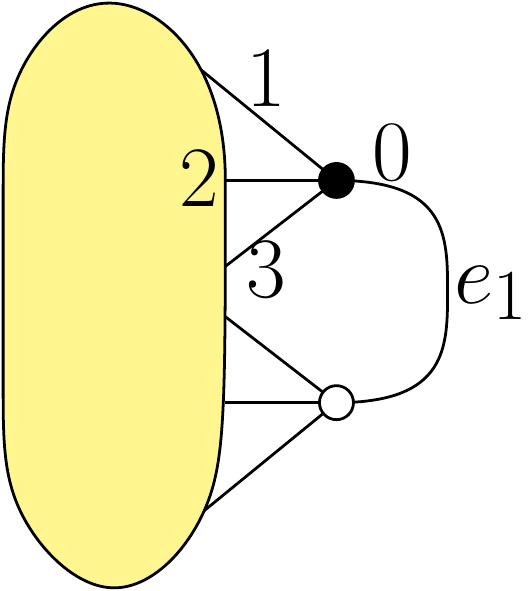}, \quad \graph_R = \includegraphics[scale=.4,valign=c]{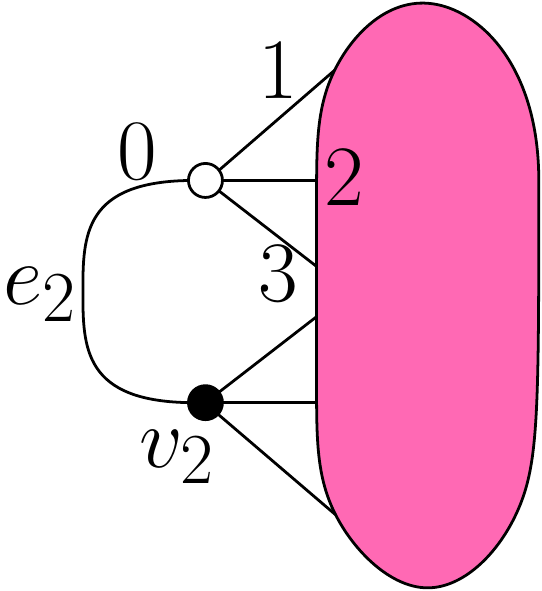}
\end{equation}
and
\begin{equation}
	C_0(\graph) = C_0(\graph_L) + C_0(\graph_R) -d,
\end{equation}
because there are $d$ bicolored cycles with colors $\{0,1\}, \dotsc, \{0,d\}$ going along both $e$ and $e'$ and each of them is split in two by the flip.

\subsection{An application to a lower bound on $\alpha(\bb)$} The maximal 2-cut family $\mathcal{P}_{n_1, \dotsc, n_N}(\bb_1, \dotsc, \bb_N)$ is the subset of $\cG_{n_1, \dotsc, n_N}(\bb_1, \dotsc, \bb_N)$ of graphs whose bubbles all satisfy the maximal 2-cut property. By induction on the formula \eqref{BicoloredCyclesMax2Cut}, one finds for $\graph\in \mathcal{P}_{n_1, \dotsc, n_N}(\bb_1, \dotsc, \bb_N)$
\begin{equation}
	C_0(\graph) = d + \sum_{i=1}^N (C(\bb_i)-d) n_i.
\end{equation}
This family thus satisfies the linear growth hypothesis with a growth coefficient $C(\bb_i)-d$. Therefore, if the model $\cG_{n_1, \dotsc, n_N}(\bb_1, \dotsc, \bb_N)$ satisfies the linear growth hypothesis, then the following lower bound holds,
\begin{equation} \label{LowerBound}
	\alpha(\bb_i) \geq C(\bb_i)-d.
\end{equation}
As it turns out, in all models where we have been able to find the growth coefficients $\alpha(\bb_i)$, they are always given by $C(\bb_i)-d$\footnote{Or by a generalization of it to pairings on bubble graphs, which we will briefly mention later.}.

Obviously $\mathcal{P}_{n_1, \dotsc, n_N}(\bb_1, \dotsc, \bb_N)$ is in bijection with a set of trees whose vertices corresponds to a bubble $\bb_i$ and the choice of a pairing $\pi\in\cG^{\max}(\vec{\bb}_i)$. An equation on the generating series is straightforwardly written. It will be given in Equation \eqref{2PtTreeEq}. %Before that, we actually define our generating series.

At $d=2$, if $\bb_i$ is the bubble with $i$ edges, we have seen in \eqref{MaxFacesPrescribedBubbles} that the linear growth hypothesis holds with $\alpha(\bb_i) = \frac{i-2}{2}$. This coincides with the above bound, because $C(\bb_i)$ is then the maximal number of vertices in 1-face maps formed by the identification of the sides of the $i$-gon, and thus $C(\bb_i) = \frac{i+2}{2}$, by Euler's relation. However at $d=2$, not all colored graphs of genus 0 satisfy the maximal 2-cut property. It is in fact what makes them very interesting: bonds of arbitrary size can contribute to vanishing genus. It is in a sense the origin of their special universality class.

\section{Universality classes vademecum}
It is a basis of analytic combinatorics \cite{FlajoletSedgewick2009book} that asymptotic behaviors of combinatorial sequences can often be extracted from the singularities of their generating functions. If $f(z) = \sum_{n\geq 1} f_n z^n$ for a sequence $(f_n)_{n\geq 1}$, then we will keep in mind the classical correspondence
\begin{equation}
	f(z)\underset{z\to\rho}{\sim} (1-z/\rho)^{-\alpha} \quad \to \quad f_n\underset{n\to\infty}{\sim} \frac{1}{\Gamma(\alpha)} \rho^{-n} n^{\alpha-1}.
\end{equation}
If it satifies a polynomial equation $P(z,f(z))=0$, then typically $f(z)\sim \sqrt{1-z/\rho}$ and $f_n \sim a \rho^{-n} n^{-3/2}$, which is the universal behavior of trees, like planar and binary trees.

However some combinatorial classes fall out of this universality class. This is the case of planar maps with $n$ edges for which $f(z) \sim (1-z/\rho)^{3/2}$ which gives $f_n \sim a \rho^{-n} n^{-5/2}$, called the universality class of pure 2-dimensional quantum gravity. More exotic behaviors can be found by coupling conformal matter to planar maps \cite{DiFrancescoGinspargZinnJustin}. Here however we will only focus on trees, maps, and an intermediate behavior between trees and planar maps, for which $f(z) \sim (1-z/\rho)^{2/3}$ and $f_n \sim a\rho^{-n} n^{-5/3}$.

\section{Universality with planar bubbles at \texorpdfstring{$d=3$}{d=3}}
The point of our main theorem is that in order to maximize the number bicolored cycles when gluing planar bubbles at $d=3$, they have to be glued using 2-bonds of color 0 only. %To make this more precise, we need the following definitions. 
%Notice that \Cref{thm:1CBB} only applies at $d=3$ for melonic bubbles, for which there is a single pairing which maximizes the number of bicolored cycles.
\subsection{Main theorem: planar bubbles satisfy the maximal 2-cut property}
\begin{theorem}{}{1Planar}\cite{Bonzom3D} 
	Let $d=3$ and $\bb_1, \dotsc, \bb_N$ arbitrary bubbles.	If $\vec{\bb}$ is planar and $\graph\in\cG^{\max}_{n_1, \dotsc, n_N}(\vec{\bb};B_1, \dotsc, B_N)$, then $\vec{\bb}\subset\graph$ satisfies the maximal 2-cut property. As a consequence, if the set $\{\bb_1, \dotsc, \bb_N\}$ satisfies the linear growth hypothesis, then
	\begin{equation}
		\langle\vec{\bb}\rangle_{\max}{}_{|p_i\to p_i x^{-\alpha(\bb_i)}} = x^{3+\alpha(\vec{\bb})} |\cG^{\max}(\vec{\bb})| K(p_1, \dotsc, p_N)^{v(\vec{\bb})/2}
	\end{equation}
	where 
	\begin{equation}
		K(p_1, \dotsc, p_N) \coloneqq \sum_{n_1, \dotsc, n_N\geq 0} |\cG^{\max}_{n_1, \dotsc, n_N}(\bb_1, \dotsc, \bb_N)| \prod_{i=1}^N p_i^{n_i}.
	\end{equation}
\end{theorem}
$K(p_1, \dotsc, p_N)$ is simply the generating series of $\cG^{\max}_{n_1, \dotsc, n_N}(\bb_1, \dotsc, \bb_N)$. That theorem basically states that $\langle\vec{\bb}\rangle_{\max}$ factorizes over 2-bonds, up to the choice of a pairing from $\cG^{\max}(\vec{\bb})$.

\begin{corollary}{}{2PtFunctionPlanarBubbles}
	Consider $\bb_1, \dotsc, \bb_N$ to be planar bubbles. Then the model $\cG_{n_1, \dotsc, n_N}(\bb_1, \dotsc, \bb_N)$ satisfies the linear growth hypothesis \eqref{LinearGrowth} with $\alpha(\bb_i) = C(\bb_i)-3$, i.e.
	\begin{equation} \label{BicoloredCyclesPlanar}
		C_{n_1, \dotsc, n_N}(\bb_1, \dotsc, \bb_N) = 3 + \sum_{i=1}^N (C(\bb_i)-3)n_i.
	\end{equation}
	The graphs from $\cG^{\max}_{n_1, \dotsc, n_N}(\bb_1, \dotsc, \bb_N)$ are in bijection with trees with $n_i$ vertices of type $i=1, \dotsc, N$, where a vertex of type $i$ has degree $v_i/2$ and $|\cG^{\max}(\vec{\bb}_i)|$ possible colors. The above generating series satisfies
	\begin{equation} \label{2PtTreeEq}
		K(p_1, \dotsc, p_N) = 1 + \sum_{i=1}^N p_i |\cG^{\max}(\vec{\bb}_i)| K(p_1, \dotsc, p_N)^{v_i/2}
	\end{equation}
\end{corollary}
It thus solves the questions of Section \ref{sec:MainQuestion} of Chapter \ref{sec:PrescribedBubbles} for models at $d=3$ with planar bubbles, except that the coefficients $C(\bb_i)$ remain unknown. \Cref{thm:2PtFunctionPlanarBubbles} in fact extends to $\cG_{n_1, \dotsc, n_N}(\vec{\bb};\bb_1, \dotsc, \bb_N)$ with a root bubble $\vec{\bb}$ which may not even be planar \cite{Bonzom3D}.

Notice that Gurau's bound on $\alpha(\bb_i)$ for a bubble $\bb_i$ of genus $g_i$ with $v_i$ vertices is $\alpha(\bb_i)\leq v_i-2+2g_i$. The inequality $C(\bb_i)-3\leq v_i-2$ for planar bubbles is obvious as the number of bicolored cycles cannot grow at the pace of one per vertex of $\bb_i$, except for melonic bubbles for which $C(\bb_i) = v_i+1$ exactly. Moreover, $C(\bb_i)$ is expected to decrease with the genus of $\bb_i$ and not grow with it.

\begin{proof}[Proof of \Cref{thm:2PtFunctionPlanarBubbles}]
	Equation \eqref{BicoloredCyclesPlanar} is proved from \Cref{thm:1Planar} with Equation \eqref{BicoloredCyclesMax2Cut} recursively.
	
	In our convention, $K(p_1, \dotsc, p_N)$ contains the ``trivial'' rooted graph, reduced to the root edge with no bubbles. If $\graph$ is not trivial, consider the bubble $\bb_i$ incident to the root vertex and any labeling of its vertices. Then one applies \Cref{thm:1Planar}. For any $\pi\in\cG^{\max}(\vec{\bb}_i)$, the partner vertices $\{v, \pi(v)\}$ can be connected either directly by an edge of color 0 (the trivial rooted graph) or via a rooted graph whose bubbles also all satisfy the maximal 2-cut property\footnote{Moreover, any $\pi\in\cG^{\max}(\vec{\bb}_i)$ can be equally used, since we saw in \eqref{BicoloredCyclesMax2Cut} that the number of bicolored cycles along the root bubble is independent of $\pi$. The other terms in \eqref{BicoloredCyclesMax2Cut} depend on the graphs $\graph_j$s, but their contributions are here ``uniformized'' by the rescaling of $p_i\to p_i x^{-\alpha(\bb_i)}$.}.
	
	This establishes the symbolic decomposition, where the edge with the arrow is the edge of color 0 incident to the root vertex,
	\begin{equation}
		\includegraphics[scale=.45,valign=c]{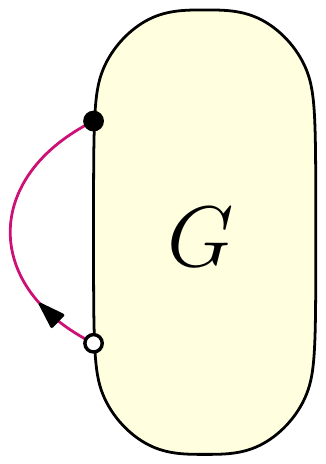} = \left\{\includegraphics[scale=.45,valign=c]{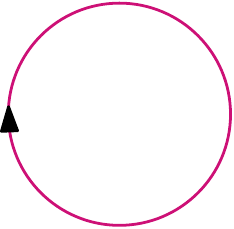}\right\}\cup \bigcup_{i=1}^N \bigcup_{\pi\in\cG^{\max}(\vec{\bb}_i)} \includegraphics[scale=.45,valign=c]{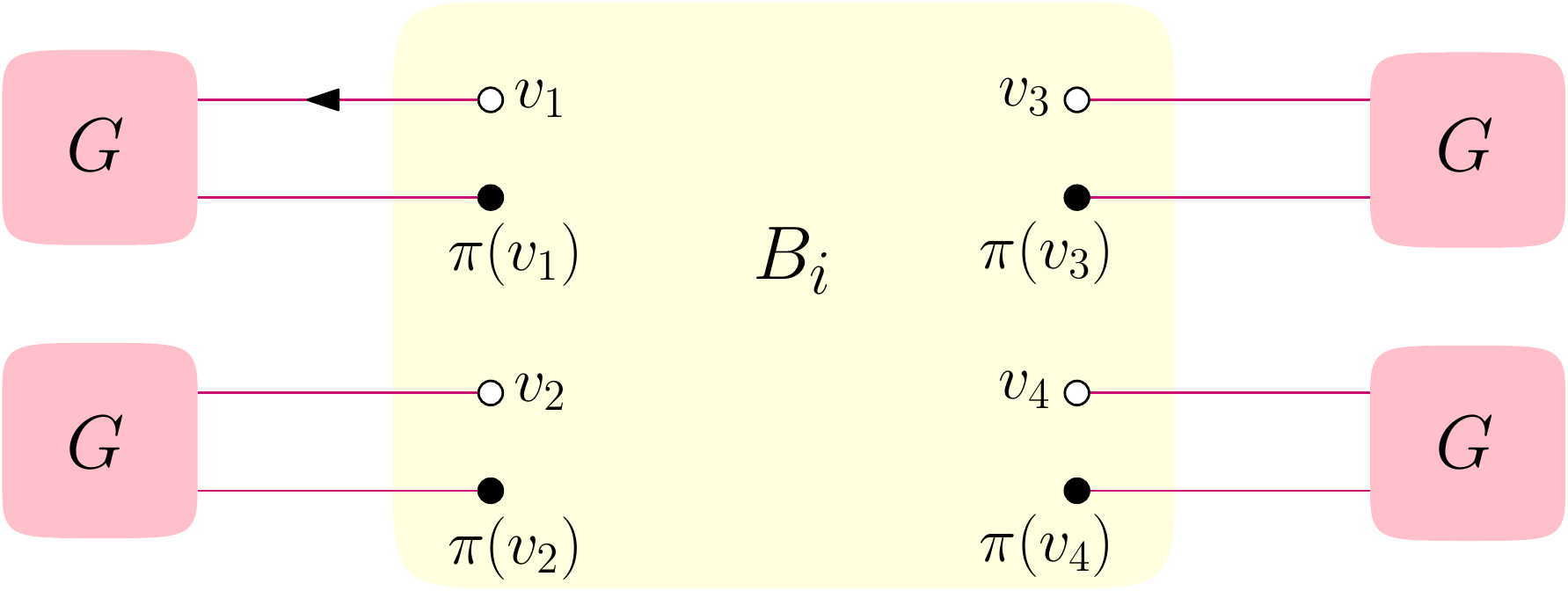}
	\end{equation}
	from which the bijection with trees follows, as well as the polynomial equation on $K(p_1, \dotsc, p_N)$.
\end{proof}

\Cref{thm:1Planar} was previously known to hold in limited situations: when all bubbles $B_1, \dotsc ,B_N$ are melonic \cite{Uncoloring}, and when there is a single type of bubble ($N=1$) which is the octahedron \cite{Octahedra}. In fact, a method detailed in \cite{SigmaReview} can be used to extend it to more bubbles of the form $\partial H$ when $H$ is a subgraph of some $G\in \cG^{\max}_{n_1, \dotsc, n_N}(B_1, \dotsc, B_N)$. This however does not bring genuinely new cases because the graphs maximizing the number of bicolored cycles with $\partial H$ form a subset of $\cG^{\max}_{n_1, \dotsc, n_N}(B_1, \dotsc, B_N)$.

Theorem \ref{thm:1Planar} is a thus a far reaching generalization of the existing results since it imposes the single constraint on one bubble to be homeomorphic to the 3-ball. Among all existing results, it only leaves out the case where all bubbles are $K_{3,3}$ which has the topology of the torus and $K_4$ (which is not bipartite, but can be used as a bubble nonetheless) which has the topology of the projective plane. The case of $K_{3,3}$ was studied in \cite{StuffedColoredMaps} with a conclusion slightly less restrictive than the maximal 2-cut property of Theorem \ref{thm:1Planar}, and similar to the result in the case of the bubble $K_4$ \cite{CarrozzaTanasa2016}. 

In fact, from the cases of $K_4$ and $K_{3,3}$, one can conjecture that there is a family of bubbles for which bubbles in $\graph\in\cG^{\max}_{n_1, \dotsc, n_N}(\bb_1, \dotsc, \bb_N)$ may not satisfy the maximal 2-cut property, but instead there exist some bubble \emph{subgraphs} $H_1, H_2, \dotsc$ which are finite and which satisfy the maximal 2-cut property. This is weaker than \Cref{thm:1Planar} because inside $H_1, H_2, \dotsc$ bubbles need not satisfy the maximal 2-cut property. Clearly, those models also have a tree-like behavior, the nodes not being bubbles anymore but bubble subgraphs instead. That is unless (possibly) the subgraphs $H_1, H_2, \dotsc$ form an infinite family where the size of the subgraphs become unbounded\ldots

\subsection{Applications to gluings of octahedra} The limiting factor in general applications of \Cref{thm:2PtFunctionPlanarBubbles} is finding the pairings $\pi\in\cG^{\max}(\vec{\bb}_i)$. We do not know much about their properties for $\bb_i$ planar in general and since $|\cG(\vec{\bb}_i)| = (v_i/2)!$, this becomes a serious issue quickly. We can nevertheless apply the theorem by hand to bubbles which do not have too many vertices.

Here, we apply it to gluings of octahedra which maximizes the number of edges of the triangulations. In the dual picture, the octahedron is represented as a bubble which is the edge-colored cube as shown in Figure \ref{fig:Octahedron} in Chapter \ref{sec:Definitions}. There are three pairings which maximize $C_0(\graph_\pi)$,
\begin{equation}
	\includegraphics[scale=.7,valign=c]{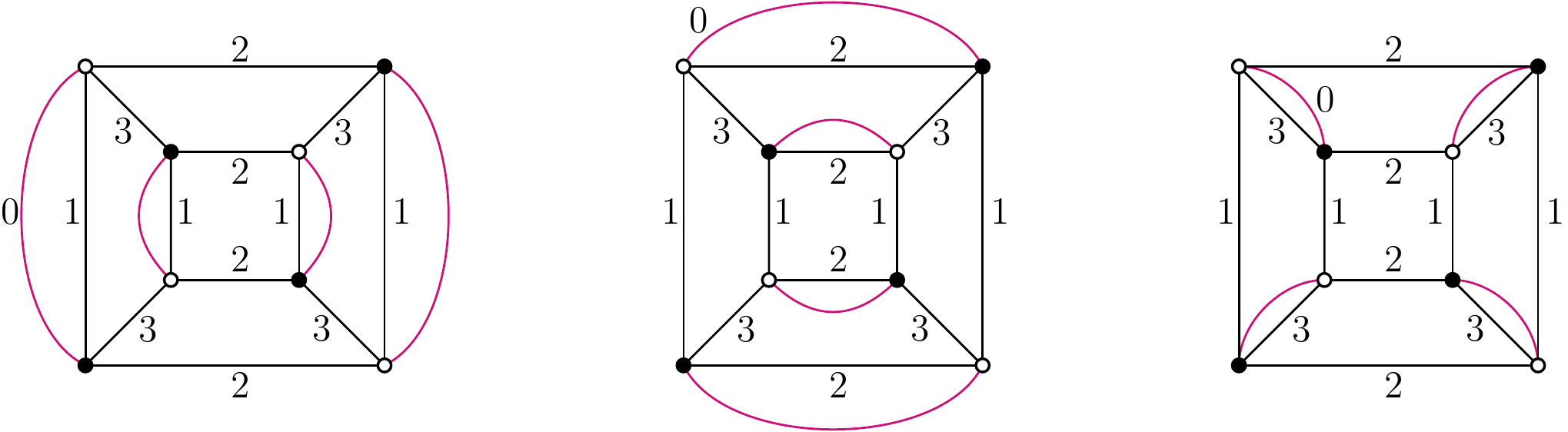}
\end{equation}
with $C(\bb) = 8\leq v(\bb)+1=9$ in agreement with Gurau's bound. For $\cG^{\max}_n(\vec{\bb})$, \Cref{thm:2PtFunctionPlanarBubbles} gives $C_n(\bb) =5n+3$ (Gurau's bound being $6n+3$). If one includes the edges inside the octahedra, we can write \eqref{BicoloredCyclesPlanar} as
\begin{equation}
	|\# \text{Edges}| \leq 3 + 11 |\# \text{Octahedra}|
\end{equation} 
while Gurau's bound has a coefficient 12 instead of 11.

The generating series $K(p) = \sum_{n\geq 0} |\cG^{\max}_n(\vec{\bb})| p^n$ satisfies $K(p) = 1 + 3p K(p)^4$, from which the critical point is found to be at $p_c = \frac{9}{256}$ and
\begin{equation}
	K(p) = \frac{4}{3} - \frac{32\sqrt{2}}{9\sqrt{3}}\sqrt{p_c-p} + \mathcal{O}(p_c-p).
\end{equation}

\subsection{Elements of proof of \Cref{thm:1Planar}} This is not the place for a complete proof of \Cref{thm:1Planar}, but since it is one of the most important results in this memoir, we want to point out a few elements from the proof. It relies on a \emph{local analysis}, in the sense that it is centered at a bubble and we study the bonds which can be attached to it. An induction on the total number of vertices of $\graph$ is used and the proof is done by a case-by-case analysis of situations where the maximal 2-cut property does not hold and how a graph with more bicolored cycles can then be built from the induction hypothesis.

If $\vec{\bb}\subset\graph$ is incident to a 2-bond, it is easily done, as well as if there is in $\vec{\bb}$ two parallel edges. Let us explain why $\vec{\bb}$ cannot be incident to a 4-bond, as this is the simplest case which has some flavor of the more intricate ones.
\begin{proposition}{}{4EdgeCuts}
	For $d$ odd and for \emph{any} (planar or not) bubbles $\bb_1, \dotsc, \bb_N$, $\graph\in\cG^{\max}_{n_1, \dotsc, n_N}(\vec{\bb};B_1, \dotsc, B_N)$ has no 4-bonds made of edges of color 0.
\end{proposition}

\begin{proof}
	Suppose $\graph$ has a 4-bond made of edges of color 0. It looks as follows,
	\begin{equation}
		G = \begin{array}{c} \includegraphics[scale=.4]{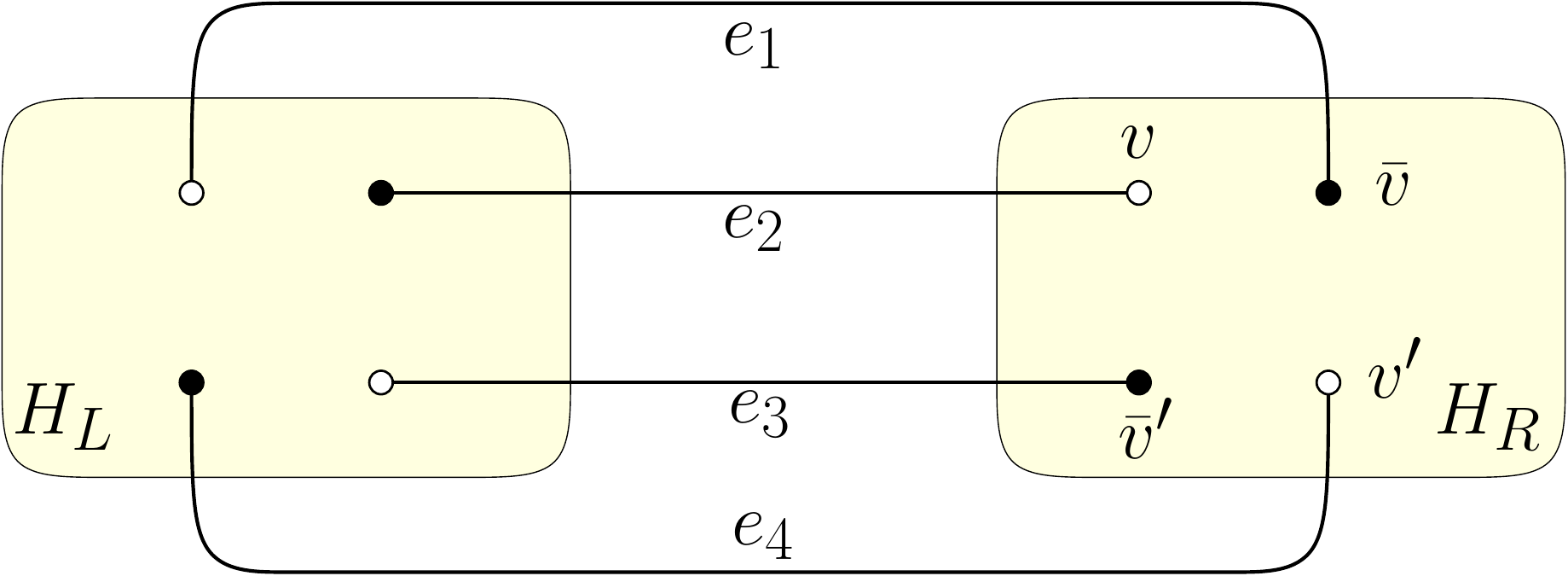} \end{array}
	\end{equation}
	Let $c\in[1..d]$ and let us look at the bicolored cycles which go along the edges of the 4-bond.
	\begin{itemize}
		\item There can be one along $e_1$ and $e_2$, and another along $e_3$ and $e_4$,
		\item or one along $e_1$ and $e_4$, and another along $e_3$ and $e_2$,
		\item or it is the same bicolored cycle along $e_1, e_2, e_3, e_4$.
	\end{itemize}
	This depends on the bicolored paths with colors $\{0,c\}$ inside $H_L$ and $H_R$. As we only need to know at which vertices those paths enter and leave, this information is captured by the boundary bubbles, which are bubbles with 4 vertices (possibly two 2-vertex bubbles).
	
	Considering the boundary bubble $\partial H_R$, it has the four vertices $v, v', \bar{v}, \bar{v}'$. Up to exchanging the roles of $e_2$ and $e_4$, we can assume that there are $d/2\leq q\leq d$ edges between $v$ and $\bar{v}$, and between $v'$ and $\bar{v}'$, and $d-q$ edges with the complementary colors between $v$ and $\bar{v}'$, and between $v'$ and $\bar{v}$. Thus the number of bicolored cycles going along $e_1, \dotsc, e_4$ in $\graph$ is the same as in 
	\begin{equation}
		\includegraphics[scale=.4,valign=c]{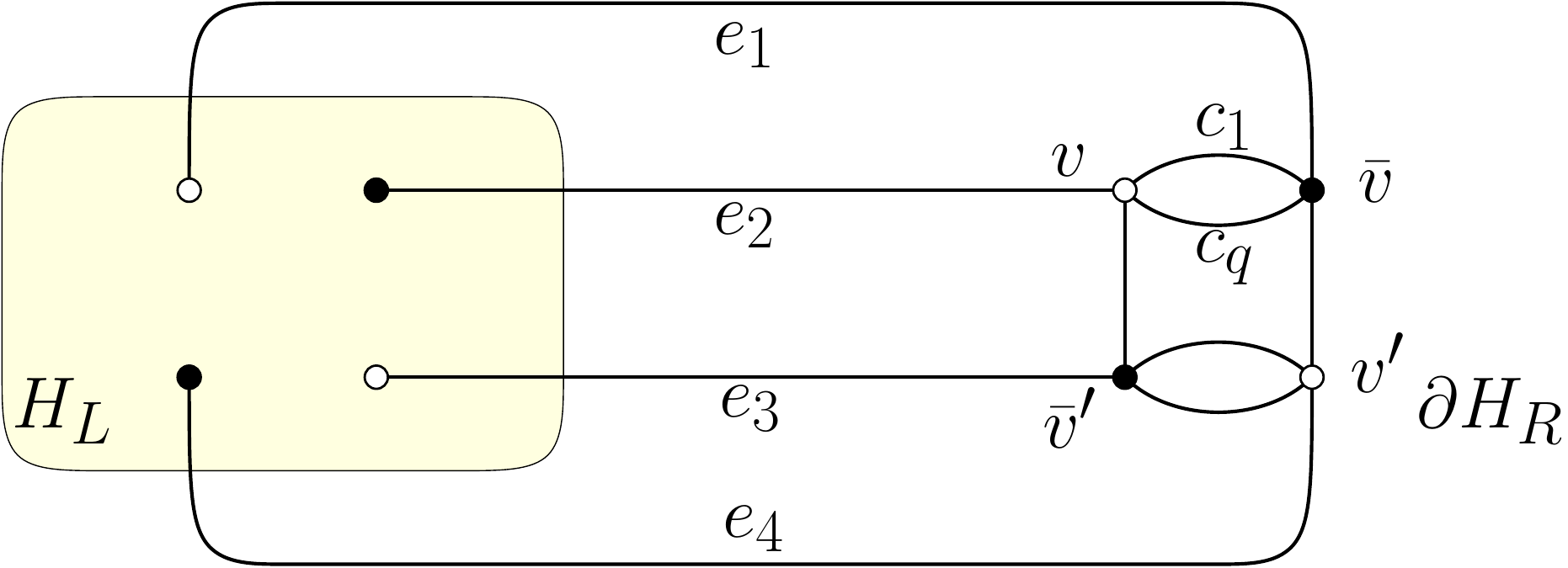}
	\end{equation}
	Crucially, since $d$ is odd, $q>d/2$, and therefore the following flip increases the number of bicolored cycles by $2q-d$
	\begin{equation}
		\includegraphics[scale=.4,valign=c]{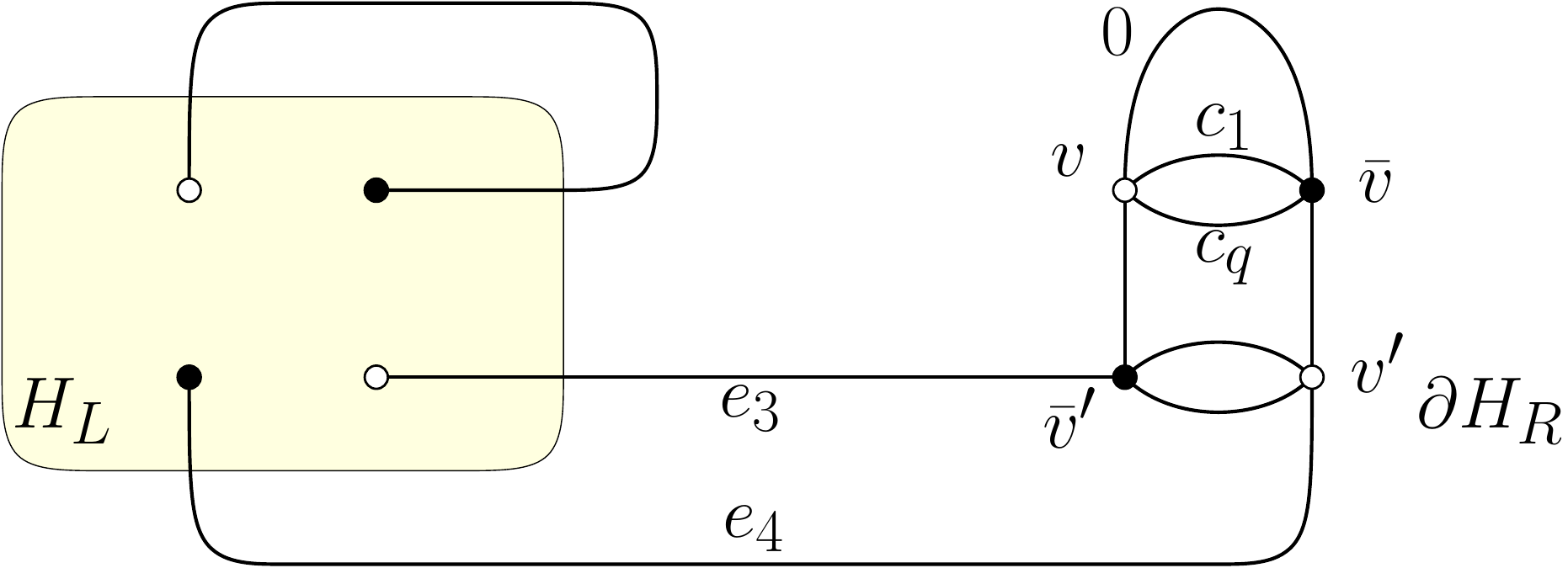}
	\end{equation}
	We thus find that performing the same flip on $e_1$, $e_2$ in $\graph$ increases the number of bicolored cycles
	\begin{equation}
		C_0\left(\includegraphics[scale=.4,valign=c]{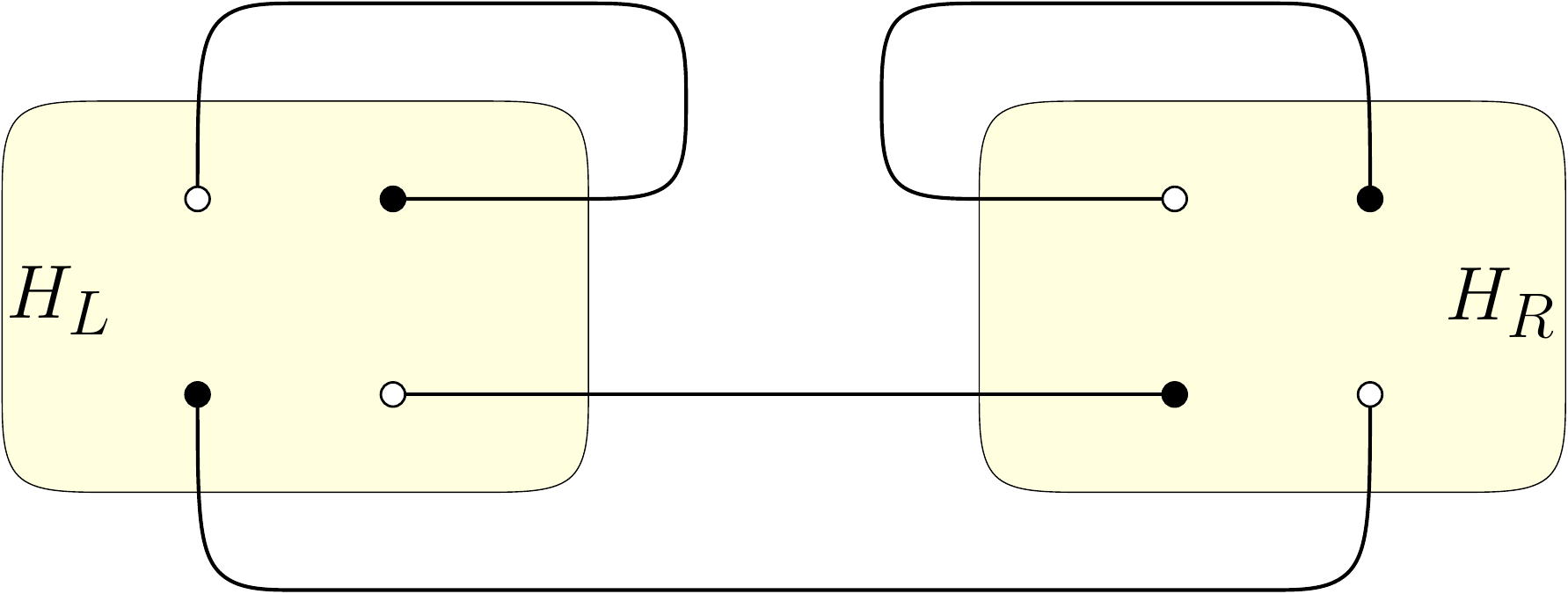}\right) = C_0(\graph) -d + 2q>C_0(\graph).
	\end{equation}
\end{proof}

For $k$-bonds with $k\geq 6$, we cannot have a statement as general as \Cref{thm:4EdgeCuts}, and we need to look at the specifity of planar bubbles at $d=3$, and combine it with the idea of boundary bubbles as in the proof above. The most important property of planar bubbles that we have used is an elementary property about the lengths of their bicolored cycles.

\begin{lemma} \label{thm:FacesPlanarBubble}
	A planar bubble without bicolored cycles of length 2 has at least six bicolored cycles of length 4.
\end{lemma}
In terms of maps, using the canonical embedding of \Cref{thm:CanonicalEmbedding} in Chapter \ref{sec:Definitions}, it is equivalent to saying that a planar, tricolored cubic map with no faces of degree 2 has at least six faces of degree 4. This can be obtained by a double counting argument of (twice) the number of edges, together with Euler's formula.

Assuming $\vec{\bb}$ has no incident 2-bonds and 4-bonds, the incident edges of color 0 are either between vertices of $\vec{\bb}$ or parts of $k$-bonds for $k\geq 6$. As it turns out, it is enough to study the ``neighborhood'' in $\graph$ of a face of degree 4 of $\vec{\bb}$, i.e. how it is connected via the edges of color 0 incident on that face. There are obviously a lot of cases. For definiteness, consider a face of degree 4 in $\vec{\bb}$ with colors $\{1,2\}$ and vertices $\bar{v}_1, v_2, \bar{v}_3, v_4$. The edges of color 0 incident to those vertices may be part of different bonds or connect to other vertices of $\vec{\bb}$ (or among themselves). For instance, there is a case where the edges $e_1, \dotsc, e_4$ incident on $\bar{v}_1, v_2, \bar{v}_3, v_4$ are all different, and $e_1$ and $e_3$ are part of the same bond, $e_2$ connects to another vertex of $\vec{\bb}$ and $e_4$ is part of another bond,
\begin{equation} \label{FaceDegree4ThreeEdgesTwoCuts}
	\graph = \includegraphics[scale=.4,valign=c]{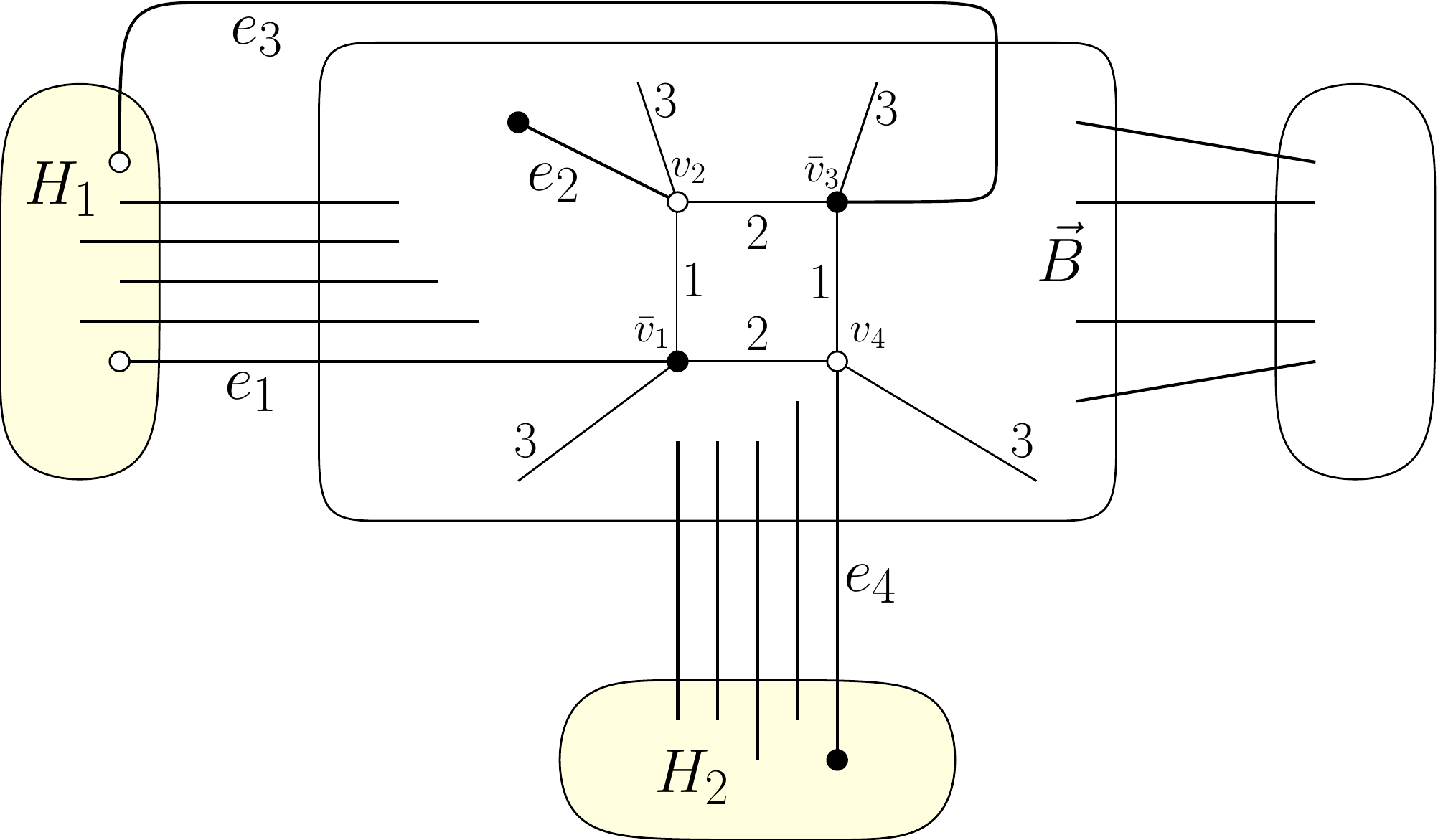}
\end{equation}
Here only the colors 1, 2, 3 are explicitly indicated.

The idea is then to perform some edge flips which do not decrease the number of bicolored cycles until the edges of color 0 incident to that face of degree 4 are parallel to the edges of color 1 or 2, like this
\begin{equation} \label{FaceDegree4TwoFlips}
	G = \includegraphics[scale=.4,valign=c]{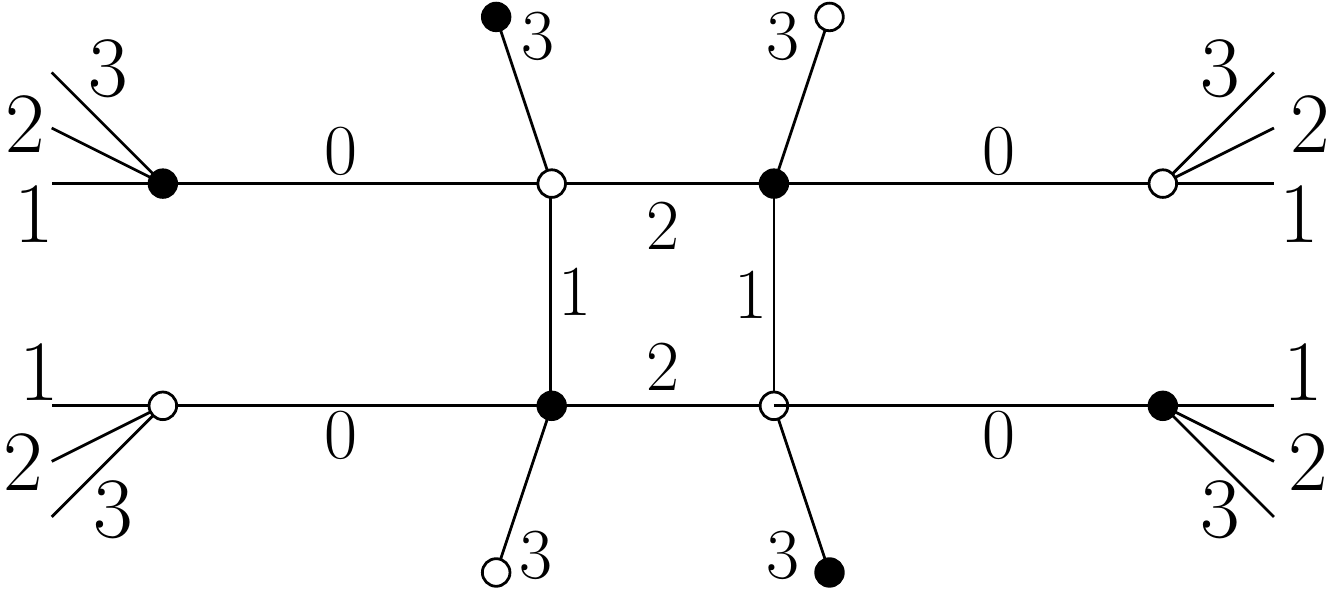} \quad\underset{\text{Flips}}{\to}\quad G_{\parallel} = \includegraphics[scale=.4,valign=c]{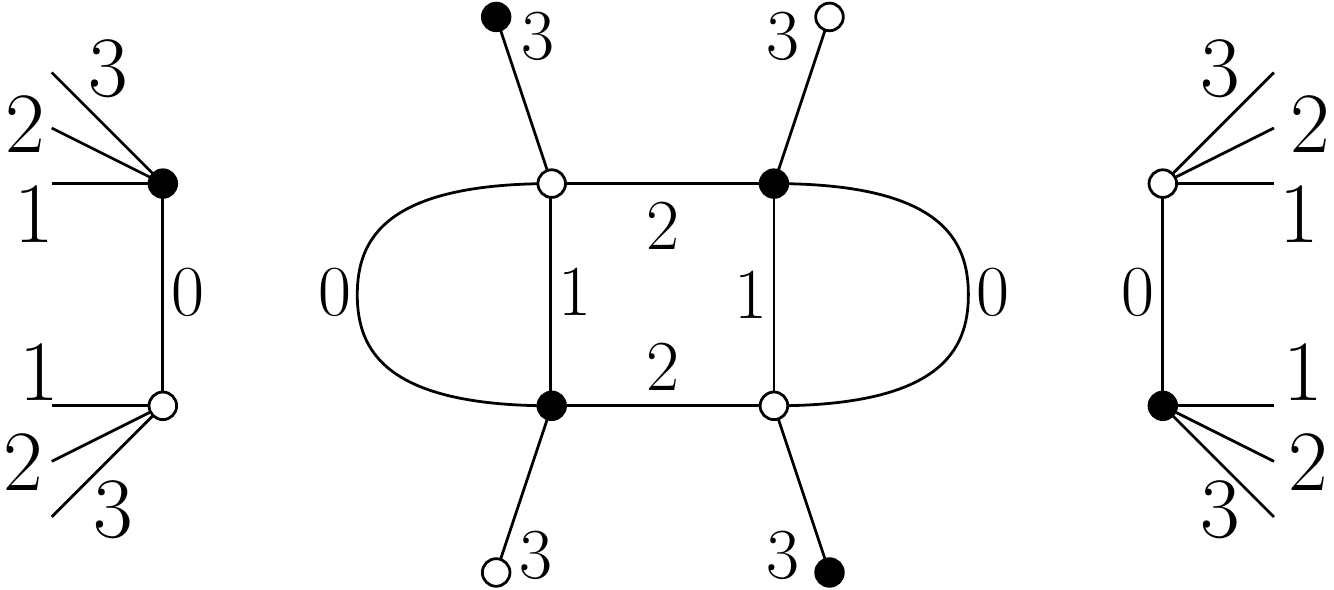}
\end{equation}
Through those flips it can be checked that the number of bicolored cycles cannot decrease, i.e. $C_0(\graph) \leq C_0(\graph_{\parallel})$. Then one performs the following transformation on the graph and $\vec{\bb}$ (e.g. if the edges of color 0 are parallel to those of color 1)
\begin{equation} \label{FaceDegree4Removed}
	\includegraphics[scale=.4,valign=c]{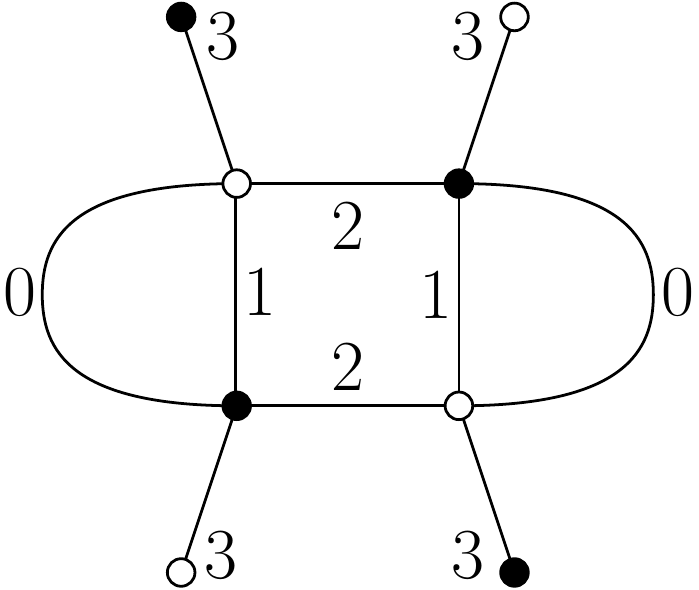} \quad\to\quad \includegraphics[scale=.4,valign=c]{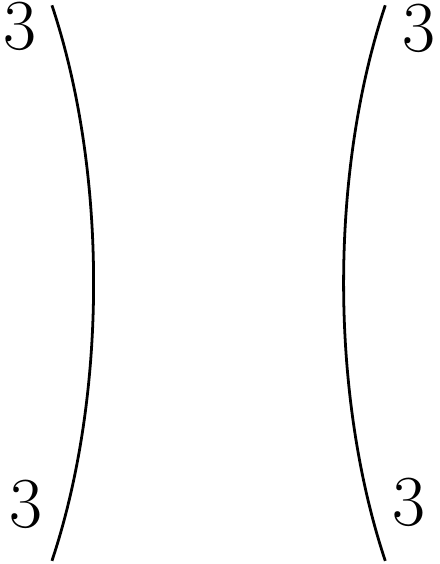}
\end{equation}
Crucially, this turns $\vec{\bb}$ into a new bubble with four vertices less and which is still planar, so that the induction hypothesis can be applied.

\subsection{Topology} 
\begin{theorem}{}{Topology}
	Let $\{\bb_1, \dotsc, \bb_N\}$ be a set of planar bubbles at $d=3$. Then $\graph\in \cG^{\max}_{n_1, \dotsc, n_N}(\bb_1, \dotsc, \bb_N)$ has the topology of the $3$-sphere.
\end{theorem}
There are two main cases to this theorem. One is when $\graph$ has at least two bubbles, and the other when $\graph$ has a single bubble. In the former, we will use \Cref{thm:2PtFunctionPlanarBubbles} and in the latter, the statement of the theorem was already given as \Cref{thm:Topology1PlanarCBB} in Chapter \ref{sec:1CBB} (with promise of details to come now).

\subsubsection{Sequences of topological moves} Let us release $d\geq 2$ for the following statement. The following moves and their inverses preserve the topology of the triangulation dual to a colored graph \cite{Ferri1982}.
\begin{equation} \label{TopologicalMoves}
	\includegraphics[scale=.4,valign=c]{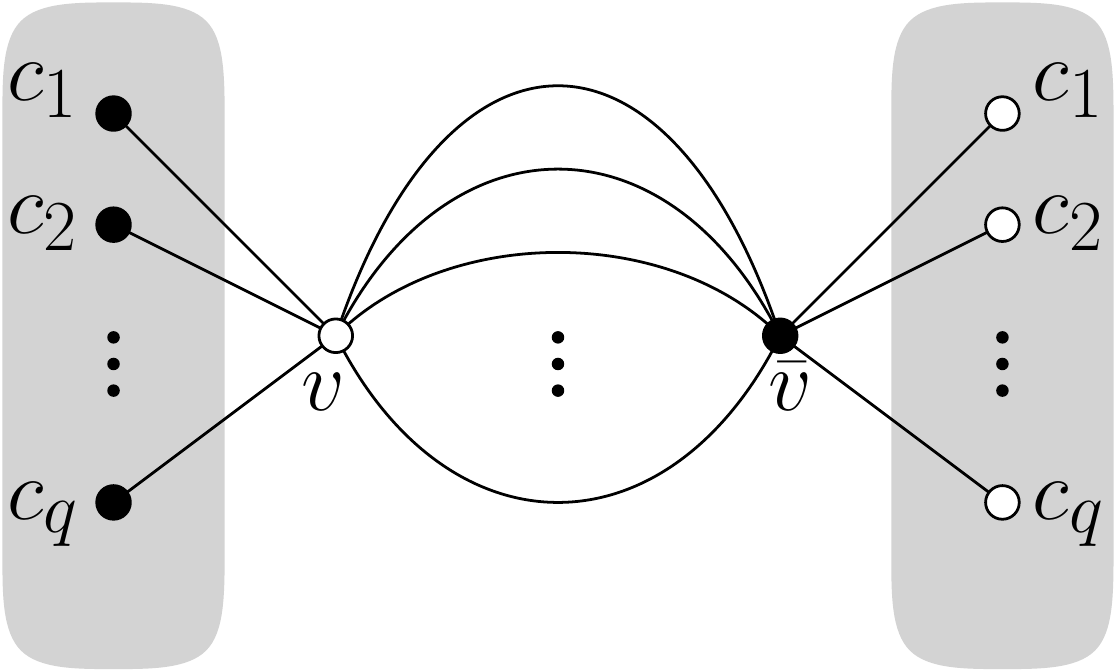} \quad \to \quad \includegraphics[scale=.4,valign=c]{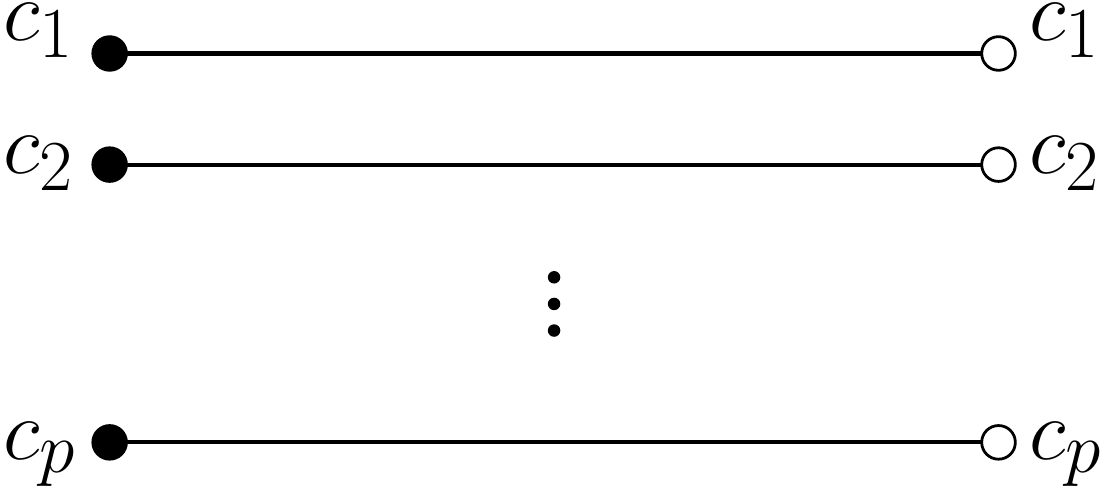}
\end{equation}
where 
\begin{itemize}
	\item $v$ and $\bar{v}$ belong to two different connected components of $\graph(c_1, \dotsc, c_q)$ (the subgraph of $\graph$ restricted to edges of colors $\{c_1, \dotsc, c_q\}$),
	\item and at least one of them is homeomorphic to a $q$-ball.
\end{itemize}
In fact, for any two homeomorphic colored triangulations $\trisp, \trisp'$, there exists a sequence of those moves transforming $\graph(\trisp)$ into $\graph(\trisp')$. However in general there is no algorithm to find such a sequence. This is a consequence of Dehn's classical isomorphism problem for groups and its undecidability. Let us explain this fascinating aspect in a few lines. A classical problem for manifolds is that of finding an algorithm to decide when two manifolds are homeomorphic. A subproblem is thus that of finding an algorithm to determine whether the fundamental groups of two manifolds are isomorphic. One of Dehn's classical problems is in fact just that, in the case of finitely presented group\footnote{i.e. finite numbers of generators and relations}: determining whether two finitely presented groups are isomorphic. It is known to be in general undecidable \cite{Stillwell1982}.

To be able to apply that result to manifolds, one needs to know the kind of groups which arise as fundamental groups. For 3-manifolds, it does not apply, because it turns out that the fundamental groups are much more restricted than ``only'' finitely presented groups. In fact, the isomorphism problem for (closed, oriented, triangulated) 3-manifolds is solvable \cite{Kuperberg2019}. We also refer to the survey \cite{Aschenbrenner2014} for more interesting decision problems on 3-manifolds and their fundamental groups. However in four dimensions, any finitely presented group is a fundamental group, and the isomorphism problem is thus non-solvable.

\subsubsection{Connected sums} We recall that the connected sum of two $d$-manifolds $M_1$ and $M_2$ is a $d$-manifold obtained by removing a ball in $M_1$ and in $M_2$ and then identifying the two resulting $(d-1)$-spheres in some canonical way. For two colored graphs $\graph_1$, $\graph_2$, a connected sum $\graph_1\#_{\{v_1, v_2\}} \graph_2$ is obtained by removing a black vertex $v_1$ from $\graph_1$, a white vertex $v_2$ from $\graph_2$ and identifying the resulting hanging edges which have the same color,
\begin{equation}
	\includegraphics[scale=.4,valign=c]{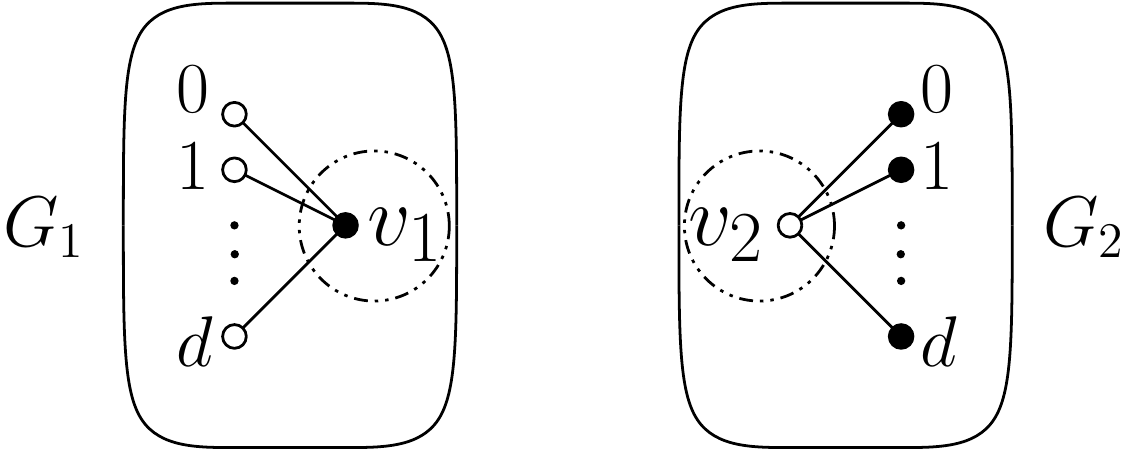} \quad \to \quad \graph_1\#_{\{v_1, v_2\}} \graph_2 = \includegraphics[scale=.4,valign=c]{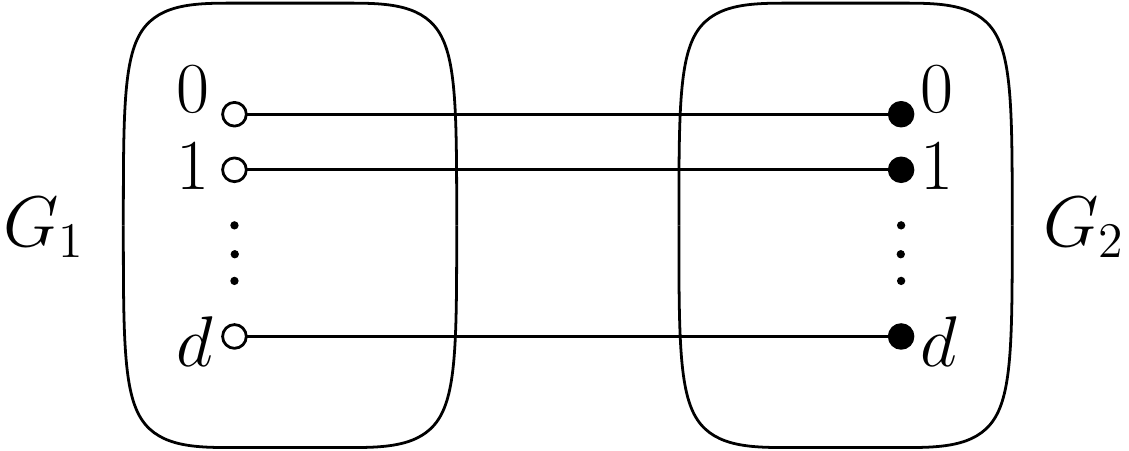}
\end{equation}
Then $\graph_1\#_{\{v_1, v_2\}} \graph_2$ is a colored graph for the connected sum of the manifolds represented by $\graph_1$ and $\graph_2$. We refer to \cite{Ferri1986} for multiple references to connected sums in the colored graph literature in topology. Notice that the sphere is a neutral element for the connected sum of manifolds.

\begin{lemma} \label{lemma:2CutConnected}
	If $\graph$ is a connected colored graph with colors $\{0, \dotsc, d\}$ which has a 2-bond formed by two edges $e, e'$ of color 0, then flipping $e$ with $e'$ gives two connected colored graphs $\graph_1, \graph_2$,
	\begin{equation}
		\graph = \includegraphics[scale=.4,valign=c]{2EdgeCutFinal.pdf} \quad \underset{\text{Flip}}{\to} \quad \graph_1 = \includegraphics[scale=.4,valign=c]{2EdgeCutLeftPart.pdf}, \quad \graph_2 = \includegraphics[scale=.4,valign=c]{2EdgeCutRightPart.pdf}
	\end{equation}
	and $\graph$ is a connected sum of $\graph_1$ and $\graph_2$.
\end{lemma}

\begin{proof}
	First add to $e_1$ two vertices connected by all colors except 0. This preserves the topology. Then choose $v_1$ and $v_2$ to perform the connected sum $\tilde{\graph}_1\#_{\{v_1, v_2\}} \graph_2$ as follows 
	\begin{equation}
		\tilde{\graph}_1 = \includegraphics[scale=.4,valign=c]{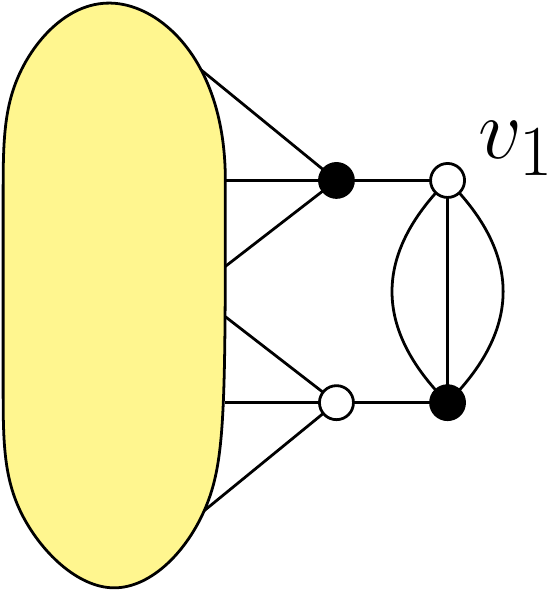} \quad \graph_2 = \includegraphics[scale=.4,valign=c]{2EdgeCutRightPart.pdf} \quad \to \quad \tilde{\graph}_1\#_{\{v_1, v_2\}} \graph_2 = \includegraphics[scale=.4,valign=c]{2EdgeCutFinal.pdf} = \graph
	\end{equation}
\end{proof}

In \Cref{thm:2PtFunctionPlanarBubbles} we saw that gluings of planar bubbles which maximize the number of bicolored cycles are performed using 2-bonds. From the above Lemma, we thus conclude that those gluings are topologically connected sums. Therefore, the remaining part of the proof of \Cref{thm:Topology} is to prove the statement for colored graphs with a single planar bubble. This shares similarities with the proof of \Cref{thm:1Planar}.

If $\vec{\bb}$ itself has a 2-edge-cut, it is easy to separate it into two (planar) bubbles and use \Cref{thm:2PtFunctionPlanarBubbles}. If $\graph\in\cG^{\max}(\vec{\bb})$ has an edge of color 0 joining two ends to an edge of $\vec{\bb}$, then the move \eqref{TopologicalMoves} can be used. Therefore, after treating some simpler cases, it is necessary to consider the case where $\vec{\bb}$ has no parallel edges, hence no faces of degree 2 in its canonical embedding, and from \Cref{thm:FacesPlanarBubble}, we can consider a face of degree 4, such that all edges of color 0 incident to its vertices are distinct. 

We then perform the same sequence of moves as in Equations \eqref{FaceDegree4TwoFlips}-\eqref{FaceDegree4Removed}
\begin{equation}
	\begin{aligned}
		G = \includegraphics[scale=.4,valign=c]{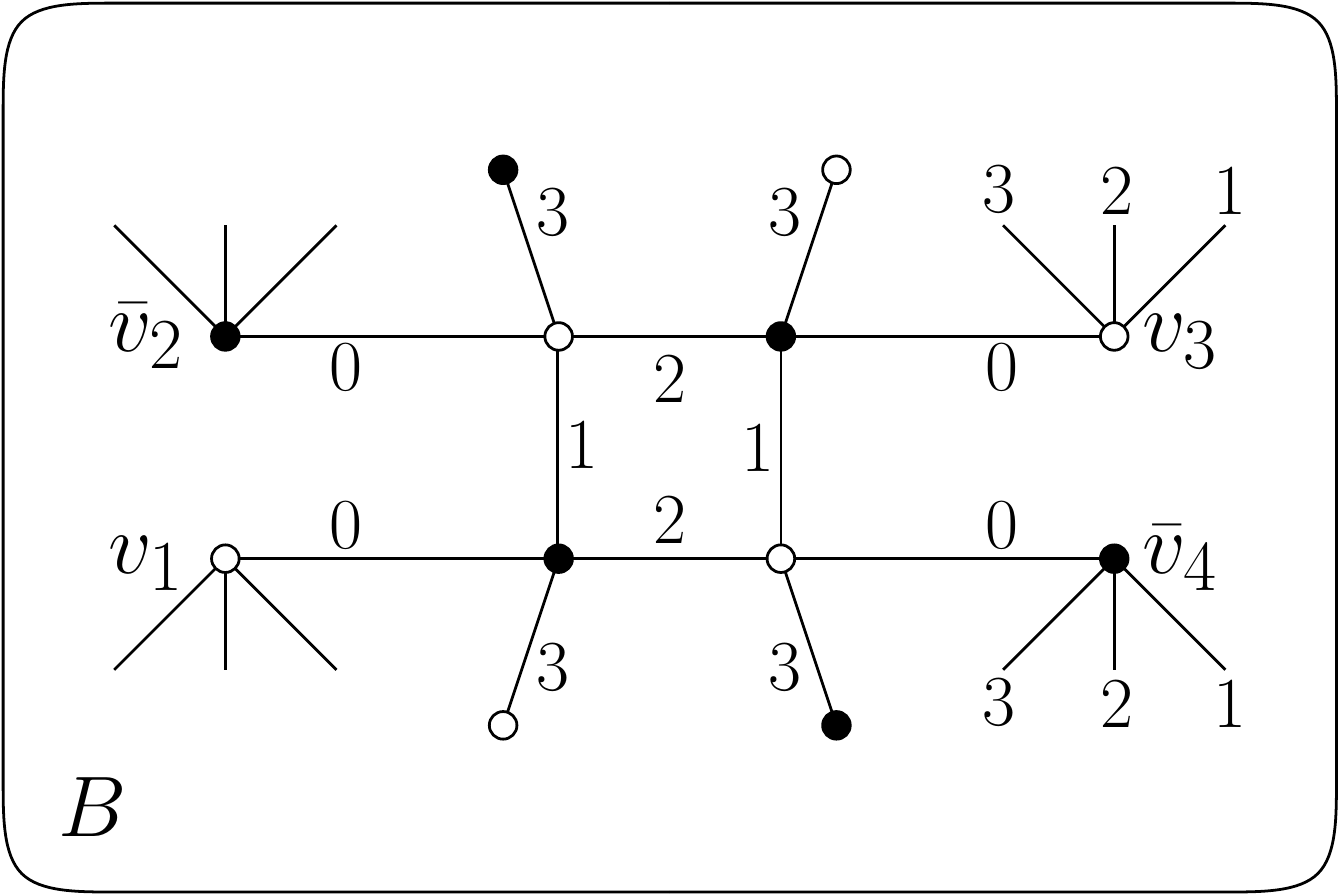}
		\to\quad &G_{\parallel} = \includegraphics[scale=.4,valign=c]{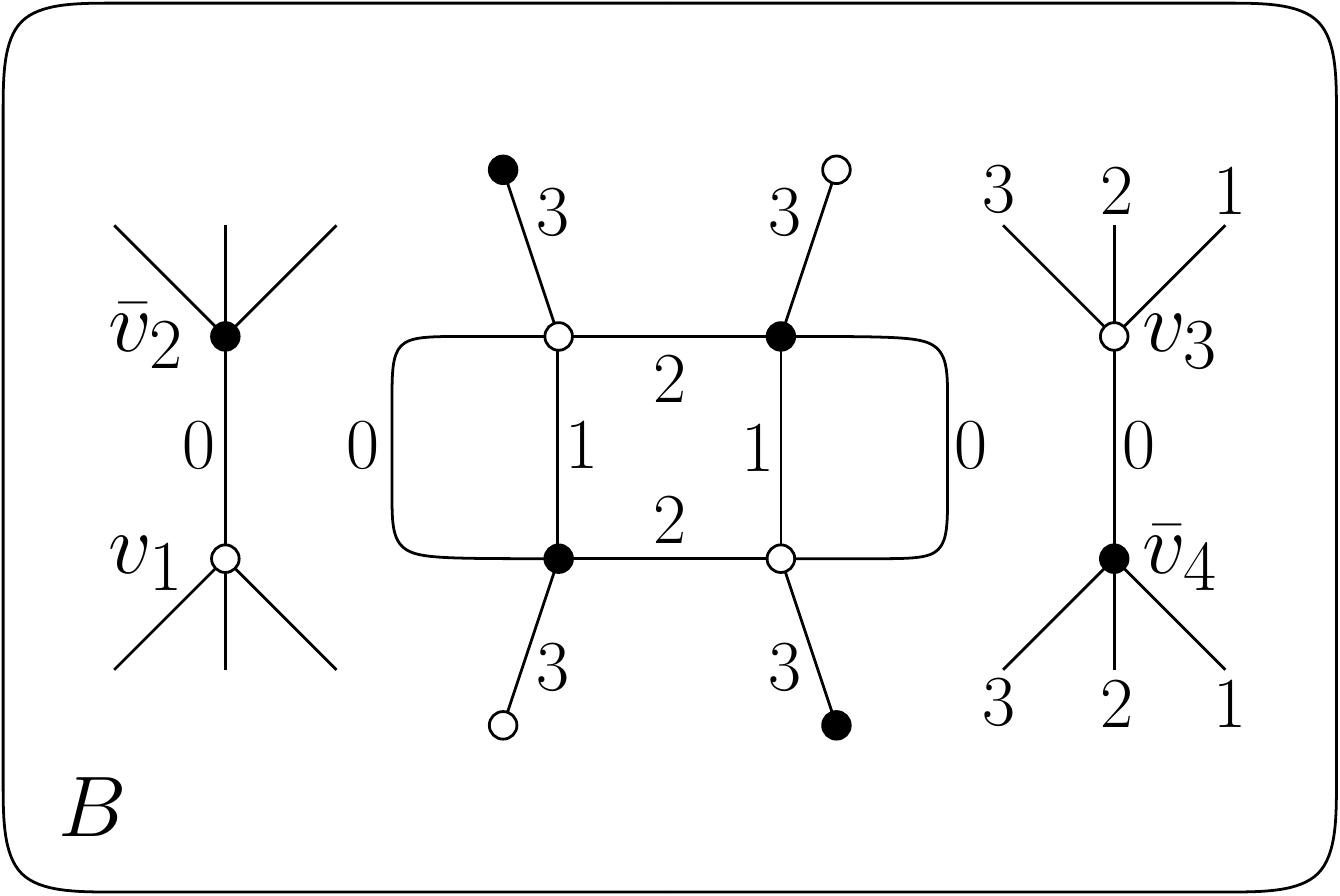}\\
		\to\quad &G' = \includegraphics[scale=.4,valign=c]{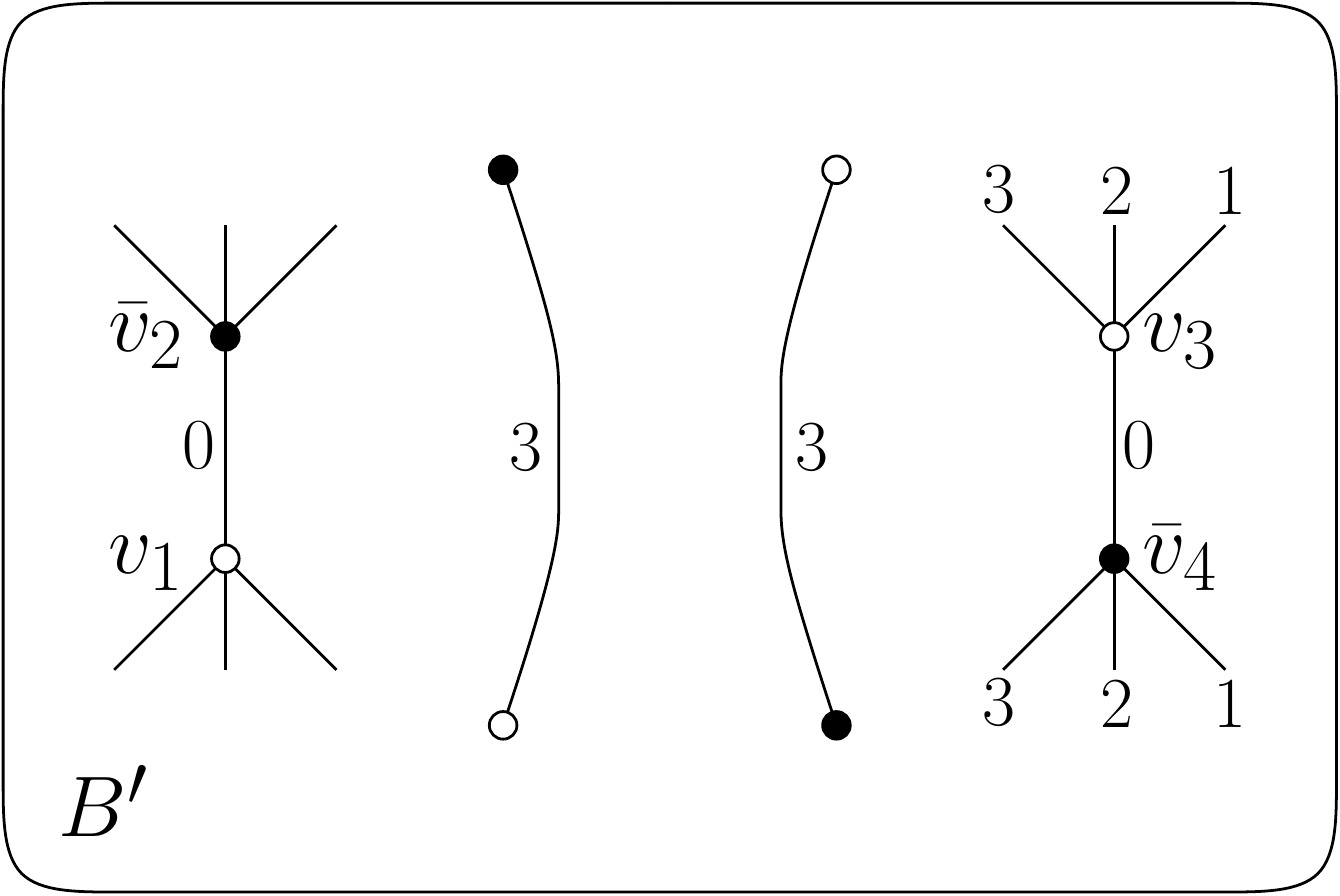}
	\end{aligned}
\end{equation}
Throughout this sequence the number of bicolored cycles is controlled, in particular
\begin{equation}
	G\in\cG^{\max}(\vec{\bb}) \ \Rightarrow G_\parallel\in\cG^{\max}(\vec{\bb}), \quad G'\in\cG^{\max}(\vec{\bb}')
\end{equation}
From the induction hypothesis, $G'$ is thus a 3-sphere. Moreover, the move from $G_\parallel$ to $G'$ is topological. In remains to show that $G\to G_\parallel$ is also topological, which is more difficult. It can be done via Euler characteristic arguments to control the number of connected components of the subgraphs $G(0, 2, 3)$ and their genera.

\begin{remark}
	We have thus found that $\graph\in\cG^{\max}(\vec{\bb})$ for a planar bubble is a 3-sphere. Although we do not know how to fully characterize those graphs, we are able to use combinatorial arguments to identify its topology.
\end{remark}

%%%%%%%%%%%%%%%%%%%%%%%%
\chapter{Universality results in higher dimensions} \label{sec:Universality2}
\section{Bijection with colored stuffed maps} \label{sec:Bijection}
\subsection{Principle} We sketch here a bijection between colored graphs with prescribed bubbles and a colored generalization of stuffed maps. The formal details can be found in \cite{StuffedColoredMaps}. This bijection can be seen as an extension of Tutte's one between quadrangulations and bipartite maps. It is simply based on the following idea: if one has a cyclically ordered set of elements $(o_1, \dotsc, o_n)$, then one can represent the elements as $n$ vertices, and the cyclic order either as an oriented cycle on those vertices, or a as star-vertex connected to the vertices with the correct cyclic order,
\begin{equation} \label{CycleToMap}
	\includegraphics[scale=.5,valign=c]{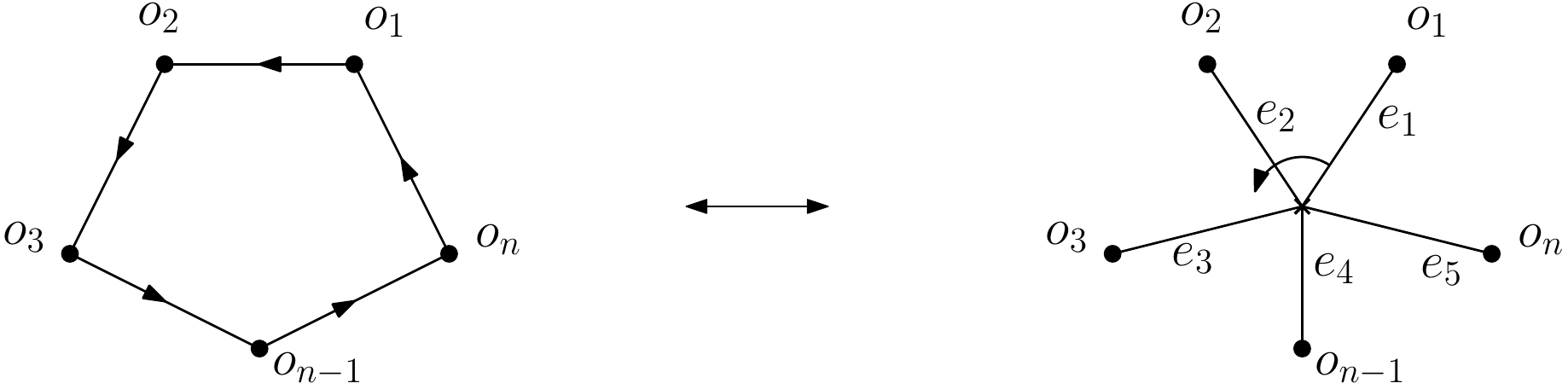}
\end{equation}
One can associate the topology of the disc to the oriented cycle on the left, by defining the interior as the region on the left of the oriented edges. Then \eqref{CycleToMap} is just the somewhat trivial equivalence between polygons and star-vertices.

\subsection{Pairing} Consider a bubble $\vec{\bb}$ with $2p$ labeled vertices. To be able to apply the above principle, we first need a pairing on $\bb$, i.e. a permutation $\pi\in \mathfrak{S}_p$. If $v$ is a white vertex, denote $\pi(v)$ its black partner. Say that $v$ and $\pi(v)$ are connected by $d-p$ edges for $p\in[0..d]$. There are $p$ edges with colors $c_1, \dotsc, c_p$ connecting $v$ to other vertices, and $p$ other edges of the same colors connecting $\pi(v)$ to other vertices. We merge $v$ and $\pi(v)$ together into a blue vertex, remove the $d-p$ edges which connected them, and orient the $2p$ edges ingoing if they were connected to $\pi(v)$ and outgoing if they were connected to $v$, while retaining their colors,
\begin{equation} \label{PairContraction}
	\includegraphics[scale=.6,valign=c]{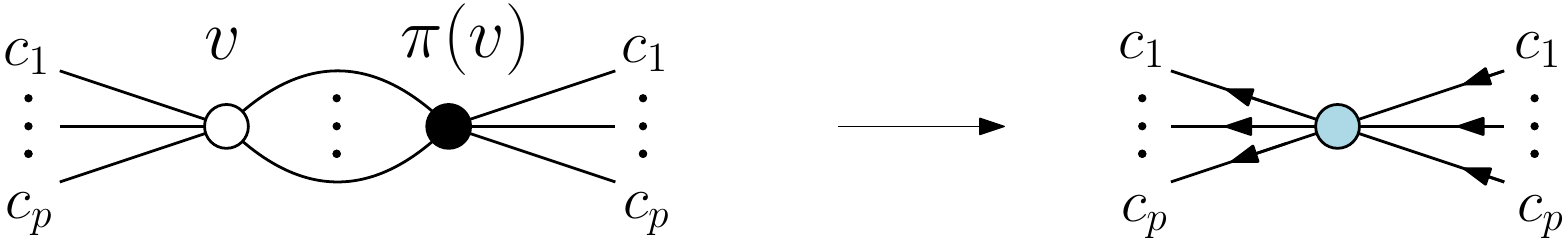}
\end{equation}

\subsection{Paired bubble} Applying \eqref{PairContraction} to every pair $(v, \pi(v))$, one obtains a \emph{paired bubble} $(\vec{\bb},\pi)$, which is an oriented graph with colored edges. Its subgraph of color $c\in[1..d]$ is the subgraph $(\vec{\bb}, \pi, c)$ obtained by removing all edges except those of color $c$. Clearly $(\vec{\bb}, \pi, c)$ is a disjoint union of $n$ oriented cycles of length $l_1,\dotsc, l_n$. They can be assembled into a surface of genus 0 with $n$ boundaries of length $l_1, \dotsc, l_n$. In this sense, $(\vec{\bb}, \pi, c)$ is a form of elementary cell of higher topology as those considered in stuffed maps. The paired bubble then has $d$ surfaces of higher topologies (one for each color; it can be empty for some) and the surfaces with different colors can share vertices. There is no order between edges of different colors at shared vertices, so although they are glued to the same vertices, they pivot with respect to each other.

It will be convenient to turn each cycle of color $c$ into a star-vertex using \eqref{CycleToMap}. An example of the process turning a bubble with a pairing into the paired bubble is shown in Figure \ref{fig:BubbleBijection}.
\begin{figure}
	$\bb=$\includegraphics[scale=.4,valign=c]{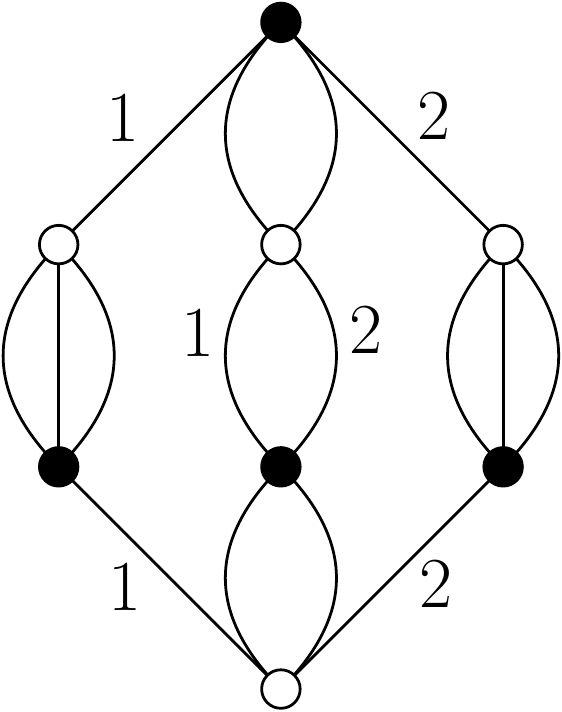} $\overset{\pi}{\to}$ \includegraphics[scale=.4,valign=c]{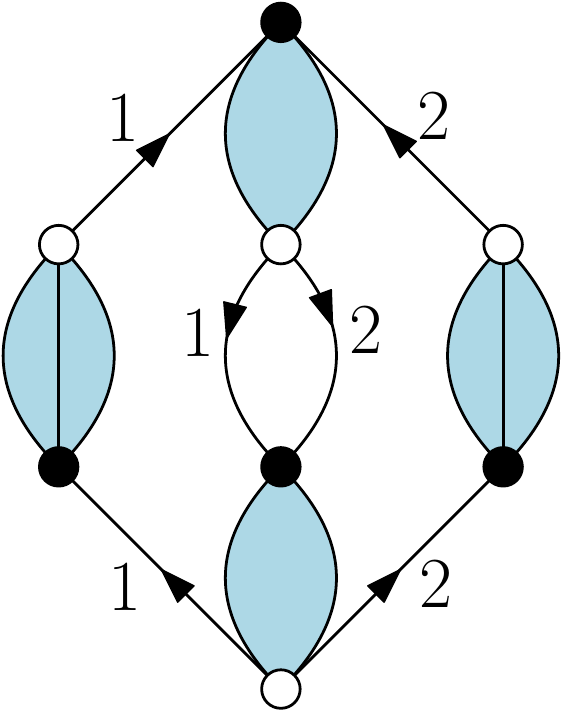} $\overset{\text{Merging}}{\to}$ \includegraphics[scale=.4,valign=c]{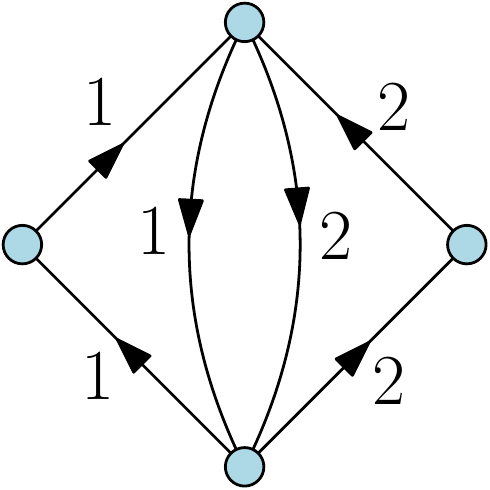} $\overset{\text{star-vertices}}{\to}$ \includegraphics[scale=.4,valign=c]{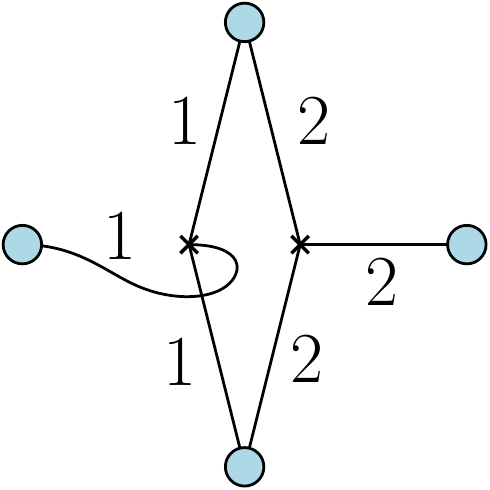}
	\caption{\label{fig:BubbleBijection} We start from a bubble with 4 colors, choose a pairing $\pi$ represented in blue, merge the vertices of each pair into blue vertices with oriented edges, then represent the oriented cycles with star-vertices, one for each cycle of each color.}
\end{figure}

\subsection{Gluing of paired bubbles} Let $\cG_{n_1, \dotsc, n_N}(\vec{\bb}_1, \dotsc, \vec{\bb}_N)$ the set of connected colored graphs built from $n_i$ copies of the bubbles $\vec{\bb}_i$ with labeled vertices. Consider $\graph\in\cG_{n_1, \dotsc, n_N}(\vec{\bb}_1, \dotsc, \vec{\bb}_N)$ and apply the above construction to all copies of $\vec{\bb}_1, \dotsc, \vec{\bb}_N$. Here the same pairing $\pi_i$ is used for all copies of $\vec{\bb}_i$. This construction takes care of the edges of all colors in $\graph$ except those of color 0. We thus apply the same technique as we did for every color of every bubble, except we use the opposite convention to orient the edges, i.e. from black to white.

After merging the vertices of every pair in $\graph$ into blue vertices, and orienting the remaining edges (from white to black for the colors $c\in[1..d]$ and from black to white for the color 0), the subgraph with color 0 only is also a disjoint union of oriented cycles. We turn them into star-vertices, as illustrated in Figure \ref{fig:CycleColor0}. We use crosses for the star-vertices of colors $c\in[1..d]$ and squares for those of color 0.
\begin{figure}
	\includegraphics[scale=.6,valign=c]{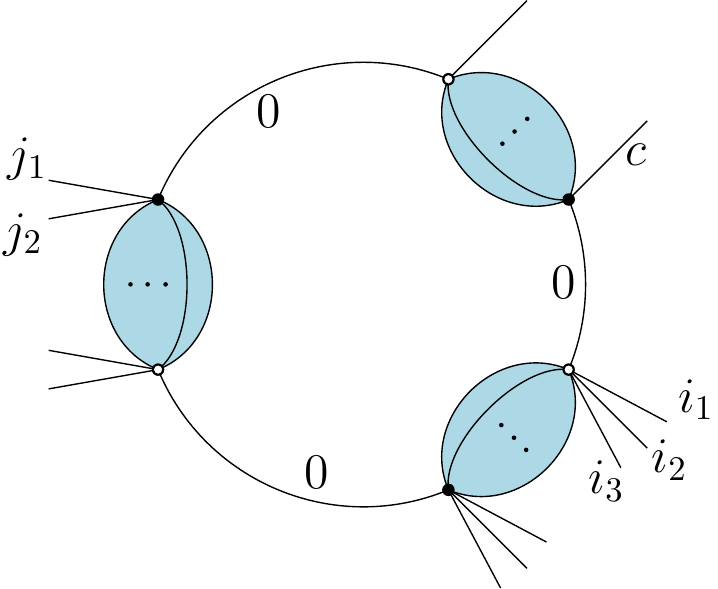} $\overset{\pi\text{ merging}}{\to}$ \includegraphics[scale=.6,valign=c]{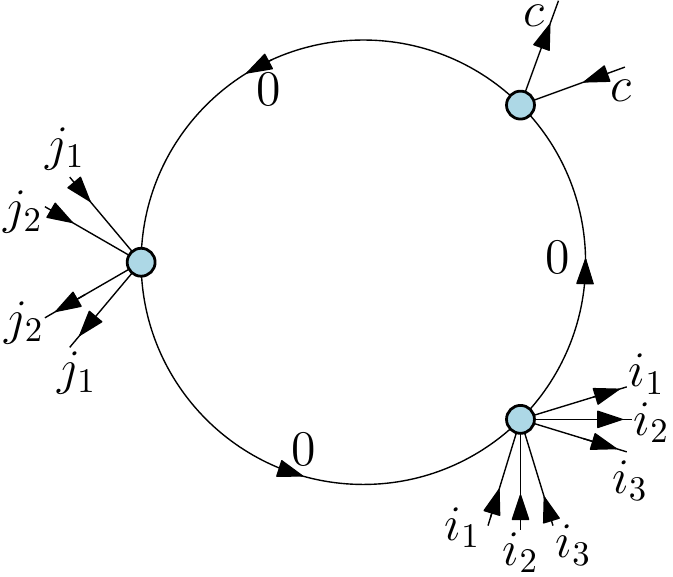} $\overset{\text{star-vertices}}{\to}$ \includegraphics[scale=.6,valign=c]{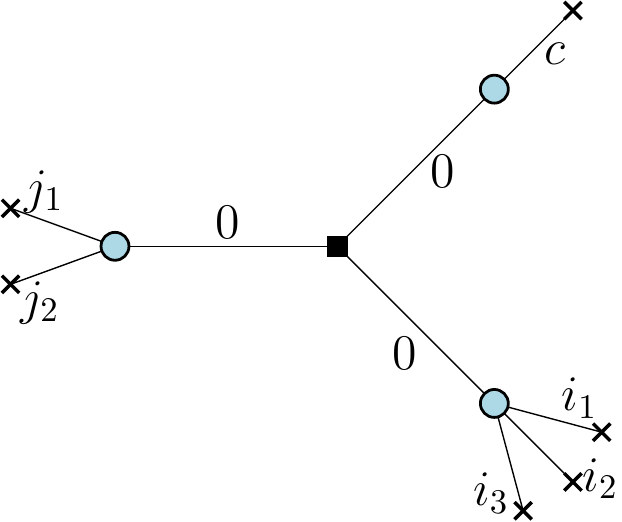}
	\caption{\label{fig:CycleColor0}A cycle made of color 0 and paired vertices (in light blue), then after merging the vertices of every pair into blue vertices and orienting edges from white to black, then after transforming all oriented cycles of each color into a star-vertex.}
\end{figure}

\subsection{The bijection} Let $\cM_{n_1, \dotsc, n_N}((\vec{\bb}_1, \pi_1), \dotsc, (\vec{\bb}_N, \pi_N))$ be the set of connected objects built from $n_i$ copies of the paired bubbles $(\vec{\bb}_i, \pi_i)$ connected by square-vertices of color 0 like in Figure \ref{fig:CycleColor0}. Although they are not maps, because the colors $[1..d]$ do not see each other, we will call them colored maps by simplicity when there is no possible confusion. In fact, for $\m\in\cM_{n_1, \dotsc, n_N}((\vec{\bb}_1, \pi_1), \dotsc, (\vec{\bb}_N, \pi_N))$, the restriction $\m_c$ to edges of colors 0 and $c\in[1..d]$ is a (possibly not connected) map. Moreover one has the following theorem.
\begin{theorem}{}{StuffedMapsBijection} \cite{StuffedColoredMaps}
	There is a bijection between $\cG_{n_1, \dotsc, n_N}(\vec{\bb}_1, \dotsc, \vec{\bb}_N)$ and $\cM_{n_1, \dotsc, n_N}((\vec{\bb}_1, \pi_1), \dotsc, (\vec{\bb}_N, \pi_N))$ which maps the bicolored cycles with colors $\{0,c\}$, for $c\in[1..d]$, to the faces of $\m_c$.
\end{theorem}
%Notice that there is in fact \emph{no order between the edges incident to a blue vertex}: the colors $[1..d]$ do not see each other, and in $\m_c$ the blue vertices are either bivalent or univalent.

\subsection{Examples}
\subsubsection{Cyclic bubbles} Consider a family of bubbles with cyclic symmetry which generalize the quartic bubbles. They are given by parallel edges with colors $c_1, \dotsc, c_q$, and parallel edges of the complementary colors, as shown in Figure \ref{fig:NecklaceBijection}. We pair the vertices connected by the edges with colors $c_1, \dotsc, c_q$. For $i=1, \dotsc, q$, $(\bb, \pi)$ has a single oriented cycle of color $c_i$. Obviously, they are all the same cycle, and we represent them as a single cycle whose edges carry the color set $\colset = \{c_1, \dotsc, c_q\}$. We then turn it into one star-vertex, all edges connecting to the blue vertices having color set $\colset$. This is shown in Figure \ref{fig:NecklaceBijection}.

If we had applied the bijection strictly, we would have $q$ star-vertices and each blue vertex would be incident to $q$ edges of different colors. Because there is no order between edges of different colors, it makes sense to pack them into ``super-edges'' which carry the color set $\colset$ and which are incident to a single star-vertex.
\begin{figure}
	\includegraphics[scale=.7,valign=c]{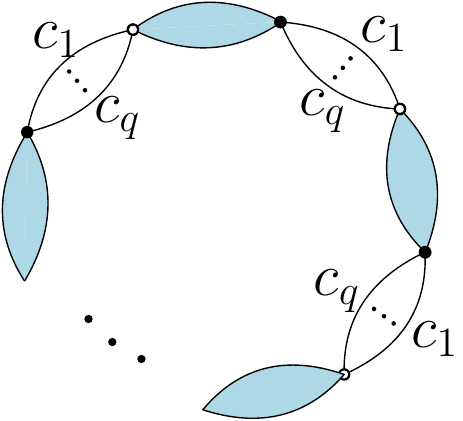} $\overset{\text{Merging}}{\to}$ \includegraphics[scale=.7,valign=c]{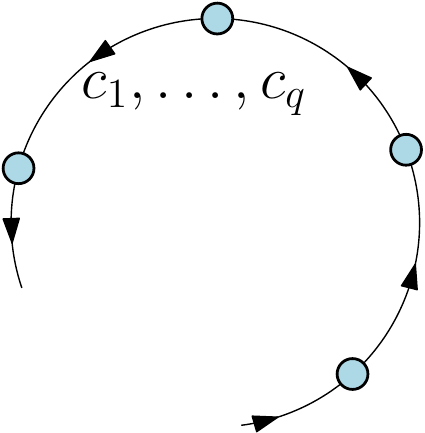} $\overset{\text{Star-vertex}}{\to}$ \includegraphics[scale=.7,valign=c]{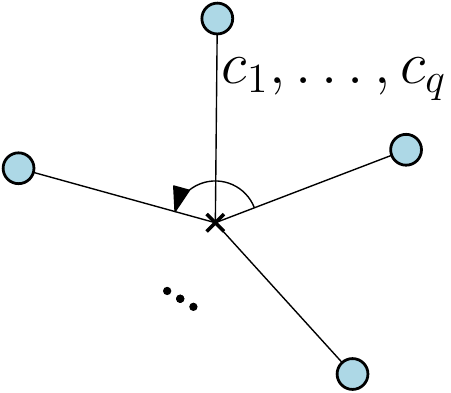}
	\caption{\label{fig:NecklaceBijection}On the left is a cyclic bubble with a pairing in blue. After merging, all cycles of colors $c_i$ are the same and can be represent as a single cycle with color set $\{c_1, \dotsc, c_q\}$. This single cycle can then be turned into a single star-vertex connected to blue vertices by edges carrying the color set $\{c_1, \dotsc, c_q\}$.} 
\end{figure}

We denote $\bb_p(\colset)$ the bubble with $p$ pairs and $\colset = \{c_1, \dotsc, c_q\}$. As a corollary of \Cref{thm:StuffedMapsBijection}
%Since the pairing is compatible with the cyclic symmetry, we find .
\begin{corollary}{}{CyclicBubblesBijection}
	There is a bijection between the set of connected rooted colored graphs with $n_p(\colset)$ copies of the bubbles $\bb_p(\colset)$ and the set of connected rooted maps $\m$ made of $n_p(\colset)$ star-vertices of degree $p$, connected to blue vertices only, with edges of colors $\colset$, and whose square-vertices have arbitrary degree and are connected to blue vertices only. The bijection maps bicolored cycles of colors $\{0,c\}$ to faces of $\m_c$.
\end{corollary}

\subsubsection{The cube} The colored cube bubble is the dual to the octahedron, see Figure \ref{fig:Octahedron}. Let us pair the vertices which are connected by the edges of color 3, see Figure \ref{fig:OctahedronBijection}. After merging the paired vertices, one finds two oriented cycles of color 1 and two of color 2, all of length 2. Therefore the star-vertices are bivalent and can be removed, as in Figure \ref{fig:OctahedronBijection}. \Cref{thm:StuffedMapsBijection} gives the following corollary.
\begin{figure}
	\includegraphics[scale=.7,valign=c]{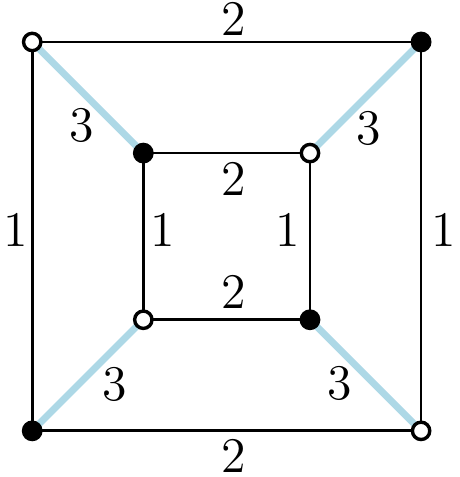} $\overset{\text{Merging}}{\to}$ \includegraphics[scale=.7,valign=c]{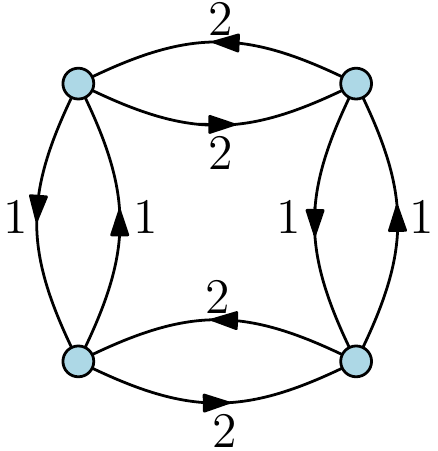} $\overset{\text{Star-vertex}}{\to}$ \includegraphics[scale=.7,valign=c]{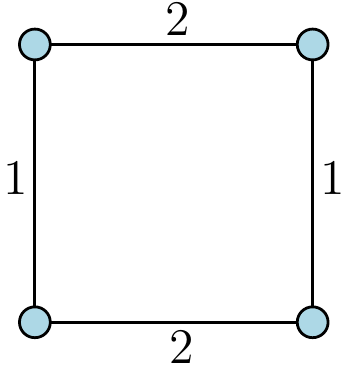}
	\caption{\label{fig:OctahedronBijection}On the left is the tricolored bubble dual to the octahedron, equipped with a pairing. After merging, all cycles of color 1 and 2 have length 2, and as star-vertices have degree 2, hence can be represented as edges.} 
\end{figure}

\begin{corollary}{}{}\cite{Octahedra}
	There is a bijection between connected rooted colored gluings of $n$ octahedra and maps made of $n$ squares with colors 1, 2 and blue vertices, and square-vertices of arbitrary degree connected to blue vertices only. The bicolored cycles with colors $\{0,1\}$, $\{0,2\}$ and $\{0,3\}$ are mapped to the faces of $\m_1$, $\m_2$ and to square-vertices respectively.
\end{corollary}

\subsubsection{Other examples} Quartic bubbles will be treated in the section below. We refer to \cite{StuffedColoredMaps, LionniThurigen, Lionni} for more examples at $d=3, 4$, such as bipyramids which generalize the octahedron.

%%%%%%%%%%%%%%%%%%%%%%%%%%
\section{Multicritical regime in even dimensions} \label{sec:EvenDim}
The bijection we have described has proved useful in several cases in identifying $\cG^{\max}_{n_1, \dotsc, n_N}(\bb_1, \dotsc, \bb_N)$. In the case of quartic bubbles in particular, it leads to an interesting set of planar maps, which can interpolate between the universality class of large uniform random trees and large uniform random planar maps.

\subsection{Gluings of quartic bubbles as edge-colored maps}
\subsubsection{Admissible color sets} We recall that quartic bubbles, see Figure \ref{fig:Bubbles}, are characterized by a color set $\colset\subset [1..d]$, but its complement $[1..d]\setminus \colset$ gives the same bubble. Here we introduce a convention: we say that $\colset$ is admissible if
\begin{itemize}
	\item if $|\colset|< d/2$,
	\item or $|\colset|=d/2$ and $1\in \colset$.
\end{itemize}
%We denote $\mathfrak{C}$ the set of such admissible color sets. 
We only consider admissible color sets in the following.

\subsubsection{The bijection} The quartic bubbles are special cases of the cyclic bubbles considered in the previous example. When applying \Cref{thm:CyclicBubblesBijection}, it turns out that the star-vertex has degree 2 because $p=2$, with the same color set label on both sides, see Figure \ref{fig:NecklaceBijection}, and can thus be removed. Moreover, the blue vertices also have degree 2 and the same color set label on both sides, and can also be removed. Every quartic bubble with color set $\colset$ is therefore mapped to an edge with label $\colset$. The only vertices are the squares, around which the cyclic order matters. The bijection is thus with general maps whose edges each have a color label $\colset\subset[1..d]$, as illustrated in Figure \ref{fig:QuarticBijection}.
\begin{figure}
	\includegraphics[scale=.4,valign=c]{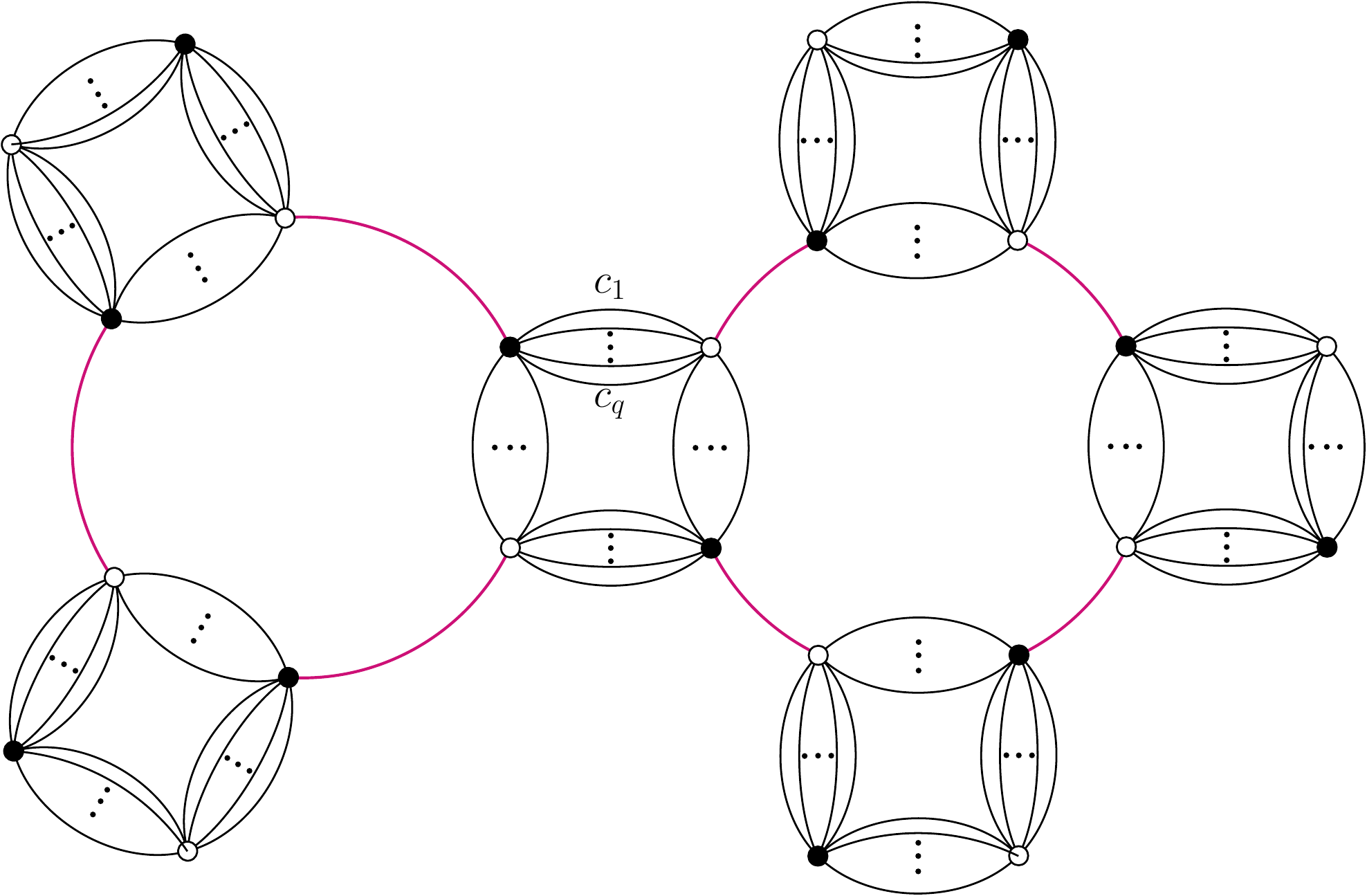}$\quad \leftrightarrow \quad$ \includegraphics[scale=.4,valign=c]{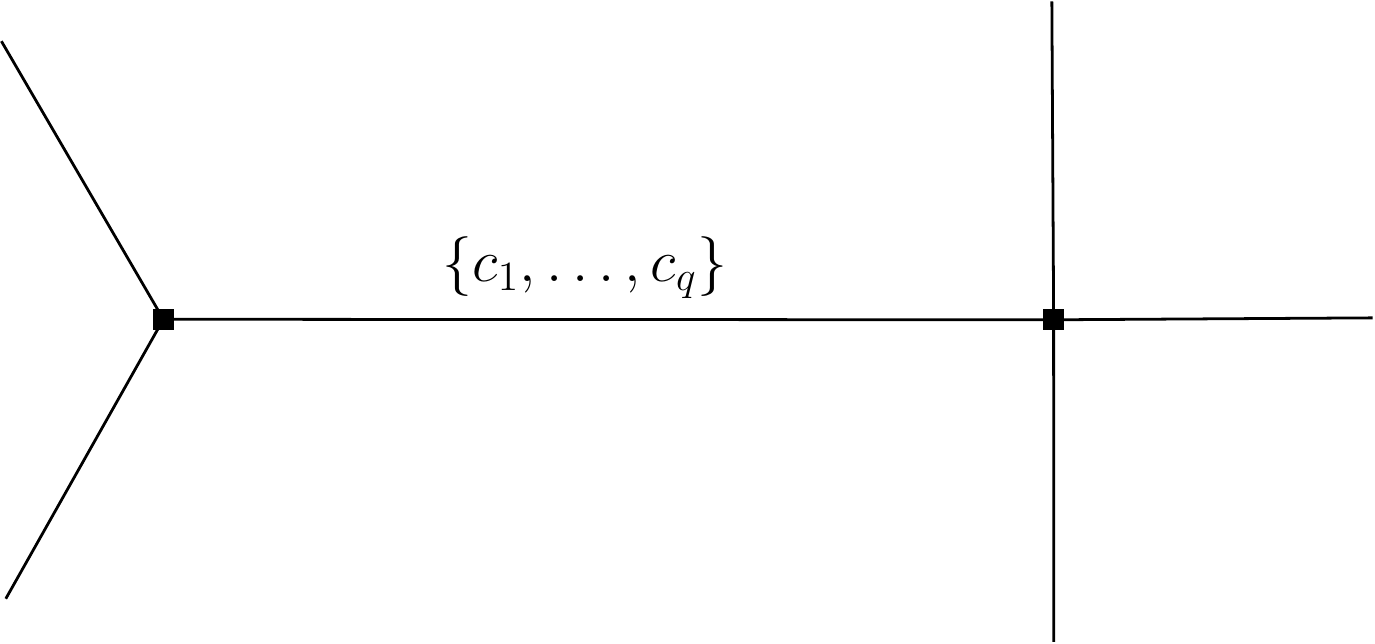}
	\caption{\label{fig:QuarticBijection} The bijection for quartic bubbles with general maps with colored edges.}
\end{figure}
\begin{corollary}{}{}
	Let $\colset_1, \dotsc, \colset_N$ be a collection of color sets. There is a bijection $J$ between $\cG_{n_1, \dotsc, n_N}(\quart(\colset_1), \dotsc, \quart(\colset_N))$ and the set $\cM_{n_1, \dotsc, n_N}(\colset_1, \dotsc, \colset_N)$ of general maps with $n_i$ edges of color labels $\colset_i$, $i=1, \dotsc, N$, and such that the bicolored cycles with colors $\{0,c\}$ are mapped to the faces of $\m_c$, $c\in[1..d]$. In particular,
	\begin{equation}
		C_0(\graph) = \sum_{c=1}^d F(\m_c).
	\end{equation}
\end{corollary}
\begin{remark}
	This bijection is classical from the point of view of tensor and matrix models. It was in fact the motivation to write the more general one of \Cref{thm:StuffedMapsBijection}. 
\end{remark}

\subsection{Maps which maximize the number of faces} 
If $\m\in\cM_{n_1, \dotsc, n_N}(\colset_1, \dotsc, \colset_N)$ we denote $\m(\colset_i)$ the (possibly not connected) submap whose edges have exactly the color set $\colset_i$. 
\begin{theorem}{}{MelonoPlanar}\cite{Enhancing,GM} 
	The model $\cG_{n_1, \dotsc, n_N}(\quart(\colset_1), \dotsc, \quart(\colset_N))$ satisfies the linear growth hypothesis \eqref{LinearGrowth} with $\alpha(\quart(\colset_i)) = d-|\colset_i|$, i.e.
	\begin{equation}
		C_{n_1, \dotsc, n_N}(\quart(\colset_1), \dotsc, \quart(\colset_N)) = d + \sum_{i=1}^N (d-|\colset_i|)n_i.
	\end{equation}
	Define $\cM^{\max}_{n_1, \dotsc, n_N}(\colset_1, \dotsc, \colset_N) = J(\cG^{\max}_{n_1, \dotsc, n_N}(\quart(\colset_1), \dotsc, \quart(\colset_N)))$ the set of maps maximizing the number of faces $\sum_{c=1}^d F(\m_c)$. It corresponds to the set of maps satisfying
	\begin{itemize}
		\item $\m$ is planar,
		\item Edges with $|\colset| <d/2$ are bridges,
		\item $\m(\colset_i)$ and $\m(\colset_j)$ for $i\neq j$ can only meet at cut-vertices.
	\end{itemize}
\end{theorem}
This theorem completely solves the questions of Section \ref{sec:MainQuestion} from Chapter \ref{sec:Bubbles} for models with quartic bubbles. The coefficients of the linear growth hypothesis are here explicit, in contrast with those of the case $d=3$ with planar bubbles where we had the generically unknown parameters $C(\bb)$. Moreover, the coefficients coincide with the general lower bound \eqref{LowerBound}, $\alpha(\bb)\leq C(\bb)-d$. Indeed, $C(\quart(\colset)) = 2d-|\colset|$.

That lower bound was found by looking at the maximal 2-cut family, but here while the bound is reached, the graphs from $\cG^{\max}_{n_1, \dotsc, n_N}(\quart(\colset_1), \dotsc, \quart(\colset_N))$ do not all satisfy the maximal 2-cut property. Through the bijection $J$, a 2-bond of color 0 is mapped to a cut-vertex\footnote{i.e. the edges of color 0 of the 2-bond are mapped to a pair of corners at a cut-vertex} and thus colored graphs whose bubbles all satisfy the maximal 2-cut property are mapped to plane trees.

While some vertices in the Theorem have to be cut-vertices, not all of them have to. If a vertex is only incident to edges with a single color set which satisfies $|\colset_i|=d/2$, then it may not be a cut-vertex. For instance, the Theorem allows for all planar maps whose edges are made of the same color set $|\colset|=d/2$. Through the bijection $J$, the lower bound \eqref{LowerBound} holds for planar maps and not just plane trees.

We can once again compare the coefficients $\alpha(\bb_i)$ to the values from Gurau's bound \eqref{GurauValue}. The latter gives $\alpha(\quart(\colset))\leq |\colset|(d-|\colset|)$. It coincides with the exact value when $|\colset|=1$ (i.e. for melonic bubbles only).

\subsection{Decomposition on non-separable planar maps and enumeration} 
\subsubsection{Non-separable planar maps} A rooted non-separable planar map is a rooted planar map such that removing any vertex does not disconnect it. Equivalently it is a rooted planar map whose underlying graph is 2-connected. A rooted planar map $\m$ has a unique maximal non-separable submap\footnote{It is the unique maximal 2-connected subgraph, equipped with the rotation system induced by that of the original map.} $\m_{\text{n.s.}}$, and it can be reconstructed from $\m_{\text{n.s.}}$ by attaching rooted maps (possibly just a vertex) in every corner, as in Figure \ref{fig:NonSeparableDecomposition}.
\begin{figure}
	\includegraphics[scale=.55]{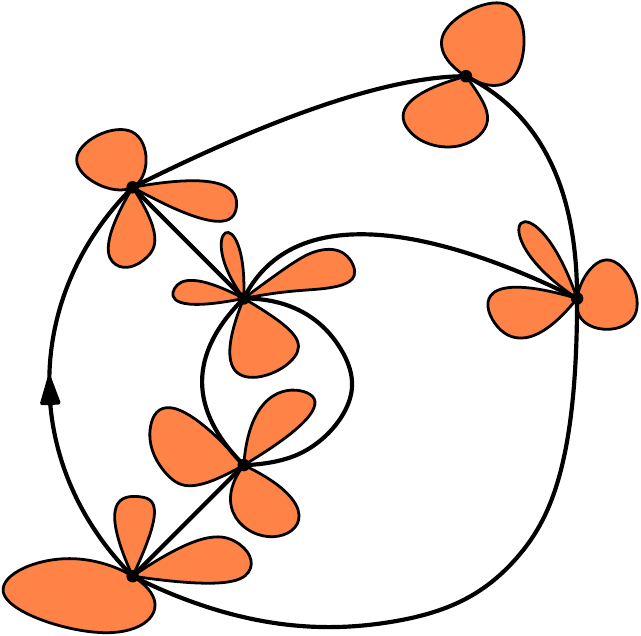}
	\caption{\label{fig:NonSeparableDecomposition} A rooted planar map can be decomposed with a unique maximal non-separable component containing the root edge, and by attaching in every corner a rooted planar map.}
\end{figure}

Let $M(t)$ (respectively $P(t)$) be the generating series of rooted planar maps (respectively non-separable rooted planar maps) counted with respect to the number of edges. Since the number of corners is twice the number of edges, one can consider that there are two maps potentially attached to every edge and therefore \cite{GouldenJacksonBook}
\begin{equation}
	M(t) = P(t(M(t)^2)
\end{equation}
which is exactly the same equation as \eqref{2IrreducibleDecomposition} and similar to \eqref{SIFEquation}. $P(t)$ can be described by the following algebraic system \cite{GouldenJacksonBook}
\begin{equation} \label{NonSeparableSystem}
	\left\{
	\begin{aligned} t &= \theta (1-\theta)^2\\
		P &= (1-\theta)(1+3\theta). \end{aligned}
	\right.
\end{equation}

\subsubsection{Enumeration} Let
\begin{equation}
	\cM^{\max}(\colset_1, \dotsc, \colset_N) = \bigcup_{n_1, \dotsc, n_N\geq 0} \cM^{\max}_{n_1, \dotsc, n_N}(\colset_1, \dotsc, \colset_N).
\end{equation}
In \Cref{thm:MelonoPlanar}, all edges whose color set has $|\colset_i|<d/2$ are bridges. They are very easy to take into account in the enumeration, and removing them entirely does not change the phase diagram, as argued below, so we do not consider them, which simplifies the equations. So we consider $|\colset_i|=d/2$ for $i=1..N$. For $\m\in \cM^{\max}_{n_1, \dotsc, n_N}(\colset_1, \dotsc, \colset_N)$, the number of edges is $E(\m) = \sum_{i=1}^N n_i$. Then consider the generating series
\begin{equation}
	f_{N}(t) = \sum_{\m\in\cM^{\max}(\colset_1, \dotsc, \colset_N)} t^{E(\m)}.
\end{equation}

$f_N(t)$ starts with $1$ (the 1-vertex map). All other terms have at least one edge. Let $\m\in \cM^{\max}_{n_1, \dotsc, n_N}(\colset_1, \dotsc, \colset_N)$ rooted on an edge $e$ with $\colset_j$. Following the same reasoning as in \cite{GouldenJacksonBook} for general maps, there is a unique non-separable connected submap $\m_{\text{n.s.}}$ which contains the root edge. In our case, the key point is that all edges of $\m_{\text{n.s.}}$ have the same color set $\colset_j$ and $\m_{\text{n.s.}}$ is thus an arbitrary non-separable, rooted planar map. Then $\m$ can be recovered uniquely from $\m_{\text{n.s.}}$ by inserting a rooted map on every corner of $\m_{\text{n.s.}}$. The corner insertions can be arbitrary rooted maps from $\cM^{\max}(\colset_1, \dotsc, \colset_N)$.

The contribution of the maps rooted on an edge of type $\colset_j$ is thus $P(t f_N(t)^2)-1$ ($P$ contains the 1-vertex map, which we take into account seperately). Adding the $N$ types of root edge gives
\begin{equation}
	f_N(t) = 1 - N + N\,P\bigl(t\, f_{N}(t)^2\bigr).
\end{equation}
From the algebraic system \eqref{NonSeparableSystem} satisfied by $P(t)$, it comes
\begin{equation} \label{AlgebraicSystem}
	\left\{ \begin{aligned}
		t f^2 &= \theta (1-\theta)^2 \\
		f &= N\,(1-\theta)(1+3\theta) - N + 1,
	\end{aligned} \right.
\end{equation}
and $f_N(t)$ is the solution of this system for $f$ after elimination of $\theta$\footnote{Doing so leads to an explicit polynomial equation of degree 6. As a sanity check, at $N=1$, $f_N(t)$ should be enumerating general planar maps. Then the polynomial equation does indeed reduce to the well-known quadratic equation on the generating function of rooted planar maps,
\begin{equation} \label{MapsGF}
	27 t^2 A(t)^2 + (1 - 18t) A(t) + 16t - 1 = 0,
\end{equation}
with $A(t) = f_{N=1}(t)$.}.
 
\subsubsection{Phase diagram} \label{sec:PhaseDiagram} {\it A priori}, $N$ is an integer. Notice however that the algebraic system \eqref{AlgebraicSystem} makes sense for $N$ a positive real number.
\begin{proposition}{}{BabyUniverses}
	\begin{itemize}
		\item For $N<9/5$, $f_N^{(1)} = 1 + \frac{N}{3}$, $t_N^{(1)} = \frac{4}{3 (N+3)^2}$. Then
		\begin{multline}
			f_N(t) = f_N^{(1)} - \frac{N(N+3)^3}{9 - 5N} (t_N^{(1)} - t) + \sqrt{3}\, N \frac{(N + 3)^{11/2}}{(9 - 5N)^{5/2}}\,(t_N^{(1)} - t)^{3/2} + o\bigl((t_N^{(1)} - t)^{3/2}\bigr),
		\end{multline}
		\item For $N>9/5$, $f_N^{(2)} = 4\bigl(1 - N + \sqrt{N(N-1)}\bigl)$, $t_N^{(2)} = \frac{N+\sqrt{N(N-1)}}{16\,N^2}$. Then
		\begin{multline}
			f_N(t) = f_N^{(2)} + 16 \sqrt{N \bigl(20 N^3 - 31 N^2 + 11 N - \sqrt{N (N-1)} (20 N^2 - 21 N + 3)\bigr)} (t_N^{(2)} - t)^{1/2} \\+ o\bigl((t_N^{(2)} - t)^{1/2}\bigr).
		\end{multline}
		\item At $N=9/5$, the two solutions $f_N^{(i)}, t_N^{(i)}$ coincide for $i=1,2$. Then
		\begin{equation}
			f_{9/5}(t) = \frac{8}{5} - \frac{432}{25 \times 5^{1/3}} \Bigl(\frac{25}{432} - t\Bigr)^{2/3} + o\Bigl(\Bigl(\frac{25}{432} - t\Bigr)^{2/3}\Bigr).
		\end{equation}
	\end{itemize}
\end{proposition}

That phase diagram can be understood heuristically as follows.
\begin{itemize}
	\item For $N$ small enough, a large typical map consists of planar components of color types $\colset_j$ connected by a finite number of cut-vertices. Therefore the criticality is expected to be that of planar maps, with a singularity $(t_c(q,u) - t)^{3/2}$.
	%resembles a generic planar map with a single color type, as branching via monocolored edges (weight $\lambda$) and via the $l$ different types of bicolored edges is weak. The singularity is thus expected to be of the form $(t_c - t)^{3/2}$, i.e. the universality class of random planar maps.
	\item For $N$ large enough, the probability of adding a component of a different color type at an existing vertex becomes large enough that cut-vertices separating planar components of different color types dominate. The planar components of fixed color type remain non-critical. The maps are thus dominated by branching processes, so a singularity $(t_c(N) - t)^{1/2}$ is expected, i.e. the universality class of trees.
	\item Between those two phases, there is a regime where planar components of type $\colset_j, j=1..N$ are in infinite number (connected by infinitely many uncolored edges or cut--vertices) and each of them may be critical too. This phase is the so-called proliferation of baby universes and the expected singularity is $(t_c - t)^{2/3}$.
\end{itemize}

In \cite{LionniThurigen}, Lionni and Thürigen found more models fitting the following general equation
\begin{equation}
	f(t) = Q_0(f(t)) + \sum_{j=1}^m P(Q_j(f(t))
\end{equation}
where $Q_0, Q_1, \dotsc, Q_m$ are polynomials. They solved it in some cases and conjecture that the same phase portrait as above holds. Heuristically, the map and tree phases correspond to either $Q_0(f(t))$ being negligible compared to the other terms, so that one gets an equation similar to the one for planar maps, or $Q_0(f(t))$ dominating so that one gets the usual polynomial equation for trees.

\subsubsection{Scaling limit} The phase of proliferation of baby universes has appeared in the physics literature for some time (see below) but only enumerative results like those above have been known and the large scale limit has remained to be described. No suprises are expected for $N$ small, as it should be the Brownian sphere, extending the convergence known for $N=1$, and for $N$ large, as it should be the continuous random tree of Aldous since the geometry is dominated by tree-like structures. However, in the phase of the proliferation of baby universes, the large scale limit was not known and it did not seem straightforward, neither to guess nor to establish. Recently, Z. Salvy and W. Fleurat\footnote{To appear, and from a research proposal of M. Albenque and É. Fusy on one side, and of G. Miermont on the other.}, studied the exact model of \Cref{thm:BabyUniverses} and found that the large scale limit is the \emph{stable tree of index $3/2$}.

\subsection{Planar stuffed maps}
\subsubsection{Stuffed maps with cylinders} The proliferation of baby universes is a phenomenon which has been first studied in the physics literature, in the context of double-trace matrix models in \cite{Das, AlvarezBarbon, Korchemsky1992, KlebanovHashimoto, BarbonDemeterfi}. Those models generate stuffed maps, as defined in Chapter \ref{sec:Definitions}, which are made by gluings of both polygons and cylinders. Here a \emph{cylinder} is an elementary cell of topology $(0,2)$, characterized by definition by the perimeters of its boundaries $\ell_1, \ell_2$. In the generating series, each appearance of a cylinder is weighted with a formal variables $p_{\ell_1, \ell_2}$, just like a disc of perimeter $\ell$ receives a weight $p_\ell$. %One can further decompose the surface of the cylinders if desired, as long as the topology remains that of a cylinder with boundary degrees $\ell_1, \ell_2$.

In the regime which maximizes the number of faces\footnote{That is the large $N$ limit of the matrix model.}, the maps must be planar. This implies in particular that removing a cylinder must disconnect the map. They must therefore be ``cut-cylinders''.

A cylinder with boundaries of perimeters $\ell_1, \ell_2$ can be represented in different ways, e.g. as two polygons of perimeters $\ell_1, \ell_2$ sitting on top of each other, like at a contact point between surfaces. Any way of indicating that the two polygons come from the same cylinder is good too, such as a special edge connecting them, which we draw here as a dashed line and call a \emph{virtual} edge, as in Figure \ref{fig:Cylinder}.

In the dual picture, a cylinder can be pictured as two vertices of respective degrees $\ell_1, \ell_2$ connected by a virtual edge. Moreover, the weights $p_{\ell_1, \ell_2}$ of cylinders in the generating series can now be associated with virtual edges.
\begin{figure}
	\includegraphics[scale=.35,valign=c]{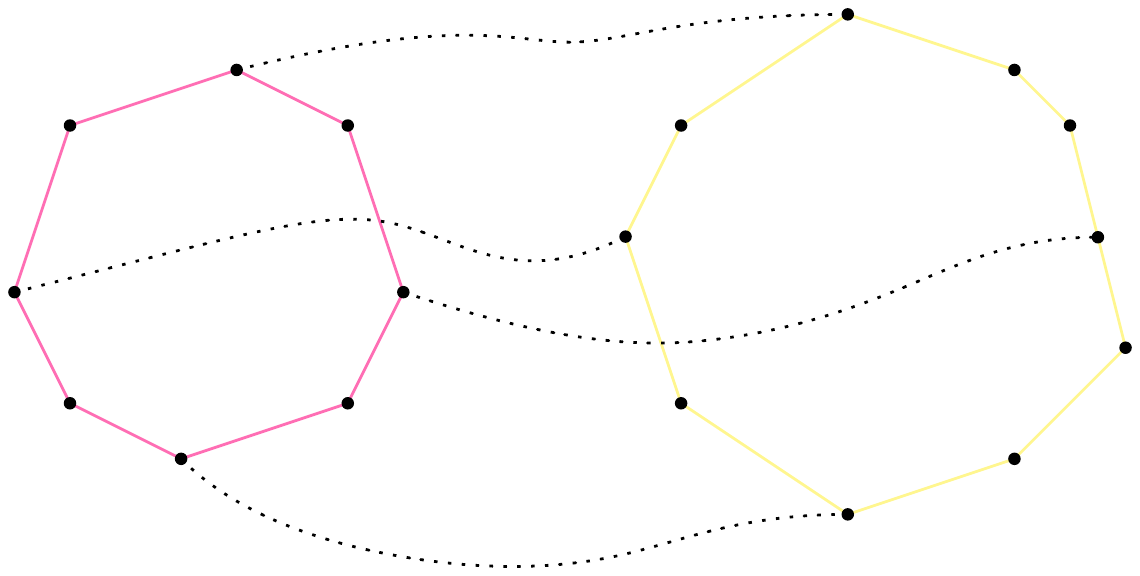} $\leftrightarrow$ \includegraphics[scale=.35,valign=c]{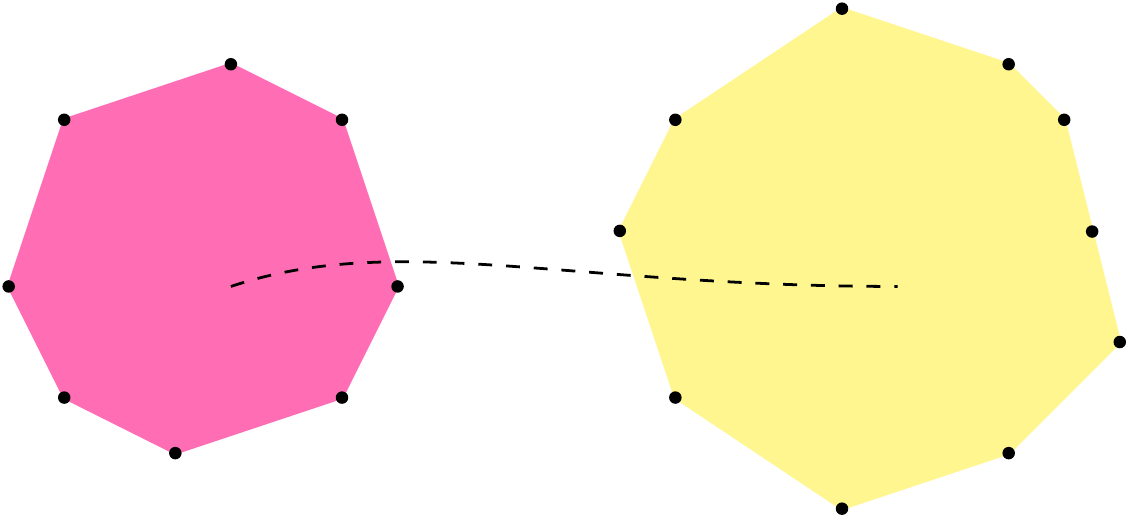} $\leftrightarrow$ \includegraphics[scale=.35,valign=c]{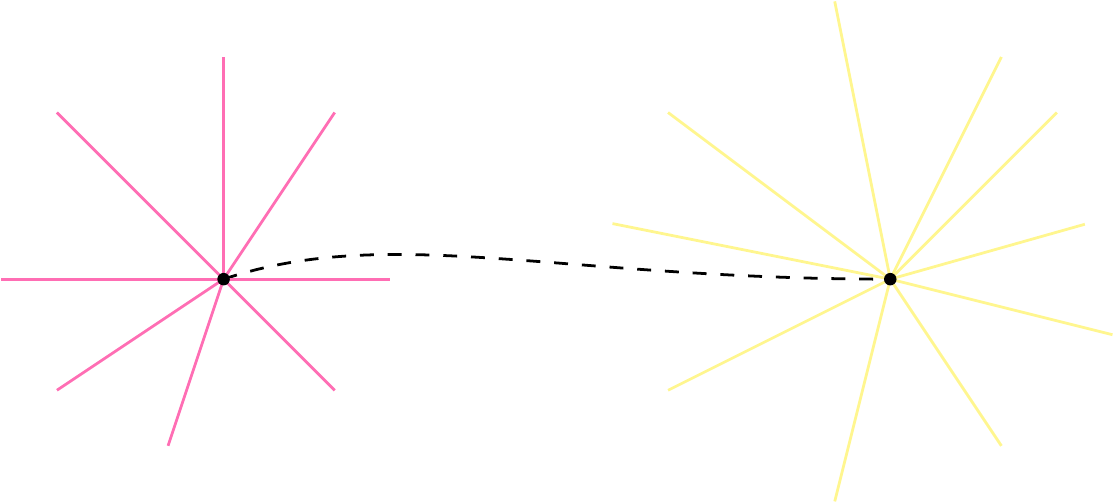} 
	\caption{\label{fig:Cylinder} On the left is a cylinder characterized as a surface of genus 0 with two boundary components of size $\ell_1, \ell_2$. In the middle, an equivalent representation as two polygons with a virtual edge indicating that they come from the same cylinder. On the right, a dual representation where the two polygons have become vertices.}
\end{figure}

Planarity of the stuffed maps translates in the ``maps with virtual edges'' world, as constraints which we are now familiar with: All virtual edges are bridges, and the components obtained after removing them are planar. If one contracts the virtual edges between two vertices (in the dual picture), this creates a cut-vertex.

Comparing stuffed maps with our model, some details are different\footnote{For instance, the two sides of a cylinder can rotate with respect to one another. This is different in our model, as the edges meeting at a cut-vertex cannot be exchanged.}. Nevertheless, one can certainly expect those stuffed maps to lie in the same universality classes at large scale. That is indeed the case at the enumerative level: this model (double-trace matrix model) has the same three universal phases as ours (planar maps, trees, and proliferation of baby universes). %The proliferation of baby universes is thus expected to converge to the stable tree of index $3/2$.

\subsubsection{Enumeration of planar stuffed maps} Tutte's equation for planar stuffed maps is given in \cite{Borot13} in a physicists' convention. Here we will write them in a more traditional manner for combinatorists, in particular {\it à la} Bousquet-Mélou--Jehanne \cite{BousquetJehanne2006} so that we can straightforwardly apply their algebraicity theorem.

Here we consider planar rooted stuffed maps (the root edge is an oriented edge, and the root face is the face to the right of the root edge). They are counted with a weight $p_{\ell_1, \dotsc, \ell_n}$ on elementary cells of topology $(0,n)$ with perimeters $\ell_1, \dotsc, \ell_n$ \cite{Borot13} (we do not consider higher genera here). We consider elementary cells with at most $r$ boundaries and whose boundary perimeters are bounded by $s$. Let $M_k\equiv M_k(t)$ be the generating series for those planar, rooted stuffed maps with root face of degree $k$, and with $t$ counting the number of edges. It is a formal series in $t$ with polynomial coefficients in $p_{\ell_1, \dotsc, \ell_n}$. Moreover, let $M(x) = \sum_{k\geq 0} M_k x^k$ with $M_0=1$ counting the atomic map, also a formal series in $t$, with polynomial coefficients in $x$ and $p_{\ell_1, \dotsc, \ell_n}$.

If $F(x) = \sum_{k\geq 0}F_k x^k$ is a formal series in $x$, we define for $i\geq 1$
\begin{equation}
	\Delta^{(i)}F(x) = \frac{F(x)-F_0-xF_1-\dotsb-x^{i-1}F_{i-1}}{x^i}.
\end{equation}
and $\Delta^{(0)} F(x) :=F(x)$.

\begin{proposition}{}{}
	For $\ell\geq 1$, let 
	\begin{equation}
		Q_\ell(x_1, \dotsc, x_s) = \sum_{n=1}^r \sum_{\ell_1, \dotsc, \ell_n=0}^{s} p_{\ell_1, \dotsc, \ell_n} \sum_{i=1}^n \delta_{\ell, \ell_i} \prod_{j\neq i} \frac{x_{\ell_j}}{\ell_j}.
	\end{equation}
	Then $M(x)$ is determined by the equation
	\begin{equation} \label{TutteStuffedMaps}
		M(x) = 1 + tx^2 M(x)^2 + tx\sum_{\ell=1}^s Q_\ell(M_1, \dotsc, M_s) \Delta^{(\ell-1)} M(x).
	\end{equation}
	and it is algebraic.
\end{proposition}

A comparison with the case of ordinary (non-stuffed) maps is in order. Tutte's equation is given in Lemma 5 in \cite{BousquetJehanne2006}. There, it is shown that ordinary maps counted with variables $q_\ell$ marking the faces of degree $\ell$ satisfy the same equation as \eqref{TutteStuffedMaps} up to the substitution $q_\ell \leftrightarrow Q_\ell(M_1, \dotsc, M_s)$. This reveals the ``stuffed'' aspect: stuffed maps are obtained from ordinary maps by substituting the face weights $q_\ell$ with a gluing of $n\leq r$ surfaces with perimeters bounded by $s$, with weights $p_{\ell_1, \dotsc, \ell_n}$.

We have already spent a lot of time on maps in this memoir while not having given a Tutte decomposition for them. This is time to fix this.
\begin{proof}
	As per usual, one considers a root face of degree $k\geq2$ and removes the root edge and two cases are distinguished. Either the root edge is a bridge, then removing it we obtain two rooted\footnote{For instance on the first edge after the root edge counter-clockwise.} stuffed maps with root faces of degrees $k_1, k_2$ such that $k_1+k_2=k-2$,
	\begin{equation}
		\includegraphics[scale=.65,valign=c]{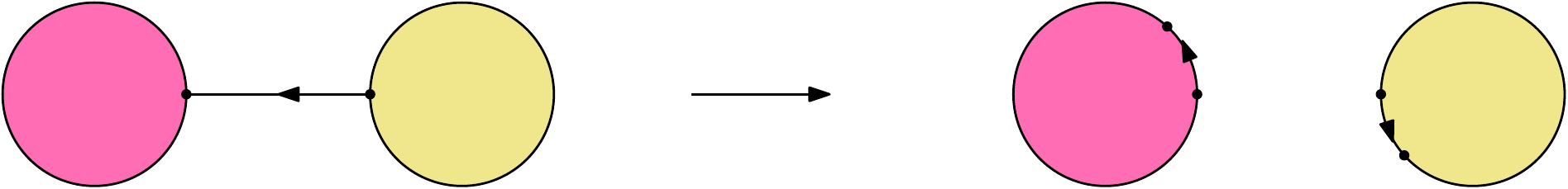}
	\end{equation}
	Or the root edge is not a bridge, then to its left can be any elementary cell of topology $(0,n)$.
	\begin{equation}
		\includegraphics[scale=.65,valign=c]{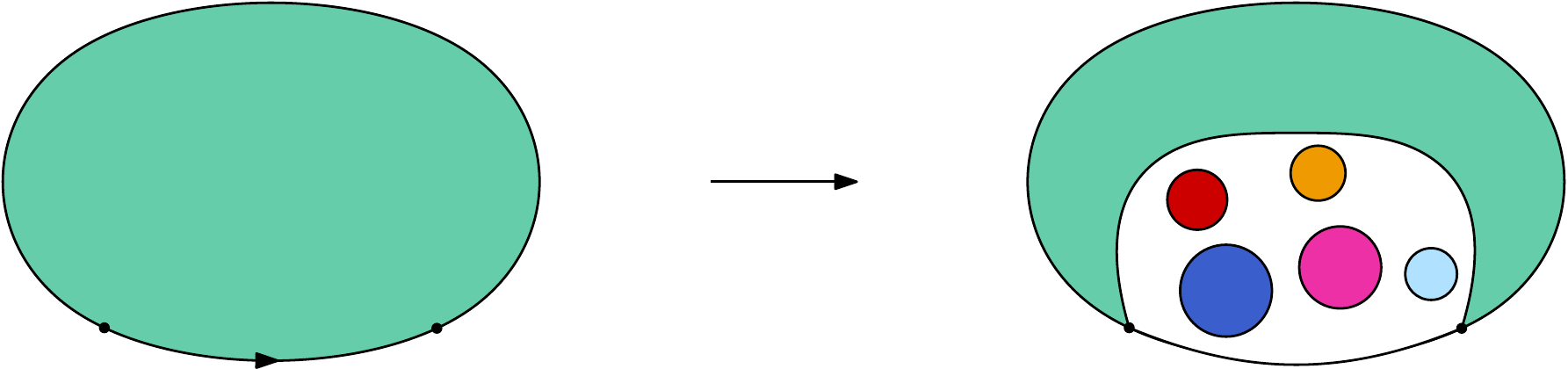}
	\end{equation}
	By planarity, its boundaries are filled with $M_{\ell_1}, \dotsc, M_{\ell_{n-1}}$ for some integers $\ell_1, \dotsc, \ell_{n-1}$, except the one incident to the root face which has size say $\ell$. After removing the root edge, the latter boundary is seen to be filled with $M_{k+\ell-2}$. That leads to
	\begin{equation} \label{SchwingerDysonStuffedMaps}
		M_k = t\sum_{l=0}^{k-2} M_l M_{k-2-l} + t\sum_{\ell=1}^s Q_\ell(M_1, \dotsc, M_r) M_{k+\ell-2}
	\end{equation}
	with the $Q_\ell$s as defined above. Here no assumptions have been made on the coefficients $p_{\ell_1, \dotsc, \ell_n}$ but symmetry assumptions can be used to simplify $Q_\ell$. We also refer to \cite{Borot13} for the higher genera equations.
	
	By expanding all $M_k$s as series in $t$, it is clear that \eqref{SchwingerDysonStuffedMaps} provides an induction which determines them order by order. Now introducing the catalytic variable $x$ and $M(x) = \sum_{k\geq 0} M_k x^k$, Equation \eqref{SchwingerDysonStuffedMaps} becomes \eqref{TutteStuffedMaps}.
	
	Algebraicity follows from a major theorem in the field due to Bousquet-Mélou and Jehanne \cite{BousquetJehanne2006}. Their theorem applies to an equation of the form 
	\begin{equation}
		P(F(x), F_1, \dotsc, F_s, t, x) = 0,
	\end{equation}
	where $P$ is a polynomial, $F(x)$ is a formal series in $t$ with polynomial coefficients in $x$, and $F_i = [x^i]F(x)$. Equation \eqref{TutteStuffedMaps} is exactly of that form (up to a global multiplication by $x^{s-2}$).
\end{proof}

As for the singularities of the $M_k$s with respect to $t$, only specific cases have been considered (in the physics literature, \cite{Das, AlvarezBarbon, Korchemsky1992, KlebanovHashimoto, BarbonDemeterfi}, with explicit calculations performed for $n=2$ and $\ell_1, \ell_2\in\{2,4\}$). They reveal the same phase portrait which we described: the universality class of maps on one side, that of trees on the other side when the weight of elementary cells with $(0,n)$ and $n>1$ dominates the weight of discs, and the proliferation of baby universes in between. % In the general case, there does not seem to be anything new compared to those, neither from the intuitive standpoint stuffed maps, nor if one tries to follow the steps of the existing calculations. From the intuitive picture, adding more boundaries or different boundary perimeters to cylinders does not seem to change the mechanism at work, and this is a manifestation of universality.

%\begin{equation} \label{SchwingerDysonMaps}
%	M(x)^2 - V'(x;M_1, \dotsc, M_s) M(x) + P(x;M_1, \dotsc, M_s) = 0
%\end{equation}
%where
%\begin{equation}
%	V(x;M_1, \dotsc, M_r) = -\frac{x^2}{2t} + \sum_{n=1}^r \sum_{\ell_1, \dotsc, \ell_n=0}^{s} p_{\ell_1, \dotsc, \ell_n} \sum_{i=1}^n \frac{x^{\ell_i}}{\ell_i} \prod_{j\neq i} \frac{M_{\ell_j}}{\ell_j},
%\end{equation}
%and $P(x)$ is a polynomial which we will write in a second. With this expression of $V$, the equation might seem complicated. However, it can be repackaged in a way which makes the ``stuffed'' aspect transparent. Expand $V(x) = \sum_{\ell=1}^{s} \frac{q_\ell}{\ell} x^{\ell}$ as a polynomial in $x$,
%\begin{equation}
%	q_\ell = \delta_{\ell,2}\frac{1}{t} + \sum_{n=1}^r \sum_{\ell_1, \dotsc, \ell_n=0}^{s} p_{\ell_1, \dotsc, \ell_n} \sum_{i=1}^n \delta_{\ell, \ell_i} \prod_{j\neq i} \frac{M_{\ell_j}}{\ell_j} = \delta_{\ell,2}\frac{1}{t} + Q_\ell(M_1, \dotsc, M_s).
%\end{equation}
%Then 
%\begin{equation}
%	P(x;M_1, \dotsc, M_s) = \sum_{\ell=1}^s \sum_{m=0}^{\ell-2} q_\ell M_{\ell-m-2} x^{m}.
%\end{equation}
%In other words, Equation \eqref{SchwingerDysonMaps} is, as a function of the variables $q_\ell$s, the same as for ordinary maps, where $q_\ell-\delta_{\ell,2}\frac{1}{t}$ marks faces of degree $\ell$. In ordinary maps, $q_\ell-\delta_{\ell,2}\frac{1}{t}$ are independent formal variables. Here those faces can additionally be stuffed, which translates as the substitution from independent $q_\ell - \delta_{\ell,2}\frac{1}{t}$ to $Q_\ell(M_1, \dotsc, M_s)$.

\subsubsection{Difference with the $O(n)$-loop model} Another model which might look similar at first sight but is in fact not really is the $O(n)$-loop model. It consists in planar triangulations decorated with self-avoiding, non-intersecting loops, where at each triangle, a piece of loop can only cross two edges. By cutting along the edges of the triangulations on each side of a loop, one gets a triangulated cylinder. Moreover, because the maps are planar, every loop has an inside and an outside region so that removing a cylinder disconnects the map. This shows that the $O(n)$-loop model \emph{is} a model of stuffed maps with discs and cylinders!

However, we know that the $O(n)$-loop model has a very rich phase structure at large scale. It is only possible if loops of arbitrarily large sizes are allowed, a necessary condition for loops to be visible in the scaling limit, and if the weights $p_{\ell_1, \ell_2}$ of cylinders are well-chosen. In the $O(n)$-loop model, each loop receives a weight $n$. After going to the stuffed maps picture, the weights are
\begin{equation}
	p^{O(n)}_{\ell_1, \ell_2} = \binom{\ell_1+\ell_2}{\ell_1}\frac{n}{\ell_1+\ell_2}.
\end{equation}

Bernardi and Bousquet-Mélou have studied in a series of two papers \cite{Bernardi2011, Bernardi2017} some Tutte equation for the $q$-colored Potts model on planar maps, with two catalytic variables. They proved that the generating series is in general differentially algebraic, and for some values of $q$, namely $q = 2 + 2\cos \frac{j\pi}{m}$, it becomes algebraic. Given the similarities between the $O(n)$-loop model and the Potts model for $q=n^2$, the same results may be expected for the $O(n)$-loop model.
%The scaling limits of the $O(n)$-loop model (depending on $n$) are unknown at the time of writing.

%%%%%%%%%%%%%%%%%%%
\section{Blobbed topological recursion for colored graphs}

Here we consider the problem of the topological recursion for colored graphs. Some similarities with stuffed maps will again be apparent. However, this can be quite technical, so we will only sketch some of the important steps.

\subsection{Effective maps with blobs as vertices} Assume that some of the $\bb_i$s are the quartic melonic bubbles $Q(\{1\}), \dotsc, Q(\{d\})$, and use the bijection of \Cref{thm:StuffedMapsBijection} between $\mathcal{G}_{n_1, \dotsc, n_N}(\bb_1, \dotsc, \bb_N)$ and $\mathcal{M}_{n_1, \dotsc, n_N}((\bb_1, \pi_1), \dotsc, (\bb_N, \pi_N))$. In the latter, quartic melonic bubbles become edges with colors $1, \dotsc, d$. Moreover, in $\mathcal{M}^{\max}_{n_1, \dotsc, n_N}((\bb_1, \pi_1), \dotsc, (\bb_N, \pi_N))$, they have to be \emph{bridges} (1-edge-cuts)\footnote{The quartic melonic bubbles satisfy the maximal 2-property, meaning that after the bijection, they become edges whose end-vertices are distinct and are cut-vertices.}. This might give some hope that ``from the point of view'' of those edges, the elements of $\mathcal{M}_{n_1, \dotsc, n_N}((\bb_1, \pi_1), \dotsc, (\bb_N, \pi_N))$ satisfy some interesting properties. Here ``from the point of view'' means that we package all other structures (from edges with more than one color) into ``blobs''. More precisely, we decompose $\m\in \mathcal{M}_{n_1, \dotsc, n_N}((\bb_1, \pi_1), \dotsc, (\bb_N, \pi_N))$ as follows,
\begin{equation}
	\includegraphics[scale=.45,valign=c]{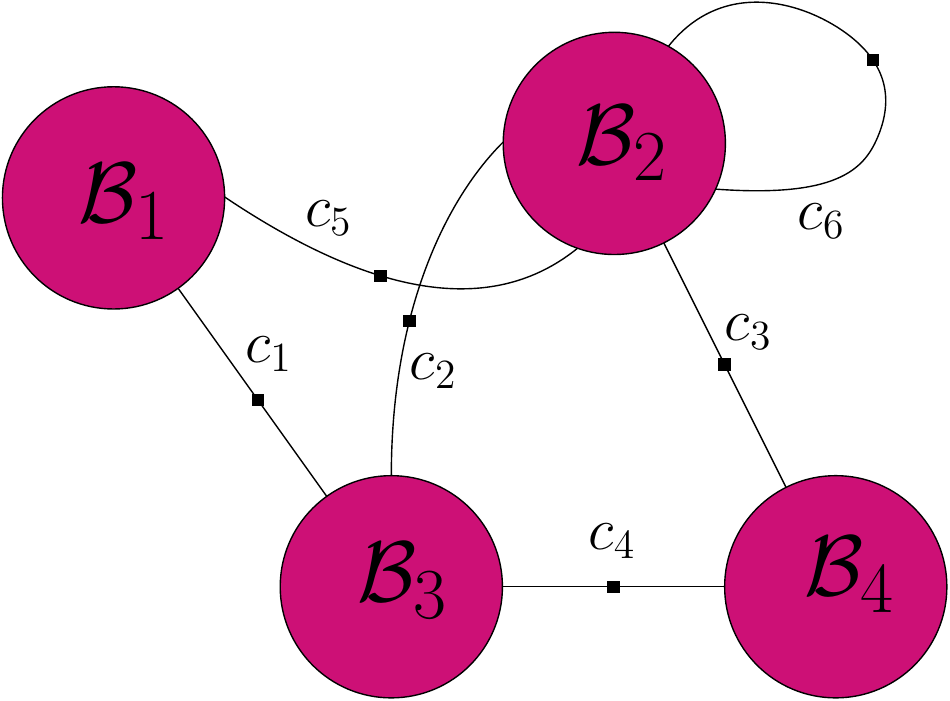}
\end{equation}
where each edge carries a color in $[1..d]$. Recall that the square-vertices represent a gluing of paired bubbles, see Section \ref{sec:Bijection}. Obviously, the blobs are obtained by removing all square-vertices of degree 2 which have edges of the same color on each side. This allows for defining some $\omega_{g,n}$s and we could prove that whatever the weights of those complicated structures, a certain form of the topological recursion is satisfied \cite{BonzomDub}.

Assume that $\mathcal{G}_{n_1, \dotsc, n_N}(\bb_1, \dotsc, \bb_N)$ satisfies the linear growth hypothesis and consider as the generating series (say rooted on an edge of color 1)
\begin{equation} \label{FullGeneratingSeries}
	F(u;p_1, \dotsc, p_N) = \sum_{n_1, \dotsc, n_N\geq 0}\sum_{\graph\in\mathcal{G}_{n_1, \dotsc, n_N}(\bb_1, \dotsc, \bb_N)} u^{C_0(\graph) - \sum_{i=1}^N \alpha(\bb_i) n_i} \prod_{i=1}^{n_i} p_i^{n_i},
\end{equation}
where $C_0(\graph)$ is the total number of bicolored cycles with colors $\{0,c\}$ for $c=1..d$, and $\alpha(\bb_i)$ is the coefficient from the linear growth hypothesis, such that $C_0(\graph) - \sum_{i=1}^N \alpha(\bb_i) n_i\leq d$.

Moving on to the description in terms of blobs, a blob $\mathcal{B}$ is characterized by a $d$-vector of partitions $\vec{\mu}=(\mu^{(1)}, \dotsc, \mu^{(d)})$ and a function $t_\mathcal{B}(u,\vec{\mu},p_1, \dotsc, p_N)$. First, the partition $\mu^{(c)}$ of color $c$ is obtained by following the walks across the blob along the edges which have color $c$. Although we do not explain how it is really done, we claim that they can be thought as describing an elementary cell of a colored stuffed map, i.e. $\mu^{(c)} = (\mu^{(c)}_1, \mu^{(c)}_2, \dotsc)$ means that there is a boundary component of color $c$ and degree $\mu^{(c)}_1$ and so on.

%In particular, while some faces are now internal to the blobs, others go along the edges of color $c$ with square-vertices. Those can be followed at each boundary of degree $\mu^{(c)}_i$ of the blobs, like in usual stuffed maps, except with colors and we denote $F(\m)$ their numbers.

The function $t_\mathcal{B}(u, \vec{\mu}, p_1, \dotsc, p_N)$ is the result of summing over all maps which have $\vec{\mu}$ as boundary. Moreover, one can show that the leading order in $u$ is as follows
\begin{equation}
	t_\mathcal{B}(u, \vec{\mu}, p_1, \dotsc, p_N) = u^{d-\sum_{c=1}^d \ell(\mu^{(c)})} \bigl(\tilde{t}_\mathcal{B}(\vec{\mu}, p_1, \dotsc, p_N)+o(1)\bigr),
\end{equation}
where $\ell(\mu^{(c)})$ is the length of the partition $\mu^{(c)}$. Denote $\mathcal{M}(\{\mathcal{B}\};(\bb_1, \pi_1), \dotsc, (\bb_N, \pi_N))$ the set of maps made of those blobs and colored edges. In particular, while some faces are now internal to the blobs, others go along the edges of color $c$. Those can be followed at each boundary of degree $\mu^{(c)}_i$ of the blobs, like in usual stuffed maps except with colors, and we denote $F(\m)$ their numbers. Then by definition
\begin{equation}
	F(u;p_1, \dotsc, p_N) = \sum_{\m\in\mathcal{M}(\{\mathcal{B}\};(\bb_1, \pi_1), \dotsc, (\bb_N, \pi_N))} u^{F(\m)} \prod_{\mathcal{B}} t_{\mathcal{B}}(u,\vec{\mu}, p_1, \dotsc, p_N).
\end{equation}

However, the TR cannot hold at this stage, because the elements maximizing the number of faces are trees (considering the blobs as vertices), and for the TR to hold, planar maps are expected instead. Let $T_c$ be the series of $\mathcal{M}^{\max}_{n_1, \dotsc, n_N}((\bb_1, \pi_1), \dotsc, (\bb_N, \pi_N))$ with a root of color $c$, i.e. a series of rooted trees. Then define new blobs as follows. For $k\geq 2$
\begin{equation}
	\includegraphics[scale=.55,valign=c]{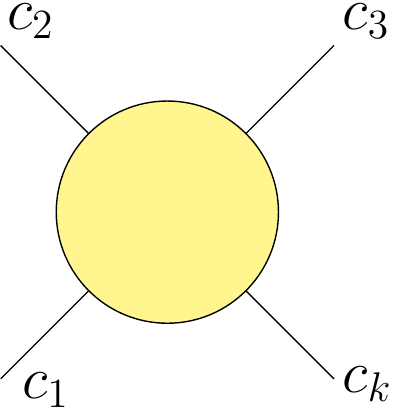} = \sum_{p\geq 0} \includegraphics[scale=.55,valign=c]{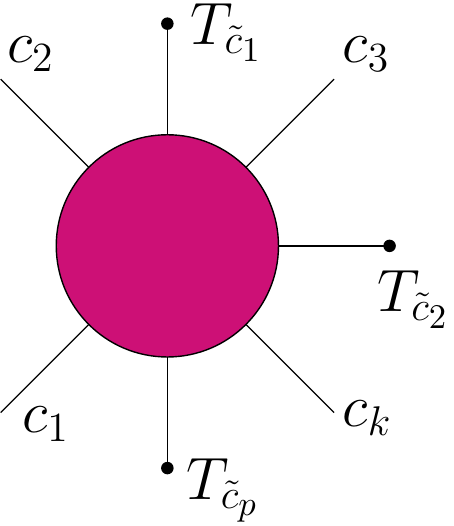}
\end{equation}
where the sum on the RHS is over all blobs with more than (or exactly) $k$ incident edges where the additional edges are turned into leaves weighted by $T$. For $k=1$, set
\begin{equation}
	\includegraphics[scale=.55,valign=c]{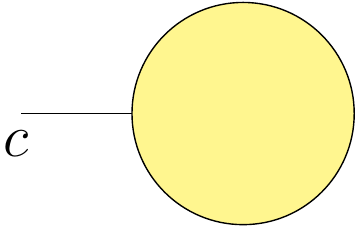} = -T + \sum_{p\geq 0} \includegraphics[scale=.55,valign=c]{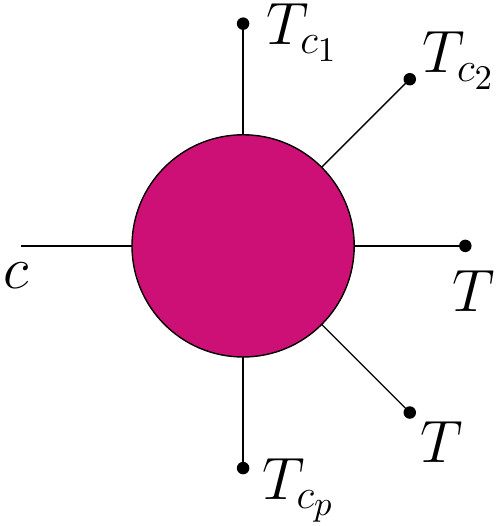} = 0,
\end{equation}
since one recognizes the equation on $T$. While it may seem artificial, this is in fact very natural from the point of view of the associated matrix integrals\footnote{Indeed, a saddle point in the large $N$ limit here selects a special value, which corresponds to the series of trees. It is then classical to substract this leading order to study the fluctuations. It is well-known that the fluctations are Gaussian to first order, as shown explicitly in \cite{NguyenDartoisEynard}. In fact this holds for any (non-pathological) saddle point approximation. Typically, one evaluates an integral like
\begin{equation}
	\int dx e^{-NV(x)}
\end{equation}
with $V(x) = \frac{x^2}{2} - \sum_{k\geq 1} \frac{p_k}{k!} x^k$ by a saddle point, i.e. around the point $\alpha$ where
\begin{equation}
	V'(\alpha) = \alpha - \sum_{k\geq 1} \frac{p_k}{(k-1)!} \alpha^{k-1} = 0.
\end{equation}
Notice that it is a tree equation, which can also be found via the Feynman expansion. To study the subleading orders, one sets $x = \alpha+y$, and
\begin{equation}
	\int dx e^{-NV(x)} = e^{-NV(\alpha)} \int dy e^{-N \sum_{k\geq 2} V^{(k)}(\alpha) \frac{y^k}{k!}}
\end{equation}
Notice that there is no linear term, so no leaf hence no trees anymore in the Feynman expansion. This is basically what we do here, except in the context of a much more complicated multi-matrix model.}. Moreover, the models with pink and yellow blobs are equivalent, since the pink blobs can be expressed in terms of the yellow blobs by using the same formulas as above with $T\mapsto -T$.

\subsection{Blobbed topological recursion} This change of blobs gives a model which can be shown to satisfy the blobbed TR of \cite{BorotShadrin2016}. It is a form of the TR, originally devised in \cite{Borot13} for stuffed maps, where the universal TR formula is supplemented with a holomorphic, model-dependent term. If the (simple) ramification points are $b_1, \dotsc, b_R$, then define the polar and holomorphic part (with respect to $z_1$) as
\begin{equation}
	\begin{aligned}
		P\omega_{g,n}(z_1, \dotsc, z_n) &\coloneqq \sum_{i=1}^R \Res_{z=b_i} \int^z \omega_{0,2}(\cdot, z_1) \omega_{n,g}(z, z_2,\dotsc, z_n),\\ H\omega_{g,n}(z_1, \dotsc, z_n) &\coloneqq \omega_{g,n}(z_1, \dotsc, z_n) - P\omega_{g,n}(z_1, \dotsc, z_n).
	\end{aligned}
\end{equation}
Then the blobbed TR reads
\begin{equation}
	P\omega_{g,n}(z_1,L) = \sum\limits_{i=1}^{R} \Res_{z=b_i}\, K_i(z_1,z)\Big(\omega_{g-1,n+1}(z,\sigma_i(z),L)+\sum\limits_{\substack{h+h'=g \\ C\sqcup C'=L}}^{'}\omega_{h,1+|C|}(z,C)\omega_{h',1+|C'|}(\sigma_i(z),C')\Big)
\end{equation}
with $L=(z_2, \dotsc, z_n)$, while
\begin{equation}
	H\omega_{g,n}(z_1, \dotsc, z_n) = \frac{1}{2i\pi}\oint_{U} \omega_{0,2}(z_1, z) \nu_{n,g}(z,z_2, \dotsc, z_n)
\end{equation}
where $U$ is the unit disc and for some functions $\nu_{n,g}$s which are model-dependent, and here related to the blobs from above. 

\subsection{Spectral curve of the model} In our model, one defines the correlation functions $W_{g,n}(x_1, c_1; \dotsc; x_n,c_n)$ as the maps made of yellow blobs and $n$ additional labeled vertices, such that the first one is incident to edges of color $c_1$ only, the second one is incident to edges of color 2 only, and so on, and with a marked corner on each of those $n$ vertices\footnote{This is the same notion of rooted boundary components as usual but in the dual and with colors.}. Then let $\mathbb{P}^1_{c_i}$ be the ``space of color $c_i$'', and we say that the catalytic variable $x_i$ lives on $\mathbb{P}^1_{c_i}$. Then,
\begin{equation}
	W_{g,n}(x_1, \dotsc, x_n) = \sum_{c_1, \dotsc, c_n\in[1..d]}W_{g,n}(x_1, c_1; \dotsc; x_n,c_n) \mathbb{1}(x_1\in \mathbb{P}^1_{c_1}) \dotsm \mathbb{1}(x_n\in \mathbb{P}^1_{c_n})
\end{equation}
where now the variables $x_i$s live on $\cup_{c=1}^d \mathbb{P}^1_c$.

Recall that a spectral curve is a tuple $\mathcal{S}=(\mathbb{P}^1,V,x(z),y(z),\omega_{02})$ where $V$ is an open subset of $\mathbb{P}^1$, $x(z)$ a covering with simple ramification points, $y(z)$ a meromorphic function and $\omega_{0,2}(z_1, z_2)$ a bidifferential which behaves locally as $\omega_{0,2}(z_1, z_2) \underset{z_1\sim z_2}{=} \frac{dz_1 dz_2}{(z_1-z_2)^2} + \text{holomorphic}$.

Here we need a slightly extended version because of the colors. Denote $a_1, \dotsc, a_d$ the leading order coefficients in $u$ of the blob with only one boundary of color $c$ and degree 2, and $b_{cc'}$ the leading order coefficients in $u$ of the blob with exactly one boundary component of color $c$ and degree 1 and one boundary component of color $c'$ and degree 1. Then denote $A=\operatorname{diag}(-1+a_1,\dotsc, -1+a_d)$ and $B=(b_{cc'})_{c,c'=1..d}$.
\begin{theorem}{}{}\cite{BonzomDub}
	The spectral curve is defined on $\cup_{c=1}^d (\mathbb{P}^1_c\setminus \{0\})$ by
	\begin{equation}
		\begin{aligned}
			x(z) &= \sum_{c=1}^d \mathbb{1}(x\in\mathbb{P}^1_c) \mathbb{1}(z\in\mathbb{P}^1_c) \frac{1}{\sqrt{1-a_c}} \bigl(z+z^{-1}\bigr)\\
			y(z) &= \sum_{c=1}^d \mathbb{1}(y\in\mathbb{P}^1_c) \mathbb{1}(z\in\mathbb{P}^1_c) \frac{\sqrt{1-a_c}}{z}
		\end{aligned}
	\end{equation}
and
\begin{multline}
	\omega_{0,2}(z_1,z_2) = \frac{dz_1 dz_2}{(z_1-z_2)^2}\sum_{c=1}^d \mathbb{1}(z_1\in\mathbb{P}^1_c) \mathbb{1}(z_2\in\mathbb{P}^1_c) \\- \frac{dz_1 dz_2}{z_1^2 z_2^2} \sum_{c_1,c_2=1}^d \frac{\mathbb{1}(z_1\in\mathbb{P}^1_{c_1}) \mathbb{1}(z_2\in\mathbb{P}^1_{c_2})}{\sqrt{(1-a_{c_1})(1-a_{c_2})}} \frac{1}{1-a_{c_1}} \Bigl(B\frac{1}{A+B}\Bigr)_{c_1 c_2}
\end{multline}
\end{theorem}

Notice that at fixed color $c$, one has $y_c^2 + (a_c-1) x_c y_c + 1-a_c=0$, which experienced readers may recognize as the spectral curve for maps with no internal faces\footnote{So the spectral curve of the GUE.}.

\subsection{Discussion on the poles of $\omega_{0,2}$} The primary example of the blobbed TR is stuffed maps, where an elementary cell of genus $h$ with boundary perimeters $\ell_1, \dotsc, \ell_n$ is counted with a formal weight $p^{(h)}_{\ell_1, \dotsc, \ell_n}$, as explained in \cite{Borot13}. It can be shown explicitly that restricting to $h=0, n=2$, i.e. cylinders, give a $\omega_{0,2}(z_1,z_2) = \frac{dz_1 dz_2}{(z_1-z_2)^2} + \text{holomorphic}$, where the part which is holomorphic for $z_1\sim z_2$ has poles at $z_i=0$. As noticed in \cite{Borot13}, the $\omega_{g,n}$s are initially defined outside the unit circle and can be analytically continued in a neighborhood of it. Here it is clear that they cannot be continued all the way to 0.

On the other hand, the $O(n)$-loop model can also be seen as stuffed maps with cylinders, where the weights are now not formal but given by $p^{O(n)}_{\ell_1, \ell_2} = \binom{\ell_1+\ell_2}{\ell_1}\frac{n}{\ell_1+\ell_2}$. The reason for mentioning this model is that the blobbed TR thus applies to it and gives rise to an $\omega_{0,2}$ with poles at 0. However, the $O(n)$-loop model has been intensely studied and was shown by G. Borot and B.~Eynard to satisfy the TR for a spectral curve of genus 1 \cite{BorotEynard2011}. 

How to reconcile those two pictures?\footnote{We thank G.~Borot, E.~Garcia-Failde and V.~Nador for discussions on this topic, in particular G.~Borot for this explanation as local vs global.} The explanation may be that with formal weights, one cannot achieve a globally-defined spectral curve. It is instead defined locally around ramification points and while trying to continue it, one may run into poles at 0. However, in the $O(n)$-loop model, the weights are not formal and the spectral curve can be ``resummed'' in a sense to produce a curve of non-zero genus.

\bibliographystyle{alpha}
\bibliography{biblio}

\end{document}